\documentclass[12pt]{amsart}
\usepackage{etex}
\usepackage{etoolbox}
\patchcmd{\thebibliography}{*}{}{}{}
\usepackage{amssymb}
\usepackage{cite}
\usepackage{booktabs}
\usepackage{url}
\usepackage{hyphenat}
\usepackage{mathtools}
\usepackage{pifont}
\usepackage[all,cmtip]{xy}
\usepackage{ifpdf}
\usepackage{enumitem}

\usepackage{xcolor}
\usepackage{graphicx}
\usepackage[lmargin=1in,rmargin=1in,tmargin=1in,bmargin=1in]{geometry}
\usepackage{subfig}
\usepackage{mathrsfs}
\usepackage{overpic}
\usepackage{xr-hyper}

\definecolor{darkblue}{rgb}{0,0,0.4} 
\usepackage[colorlinks=true, citecolor=darkblue, filecolor=darkblue, linkcolor=darkblue,urlcolor=darkblue]{hyperref} 
\usepackage[all]{hypcap}

\usepackage{tikz}
\usetikzlibrary{matrix,arrows,3d}
\tikzset{zxplane/.style={canvas is zx plane at y=#1,very thin}}
\tikzset{xyplane/.style={canvas is xy plane at z=#1,very thin}}
\tikzset{yzplane/.style={canvas is yz plane at x=#1,very thin}}

\makeatletter
\providecommand\@dotsep{5}
\def\listtodoname{List of Todos}
\def\listoftodos{\@starttoc{tdo}\listtodoname}
\makeatother
\graphicspath{{draws/}{}}

\def\mathcenter#1{%
	\vcenter{\hbox{$#1$}}%
}

\newcommand{\RR}{\mathbb R}
\newcommand{\CC}{\mathbb C}
\newcommand{\ZZ}{\mathbb Z}
\newcommand{\QQ}{\mathbb Q}

\newcommand{\FF}{\mathbb F}
\newcommand{\NN}{\mathbb N}

\newcommand{\Image}{\mathrm{Im}}
\newcommand{\bD}{\mathbb{D}}

\newcommand{\co}{\nobreak\mskip2mu\mathpunct{}\nonscript
  \mkern-\thinmuskip{:}\penalty300\mskip6muplus1mu\relax}
\newcommand{\from}{\co}

\newcommand{\bdy}{\partial}
\newcommand{\into}{\hookrightarrow}

\newcommand{\lbracket}{[}
\newcommand{\rbracket}{]}

\newcommand{\spinc}{\mathfrak s}
\newcommand{\spinct}{\mathfrak t}
\newcommand{\tors}{\mathit{tors}}

\DeclareMathOperator{\Spinc}{\mathrm{Spin}^c}

\DeclareMathOperator{\Sym}{Sym}

\DeclareMathOperator{\Hom}{Hom}
\DeclareMathOperator{\Map}{Map}

\DeclareMathOperator{\RHom}{RHom}
\DeclareMathOperator{\Ext}{Ext}
\DeclareMathOperator{\Tor}{Tor}
\DeclareMathOperator{\Ab}{Ab}

\DeclareMathOperator{\rank}{rank}

\DeclareMathOperator{\spin}{spin}
\newcommand{\SpinC}{\spin^c}

\DeclareMathOperator{\gr}{gr}

\renewcommand{\emptyset}{\varnothing}

\DeclareMathOperator{\SO}{\mathit{SO}}
\DeclareMathOperator{\BSO}{\mathit{BSO}}
\DeclareMathOperator{\GL}{\mathit{GL}}

\theoremstyle{plain}

\numberwithin{equation}{section}
\newtheorem{theorem}[equation]{Theorem}

\newtheorem{proposition}[equation]{Proposition}
\newtheorem{lemma}[equation]{Lemma}
\newtheorem{corollary}[equation]{Corollary}

\newtheorem{conjecture}[equation]{Conjecture}
\newtheorem{observation}[equation]{Observation}
\newtheorem{convention}[equation]{Convention}
\newtheorem{definition}[equation]{Definition}
\newtheorem{construction}[equation]{Construction}
\newtheorem{hypothesis}[equation]{Hypothesis}

\theoremstyle{definition}

\newtheorem{question}[equation]{Question}

\theoremstyle{remark}
\newtheorem{example}[equation]{Example}
\newtheorem{remark}[equation]{Remark}

\newcommand{\Kh}{\mathit{Kh}}
\newcommand{\KhCx}{C_{\mathit{Kh}}}
\newcommand{\BNCx}{C_{\mathit{BN}}}
\newcommand{\SzCx}{C_{\mathit{Sz}}}
\newcommand{\totCx}{C_{\mathit{tot}}}

\newcommand{\HF}{\mathit{HF}}
\newcommand{\red}{\mathit{red}}

\newcommand{\HFa}{\widehat{\mathit{HF}}}

\newcommand{\HFm}{{\HF}^-}

\newcommand{\CF}{{\mathit{CF}}}
\newcommand{\CFa}{\widehat{\mathit{CF}}}
\newcommand{\CFm}{\mathit{CF}^-}
\newcommand{\HFKa}{\widehat{\mathit{HFK}}}
\newcommand{\HFLa}{\widehat{\mathit{HFL}}}
\newcommand{\CFLa}{\widehat{\mathit{CFL}}}
\newcommand{\HFLt}{\widetilde{\mathit{HFL}}}
\newcommand{\HFKt}{\widetilde{\mathit{HFK}}}
\newcommand{\x}{\mathbf x}
\newcommand{\y}{\mathbf y}

\newcommand\HH{\mathit{HH}}
\newcommand\HC{\mathit{HC}}
\newcommand\Hochschild\HH

\newcommand{\Ainf}{A_\infty}

\newcommand{\Alg}{\mathcal{A}}

\newcommand{\zs}{\mathbf{z}}
\newcommand{\ws}{\mathbf{w}}
\newcommand{\alphas}{{\boldsymbol{\alpha}}}
\newcommand{\betas}{{\boldsymbol{\beta}}}
\newcommand{\gammas}{{\boldsymbol{\gamma}}}

\newcommand{\cM}{\mathcal{M}}

\newcommand{\CFDA}{\mathit{CFDA}}
\newcommand{\CFDAa}{\widehat{\CFDA}}

\newcommand{\CFK}{\mathit{CFK}}
\newcommand{\CFKa}{\widehat{\CFK}}

\newcommand{\s}{\mathfrak{s}}

\newcommand{\dg}{\textit{dg} }

\newcommand\Id{\mathbb{I}}

\newcommand\DTP{\mathbin{\otimes^L}}

\newcommand{\Field}{{\FF_2}}

\newcommand{\dbar}{\bar{\partial}}
\newcommand{\Heegaard}{\mathcal{H}}
\newcommand{\HD}{\Heegaard}

\renewcommand{\th}{^\text{th}}

\DeclareMathOperator{\Fix}{Fix}

\newcommand{\Cat}{\mathscr{C}}
\newcommand{\Dat}{\mathscr{D}}

\DeclareMathOperator{\ob}{Ob}

\newcommand{\Denis}{\mathsf{AZ}}

\newcommand{\MirrorDenis}{\overline{\mathsf{AZ}}}

\newcommand{\ol}[1]{\overline{#1}{}}
\newcommand{\wt}[1]{\widetilde{#1}{}}

\makeatletter
\newcommand\honestalg[3]{\bigl\lbracket
\begin{smallmatrix} #1\@ifempty{#3}{}{&#3} \\ #2 \end{smallmatrix}
\bigr\rbracket}

\makeatother
\newcommand{\lab}[1]{$\scriptstyle #1$}

\newcommand{\lsup}[2]{{}^{#1}\mskip-.6\thinmuskip#2}

\newcommand{\Filt}{\mathcal{F}}

\newcommand{\pt}{\mathit{pt}}

\newcommand{\Crit}{\mathit{Crit}}

\newcommand{\JSpace}{\mathcal{J}}
\newcommand{\HSpace}{\mathcal{H}}
\newcommand{\cyl}{\mathrm{cyl}}
\newcommand{\ECat}{\mathscr{E}}
\newcommand{\ICat}{\mathscr{I}}
\newcommand{\fix}{\mathit{fix}}
\newcommand{\Complexes}{\mathsf{Kom}}
\DeclareMathOperator{\hocolim}{hocolim}
\newcommand{\KhSymp}{\mathit{Kh}_{\mathit{symp}}}
\newcommand{\rKhSymp}{\widetilde{\mathit{Kh}}_{\mathit{symp}}}
\newcommand{\KCSymp}{\mathcal{C}_{\mathit{Kh},\mathit{symp}}}
\newcommand{\eKCSymp}{\mathcal{C}_{\mathit{Kh}}^{\mathit{symp,free}}}
\newcommand{\HFaNab}{\HFa_{\!\!\nabla}^*}
\DeclareMathOperator{\Hilb}{Hilb}
\newcommand{\CipLag}{\mathcal{K}}
\newcommand{\ssspace}[1]{\mathcal{Y}_{#1}}
\newcommand{\rssspace}[1]{\widetilde{\mathcal{Y}}_{#1}}
\newcommand{\rCipLag}{\widetilde{\CipLag}}

\newcommand{\onto}{\twoheadrightarrow}
\newcommand{\ul}{\underline}

\newcommand{\eCF}[1][{\ZZ/2}]{\CF_{\!#1}}
\newcommand{\eH}[1][{\ZZ/2}]{H_{#1}}
\newcommand{\eHF}[1][{\ZZ/2}]{\HF_{\!#1}}
\newcommand{\eHFa}[1][{\ZZ/2}]{\HFa_{\!#1}}
\newcommand{\eCFa}[1][{\ZZ/2}]{\CFa_{\!#1}}
\newcommand{\ECF}[1][{}]{\widetilde{\CF}^{\raisebox{-4pt}{$\scriptstyle #1$}}}
\newcommand{\ECFm}{\widetilde{\CF}^{\raisebox{-3pt}{$\scriptstyle -$}}}
\newcommand{\HFaDual}{\HFa^{\raisebox{-3pt}{$\scriptstyle *$}}}
\newcommand{\CFaDual}{\CFa^{\raisebox{-3pt}{$\scriptstyle *$}}}
\newcommand{\HFKaDual}{\HFKa^{\raisebox{-3pt}{$\scriptstyle *$}}}
\newcommand{\HFKtDual}{\HFKt^{\raisebox{-3pt}{$\scriptstyle *$}}}
\newcommand{\HFLtDual}{\HFLt^{\raisebox{-3pt}{$\scriptstyle *$}}}
\newcommand{\HFLaDual}{\HFLa^{\raisebox{-3pt}{$\scriptstyle *$}}}

\makeatletter
\newcommand*\wtt[1]{\mathpalette\wthelper{#1}}
\newcommand*\wthelper[2]{%
        \hbox{\dimen@\accentfontxheight#1%
                \accentfontxheight#11.3\dimen@
                $\m@th#1\widetilde{#2}$%
                \accentfontxheight#1\dimen@
        }%
}
\newcommand*\accentfontxheight[1]{%
        \fontdimen5\ifx#1\displaystyle
                \textfont
        \else\ifx#1\textstyle
                \textfont
        \else\ifx#1\scriptstyle
                \scriptfont
        \else
                \scriptscriptfont
        \fi\fi\fi3
}
\makeatother
\newcommand{\ewtHF}[1][{\ZZ/2}]{\wtt{\widetilde{\HF}}_{#1}}
\newcommand{\ewtCF}[1][{\ZZ/2}]{\wtt{\widetilde{\CF}}_{#1}}

\newcommand{\free}{\mathrm{free}}

\newcommand{\HomO}[1]{\Hom_{#1}}
\newcommand{\ExtO}[1]{\Ext_{#1}}

\newcommand{\RHomO}[1]{\RHom_{#1}}

\newcommand{\ttop}{\Theta_{\mathit{top}}}
\newcommand{\tbot}{\Theta_{\mathit{bot}}}

\begin{document}
\title[Flexible construction of equivariant Floer homology]{A flexible construction of equivariant Floer homology and applications}

\author{Kristen Hendricks}
 \address{Mathematics Department, University of California\\
   Los Angeles, CA 90095}
   \thanks{KH was supported by NSF grant DMS-1506358.}
\email{\href{mailto:hendricks@math.ucla.edu}{hendricks@math.ucla.edu}}

\author{Robert Lipshitz}
 \address{Department of Mathematics, University of Oregon\\
   Eugene, OR 97403}
\thanks{RL was supported by NSF grant DMS-1149800.}
\email{\href{mailto:lipshitz@uoregon.edu}{lipshitz@uoregon.edu}}

\author{Sucharit Sarkar}
\thanks{SS was supported by NSF grant DMS-1350037}
\address{Department of Mathematics, University of California\\
  Los Angeles, CA 90095}
\email{\href{mailto:sucharit@math.ucla.edu}{sucharit@math.ucla.edu}}

\date{\today}

\begin{abstract}
	Seidel-Smith and Hendricks used equivariant Floer cohomology to define some 
	spectral sequences from symplectic Khovanov homology and Heegaard Floer 
	homology. These spectral sequences give rise to Smith-type inequalities. 
	Similar-looking spectral sequences have been defined by Lee, Bar-Natan, 
	Ozsv\'ath-Szab\'o, Lipshitz-Treumann, Szab\'o, Sarkar-Seed-Szab\'o, and 
	others. In 
	this paper we give another construction of equivariant Floer cohomology 
	with 
	respect to a finite group action and use it to prove some invariance properties of 
	these spectral sequences; prove that some of these spectral sequences agree; 
	improve Hendricks's Smith-type inequalities; give some theoretical and 
	practical computability results for these spectral sequences; define some 
	new spectral sequences conjecturally related to Sarkar-Seed-Szab\'o's; and 
	introduce a new concordance homomorphism and concordance invariants. We also digress to prove 
	invariance of Manolescu's reduced symplectic Khovanov homology.
\end{abstract}

\maketitle

\tableofcontents


\section{Introduction}\label{sec:intro}

Given a space $X$ with a $\ZZ/2$-action $\tau\co X\to X$, one can form
the Borel equivariant cohomology
$\eH^*(X)=H^*(X\times_{\ZZ/2}E\ZZ/2)$. (Here and throughout, we always
work over the field $\Field$ with two elements.) The Serre spectral
sequence for the fibration $X\to X\times_{\ZZ/2}E\ZZ/2\to B\ZZ/2$ and
the fact that the cellular chain complex of $B\ZZ/2=\RR P^\infty$
satisfies $C^*(B\ZZ/2;\Field)\cong\Field[\theta]$ combine to give a
spectral sequence $H^*(X)\otimes \Field[\theta]\Rightarrow \eH^*(X)$
which is, transparently, an invariant of the pair $(X,\tau)$. The
$E_2$-page is the group cohomology of $\ZZ/2$ with coefficients in
$H^*(X)$, $H^*(\ZZ/2, H^*(X))$.  If $X$ is a finite CW complex then
Smith theory relates $\eH^*(X)$ with the homology of the fixed set
$X^\fix$ of $\tau$.

Given a chain complex $C_*$ with a $\ZZ/2$ action, there is a purely
algebraic construction of the equivariant cohomology $\eH^*(C_*)$,
which is a module over $\Field[\theta]$, and of the spectral sequence
from $H^*(C_*)\otimes \Field[\theta]$ to $\eH^*(C_*)$, with $E_2$-page
equal to $H^*(\ZZ/2,H^*(C_*))$; see Section~\ref{sec:background} for a
brief review of this homological algebra. Further, up to isomorphism,
$\eH^*(C_*)$, as an $\Field[\theta]$-module, and every page of the
spectral sequence $H^*(C_*)\otimes \Field[\theta]\Rightarrow
\eH^*(C_*)$ depends only on the quasi-isomorphism type of $C_*$ as a
chain complex over the group ring $\Field[\ZZ/2]$.

Seidel-Smith gave an analogue of this story in Lagrangian intersection
Floer homology. In particular, given a symplectic manifold
$(M,\omega)$, a pair of Lagrangians $L_0,L_1\subset M$, and a
symplectic involution $\tau\co M\to M$ satisfying certain geometric
assumptions, they construct a Floer complex $\CF(L_0,L_1)$ with an
action of $\ZZ/2$, and from it an equivariant Floer cohomology
$\eHF(L_0,L_1)$ and a spectral sequence
$\HF^*(L_0,L_1)\otimes\Field[\theta]\Rightarrow \eHF(L_0,L_1)$. Again,
the $E_2$-page is the group cohomology $H^*(\ZZ/2,\HF(L_0,L_1))$ of
$\ZZ/2$ with coefficients in $\HF^*(L_0,L_1)$.\footnote{We will use $\HF$ to denote Floer homology, $\HF^*$ to denote Floer cohomology, and $\eHF$ to denote equivariant Floer cohomology.} 
More remarkably, they
relate $\eHF(L_0,L_1)$ to the Floer cohomology of the fixed parts of
$L_0$ and $L_1$, in the fixed set of $M$.

Seidel-Smith's construction starts by choosing a large Hamiltonian
isotopy of $L_0$ and $L_1$ with certain properties. In particular, it
is not immediately obvious that $\eHF(L_0,L_1)$ is an invariant
of the equivariant isotopy class of $L_0$ and $L_1$. Proving
invariance is the first goal of this paper:
\begin{theorem}\label{thm:SS-invt}
  The quasi-isomorphism type of the chain complex $\CF(L_0,L_1)$ over $\Field[\ZZ/2]$
  is an invariant of $(L_0,L_1)$ up to
  Hamiltonian isotopy through $\ZZ/2$-equivariant Lagrangians.  In
  particular, the equivariant Floer cohomology $\eHF(L_0,L_1)$
  and the spectral sequence
  $\HF^*(L_0,L_1)\otimes\Field[\theta]\Rightarrow \eHF(L_0,L_1)$
  are invariants of the equivariant Hamiltonian isotopy type of
  $(L_0,L_1)$.
\end{theorem}
In fact, the equivariant Floer cohomology is often invariant under
even non-equivariant isotopies; see
Proposition~\ref{prop:non-equi-invar}.

The techniques involved in the proof of Theorem~\ref{thm:SS-invt} are
suggested by Seidel-Smith's discussion of Morse theory~\cite[Section
2(a)]{SeidelSmith10:localization}, and the result would presumably
have appeared in Seidel-Smith's paper had they a use for it. (Indeed,
some version seems implicit in the discussion around~\cite[Formula
(97)]{SeidelSmith10:localization}.) Seidel-Smith's definition of
$\eHF(L_0,L_1)$ requires equivariant transversality, which is
guaranteed only after the large Hamiltonian isotopy. The strategy in
proving Theorem~\ref{thm:SS-invt} is to define $\eHF(L_0,L_1)$ without
using equivariant complex structures, making transversality easier to
achieve. This construction has a simple extension to any finite group,
though we continue to work with coefficients in $\Field$ to avoid
discussing orientations of moduli spaces; see
Section~\ref{sec:2-groups}.  (Since we are working with coefficients
in $\Field$, the construction depends only on the $2$-subgroups of the
finite group, and is therefore not interesting for odd-order groups.)

Seidel-Smith apply their construction to their symplectic Khovanov
homology, to obtain a spectral sequence 
\begin{equation}\label{eq:SS-to-nabla}
\KhSymp(K)\otimes\Field[\theta,\theta^{-1}]\Rightarrow
\HFaNab(\Sigma(K))\otimes\Field[\theta,\theta^{-1}]
\end{equation}
from Seidel-Smith's symplectic Khovanov homology of
$K$~\cite{SeidelSmith6:Kh-symp} to a variant of the Heegaard Floer
cohomology of the branched double cover of the mirror of $K$.
(Seidel-Smith denote $\HFaNab$ by $\Kh_{\mathit{symp,inv}}(K)$; we
choose the notation $\HFaNab(\Sigma(K))$ because Seidel-Smith
conjecture that the invariant is isomorphic to $\HFaDual(\Sigma(K)\#
(S^2\times S^1))$.)  Given a 2-periodic knot $\wt{K}$ with quotient
$K$, they also obtain a spectral sequence
\begin{equation}\label{eq:SS-periodic}
\KhSymp(\wt{K})\otimes\Field[\theta,\theta^{-1}]\Rightarrow\KhSymp(K)\otimes\Field[\theta,\theta^{-1}].
\end{equation}
As a corollary of Theorem~\ref{thm:SS-invt} we obtain:
\begin{corollary}\label{cor:SSss-invt}
  Every page of Seidel-Smith's spectral sequence~\eqref{eq:SS-to-nabla}
  is an invariant of $K$.
\end{corollary}
A version of Corollary~\ref{cor:SSss-invt} was observed by
Seidel-Smith~\cite[p.~1496]{SeidelSmith10:localization}, though they
do not spell out the details of the argument. Note that we are not
asserting that the isomorphism between the $E_\infty$-page of the
sequence~\eqref{eq:SS-to-nabla} and
$\HFaNab\otimes\Field[\theta,\theta^{-1}]$ is itself a knot invariant;
see also Seidel-Smith's discussion~\cite[Section
4.4]{SeidelSmith10:localization}.
The fact that $\HFaNab(\Sigma(K))$ is an invariant of $K$ was also
proved by Tweedy~\cite{Tweedy14:anti-diag}; a sketch of a proof was also given by 
Seidel-Smith~\cite[Section 4.4]{SeidelSmith10:localization}.
(In fact, there is a spectral sequence
$\HFaNab(\Sigma(K))\Rightarrow \HFa^*(\Sigma(K)\#(S^2\times S^1))$~\cite[Section 4d]{SeidelSmith6:Kh-symp}, and
Tweedy showed that this spectral sequence is a knot invariant.)

\begin{corollary}\label{cor:SS-periodic-ss-invt}
  Every page of Seidel-Smith's spectral sequence~\eqref{eq:SS-periodic} is an
  invariant of the knot $\wt{K}$ together with its free period, or
  equivalently of the pair $(K,A)$ of the knot $K$ and the axis $A$ of
  the period.
\end{corollary}

Using Seidel-Smith's localization theorem, Hendricks constructed four
spectral sequences in Heegaard Floer cohomology:
\begin{itemize}
    \item Given a knot $K$ in $S^3$ and a sufficiently large integer $n$, a spectral sequence 
    \begin{equation}\label{eq:Sig-to-K}
    \HFKaDual(\Sigma(K),\widetilde{K})\otimes V^{\otimes n}\otimes \Field[\theta,\theta^{-1}]\Rightarrow \HFKaDual(S^3,K)\otimes V^{\otimes n}\otimes \Field[\theta,\theta^{-1}]
    \end{equation}
    from the knot Floer cohomology of the preimage $\widetilde{K}$ of $K$ in its
    branched double cover to the knot Floer cohomology of $K$ in $S^3$ and a spectral sequence
	\begin{equation}\label{eq:Sig-to-S3}    
    \HFaDual(\Sigma(K))\otimes V^{\otimes n} \otimes \Field[\theta,\theta^{-1}] \Rightarrow V^{\otimes n} \otimes \Field[\theta,\theta^{-1}].
    \end{equation}
  \item Given a 2-periodic knot $\wt{K}$ in $S^3$ with axis $\wt{A}$ whose quotient knot $K$ has axis $A$, and a
    sufficiently large integer $n$, a spectral sequence
    \begin{equation}\label{eq:HFL-periodic}
    \HFLaDual(\wt{K}\cup \wt{A})\otimes V_1^{\otimes (2n-1)}\otimes \Field[\theta,\theta^{-1}]\Rightarrow\HFLaDual(K\cup A)\otimes V_1^{\otimes (n-1)}\otimes \Field[\theta,\theta^{-1}]
    \end{equation}
    from the link Floer cohomology of $\wt{K}\cup A$ to the link Floer
    cohomology of $K\cup A$, and a spectral sequence 
     \begin{equation}\label{eq:HFK-periodic}
    \HFKaDual(\wt{K})\otimes V^{\otimes (2n-1)}\otimes W \otimes \Field[\theta,\theta^{-1}]\Rightarrow\HFKaDual(K)\otimes V^{\otimes (n-1)}\otimes W \otimes \Field[\theta,\theta^{-1}].
    \end{equation}
\end{itemize}
(Here, $V$, $V_1$, and $W$ are 2-dimensional graded vector spaces. In all cases there are similar, albeit more complex, statements for links.)
Theorem~\ref{thm:SS-invt} also implies:
\begin{corollary}\label{cor:Hen-dcov-invt}
  Every page of the spectral sequence~\eqref{eq:Sig-to-K} is an invariant of the pair $(K,n)$.
\end{corollary}

\begin{corollary}\label{cor:Hen-periodic-invt}
  Every page of the spectral sequences~\eqref{eq:HFL-periodic} and~\eqref{eq:HFK-periodic} is  an invariant of the
  integer $n$, the knot $\wt{K}$, and the 2-period $\tau\co
  (S^3,\wt{K})\to (S^3,\wt{K})$ or, equivalently, of the triple $(K,A,n)$.
\end{corollary}

Typically, direct computations of Floer cohomology take advantage of
symmetries of the Lagrangians. A side effect of Seidel-Smith's
construction of $\eHF(L_0,L_1)$ is that one tends to lose such
symmetries, making computations challenging. The techniques used to
prove Theorem~\ref{thm:SS-invt} mean that often one can compute the
equivariant Floer cohomology without perturbing the Lagrangians. A
precise statement is given in Proposition~\ref{prop:equi-is-equi};
here are some consequences:
\begin{corollary}\label{cor:Hen-from-diag}
  Let $K$ be a knot in $S^3$ and let
  $\HD=(S^2,\alphas,\betas,\mathbf{z},\mathbf{w})$ be a genus-0
  Heegaard diagram for $K$ with $|\alphas|=n$. Let
  $\wt{\HD}=(\Sigma,\wt{\alphas},\wt{\betas},\wt{\mathbf{z}},\wt{\mathbf{w}})$
  be the double cover of $\HD$ branched along
  $\mathbf{z}\cup\mathbf{w}$ so that $\wt{\HD}$ represents
  $(\Sigma(K),K)$. Then for a generic one-parameter family $\wt{J}$ of
  $\ZZ/2$-invariant almost complex structures on $\Sym^{2n}(\Sigma)$,
  the spectral sequence~\eqref{eq:Sig-to-K} is induced by the double
  complex
 \[
  0\longrightarrow (\CF(T_{\wt{\alpha}},T_{\wt{\beta}}),\bdy_{\wt{J}})
  \stackrel{\Id+\tau^\#}{\longrightarrow} (\CF(T_{\wt{\alpha}},T_{\wt{\beta}}),\bdy_{\wt{J}}) 
  \stackrel{\Id+\tau^\#}{\longrightarrow} \cdots
 \]
 
 An analogous statement holds for the spectral
 sequence~\eqref{eq:HFL-periodic} and any genus-0 Heegaard diagram
 $\HD$ for $K\cup A$ and the induced diagram for $\wt{K}\cup \wt{A}$.
\end{corollary}

\begin{corollary}\label{cor:Hen-nice}
  For any knot $K$ in $S^3$, the spectral
  sequences~\eqref{eq:Sig-to-K},~\eqref{eq:HFL-periodic}, and~\eqref{eq:HFK-periodic} can be
  computed combinatorially from a genus-0 nice Heegaard diagram.
\end{corollary}
For the spectral sequences~\eqref{eq:Sig-to-K}
and~\eqref{eq:HFL-periodic}, any genus-$0$ nice diagram, such as
Beliakova's planar grid diagrams~\cite{Beliakova10:grid}, suffices. For the
spectral sequence~\eqref{eq:HFK-periodic}, we need a specially adapted
diagram, like those constructed in Section~\ref{sec:nice-diags}.

Using the fact that we can compute the spectral sequences from an
equivariant Heegaard diagram, we can also understand the behavior
under changing $n$:
\begin{theorem}\label{thm:Hen1-small}
  For any knot $K$ in $S^3$ there is a spectral sequence
  $\HFKaDual(\Sigma(K),\wt{K})\otimes\Field[\theta,\theta^{-1}]\Rightarrow \HFKaDual(S^3,K)\otimes\Field[\theta,\theta^{-1}]$ so that the spectral
  sequence~\eqref{eq:Sig-to-K} is obtained by tensoring with
  $V^{\otimes n}$.
\end{theorem}

A consequence of Hendricks' spectral sequence (\ref{eq:Sig-to-K}) was a rank inequality
\[
\dim(\HFKa(\Sigma(K),\wt{K}, \spinc_0) \geq \dim(\HFKa(S^3,K)),
\]
where $\spinc_0$ is the unique spin structure on $\Sigma(K)$, and the same rank inequality for the restriction of the theory to the extremal Alexander gradings $\pm g(K)$, to wit
\[
\dim(\HFKa(\Sigma(K), \wt{K}, \spinc_0, \pm g(K)) \geq \dim(\HFKa(S^3,K, \pm g(K)).
\]
Theorem \ref{thm:Hen1-small} allows us to extract this inequality for all Alexander gradings.

\begin{corollary} \label{cor:rank-inequality-refined} For any knot $K$
  in $S^3$, let $\spinc_0$ be the unique spin structure on $\Sigma(K)$
  and let $i$ be any Alexander grading. There is a rank inequality
  \[
  \dim(\HFKa(\Sigma(K), \wt{K}, \spinc_0, i)) \geq \dim(\HFKa(S^3,K,i)).
  \]
\end{corollary}

To state the analogue of Theorem \ref{thm:Hen1-small} for the spectral
sequences~\eqref{eq:HFL-periodic} and~\eqref{eq:HFK-periodic} we need
a little more notation.  Let $X=\Field\langle xx, xy, yx, yy\rangle$
with differential $d_1(xx)=d_1(yy)=0$ and $d_1(xy)=d_1(yx)=xy+yx$. Let
$Y$ be the homology of $X$.
\begin{theorem}\label{thm:Hen2-small}
  For any $2$-periodic link $\wt{K}$ in $S^3$ with axis $\wt{A}$ and
  quotient link $K$ there is a spectral sequence 
  \begin{equation}\label{eq:Hen2-small-seq}
    \HFLaDual(\wt{K}\cup \wt{A})\otimes\Field[\theta,\theta^{-1}]\otimes V_1\Rightarrow\HFLaDual(K\cup A)\otimes\Field[\theta,\theta^{-1}]
  \end{equation}
  so that the spectral sequence~\eqref{eq:HFL-periodic} is obtained by
  tensoring with $X^{\otimes n-1}$. That is, let $E_i(K)$ be the
  spectral sequence~(\ref{eq:HFL-periodic}) and let $\wt{E}_i(K)$ be
  the spectral sequence~\eqref{eq:Hen2-small-seq}. Then:
  \begin{itemize}
  \item $E_1(K)\cong \wt{E}_1(K)\otimes X^{\otimes n-1}$, where the
    differential on the right hand side is the tensor product
    differential.
  \item For $i>1$, $E_i(K)\cong \wt{E}_i(K)\otimes Y^{\otimes n-1}$,
    where the differential on the right hand side is induced by the
    differential on $\wt{E}_i(K)$.
  \end{itemize}

  Similarly, there is a spectral sequence 
  \begin{equation}\label{eq:Hen2-oth-seq}
    \HFKaDual(\wt{K})\otimes V\otimes W\otimes\Field[\theta,\theta^{-1}]\Rightarrow\HFKaDual(K)\otimes W\otimes\Field[\theta,\theta^{-1}]
  \end{equation}
  so that the spectral sequence~(\ref{eq:HFK-periodic}) is obtained by
  tensoring with $X^{\otimes n-1}$.
\end{theorem}

(There is a version of Theorem~\ref{thm:Hen2-small} for links, but the
notation is more cumbersome, so we omit it.) Using the spectral
sequence (\ref{eq:Hen2-small-seq}), we prove the following corollary.
\begin{corollary} \label{cor:dim-periodic-link} For any $2$-periodic link $\wt{K}$ in $S^3$ with axis $\wt{A}$ and
  quotient link $K$, there is a rank inequality
\[
\dim( \HFLa(S^3, \wt{K} \cup \wt{A})) \geq \dim(\HFLa(S^3,K \cup A)).
\]
\end{corollary}

Notice that there is no factor of two in this rank inequality, despite
the presence of the vector space $V_1$ in
Formula~\eqref{eq:Hen2-small-seq}.

Using Hochschild homology and bordered Floer homology,
Lipshitz-Treumann also constructed spectral sequences similar
to Formulas~\eqref{eq:Sig-to-K}
and~\eqref{eq:HFL-periodic}~\cite{LT:hoch-loc}. In particular, for a
genus $\leq 2$, null-homologous knot $K\subset Y^3$, they give a
spectral sequence $\HFKa(\Sigma(K),K)\otimes\Field[\theta,\theta^{-1}]\Rightarrow
\HFKa(Y,K)\otimes \Field[\theta,\theta^{-1}]$~\cite[Theorem 2]{LT:hoch-loc}. In the case that the
Hendricks and Lipshitz-Treumann spectral sequences are both defined,
we prove they agree:
\begin{theorem}\label{thm:LT-is-Hen}
  Let $K$ be a genus $\leq 2$ knot in $S^3$. Then the spectral
  sequence from Theorem~\ref{thm:Hen1-small} agrees with the spectral
  sequence from~\cite[Theorem 2]{LT:hoch-loc}.
\end{theorem}

\begin{remark}
  To be more precise,~\cite[Theorem 2]{LT:hoch-loc} worked with
  equivariant Floer homology, not cohomology, so the identification is
  between Theorem~\ref{thm:Hen1-small} for $(Y,K)$ and~\cite[Theorem
  2]{LT:hoch-loc} for $(-Y,K)$. In this paper, when we refer to
  complexes from~\cite{LT:hoch-loc}, we mean the \emph{duals} of the
  complexes presented there. The paper~\cite{LT:hoch-loc} typically
  focuses on the Tate version of equivariant cohomology, but here we
  mean the Borel variant.
\end{remark}

It also follows from the techniques in this paper that the spectral
sequences constructed by Lipshitz-Treumann are invariants of the
corresponding topological objects, and that equivariant Floer cohomology
can be computed using bordered Floer homology. Rather than further
trying the reader's patience, we leave most cases as an exercise, and
spell out only one:
\begin{theorem}\label{thm:LT-invt}
  For any $\ZZ/2$-cover $\pi\co \wt{Y}\to Y$ and torsion $\SpinC$-structure $\spinc$ on $Y$, there is an equivariant
  Heegaard Floer homology $\eHFa(\wt{Y},\pi^*(\spinc))$ and a spectral
  sequence $\HFaDual(\wt{Y},\pi^*(\spinc))\otimes V\otimes\Field[\theta]\Rightarrow \eHFa(\wt{Y},\pi^*(\spinc))$
  which are invariants of the pair $(\pi\co \wt{Y}\to Y,\spinc)$. This spectral sequence
  can be computed combinatorially using nice diagrams. If $\wt{Y}\to
  Y$ is induced by a $\ZZ$-cover then the equivariant Floer cohomology
  and the spectral sequence can be computed combinatorially using
  bordered Floer homology.
\end{theorem}
A $\ZZ/2$ cover is induced by a $\ZZ$-cover if there is a pullback square
\[
\xymatrix{
\wt{Y}\ar[r]\ar[d] & S^1\ar[d]^{z\mapsto z^2}\\
Y\ar[r] & S^1.
}
\]
If $\wt{Y}\to Y$ is a $\ZZ/2$-cover induced by a $\ZZ$-cover and
$\spinc$ is a torsion $\SpinC$-structure on $Y$ then the localized
equivariant Floer cohomology $\theta^{-1}\eHFa(\wt{Y},\pi^*\spinc)$ is
isomorphic to
$\HFaDual(Y,\spinc)\otimes\Field[\theta,\theta^{-1}]$~\cite[Theorem
3]{LT:hoch-loc}.  The restriction to torsion $\SpinC$-structures is to
ensure that the relevant objects are $\ZZ$-graded.

(In~\cite{LT:hoch-loc}, the localized equivariant cohomology is referred to as the \emph{Tate equivariant 
cohomology}. The statement of~\cite[Theorem 3]{LT:hoch-loc} combines the spectral sequence 
$\HFaDual(\wt{Y},\pi^*\spinc)\otimes V\otimes\Field[\theta,\theta^{-1}]\Rightarrow \theta^{-1}\HFaDual_{\ZZ/2}(\wt{Y},\pi^*\spinc)$, which exists simply because $\CFa(\wt{Y},\pi^*\spinc)\otimes V$ is a module over $\Field[\ZZ/2]$, and the more subtle localization isomorphism 
$\theta^{-1}\HFaDual_{\ZZ/2}(\wt{Y},\pi^*\spinc)\cong \HFaDual(Y,\spinc)\otimes\Field[\theta,\theta^{-1}]$,
to state a spectral sequence $\HFaDual(\wt{Y},\pi^*\spinc)\otimes V\otimes\Field[\theta,\theta^{-1}]\Rightarrow\HFaDual(Y,\spinc)\otimes\Field[\theta,\theta^{-1}]$.)

Theorem~\ref{thm:LT-invt} extends to finite covering spaces; see Theorem~\ref{thm:cov-eq-invt}.

Finally, we discuss two spectral sequences which were partial
inspiration for work of
Sarkar-Seed-Szab\'o~\cite{SSS:spectrals}. First, in
Section~\ref{sec:new-HFa} we observe that applying Seidel-Smith's
localization theorem to the involution on the branched double cover
$\Sigma(L)$ of a based link $(L,p)$ in $S^3$ gives a spectral sequence
\begin{equation}\label{eq:sar-dcov}
  \HFaDual(\Sigma(L))\otimes\Field[\theta,\theta^{-1}]\Rightarrow (\Field\oplus\Field)^{\otimes(|L|-1)}\otimes\Field[\theta,\theta^{-1}].
\end{equation}

\begin{theorem}\label{thm:sar-dcov-invt}
  The spectral sequence~\eqref{eq:sar-dcov} is an invariant of the
  based link $(L,p)$.
\end{theorem}

From this spectral sequence, we extract a numerical invariant
$q_\tau(K)\in\ZZ$, and prove:
\begin{theorem}\label{thm:q-tau-conc-homo}
  The invariant $q_\tau(K)$ is a homomorphism from the smooth concordance group.
\end{theorem}

There are analogous sequences for $\HF^-(\Sigma(K))$
(Theorem~\ref{thm:new-dcov-minus}), from which we extract a sequence
of integer-valued concordance invariants $d_\tau(K,i)$,
$i>0\in\ZZ$. The invariant $d_\tau(K,1)$ is the $d$-invariant of the
spin structure on $\Sigma(K)$, denoted $2\delta(K)$ by
Manolescu-Owens~\cite{ManolescuOwens07:delta}. The invariants $d_\tau$
satisfy
\[
d_\tau(K,1)\leq d_\tau(K,2)\leq d_\tau(K,3)\leq\cdots.
\]

Explicit computations (Section~\ref{sec:computations}) show that the
invariants $q_\tau$ and $d_\tau(K,2)$ do not agree with previously
defined concordance invariants.

In the setting of symplectic Khovanov homology, there is an action of
the dihedral group $D_{2^m}$ for any $m$; see
Section~\ref{sec:ssseq}. The techniques of
Section~\ref{sec:equi-complex} give a chain complex $\eKCSymp(L)$,
quasi-isomorphic to the
symplectic Khovanov cochain complex and is a module over the
group ring $\Field[D_{2^m}]$.
\begin{theorem}\label{thm:KC-equi-invt}
  The quasi-isomorphism type of $\eKCSymp(L)$ over $\Field[D_{2^m}]$ is an invariant of the link $L$.
\end{theorem}
Theorem~\ref{thm:KC-equi-invt} gives several spectral sequences from
$\KCSymp(L)$. Restricting attention to the $\ZZ/2$ subgroup of
$D_{2^m}$ corresponding to a (particular) reflection gives
Seidel-Smith's spectral sequence~\eqref{eq:SS-to-nabla}. Restricting
attention to the cyclic group $\ZZ/2\subset D_{2^m}$ generated by rotation by $\pi$ gives a
spectral sequence
\begin{equation}\label{eq:sar-kh-symp-1}
  \KhSymp(L)\otimes\Field[\theta,\theta^{-1}]\Rightarrow (\Field\oplus\Field)^{\otimes|L|}\otimes\Field[\theta,\theta^{-1}]
\end{equation}
while restricting to the cyclic group $\ZZ/2^m$, $m>1$, gives a
spectral sequence
\begin{equation}\label{eq:sar-kh-symp-2}
  \KhSymp(L)\otimes\Field[\eta,\eta^{-1}]\Rightarrow (\Field\oplus\Field)^{\otimes|L|}\otimes\Field[\eta,\eta^{-1}],
\end{equation}
and these spectral sequences are invariants of
$L$. (Formulas~\eqref{eq:sar-kh-symp-1} and~\eqref{eq:sar-kh-symp-2}
are proved in Theorem~\ref{thm:khsymp-loc}, though with slightly
different variable names.)

\textbf{Additional context.} Other versions of equivariant Lagrangian
intersection Floer homology substantially predate Seidel-Smith's
work. In particular, Khovanov-Seidel~\cite[Section
5c]{KhS02:BraidGpAction} consider Floer homology in the presence of an involution, 
under hypotheses which allow one to achieve
equivariant transversality.  Equivariant versions of the Fukaya
category of a symplectic manifold with an involution were studied by
Seidel~\cite[Chapter 14]{SeidelBook}, again under restrictions on the
Lagrangians which allow equivariant transversality. Equivariant Fukaya
categories with respect to finite group actions were further developed
by Cho-Hong~\cite{CH:group-actions} with coefficients in fields of
characteristic $0$, by averaging over multi-valued perturbations to
avoid equivariant transversality issues. In a slightly different
direction, the case of free actions has been exploited by
Seidel~\cite[Section 8b]{Seidel:quartic} and Wu~\cite{Wu:free-action}.

\textbf{Organization.}
This paper is organized as follows. In Section~\ref{sec:background} we
recall a little background about classical equivariant cohomology.  In
Section~\ref{sec:equi-complex} we give a definition of the equivariant
Floer complex which does not require equivariant transversality.
When equivariant transversality can be achieved, 
we identify the new construction with the
more obvious definition of equivariant Floer homology in
Section~\ref{sec:equi-equi}, and prove Theorem~\ref{thm:SS-invt}. In
Section~\ref{sec:HF-applications} we turn to the first Heegaard Floer
applications, to the spectral sequences~(\ref{eq:Sig-to-K})
and~(\ref{eq:HFL-periodic}), and prove
Corollaries~\ref{cor:Hen-from-diag} and \ref{cor:Hen-nice} and
Theorems~\ref{thm:Hen1-small},~\ref{thm:Hen2-small},
and~\ref{thm:LT-is-Hen}, Corollaries~\ref{cor:rank-inequality-refined} and~\ref{cor:dim-periodic-link}, and Proposition~\ref{thm:LT-invt}.  
We conclude
Section~\ref{sec:equi-complex} by proving that equivariant Floer
homology gives an invariant of $2^m$-fold covers.  
In Section~\ref{sec:new-dcov} we
construct the spectral sequence~\eqref{eq:sar-dcov} and the concordance invariants $q_\tau(K)$ and $d_\tau(K,n)$ and study their basic properties. 
In
Section~\ref{sec:ssseq} we turn to the Seidel-Smith spectral sequence,
constructing the spectral sequences~\eqref{eq:sar-kh-symp-1} and~\eqref{eq:sar-kh-symp-2} proving
Theorem~\ref{thm:KC-equi-invt}, as well as
Corollaries~\ref{cor:SSss-invt} and \ref{cor:SS-periodic-ss-invt}. 

Since some of the notions, like homotopy coherence, may be unfamiliar
to some readers, we have added a number of explicit, elementary
examples. To avoid interrupting the narrative flow of the paper, we
have moved these examples to Appendix~\ref{app:examples}, which the
reader familiar with homotopy coherence is discouraged from
reading. Other useful introductions to homotopy coherence include the
books by Kamps-Porter~\cite[Chapter 5]{KampsPorter97:abstract} and
Lurie~\cite{Lurie09:topos}.

Throughout this paper, all chain complexes are over the field $\Field$
with two elements. Involutions will usually be denoted $\tau$.

\textbf{Acknowledgments.} We thank the contributors to nLab, which has
been invaluable in sorting through the (to us unfamiliar) higher
category theory literature. We also thank Mohammed Abouzaid, Aliakbar Daemi, Ailsa
Keating, Tyler Lawson, Tye Lidman, Ciprian Manolescu, Davesh Maulik, Steven Sivek,
Ivan Smith, Stephan Wehrli, and Jingyu Zhao for helpful
conversations. Finally, we thank the referees for many helpful
suggestions and corrections.


\section{Background on equivariant cohomology}\label{sec:background}
We briefly recall some homological algebra behind equivariant
cohomology. Recall that our base ring is always the field $\Field$
with two elements. Let $C_*$ be a chain complex (over $\Field$) with a
$\ZZ/2$-action, or equivalently a chain complex over
$\Field[\ZZ/2]$. Further assume that $C_*$ is bounded below, and is
finite dimensional in each homological grading. The \emph{equivariant
  cohomology} of $C_*$ is
\[
\eH^*(C_*)\coloneqq \ExtO{\Field[\ZZ/2]}(C_*,\Field),
\]
where $\ZZ/2$ acts trivially on $\Field$.  So, to compute
$\eH^*(C_*)$, one starts by taking, say, a projective
resolution $\wt{C}_*$ of $C_*$ over $\Field[\ZZ/2]$ and then computing
the homology of the complex
\[
\RHomO{\Field[\ZZ/2]}(C_*,\Field)\coloneqq \HomO{\Field[\ZZ/2]}(\wt{C}_*,\Field).
\]
Writing $\ZZ/2=\{1,\tau\}$, a particularly efficient projective
resolution $\wt{C}_*$ is the total complex of the bicomplex
\begin{equation}\label{eq:standard-res}
\wt{C}_*\coloneqq \bigl(0 \longleftarrow C_*\otimes\Field[\ZZ/2] \stackrel{1+\tau}{\longleftarrow} C_*\otimes\Field[\ZZ/2]\stackrel{1+\tau}{\longleftarrow}\cdots\bigr).
\end{equation}
for which the complex $\RHomO{\Field[\ZZ/2]}(C_*,\Field)$ takes the form
\begin{equation}\label{eq:standard-res-rhom}
0 \longrightarrow C^* \stackrel{1+\tau}{\longrightarrow} C^*\stackrel{1+\tau}{\longrightarrow}C^* \stackrel{1+\tau}{\longrightarrow}\cdots.
\end{equation}
(Here by an abuse of notation $\tau$ also denotes the automorphism of
$C^*$ induced by the $\ZZ/2$-action.)

In the special case that $C_*$ is the singular chain complex of a
space $X$ with a $\ZZ/2$-action, the complex $C_*(X\times E\ZZ/2)$ is
a projective resolution of $C_*(X)$, so the equivariant cohomology as
defined above agrees with the Borel equivariant cohomology of $X$.

Given a module $M$ over $\Field$, the (mod 2) \emph{group cohomology of
  $\ZZ/2$ with coefficients in $M$} is
\[
H^*(\ZZ/2,M)\coloneqq \ExtO{\Field[\ZZ/2]}(\Field,M).
\]
The special case $M=\Field$ is simply called the group cohomology
of $\ZZ/2$ and written $H^*(\ZZ/2)\coloneqq
H^*(\ZZ/2,\Field)=\eH^*(\Field)$. Composition gives a product on
$H^*(\ZZ/2)$ and an action of $H^*(\ZZ/2)$ on $\eH^*(C_*)$ for any
complex $C_*$. An easy computation shows that $H^*(\ZZ/2)\cong
\Field[\theta]$ where $\theta$ lies in degree $1$; the action of
$\theta$ on the complex~\eqref{eq:standard-res-rhom} is shifting it
one unit to the right.

The horizontal filtration (or equivalently, the $\theta$-action) on
the bicomplex~\eqref{eq:standard-res-rhom} induces a spectral sequence
with $E_1$-page given by
\[
\{0 \qquad H^*(C_*)\qquad H^*(C_*)\qquad \cdots\} = H^*(C_*)\otimes\Field[\theta],
\]
$E_2$-page given by $H^*(\ZZ/2, H^*(C_*))$, and $E_\infty$-page given
by $\eH^*(C_*)$.

One can recover the dual of the original chain complex from the
equivariant version by simply setting $\theta=0$, namely, $C^*\cong
\RHomO{\Field[\ZZ/2]}(C_*,\Field)/(\theta=0)$. To be more precise,
$C^*$ fits into a short exact sequence with the $\RHom$ complex
from~\eqref{eq:standard-res-rhom},
\[
0\to\RHomO{\Field[\ZZ/2]}(C_*,\Field)\stackrel{\cdot\theta}{\longrightarrow}\RHomO{\Field[\ZZ/2]}(C_*,\Field)\stackrel{\pi}{\longrightarrow}
C^*\to 0,
\]
where $\pi$ is the projection to the leftmost $C^*$
in~\eqref{eq:standard-res-rhom}.

If $C_*'$ is a chain complex over $\Field[\ZZ/2]$, and if $C_*\to
C'_*$ is a quasi-isomorphism over $\Field[\ZZ/2]$, we get an induced
map of bicomplexes~\eqref{eq:standard-res-rhom} so that the map of
associated graded complexes is a quasi-isomorphism.  Thus, we get
induced isomorphisms between the entire spectral sequences
$H^*(C'_*)\otimes\Field[\theta]\Rightarrow \eH^*(C'_*)$ and
$H^*(C_*)\otimes\Field[\theta]\Rightarrow \eH^*(C_*)$. Furthermore,
the induced map
$\RHomO{\Field[\ZZ/2]}(C'_*,\Field)\to\RHomO{\Field[\ZZ/2]}(C_*,\Field)$
gives an isomorphism $\eH^*(C'_*)\stackrel{\cong}{\longrightarrow}
\eH^*(C_*)$ on homology, as modules over $\Field[\theta]$. These quasi-isomorphisms further produce quasi-isomorphisms between short exact sequences
\[
\xymatrix{
0\ar[r]&\RHomO{\Field[\ZZ/2]}(C'_*,\Field)\ar[d]\ar[r]^-{\theta}&\RHomO{\Field[\ZZ/2]}(C'_*,\Field)\ar[d]\ar[r]^-{\pi}&(C')^*\ar[r]\ar[d]& 0\\
0\ar[r]&\RHomO{\Field[\ZZ/2]}(C_*,\Field)\ar[r]^-{\theta}&\RHomO{\Field[\ZZ/2]}(C_*,\Field)\ar[r]^-{\pi}&C^*\ar[r]& 0
}
\]
where the rightmost arrow is the dual of the original quasi-isomorphism between $C_*$ and $C'_*$.

Similar constructions apply when $K$ is any finite group: given a
chain complex $C_*$ with an action of $K$, we define
\begin{align*}
\eH[K]^*(C_*)&\coloneqq \ExtO{\Field[K]}(C_*,\Field)\\
H^*(K,M)&\coloneqq \ExtO{\Field[K]}(\Field,M)\\
H^*(K)&\coloneqq H^*(K,\Field)=\eH[K]^*(\Field);
\end{align*}
and $\eH[K]^*(C_*)$ comes with an action of the algebra $H^*(K)$, by
composition.  There is no longer as nice a resolution as
Formula~\eqref{eq:standard-res}, but choosing any particular free
resolution $R_*$ of $\Field$ over $\Field[K]$ (for instance, the bar
resolution) we get a resolution $C_*\otimes_\Field R_*$ of $C_*$, and
the grading on $R_*$ induces a filtration on this resolution. The
$E_1$-page now depends on $R_*$, but the $E_2$-page is $H^*(K,
H^*(C_*))$. 
Also, there is no longer an easy way to recover the
original chain complex from the equivariant version, but there is a
spectral sequence $\Tor_{H^*(K)}(\Field,\eH[K]^*(C_*))\Rightarrow
H^*(C_*).$

Given a chain complex $C'_*$ quasi-isomorphic to $C_*$ over
$\Field[K]$, the invariants $\eH[K]^*(C_*)$ and $\eH[K]^*(C'_*)$ are
isomorphic $H^*(K)$-modules, and each page (from the second on) of the
spectral sequences $H^*(K,H^*(C_*))\Rightarrow \eH[K]^*(C_*)$ and
$H^*(K,H^*(C'_*))\Rightarrow \eH[K]^*(C'_*)$ is isomorphic.

The upshot is that to prove equivariant cohomology, with its
module structure, and the spectral sequence converging to it, is an
invariant it suffices to prove that the original complex over
$\Field[K]$ is well-defined up to quasi-isomorphism over
$\Field[K]$.


\section{The equivariant Floer complex via non-equivariant complex structures}\label{sec:equi-complex} \label{sec:equicomplex}
Fix a symplectic manifold $(M,\omega)$, Lagrangians $L_0,L_1\subset
M$, and a symplectic involution $\tau\co M\to M$ preserving the $L_i$.
In this section we discuss how to define the (Borel) equivariant Floer
cohomology of $(L_0,L_1)$ without the need for equivariant
transversality. Other constructions with similar invariance properties are given
in~\cite[Section 2(a)]{SeidelSmith10:localization} (written in the
context of Morse theory) and~\cite[Section 3(b)]{Seidel:equi-pants}
(in the context of fixed-point Floer homology).

Roughly, the idea is that the space of almost complex structures 
is contractible and admits an action by $\ZZ/2$, and we will construct Floer cohomology parameterized
by points in this space. There are many ways to formulate precisely
what one means by Floer cohomology with this parameter
space. In~\cite{SeidelSmith10:localization,Seidel:equi-pants}, the
authors choose a Morse function on their parameter space
and
couple the $\dbar$ equation to the Morse flow equation on the space.
Here, we take an approach more along the lines of simplicial sets,
building a homotopy coherent functor from a category $\ECat\ZZ/2$ (see
Section~\ref{sec:hoco-diag-cx-strs}) to
the category of chain complexes, using an intermediary category of
almost complex structures. One feature of this construction is that it
generalizes easily from $\ZZ/2$ to other finite groups.

\subsection{Hypotheses and statement of result}\label{sec:hyps-and-statement}
While we do not require equivariant transversality, we do make
assumptions which allow us to construct a $\ZZ$-graded Floer complex
with coefficients in $\Field$ (rather than a Novikov ring).
Because we want Heegaard Floer invariants to fall under this rubric,
we give fairly weak but intricate assumptions. 

\begin{convention}
  For the rest of Section~\ref{sec:equicomplex}, fix a collection
  $\eta$ of homotopy classes of paths from $L_0$ to $L_1$; by the
  Floer complex we mean the summand of the Floer complex spanned by
  intersection points which, viewed as constant paths from $L_0$ to
  $L_1$, lie in one of the homotopy classes in $\eta$; in Heegaard
  Floer terminology, this corresponds to fixing a collection of
  $\SpinC$-structures.
\end{convention}

The hypotheses we require are:%
\begin{hypothesis}\label{hyp:Floer-defined}
  \begin{enumerate}[label=(J-\arabic*)]
  \item\label{item:J-1}
    For any loop of paths
    \[
    v\co \bigl([0,1]\times S^1,\{0\}\times S^1,\{1\}\times S^1\bigr)\to (M,L_0,L_1)
    \]
    such that the homotopy class of the path $[v|_{[0,1]\times
      \{\pt\}}]$ is in $\eta$, the area of $v$ (with respect to
    $\omega$) and the Maslov index of $v$ both vanish.
    (Compare~\cite[Theorem 1.0.1]{WW12:compose-correspond}.)

    (Note that this condition holds for the hat version of the Heegaard
    Floer complex for admissible Heegaard diagrams;
    compare~\cite[Lemma 4.12]{OS04:HolomorphicDisks}.)
  \item\label{item:J-2} The manifold $M$ is either compact or else is convex at
    infinity~\cite[Section 1.7.1]{EliashbergGromov91:convex}. Further,
    the Lagrangians $L_0$ and $L_1$ are either compact or else $M$ has
    finite type and $L_0$ and $L_1$ are conical at infinity and
    disjoint outside a compact set (see, e.g.,~\cite[Section
    5]{KhS02:BraidGpAction} or~\cite[Section 2.1]{Hendricks:symplecto}
    and the references there).
  \item There is an action of a finite group $H$ on $M$ by
    symplectomorphisms which preserve $L_0$, $L_1$, and the collection
    of homotopy classes of paths $\eta$ set-wise.
  \end{enumerate}
\end{hypothesis}

Under these hypotheses, we will define a model $\ECF(L_0,L_1)$ for the
Lagrangian Floer complex, which we will call the \emph{freed Floer
  complex} (Definition~\ref{def:equi-Floer} for $H=\ZZ/2$ and
Section~\ref{sec:2-groups} for arbitrary finite groups). The freed
Floer complex
is a module over $\Field[H]$; regarded as simply a complex over $\Field$ (by
restriction of scalars), $\ECF(L_0,L_1)$ is quasi-isomorphic to the
usual Lagrangian Floer complex $\CF(L_0,L_1)$.  The construction
depends on choices of (many) almost complex structures, but the
resulting quasi-isomorphism type, over $\Field[H]$, is independent
of these choices (Propositions~\ref{prop:indep-of-cx-str} and~\ref{prop:invariance-gen-group}). Further,
the equivariant quasi-isomorphism type of $\ECF(L_0,L_1)$ is invariant
under Hamiltonian isotopies through equivariant Lagrangians
(Proposition~\ref{prop:indep-of-Ham-isotopy} and~\ref{prop:invariance-gen-group}).  Under additional
hypotheses, $\ECF(L_0,L_1)$ is even invariant under non-equivariant
Hamiltonian isotopies (Propositions~\ref{prop:non-equi-invar} and~\ref{prop:non-equi-invar-gen-group}).
Conveniently, $\ECF(L_0,L_1)$ is a free module over $\Field[H]$,
so the equivariant Floer cochain complex $\eCF(L_0,L_1)$ is simply
$\HomO{\Field[H]}(\ECF(L_0,L_1),\Field)$, and the equivariant
Floer cohomology $\eHF(L_0,L_1)$ is
$\ExtO{\Field[H]}(\ECF(L_0,L_1),\Field)$. It follows from
homological algebra that the isomorphism type of $\eHF(L_0,L_1)$ and, in the case $H=\ZZ/2$,
the pages of the spectral sequence
$\HF^*(L_0,L_1)\otimes\Field[\theta]\Rightarrow \eHF(L_0,L_1)$, are
invariants of $(L_0,L_1)$ up to equivariant Hamiltonian isotopy. 

\subsection{The extended space of eventually cylindrical almost complex structures}
By a \emph{cylindrical complex structure} on $(M,\omega)$ we mean a
smooth, one-parameter family $J=J(t)$, $t\in[0,1]$, of smooth almost complex 
structures on $M$ so that each $J(t)$ is compatible with
$\omega$. Let $\JSpace_\cyl$ denote the space of cylindrical almost complex
structures with the $C^\infty$ topology. By an
\emph{eventually cylindrical almost complex structure} on
$(M,\omega)$ we mean a smooth map $\wt{J}\co \RR\to \JSpace_\cyl$
which is constant outside some compact set, modulo the equivalence
relation that for any $s_0\in\RR$, $\wt{J}(\cdot)\sim \wt{J}(\cdot +
s_0)$ (i.e., modulo translation in the source).  Let $\JSpace$ denote
the set of eventually cylindrical almost complex structures.

Define a topology on $\JSpace$ as follows. Given a compact interval
$I\subset\RR$ and an open set $V$ in the space of smooth maps
$\RR\times [0,1]\times TM\to TM$ (with the $C^\infty$ topology), let
$U(I,V)$ be the set of $\wt{J}\in\JSpace$ so that there is a representative for $\wt{J}$ which
is constant on $\RR\setminus I$ and which, when viewed as a map
$\RR\times [0,1]\times TM\to TM$, lies in $V$. Define the topology on
$\JSpace$ to be generated by the sub-basis consisting of all such sets
$U(I,V)$. In particular, a sequence
$\{\wt{J}_i\}$ converges if, for some representatives of the
equivalence classes of the $\wt{J}_i$, there is a compact set
$I\subset \RR$ so that all of the $\wt{J}_i$ are constant on each
component of $\RR\setminus I$ and the $\wt{J}_i|_I$ converge in the
$C^\infty$ topology.

An eventually cylindrical almost complex structure $\wt{J}$ specifies a
cylindrical almost complex structure $J_{-\infty}$ (respectively
$J_{+\infty}$) near $-\infty$ (respectively $+\infty$).  Let
$\JSpace(J_{-\infty},J_{+\infty})$ be the subspace of $\JSpace$
specifying $J_{\pm\infty}$ near $\pm\infty$, respectively. (The
$\JSpace(J_{-\infty},J_{+\infty})$'s are the path components of
$\JSpace$.)  If $\wt{J}\in \JSpace(J_{-\infty},J_{+\infty})$ then we
may write $\wt{J}|_{-\infty}=J_{-\infty}$ and
$\wt{J}|_{+\infty}=J_{+\infty}$. Note that $\JSpace_{\cyl}\subset
\JSpace$ as the constant maps $\RR\to\JSpace_{\cyl}$.

It will be convenient to add some extra points to $\JSpace$,
corresponding to multi-story almost complex structures. Specifically,
let $\ol{\ol{\JSpace}}(J,J')$ denote the set of finite, non-empty sequences 
\[
(\wt{J}^1,\wt{J}^2,\dots, \wt{J}^n)\in \JSpace(J=J_0,J_1)\times \JSpace(J_1,J_2)\times\cdots\times \JSpace(J_{n-1},J_n=J')
\]
of eventually cylindrical complex structures so that
$\wt{J}^i|_{+\infty}=\wt{J}^{i+1}|_{-\infty}$, starting at $J$ and
ending at $J'$. Define an equivalence
relation $\sim$ on $\ol{\ol{\JSpace}}$ be declaring that
\[
(\wt{J}^1,\dots,\wt{J}^{i-1},\wt{J}^i,\wt{J}^{i+1},\dots \wt{J}^n) \sim
(\wt{J}^1,\dots,\wt{J}^{i-1},\wt{J}^{i+1},\dots \wt{J}^n)
\]
if $\wt{J}^i\in \JSpace_{\cyl}$. Let $\ol{\JSpace}$ be the quotient of
$\ol{\ol{\JSpace}}$ by the equivalence relation $\sim$.

There is a composition operation 
$\circ\co \ol\JSpace(J_2,J_3)\times\ol\JSpace(J_1,J_2)\to\ol\JSpace(J_1,J_3)$ 
defined by $\wt{J}^2\circ \wt{J}^1=(\wt{J}^1,\wt{J}^2)$.

Next, we define a topology on $\ol{\JSpace}$. Fix a (one-story)
cylindrical-at-infinity almost complex structure $\wt{J}$ (or rather,
a representative of the equivalence class $\wt{J}$) and real numbers
$a_1<b_1<a_2<b_2<\cdots<a_k<b_k$ so that $\wt{J}$ is locally
constant in the $\RR$ coordinate except on the intervals $[a_i,b_i]$, $1\leq i\leq k$. Say that a cylindrical-at-infinity
almost complex structure $\wt{J}'$ is obtained from $(\wt{J},\{a_i,b_i\})$ by an
\emph{elementary expansion} if there is some $1\leq j<k$ and an
$r\in\RR$ so that for $t\in\RR$,
\[
  \wt{J}'(t)=
  \begin{cases}
    \wt{J}(t) & \text{if }t\leq b_i\\
    \wt{J}(b_i)=\wt{J}(a_{i+1}) & \text{if }b_i\leq t\leq a_{i+1}+r\\
    \wt{J}(t-r) & \text{if }t\geq a_{i+1}.
  \end{cases}
\]
We also say that the two-story almost complex structure $\wt{J}'$
whose first story $\wt{J}^1$ is given by 
\[
  \wt{J}^1(t)=
  \begin{cases}
    \wt{J}(t) & \text{if }t\leq b_i\\
    \wt{J}(b_i) &\text{if }t\geq b_i
  \end{cases}
\]
and whose second story $\wt{J}^2$ is given by
\[
  \wt{J}^2(t)=
  \begin{cases}
    \wt{J}(b_i) &\text{if }t\leq a_{i+1}\\
    \wt{J}(t) & \text{if }t\geq a_{i+1}
  \end{cases}
\]
is obtained from $\wt{J}$ by an elementary expansion. Note that we
have a distinguished collection of points $a'_1,b'_1,\dots,a'_k,b'_k$
so that $\wt{J}'$ is constant on $[b'_i,a'_{i+1}]$. (Here, we abuse
notation slightly in the case that $\wt{J}'$ has two stories.) More
generally, consider a pair $(\wt{J},\{a_i,b_i\})$, where $\wt{J}$ is a
multi-story almost complex structure and the $a_1<b_1<\cdots<a_k<b_k$
are a sequence of points in the source of $\wt{J}$ (a sequence of
copies of $\RR$).  We say that a multi-story almost complex structure
$(\wt{J}',\{a'_i,b'_i\})$ is obtained from $(\wt{J},\{a_i,b_i\})$ by
an \emph{elementary expansion} if $\wt{J}'$ is obtained by replacing
some story in $\wt{J}$ by an elementary expansion of that story and
$\{a'_i,b'_i\}$ are the new distinguished points. Finally, we say that
$\wt{J}'$ is obtained from $(\wt{J},\{a_i,b_i\})$ by an
\emph{expansion} if $\wt{J}'$ is obtained by some sequence of
elementary expansions.  Informally, $\wt{J}'$ is obtained from
$\wt{J}$ by an expansion if $\wt{J}'$ comes from making some of the
distinguished intervals on which $\wt{J}$ is constant longer,
including allowing those intervals to become infinitely long.

Now, fix an open neighborhood $V$ of $\wt{J}$ in $\JSpace$ and $2n$
points $\{a_i,b_i\}$ in $\RR$ (for some $n$) so that each $\wt{J}'\in V$ is
constant on each component of
$\RR\setminus\bigcup_{i=1}^{n}[a_i,b_i]$, and consider the set
$U(\wt{J},V,\{a_i,b_i\})$ of multi-story almost complex structures
which are obtained from $(\wt{J}',\{a_i,b_i\})$ for some $\wt{J}'\in
V$ by an expansion. The sets $U(\wt{J},V,\{a_i,b_i\})$ form a
sub-basis for the topology on $\ol{\JSpace}$.

Note, in particular, that if $\{a_i,b_i\}$ has $2n$ elements (i.e.,
$n$ intervals) then any complex structure in $U(\wt{J},V,\{a_i,b_i\})$
has at most $n$ stories. So, the space $\JSpace$ is an open subset of $\ol{\JSpace}$.

It is perhaps clearer to spell out what it means for a sequence of almost complex structures in $\ol{\JSpace}$ to converge.
Given a sequence $\wt{J}^i$ in $\JSpace$, we say that
$\wt{J}^i$ converges to a point $(\wt{J}^1,\dots,\wt{J}^n)\in
\ol{\JSpace}$ if for some representatives of the $\wt{J}^i$
there is a sequence of $n$-tuples of open intervals $(I_1^i,\dots,
I_n^i)$ in $\RR$, $i=1,\dots,\infty$, so that:
\begin{itemize}
\item For each sufficiently large $i$, the $I^i_k$ are disjoint.
\item For each $k$, the length of $I_k^i$ is bounded independently of $i$.
\item The distance between $I_k^i$ and $I_{k+1}^i$ goes to infinity as $i\to\infty$.
\item Each $\wt{J}^i$ is constant on each component of
  $\RR\setminus(I_1^i\cup\cdots\cup I_n^i)$.
\item Let $\wt{J}^{i,k}$ be the result of restricting $\wt{J}^i$ to
  $I_k$ and extending by the constant map, to get an
  eventually cylindrical almost complex structure on all of
  $\RR$. Then we require that for each $k$, the $\wt{J}^{i,k}$
  converge to $\wt{J}^k$.
\end{itemize}

Like $\JSpace$, $\ol{\JSpace}$ is a disjoint union of path components
$\ol{\JSpace}(J_{-\infty},J_{+\infty})$.

As the composition notation suggests, the space $\ol{\JSpace}$ is a
topological category. The objects of $\ol{\JSpace}$ are the
cylindrical almost complex structures, $\JSpace_{\cyl}$,
$\Hom(J_1,J_2)=\ol{\JSpace}(J_1,J_2)$, which is a topological space, and
composition is $\circ$, which is continuous. The identity map of any
cylindrical almost complex structure $J$ is the constant path at
$J$; because of the equivalence relation we imposed, this is a
strict identity.

If there is a symplectic action of a group $G$ on $M$ then there is an
induced action of $G$ on $\ol{\JSpace}$ by functors;
$g_*(\wt{J})=dg\circ \wt{J}\circ (dg)^{-1}$.

It is well known that the space of almost complex structures on $M$
compatible with $\omega$ is contractible. It follows immediately that
$\JSpace_{\cyl}$ is contractible. Further, it is not hard to see that
$\JSpace(J_{-\infty},J_{+\infty})$ and
$\ol{\JSpace}(J_{-\infty},J_{+\infty})$ are at least weakly
contractible (i.e., have trivial homotopy groups). (This is used
in the proof Lemma~\ref{lem:exist-equi-diagram}.)

\subsection{Homotopy coherent diagrams of almost complex structures}
\label{sec:hoco-diag-cx-strs}
To build our equivariant complex we will need to fix a number of
choices.  Let $\ECat\ZZ/2$ denote the category with two objects, $a$
and $b$, with $\Hom(x,y)$ a single element for any
$x,y\in\{a,b\}$. (The category $\ECat\ZZ/2$ is a groupoid.)
Graphically:
\[
\ECat\ZZ/2=
\xymatrix{
a \ar@(dl, ul)^{\Id} \ar@/_/[r]_{\alpha}& b\ar@/_/[l]_\beta \ar@(dr, ur)_{\Id}
}
\]
There is an obvious $\ZZ/2$-action on $\ECat\ZZ/2$, exchanging $a$ and
$b$.

The choices we will need in order to define the equivariant complex
can be summarized as a generic, $\ZZ/2$-equivariant homotopy coherent
diagram $\ECat\ZZ/2\to \ol{\JSpace}$.  Given a category $\Cat$ and a
homotopy coherent diagram $\Cat\to\ol{\JSpace}$, we can apply Floer
theory to obtain a homotopy coherent diagram from $\Cat$ to simplicial
abelian groups (i.e., chain complexes). In particular, the homotopy
colimit of this diagram is a chain complex. In the case of the
$\ZZ/2$-equivariant diagram $\ECat\ZZ/2\to \ol{\JSpace}$, we get a
chain complex over $\Field[\ZZ/2]$ which, because the action of
$\ZZ/2$ on $\ECat\ZZ/2$ is free, is a free module.  Taking $\Hom$ to
the trivial $\Field[\ZZ/2]$-module $\Field$ and taking homology yields
the desired equivariant cohomology. The chain complex whose homology
is the equivariant Floer cohomology is given explicitly in
Observation~\ref{obs:concrete-hocolim}.

We explain this in a little more detail next. In this section, we
focus on diagrams $\Cat\to \ol{\JSpace}$, and restrict to the special
case of equivariant diagrams $\ECat\ZZ/2\to\ol{\JSpace}$ in
Section~\ref{sec:build-equi-cx}. Roughly, a homotopy coherent diagram
is a homotopy commutative diagram in which the homotopies are part of
the data, and commute up to appropriate higher homotopies (also part
of the data). The first precise definition of this notion was given by
Vogt:

\begin{definition}\label{def:ho-coh-diag-strs}\cite[Definition 2.3]{Vogt73:hocolim}
  Given a small category $\Cat$, a homotopy coherent $\Cat$-diagram in
  $\ol{\JSpace}$ consists of:
  \begin{itemize}
  \item For each object $x$ of $\Cat$, an object $F(x)$ of
    $\ol{\JSpace}$, i.e., a cylindrical almost complex structure.
  \item For each integer $n\geq 1$ and each sequence
    $x_0\stackrel{f_1}{\longrightarrow}\cdots\stackrel{f_n}{\longrightarrow}
    x_n$ of composable morphisms a continuous
    map $F(f_n,\dots,f_1)\co [0,1]^{n-1}\to\ol{\JSpace}(F(x_0),F(x_n))$
  \end{itemize}
  such that
  \begin{gather*}
    \begin{split}
      F(f_n,\dots,f_2,\Id)(t_1,\dots,t_{n-1})&=F(f_n,\dots,f_2)(t_{2},\dots,t_{n-1})\\
      F(\Id,f_{n-1},\dots,f_1)(t_1,\dots,t_{n-1}) &= F(f_{n-1},\dots,f_1)(t_{1},\dots,t_{n-2})
    \end{split}\\
    \begin{split}
      F(f_n,\dots,f_{i+1},&\Id,f_{i-1},\dots,f_1)(t_1,\dots,t_{n-1}) \\
      &= F(f_n,\dots,f_{i+1},f_{i-1},\dots,f_1)(t_{1},\dots,t_{i-1}\cdot t_{i},\dots,t_{n-1})\\
      F(f_n,\dots,f_1)&(t_1,\dots,t_{i-1},1,t_{i+1},\dots,t_{n-1}) \\
      &= F(f_n,\dots,f_{i+1}\circ f_i,\dots,f_1)(t_{1},\dots,t_{i-1},t_{i+1},\dots,t_{n-1})\\
      F(f_n,\dots,f_1)&(t_1,\dots,t_{i-1},0,t_{i+1},\dots,t_{n-1})\\
      &=
      [F(f_n,\dots,f_{i+1})(t_{i+1},\dots,t_{n-1})]\circ [F(f_i,\dots,f_1)(t_{1},\dots,t_{i-1})].
    \end{split}
  \end{gather*}
  
  (Here, $(t_1,\dots,t_{n-1})\in [0,1]^{n-1}$.)

  We will often drop the words ``homotopy coherent,'' and simply refer
  to a homotopy coherent $\Cat$-diagram as a \emph{$\Cat$-diagram} or
  as a diagram $F\co\Cat\to\ol{\JSpace}$.
\end{definition}

\begin{remark}
  A strictly commutative diagram is the special case that
  \[
    F(f_n,\dots,f_1)(t_{1},\dots,t_{n-1})=F(f_n\circ\cdots\circ f_1),
  \]
  for all $t_1,\dots,t_{n-1}$. (This example is important in
  Section~\ref{sec:equi-equi}.)
\end{remark}

See Example~\ref{eg:coherent-diag} for an explicit example of the first few terms in a
homotopy coherent $\ECat\ZZ/2$-diagram in $\ol{\JSpace}$.

\begin{definition}
  Given a point $\wt{J}=(\wt{J}^1,\wt{J}^2,\dots, \wt{J}^n)$ in
  $\ol{\JSpace}$, where each $\wt{J}^i\in \JSpace(J_{i-1},J_i)$ is not
  $\RR$-invariant, and points
  $x,y\in L_0\cap L_1$, by a \emph{$\wt{J}$-holomorphic disk from $x$ to $y$}
  we mean a sequence 
  \[
  (v^{0,1},\dots,v^{0,m_0},u^1,v^{1,1},\dots,v^{1,m_1},u^2,\dots,u^n,v^{n,1},\dots,v^{n,m_n})
  \]
  where
  \begin{itemize}
  \item each $v^{i,j}$ is a $J_i$-holomorphic Whitney disk with boundary on
    $L_0$ and $L_1$ connecting some points
    $x^{i,j-1}$ and $x^{i,j}$ in $L_0\cap L_1$,
  \item each $u^i$ is a $\wt{J}_i$-holomorphic Whitney disk with boundary on
    $L_0$ and $L_1$ connecting some points $x^{i-1}$ and $x^{i}$ in
    $L_0\cap L_1$,
  \item for $i=1,\dots, n$, $x^{i,0}=x^i$,
  \item for $i=0,\dots,n-1$, $x^{i,m_i}=x^{i+1}$, 
  \item $x^{0,0}=x$, and $x^{n,m_n}=y$.
  \end{itemize}
  (The numbers $m_i$ are allowed to be zero, i.e., the $v^{i,j}$ are
  not required to appear.)

  Given a map $\wt{J}=\wt{J}(t_1,\dots,t_{k})\co [0,1]^k\to
  \ol{\JSpace}$ and points $x,y\in L_0\cap L_1$ let $\cM(x,y;\wt{J})$
  denote the moduli space of pairs $(\vec{t},u)$ where
  $\vec{t}\in[0,1]^k$ and $u$ is a $\wt{J}(\vec{t})$-holomorphic disk
  from $x$ to $y$.

  We say that a diagram $F\co\Cat\to \ol{\JSpace}$ is
  \emph{sufficiently generic} if for any sequence of composable arrows
  $(f_1,\dots,f_n)$ in $\Cat$ and any $x,y\in L_0\cap L_1$, the moduli
  space $\cM(x,y;F(f_n,\dots,f_1))$ is transversely cut out by the
  $\dbar$-equation.
\end{definition}

Note that the space $\cM(x,y;\wt{J})$ decomposes as a union 
\[
\cM(x,y;\wt{J})=\bigcup_{\phi\in\pi_2(x,y)}\cM(\phi;\wt{J})
\]
according to homotopy classes of Whitney disks. If $\wt{J}$ is a
$k$-parameter family of almost complex structures then the expected
dimension of $\cM(\phi;\wt{J})$ is $\mu(\phi)+k$, where $\mu$ is the Maslov index.

\begin{lemma}\label{lem:exist-suff-generic}
  Let $F\co\Cat\to\ol{\JSpace}$ be a diagram and suppose that for each
  object $x\in\ob(\Cat)$ the moduli spaces of disks with respect to
  the cylindrical complex structure $F(x)$ are transversely cut out.
  Then there is a sufficiently generic $F'\co\Cat\to\ol{\JSpace}$
  arbitrarily close to $F$. Further, if the restriction of
  $F(f_m,\dots,f_1)$ is already sufficiently generic for all sequences
  $(f_1,\dots,f_m)$ of length $m\leq n$ then one can choose $F'$ to
  agree with $F$ for all such sequences.
\end{lemma}
\begin{proof}[Proof sketch]
  This follows from standard transversality arguments, which are
  explained nicely by McDuff-Salamon in a closely related
  setting~\cite[Chapter 3]{MS04:HolomorphicCurvesSymplecticTopology}. 
  We outline the key steps.

  The proof is by induction on $n$, the length of the sequence, as in
  the second half of the lemma's statement. Suppose that a
  sufficiently generic $F'$ has been chosen already for all sequences
  of length $m<n$. Consider a sequence $(f_1,\dots,f_n)$ of
  non-identity morphisms. Then $F'(f_n,\dots,f_1)$ should be a family
  of almost complex structures parameterized by $[0,1]^{n-1}$, close
  to $F(f_n,\dots,f_1)$, and so that the restriction
  $F'(f_n,\dots,f_1)|_{\bdy([0,1]^{n-1})}$ of $F'$ to the boundary of
  the cube is determined by Definition~\ref{def:ho-coh-diag-strs}. 

  Note that the condition of being sufficiently generic is local, in
  the following sense: given an open set $U\subset [0,1]^{n-1}$, and
  $F'\co U\to \ol{\JSpace}$, it makes sense to ask whether $F'$ is
  sufficiently generic on $U$, i.e., whether the moduli space
  $\cM(x,y;F'|_{U})$ is transversely cut out.

  Let $U$ be a small neighborhood of $\bdy([0,1]^{n-1})$ so that the
  boundary of $U$ consists of the boundary of
  $[0,1]^{n-1}$ and a smooth sphere $S^{n-2}$. Extend $F'$ to $U$
  continuously, and so that on $U\setminus\bdy([0,1]^{n-1})$ the map
  $F'$ takes values $\JSpace$, the space of 1-story
  cylindrical-at-infinity almost complex structures. By shrinking $U$
  we may arrange that $F'$ is arbitrarily close to $F$ on $U$. If $U$
  is small enough, then the $F'$ is sufficiently generic over $U$:
  $F'|_{\bdy([0,1]^{n-1})}$ is sufficiently generic by induction, and
  it is not hard to show that the condition of being sufficiently
  generic is open.

  Now, let $V$ be the inside of $S^{n-2}$, i.e., the component of
  $[0,1]^{n-1}\setminus S^{n-2}$ on which $F'$ has not yet been
  defined.  The next step is to use Sard's theorem to extend $F'$ to
  $V$.  Given intersection points $x,y\in L_0\cap L_1$, let
  $\mathcal{B}(x,y)$ be an appropriate weighted Sobolev space of maps 
  \[
    (\RR\times [0,1],\RR\times\{0\},\RR\times\{1\})\to (M,L_0,L_1)
  \]
  converging exponentially to $x$ at $-\infty$ and $y$ at $+\infty$.
  (See, for example, Seidel~\cite[Section (8h)]{SeidelBook} for
  details on the functional analytic setup.) There is also a space
  $\Map(\overline{V},\JSpace)$ of maps $\overline{V}\to \JSpace$
  extending $F'|_{S^{n-2}}$, which we topologize as a subspace of the
  space of $C^\ell$ maps $\overline{V}\times [0,1]\times TM\to TM$,
  for an appropriate $\ell$. Finally, there is a bundle
  $\mathcal{E}\to\mathcal{B}(x,y)\times\Map(\overline{V},\JSpace)\times\overline{V}$
  whose fiber over $(u,J,p)$ is a weighted Sobolev space of
  $(0,1)$-forms on $[0,1]\times\RR$ (with respect to $J(p)$) valued in
  $u^*TM$. The universal $\dbar$-operator is a section
  \[
    \dbar\co \mathcal{B}(x,y)\times\Map(\overline{V},\JSpace)\times\overline{V}\to \mathcal{E}.
  \]
  It follows from the argument given in McDuff-Salamon~\cite[Proof of Proposition 3.2.1]{MS04:HolomorphicCurvesSymplecticTopology} that
  the operator $\dbar$ is transverse to the $0$-section. Thus, the
  \emph{universal moduli space} $\dbar^{-1}(0)$ is a $C^k$ Banach
  manifold (where $k$ depends on $\ell$ and the Maslov index
  difference between $x$ and $y$). Any regular value of the projection
  $\dbar^{-1}(0)\to \Map(\overline{V},\JSpace)$ is a sufficiently
  generic almost complex structure. Smale's infinite dimensional
  version of Sard's theorem guarantees that the set of regular values
  is dense, so, in particular, we can find a regular value arbitrarily
  close to $F|_{V}$.
  
  Taking $F'|_V$ to be such a regular value gives a $C^k$ sufficiently
  generic almost complex structure as close as desired to $F$.  While
  $C^k$ almost complex structures for finite $k$ are enough for the
  applications in this paper, a short further argument as in
  McDuff-Salamon\cite[Proof of Theorem 3.1.6(II)]{MS04:HolomorphicCurvesSymplecticTopology} 
  guarantees the existence of generic smooth
  complex structures, as well.
\end{proof}

Notice that in order to define a homotopy coherent diagram, the key
ingredients we needed about $\ol{\JSpace}$ (or any other topological
category) were the product and the interval $[0,1]$. We can give an
analogous definition in the category of chain complexes, using
$\otimes$ in place of the product and the chain complex 
\[
I_*\coloneqq C_*^{\mathrm{simp}}([0,1]) = 
\vcenter{\xymatrix{
& \{0,1\} \ar[dl] \ar[dr]& \\
\{0\} & & \{1\} 
}}
\]
in place of $[0,1]$. The inclusions $\{0\}\hookrightarrow [0,1]$ and
$\{1\}\hookrightarrow [0,1]$ induce chain maps $\iota_0,\iota_1\co
\Field\hookrightarrow I_*$ and the projection $[0,1]\to \{\pt\}$
induces a chain map $\pi\co I_*\to\Field$.  The multiplication map
$[0,1]^2\to [0,1]$, $(x,y)\mapsto xy$ induces the ``multiplication''
$m\co I_*\otimes I_*\to I_*$ given by
\begin{align*}
&m(\{0\}\otimes \{0\})=m(\{0\}\otimes\{1\})=m(\{1\}\mathrlap{\otimes\{0\})=\{0\}} & m(\{1\}\otimes \{1\})=\{1\}& \\
&m(\{0,1\}\otimes \{0\})=m(\{0\}\otimes \{0,1\})=0 & 
{m(\{0,1\} \otimes \{1\}) = m(\{1\}\otimes \{0,1\})=\{0,1\}}& \\
&m(\{0,1\}\otimes\{0,1\})=0.
\end{align*}

(For a formal setting which generalizes both
examples, see for instance~\cite[Chapter 5]{KampsPorter97:abstract}).
We spell out the notion of homotopy coherent diagrams:
\begin{definition}\label{def:ho-coh-diag-cxs}
 A homotopy coherent $\Cat$-diagram in chain complexes consists of:
 \begin{itemize}
 \item For each object $x$ of $\Cat$, a chain complex $G(x)$.
 \item For each $n\geq 1$ and each sequence 
   $x_0\stackrel{f_1}{\longrightarrow}\cdots\stackrel{f_n}{\longrightarrow} x_n$
   of composable morphisms a chain map $G(f_n,\dots,f_1)\co
   I_*^{\otimes (n-1)}\otimes G(x_0)\to G(x_n)$
 \end{itemize}
 such that
 \begin{multline*}
   G(f_n,\dots,f_1)(t_1\otimes\dots\otimes t_{n-1})\\
   =
   \begin{cases}
     G(f_n,\dots,f_2)(\pi(t_1)\otimes t_{2}\otimes\dots\otimes t_{n-1}) & f_1=\Id\\
     G(f_n,\dots,f_{i+1},f_{i-1},\dots,f_1)(t_{1} \otimes \dots \otimes m(t_{i-1}\otimes t_{i}) \otimes \dots \otimes t_{n-1}) & f_i=\Id, 1<i<n\\
     G(f_{n-1},\dots,f_1)(t_{1}\otimes\dots\otimes t_{n-2}\otimes \pi(t_{n-1})) & f_n=\Id\\
     G(f_n,\dots,f_{i+1}\circ f_i,\dots,f_1)(t_{1}\otimes\dots \otimes t_{i-1}\otimes t_{i+1}\otimes \dots\otimes t_{n-1}) & t_i=\{1\}\\
     [G(f_n,\dots,f_{i+1})(t_{i+1}\otimes\dots\otimes t_{n-1})]\circ [G(f_i,\dots,f_1)(t_{1}\otimes\dots\otimes t_{i-1})] & t_i=\{0\}.
   \end{cases}
 \end{multline*}
 
 (Here, $(t_1\otimes\dots\otimes t_{n-1})\in I_*^{\otimes(n-1)}$.)

 As in the case of diagrams in $\ol{\JSpace}$, we will usually drop
 the words ``homotopy coherent'', and refer to homotopy coherent diagrams in chain complexes simply as diagrams
 $G\co \Cat\to \Complexes$.
\end{definition}

Observe that a homotopy coherent $\Cat$-diagram of chain complexes is
determined by the complexes $G(x)$ and the maps 
\[
G_{f_n,\dots,f_1} = G(f_n,\dots,f_1)(\{0,1\}\otimes\cdots\otimes\{0,1\})\co G(x)_m\to G(y)_{m+n-1}
\]
where $x$ is the source of $f_1$, $y$ is the target of $f_n$, and the
subscripts on $G(x)$ and $G(y)$ indicate the gradings. These maps satisfy compatibility conditions
\begin{itemize}
\item $G_{\Id_{x_0}} = \Id_{G(x_0)}$
\item $G_{f_n,\dots,f_1} = 0$ if $n>1$ and any $f_i = \Id$

\item $\partial \circ G_{f_n,\dots,f_1} = G_{f_n,\dots,f_1} \circ \partial + \sum_{i=1}^{n-1} G_{f_n,\dots, f_{i+1} \circ f_i, \dots,f_1} + \sum_{i=1}^{n-1} G_{f_n,\dots,f_{i+1}} \circ G_{f_i,\dots,f_1}$.
\end{itemize}

See Example~\ref{eg:coherent-diag-complexes} for the first few terms
in a homotopy coherent $\ECat\ZZ/2$-diagram in $\Complexes$. 

A homotopy coherent diagram $G$ is called \emph{strict} if
$G_{f_n,\dots,f_1}=0$ whenever $n>1$.

Given a sufficiently generic $\Cat$-diagram $F$ in $\ol{\JSpace}$,
Floer theory produces a homotopy coherent diagram of chain
complexes.
\begin{construction}\label{const:F-to-G}
  Fix a sufficiently generic $F\co \Cat\to \ol{\JSpace}$. We build a
  diagram $G\co \Cat\to \Complexes$ as follows:
  \begin{itemize}
  \item For each object $a\in\ob(\Cat)$, let
    $G(a)=(\CF(L_0,L_1),\bdy_{F(a)})$ be the Lagrangian intersection
    Floer chain complex computed with respect to the almost complex
    structure $F(a)$.
  \item Fix a sequence $(f_1,\dots,f_n)$ of composable morphisms
    and a generator $[\sigma]\in I_*^{\otimes (n-1)}$
    corresponding to some $k$-dimensional face
    $\sigma$ of $[0,1]^{n-1}$. There is an associated $k$-parameter
    family of almost complex structures
    $F(f_n,\dots,f_1)|_\sigma$. For each intersection point $x\in
    L_0\cap L_1$ define
    \[
    G(f_n,\dots,f_1)(\sigma\otimes x) = \sum_{y\in L_0\cap L_1}\sum_{\substack{\phi\in\pi_2(x,y)\\ \mu(\phi)=1-n}}\#\cM(\phi;F(f_n,\dots,f_1)|_\sigma)y.
    \]
  \end{itemize}
\end{construction}

Note that, while we will construct the equivariant Floer cohomology
below, the chain complexes $\CF(L_0,L_1)$ compute the Lagrangian
intersection Floer homology: we will dualize later. 

\begin{remark}\label{remark:there-is-relative-z-grading}
  Recall that we have fixed a collection of homotopy classes $\eta$ of
  paths from $L_0$ to $L_1$, and by Floer complex we mean the summand
  generated by intersection points in the classes
  $\eta$. Consequently, by Hypothesis~\ref{hyp:Floer-defined}, the
  homotopy coherent diagram $G$ from Construction~\ref{const:F-to-G}
  is a direct sum of relatively $\ZZ$-graded diagrams in the following
  sense: 
  There is a decomposition at every vertex,
  $G(a)=\oplus_{\spinc\in\eta}G(a)_{(\spinc)}$, and at every map,
  $G(f_n,\dots,f_1)=\oplus_{\spinc\in\eta}G(f_n,\dots,f_1)_{(\spinc)}$. For
  each $\spinc\in\eta$, each of the groups $G(a)_{(\spinc)}$ is
  $\ZZ$-graded so that each of the maps $G(f_n,\dots,f_1)_{(\spinc)}$
  is grading preserving; and this $\ZZ$-grading is well-defined up to
  an overall translation (depending on $\spinc$).
\end{remark}

\begin{lemma}
  Construction~\ref{const:F-to-G} defines a homotopy coherent diagram
  of chain complexes.
\end{lemma}
\begin{proof}
  This follows by standard arguments, by considering the ends of
  \[
  \bigcup_{\mu(\phi)=2-n}\#\cM(\phi;F(f_n,\dots,f_1)|_\sigma).
  \]
  Note that bubbling of spheres or disks is precluded by
  Hypothesis~\ref{item:J-1}, and curves can not escape to infinity by
  Hypothesis~\ref{item:J-2}.
\end{proof}

Finally, we can turn a homotopy coherent diagram of chain complexes
into a single chain complex by taking the homotopy colimit
(compare~\cite[Paragraph (5.10)]{Vogt73:hocolim}), an analogue of the
direct limit which behaves well with respect to homotopies and
homotopy coherence:
\begin{definition}\label{def:hocolim}
  Given a homotopy coherent diagram $G\co\Cat\to\Complexes$, the
  \emph{homotopy colimit} of $G$ is defined by
  \begin{equation}\label{eq:chain-hocolim-def}
  \hocolim G = \bigoplus_{n\geq 0}
  \bigoplus_{a_0\stackrel{f_1}{\longrightarrow}\cdots\stackrel{f_n}{\longrightarrow}a_n}
  I_*^{\otimes n}\otimes G(a_0)/\sim,
  \end{equation}
  where the coproduct is over $n$-tuples of composable morphisms in
  $\Cat$ and the case $n=0$ corresponds to the objects
  $a_0\in\ob(\Cat)$.  The equivalence relation $\sim$ is given by
  \begin{multline*}
  (f_n,\dots,f_1;t_1\otimes\dots\otimes t_n;x)\\
  \sim
  \begin{cases}
    (f_n,\dots,f_2;\pi(t_1)\otimes t_2\otimes\dots\otimes t_n;x) &  f_1=\Id\\
    (f_n,\dots,f_{i+1},f_{i-1},\dots,f_1;t_1\otimes\dots\otimes m(t_{i-1},t_{i})\otimes\dots \otimes t_n;x) & f_i=\Id,\ i>1\\
    (f_n,\dots,f_{i+1}\circ f_i,\dots,f_1;t_1\otimes\dots\otimes t_{i-1}\otimes t_{i+1} \otimes\dots \otimes t_n;x) & t_i=\{1\},\ i<n\\
    (f_{n-1},\dots,f_1;t_{1}\otimes\dots\otimes t_{n-1};x) & t_n=\{1\}\\
    (f_n,\dots,f_{i+1};t_{i+1}\otimes\dots\otimes t_{n};G(f_i,\dots,f_1)(t_{1}\otimes\dots\otimes t_{i-1} \otimes x)) & t_i=\{0\}.
  \end{cases}
  \end{multline*}
  The differential is induced by the tensor product differential in
  Formula~\eqref{eq:chain-hocolim-def}, i.e.,
  \[
  \bdy((f_n,\dots,f_1;t_1\otimes\dots\otimes
  t_n;x))=(f_n,\dots,f_1;\bdy(t_1\otimes\dots\otimes t_n);x)+(f_n,\dots,f_1;t_1\otimes\dots\otimes t_n;\bdy(x)).
  \]

  The homotopy colimit has a filtration $\Filt$ where $\Filt_n\hocolim
  G$ is the span of the expressions $(f_m,\dots, f_1;
  \{0,1\}\otimes\dots\otimes\{0,1\}; x)$, $m\leq n$, i.e., the length
  $\leq n$ composable sequences with all of the $t_i$ the $1$-simplex
  $\{0,1\}$. (The special case $\Filt_0\hocolim G$ is
  $\bigoplus_{a\in\ob(\Cat)}G(a)$.)
\end{definition}

See Example~\ref{eg:hocolim} for an illustration of homotopy colimits
of complexes.

The homotopy colimit $\hocolim G$ of a diagram from
Construction~\ref{const:F-to-G} is a relatively $\ZZ$-graded (in the
sense of Remark~\ref{remark:there-is-relative-z-grading}),
bounded-below chain complex.

\subsection{Building the equivariant cochain complex}\label{sec:build-equi-cx}
Throughout this section we fix a symplectic manifold $(M,\omega)$, a
symplectic involution $\tau\co M\to M$, Lagrangians $L_0,L_1\subset M$
fixed by $\tau$, and a collection of homotopy classes of paths $\eta$ from $L_0$ to
$L_1$, such that this data satisfies Hypothesis~\ref{hyp:Floer-defined}.

\begin{lemma}\label{lem:exist-equi-diagram}
  There is a sufficiently generic homotopy coherent diagram $F\co
  \ECat\ZZ/2\to \ol{\JSpace}$ which is $\tau$-equivariant, i.e., so
  that $F(b)=\tau_*(F(a))=d\tau\circ F(a)\circ d\tau$ and for any
  sequence $(f_1,\dots,f_n)$ of composable morphisms and any
  $(t_1,\dots,t_{n-1})\in [0,1]^{n-1}$,
  \[
  \tau_*\left(F(f_n,\dots,f_1)(t_1,\dots,t_{n-1})\right) 
  = F(\tau(f_n),\dots,\tau(f_1))(t_1,\dots,t_{n-1}).
  \]
\end{lemma}
\begin{proof}
  The construction of the $F(f_n,\dots,f_1)$ is inductive in $n$,
  using the facts that $\ol{\JSpace}$ is weakly contractible;
  Lemma~\ref{lem:exist-suff-generic}; and the fact that the action of
  $\tau$ on $\ECat\ZZ/2$ (and hence on sequences of morphisms in
  $\ECat\ZZ/2$) is free. Details are left to the reader.
\end{proof}

The equivariance condition implies that for any sequence
$(\alpha,\beta,\dots)$, say, of composable morphisms in $\ECat$ and
any $x\in \CF(L_0,L_1)=G(a)$,
\[
G_{\dots,\beta,\alpha}(x) = \tau_\#(G_{\dots,\alpha,\beta}(\tau_\#(x))),
\]
where $\tau_\#\co \CF(L_0,L_1)\to\CF(L_0,L_1)$ is the map induced by
$\tau$ (i.e., takes a point $x\in L_0\cap L_1$ to $\tau(x)$). (The map
$\tau_\#$ is a chain map from $\CF(L_0,L_1;\bdy_{F(a)})$ to $\CF(L_0,L_1;\bdy_{F(b)})$.)

\begin{lemma}\label{lem:ECF-is-CF}
  For any sufficiently generic diagram $F\co \ECat\ZZ/2\to \ol{\JSpace}$ with
  associated diagram $G\co \ECat\ZZ/2\to \Complexes$ as in
  Construction~\ref{const:F-to-G}, the complex $\hocolim G$ is
  quasi-isomorphic to $G(a)=(\CF(L_0,L_1),\bdy_{F(a)})$.
\end{lemma}
\begin{proof}
  Since $a$ is a terminal object of $\ECat\ZZ/2$,
  $\hocolim_{\ECat\ZZ/2} G\simeq G(a)$~\cite[Example
  XII.3.1]{BousfieldKan72}. (One can also give a concrete proof using
  the explicit description of $\hocolim_{\ECat\ZZ/2} G$ from
  Observation~\ref{obs:concrete-hocolim}, below.)
\end{proof}

\begin{definition}\label{def:equi-Floer}
  Fix a $\tau$-equivariant, sufficiently generic $F\co \ECat\ZZ/2\to
  \ol{\JSpace}$ and let $G\co \ECat\ZZ/2\to \Complexes$ be as in
  Construction~\ref{const:F-to-G}.  We define the \emph{freed Floer complex} of
  $(L_0,L_1)$ to be
  \[
  \ECF(L_0,L_1)\coloneqq \hocolim G.
  \]
  By Lemma~\ref{lem:ECF-is-CF}, $\ECF(L_0,L_1)$ is quasi-isomorphic, over $\Field$, to $\CF(L_0,L_1)$.
  The action of $\ZZ/2$ on $\ECat\ZZ/2$ makes
  $\ECF(L_0,L_1)$ into a chain complex of free modules over $\Field[\ZZ/2]$. The
  \emph{equivariant Floer cochain complex} of $(L_0,L_1)$ is the
  complex
  \[
  \eCF(L_0,L_1)\coloneqq\HomO{\Field[\ZZ/2]}(\hocolim G, \Field).
  \]
  Let $\eHF(L_0,L_1)$ be the homology of $\eCF(L_0,L_1)$, i.e., the
  \emph{equivariant Floer cohomology} of $(M,L_0,L_1,\tau)$. Since the
  complex $\ECF(L_0,L_1)$ is bounded below and free over
  $\Field[\ZZ/2]$, so $\eHF(L_0,L_1)$ is the same as
  $\ExtO{\Field[\ZZ/2]}(\hocolim G,\Field)$.

  The filtration $\Filt$ on $\hocolim G$ is preserved by the
  $\Field[\ZZ/2]$-action, and hence induces a filtration on
  $\eCF(L_0,L_1)$, which we also denote $\Filt$. 
\end{definition}

\begin{observation}\label{obs:concrete-hocolim}
In Section~\ref{sec:equi-equi} we will need a somewhat more concrete
description of parts of the complexes $\ECF(L_0,L_1)$ and $\eCF(L_0,L_1)$. So, fix
diagrams $F$ and $G$ as in Definition~\ref{def:equi-Floer}. The
homotopy colimit $\hocolim G=\ECF(L_0,L_1)$ has a basis consisting of elements of
the form $(f_n,\dots,f_1;\{0,1\}\otimes\cdots\otimes\{0,1\};x)$ where
$x\in L_0\cap L_1$ and $(f_1,\dots,f_n)$ is a sequence of composable,
non-identity elements. Thus, the elements have the form 
\[
\alpha_n\otimes x\coloneqq (\alpha,\beta,\alpha,\beta,\dots; \{0,1\},\dots,\{0,1\};x)
\text{ or }
\beta_n\otimes x\coloneqq (\beta,\alpha,\beta,\alpha,\dots; \{0,1\},\dots,\{0,1\};x)
\]
where the string of $\alpha$'s and $\beta$'s has length $n\geq 0$. The
differential is given by
\begin{align}
\bdy(\alpha_n\otimes x)&=\alpha_n\otimes (\bdy x)+\beta_{n-1}\otimes x + \sum_{i=1}^n\alpha_{n-i}\otimes
\begin{cases}
  G_{\beta,\alpha,\dots}(x) & i\equiv n+1\pmod{2}\\
  G_{\alpha,\beta,\dots}(x) & i\equiv n\pmod{2}
\end{cases}\label{eq:d-of-alpha}\\
\bdy(\beta_n\otimes x)&=\beta_n\otimes (\bdy x)+\alpha_{n-1}\otimes x + \sum_{i=1}^n\beta_{n-i}\otimes
\begin{cases}
  G_{\alpha,\beta,\dots}(x) & i\equiv n+1\pmod{2}\\
  G_{\beta,\alpha,\dots}(x) & i\equiv n\pmod{2}
\end{cases}\label{eq:d-of-beta}
\end{align}
where the sequence of morphisms appearing in $G_{\alpha,\beta,\dots}$
or $G_{\beta,\alpha,\dots}$ in the $i\th$ term has length $i$.  The
second term in $\bdy(\alpha_n \otimes x)$ come from the fourth case in
the definition of $\sim$. The third term in $\bdy(\alpha_n\otimes x)$
comes from the last case in the definition of $\sim$. The expression
$\bdy x$ means the differential with respect to the complex structure
$F(a)$ or $F(b)$, depending on whether $x\in G(a)$ or $G(b)$. The
$\ZZ/2$-action exchanges $\alpha_n\otimes x$ and $\beta_n\otimes
\tau_\#(x)$.

Passing to $\eCF(L_0,L_1)$ dualizes the complex and identifies
the elements $(\alpha_n\otimes x)^*$ and $(\beta_n\otimes
\tau_\#(x))^*$. We can also write $(\alpha_n\otimes x)^*$ as
$\alpha_n^*\otimes x^*$; note that there is a Floer cochain
differential $d(x^*)$, the dual (over $\ZZ$) to the differential on
the Floer chain complex, and a cochain element $\tau^\#(x^*)=x^*\circ \tau_\#$.
 The differential on the equivariant cochain complex is given by
\begin{equation}\label{eq:diff-on-equicx}
d(\alpha_n^*\otimes x^*)=\alpha_n^*\otimes (d x^*)+\alpha_{n+1}^*\otimes \tau^\#(x^*) + 
\sum_{i=1}^{\infty}\alpha_{n+i}^*\otimes 
\begin{cases}
  x^*\circ G_{\beta,\alpha,\dots} & n\equiv 1\pmod{2}\\
  x^*\circ G_{\alpha,\beta,\dots} & n\equiv 0\pmod{2}.
\end{cases}
\end{equation}
The fact that $\CF(L_0,L_1)$ is finite-dimensional and $\ZZ$-graded
implies that the sum in Formula~\eqref{eq:diff-on-equicx} is finite.

There is an action of $\Field[\theta]$ on $\eCF(L_0,L_1)$ by
$\theta\cdot (\alpha_n^*\otimes x^*)=\alpha_{n+1}^*\otimes
\tau^{\#}(x^*) = \beta_{n+1}^* \otimes x^*$.
\end{observation}

\begin{lemma}
  The homology of the associated graded complex to
  $(\eCF(L_0,L_1),\Filt)$ is given by
  $\HF^*(L_0,L_1)\otimes_\Field \Field[\theta]$. In particular, there
  is a spectral sequence $\HF^*(L_0,L_1)\otimes_\Field
  \Field[\theta]\Rightarrow\eHF(L_0,L_1)$ from the Floer
  cohomology of $(L_0,L_1)$ tensored with $\Field[\theta]$ to the
  equivariant Floer cohomology. 
  Further, both this spectral sequence and the module structure over
  $\Field[\theta]$ agree with the spectral sequence and module structure
  described in Section~\ref{sec:background} for the chain complex
  $\ECF(L_0,L_1)$ over $\Field[\ZZ/2]$.
\end{lemma}
\begin{proof}
  The explicit description in Observation~\ref{obs:concrete-hocolim}
  implies the spectral sequence has the specified form: we have
  $\Filt(\alpha_n^* \otimes x^*)=n$, and the only term of
  $d(\alpha_n^*\otimes x^*)$ in filtration $n$ is $\alpha_n^*\otimes(d
  x^*)$. 
  
  To see that this agrees with the spectral sequence in
  Section~\ref{sec:background}, first note that the homotopy coherent
  diagram $G$ is equivalent to a strict diagram
  $\wt{G}\co\ECat\ZZ/2\to\Complexes$. That is, there is a strict
  diagram $\wt{G}\co\ECat\ZZ/2\to\Complexes$ and a homotopy coherent
  diagram $H\co \ICat\times\ECat\ZZ/2\to \Complexes$, where
  $\ICat=\left(\xymatrix{ 0\ar[r] & 1}\right)$, so that
  $H|_{\{1\}\times \ECat\ZZ/2}=\wt{G}$ and
  $H|_{\{0\}\times\ECat\ZZ/2}=G$~\cite[Corollary
  4.5]{CordierPorter86:Vogt}.  Further, it follows from the
  construction that we can arrange for $\wt{G}$ and $H$ to be
  $\ZZ/2$-equivariant. Associated to $H$,
  then, is a quasi-isomorphism $\ol{H}\co \hocolim{G}\to\hocolim \wt{G}$
  of complexes over $\Field[\ZZ/2]$. It follows from the construction
  of $\ol{H}$ that it induces a filtered, $\Field[\theta]$-equivariant
  map $\ol{H}^*\co \HomO{\Field[\ZZ/2]}(\hocolim \wt{G},\Field)\to
  \HomO{\Field[\ZZ/2]}(\hocolim G,\Field)$.  (The
  differential on $\hocolim\wt{G}$ is given by a formula analogous
  to Equation~\eqref{eq:diff-on-equicx}, with the simplification that 
  the sum over $i$ only involves the single term $i=1$.)

  Thus, it suffices to prove the result when $G$ is a strict
  diagram. In this case, the complex $\hocolim G$ is given by
  \[
  \begin{tikzpicture}[x=1.5cm, y=1.5cm]
  \node at (0,0) (Ca1) {$C_a$};
  \node at (0,-.5) (oplus1) {$\oplus$};
  \node at (0,-1) (Cb1) {$C_b$};
  \node at (2,0) (Ca2) {$C_a$};
  \node at (2,-.5) (oplus2) {$\oplus$};
  \node at (2,-1) (Cb2) {$C_b$};
  \node at (4,0) (Ca3) {$\cdots$};
  \node at (4,-.5) (oplus3) {$\cdots$};
  \node at (4,-1) (Cb3) {$\cdots$};
  \draw[->] (Ca2) to (Ca1);
  \draw[->, dashed, bend right=10] (Ca2) to node[above]{\lab{G_\alpha}} (Cb1);
  \draw[->] (Cb2) to (Cb1);
  \draw[->, dashed, bend left=10] (Cb2) to node[below]{\lab{G_\beta}} (Ca1);
  \draw[->] (Ca3) to (Ca2);
  \draw[->, dashed, bend right=10] (Ca3) to node[above]{\lab{G_\alpha}} (Cb2);
  \draw[->] (Cb3) to (Cb2);
 \draw[->, dashed, bend left=10] (Cb3) to node[below]{\lab{G_\beta}} (Ca2);
  \end{tikzpicture}
  \]
  where $C_a=G(a)$, $C_b=G(b)$, and all solid arrows are the identity map. (We have not drawn the internal differentials of $C_a$ and $C_b$.)
  The complex $\eCF(L_0,L_1)$ is, thus, 
  \begin{equation}\label{eq:single-complex}
  \mathcenter{\begin{tikzpicture}[x=1.25cm, y=1cm]
  \node at (1,0) (zero) {$0$};
  \node at (2,0) (Ca1) {$C_a^*$};
  \node at (4,0) (Ca2) {$C_a^*$};
  \node at (6,0) (Ca3) {$\cdots$};
  \draw[->] (zero) to (Ca1);
  \draw[->] (Ca1) to node[above]{\lab{\Id+\tau^{\#}\circ G_{\beta}^*}} (Ca2);
  \draw[->] (Ca2) to node[above]{\lab{\Id+\tau^{\#}\circ G_{\beta}^*}} (Ca3);
  \end{tikzpicture}}
  \end{equation}
  This complex is filtered by the columns, and the action of $\theta$
  moves elements one step to the right. (In the notation of Observation~\ref{obs:concrete-hocolim}, the copy of $C_a^*$
  in the $n$th filtration level is generated by the elements
  $\{\beta_{n}^* \otimes x^*\}$, for $x \in L_0 \cap L_1$, if $n$ is
  even, and $\{\alpha_n^* \otimes x^*\}$ if $n$ is odd.) By contrast,
  tensoring $\hocolim G$ with the standard resolution of the
  $\Field[\ZZ/2]$-module $\Field$ (Section~\ref{sec:background}) gives
  \[
  \begin{tikzpicture}[x=1.5cm, y=1.5cm]
  \node at (0,3) (Ca01) {$1C_a$};
  \node at (.5,3) (oplus01a) {$\oplus$};
  \node at (1,3) (tCa01) {$\tau C_a$};
  \node at (0,2.5) (oplus01b) {$\oplus$};
  \node at (1, 2.5) (oplus01c) {$\oplus$};
  \node at (0,2) (Cb01) {$1C_b$};
  \node at (0.5, 2) (oplus01d) {$\oplus$};
  \node at (1, 2) (tCb01) {$\tau C_b$};
  \node at (0,0) (Ca00) {$1C_a$};
  \node at (.5,0) (oplus00a) {$\oplus$};
  \node at (1,0) (tCa00) {$\tau C_a$};
  \node at (0,-.5) (oplus00b) {$\oplus$};
  \node at (1, -.5) (oplus00c) {$\oplus$};
  \node at (0,-1) (Cb00) {$1C_b$};
  \node at (0.5, -1) (oplus00d) {$\oplus$};
  \node at (1, -1) (tCb00) {$\tau C_b$};
  \node at (3,3) (Ca11) {$1C_a$};
  \node at (3.5,3) (oplus11a) {$\oplus$};
  \node at (4,3) (tCa11) {$\tau C_a$};
  \node at (3,2.5) (oplus11b) {$\oplus$};
  \node at (4, 2.5) (oplus11c) {$\oplus$};
  \node at (3,2) (Cb11) {$1C_b$};
  \node at (3.5, 2) (oplus01d) {$\oplus$};
  \node at (4, 2) (tCb11) {$\tau C_b$};
  \node at (3,0) (Ca10) {$1C_a$};
  \node at (3.5,0) (oplus10a) {$\oplus$};
  \node at (4,0) (tCa10) {$\tau C_a$};
  \node at (3,-.5) (oplus10b) {$\oplus$};
  \node at (4, -.5) (oplus10c) {$\oplus$};
  \node at (3,-1) (Cb10) {$1C_b$};
  \node at (3.5, -1) (oplus10d) {$\oplus$};
  \node at (4, -1) (tCb10) {$\tau C_b$};
  \node at (5,-.5) (dots1) {$\cdots$};
  \node at (5,2.5) (dots2) {$\cdots$};
  \node at (.5, 4) (dots3) {$\vdots$};
  \node at (3.5, 4) (dots4) {$\vdots$};
  \draw[->, bend right=20] (Ca01) to (Ca00);
  \draw[->, bend left=20] (tCa01) to (tCa00);
  \draw[->, bend right=20] (Cb01) to (Cb00);
  \draw[->, bend left=20] (tCb01) to (tCb00);
  \draw[->] (Ca01) to (tCa00);
  \draw[->] (Cb01) to (tCb00);
  \draw[->] (tCa01) to (Ca00);
  \draw[->] (tCb01) to (Cb00);
  \draw[->, bend right=20] (Ca11) to (Ca10);
  \draw[->, bend left=20] (tCa11) to (tCa10);
  \draw[->, bend right=20] (Cb11) to (Cb10);
  \draw[->, bend left=20] (tCb11) to (tCb10);
  \draw[->] (Ca11) to (tCa10);
  \draw[->] (Cb11) to (tCb10);
  \draw[->] (tCa11) to (Ca10);
  \draw[->] (tCb11) to (Cb10);
  \draw[->, bend right=20] (Ca10) to (Ca00);
  \draw[->, bend right=20] (tCa10) to (tCa00);
  \draw[->, bend left=20] (Cb10) to (Cb00);
  \draw[->, bend left=20] (tCb10) to (tCb00);
  \draw[->, dashed] (Ca10) to (Cb00);
  \draw[->, dashed] (Cb10) to (Ca00);
  \draw[->, dashed] (tCa10) to (tCb00);
  \draw[->, dashed] (tCb10) to (tCa00);
  \draw[->, bend right=20] (Ca11) to (Ca01);
  \draw[->, bend right=20] (tCa11) to (tCa01);
  \draw[->, bend left=20] (Cb11) to (Cb01);
  \draw[->, bend left=20] (tCb11) to (tCb01);
  \draw[->, dashed] (Ca11) to (Cb01);
  \draw[->, dashed] (Cb11) to (Ca01);
  \draw[->, dashed] (tCa11) to (tCb01);
  \draw[->, dashed] (tCb11) to (tCa01);
  \end{tikzpicture}
  \]
  Here, we have identified $\Field[\ZZ/2]=\Field[\tau]/(\tau^2=1)$ and
  written, e.g., $\tau C_a$ for $\tau\otimes C_a$. The action of
  $\tau$ exchanges $1C_a$ and $\tau C_b$ and exchanges $1C_b$ and
  $\tau C_a$ in each block. Solid arrows are, again, identity maps,
  while dashed arrows are $G_\alpha$ or $G_\beta$, depending on
  whether the source is $C_a$ or $C_b$.
  
  Taking $\Hom$ over $\Field[\ZZ/2]$ to $\Field$ gives
  \begin{equation}\label{eq:double-cx}
  \mathcenter{
  \begin{tikzpicture}[x=1.5cm, y=1cm]
  \node at (0,0) (Ca00) {$C_a^*$};
  \node at (0,-.5) (oplus00) {$\oplus$};
  \node at (0,-1) (Cb00) {$C_b^*$};
  \node at (2,0) (Ca10) {$C_a^*$};
  \node at (2,-.5) (oplus10) {$\oplus$};
  \node at (2,-1) (Cb10) {$C_b^*$};
  \node at (0,3) (Ca01) {$C_a^*$};
  \node at (0,2.5) (oplus01) {$\oplus$};
  \node at (0,2) (Cb01) {$C_b^*$};
  \node at (2,3) (Ca11) {$C_a^*$};
  \node at (2,2.5) (oplus11) {$\oplus$};
  \node at (2,2) (Cb11) {$C_b^*$};
  \node at (3,2.5) (dots1) {$\cdots$};
  \node at (3,-.5) (dots2) {$\cdots$};
  \node at (0,4) (dots3) {$\vdots$};
  \node at (2,4) (dots4) {$\vdots$};
  \draw[->] (Ca00) to (Ca10);
  \draw[->] (Cb00) to (Cb10);
  \draw[->, dashed] (Ca00) to (Cb10);
  \draw[->, dashed] (Cb00) to (Ca10);
  \draw[->] (Ca01) to (Ca11);
  \draw[->] (Cb01) to (Cb11);
  \draw[->, dashed] (Ca01) to (Cb11);
  \draw[->, dashed] (Cb01) to (Ca11);
  \draw[->, dotted, bend right=20] (Ca00) to (Cb01);
  \draw[->, dotted, bend right=20] (Cb00) to (Ca01);   
  \draw[->, bend left=20] (Ca00) to (Ca01);
  \draw[->, bend left=20] (Cb00) to (Cb01);
  \draw[->, dotted, bend right=20] (Ca10) to (Cb11);
  \draw[->, dotted, bend right=20] (Cb10) to (Ca11);   
  \draw[->, bend left=20] (Ca10) to (Ca11);
  \draw[->, bend left=20] (Cb10) to (Cb11);
  \end{tikzpicture}}
  \end{equation}
  where the solid arrows are again identity maps and the dashed arrows are $G_\alpha^*$ or $G_\beta^*$, depending on whether the source is $C_b^*$ or $C_a^*$. Dotted arrows indicate the map $\tau^\#$.
  There is a quasi-isomorphism from the complex~\eqref{eq:single-complex} to the complex~\eqref{eq:double-cx}, given by
\[
  \begin{tikzpicture}[x=1.5cm, y=1cm]
  \node at (0,0) (Ca00) {$C_a^*$};
  \node at (0,-.5) (oplus00) {$\oplus$};
  \node at (0,-1) (Cb00) {$C_b^*$};
  \node at (2,0) (Ca10) {$C_a^*$};
  \node at (2,-.5) (oplus10) {$\oplus$};
  \node at (2,-1) (Cb10) {$C_b^*$};
  \node at (0,3) (Ca01) {$C_a^*$};
  \node at (0,2.5) (oplus01) {$\oplus$};
  \node at (0,2) (Cb01) {$C_b^*$};
  \node at (2,3) (Ca11) {$C_a^*$};
  \node at (2,2.5) (oplus11) {$\oplus$};
  \node at (2,2) (Cb11) {$C_b^*$};
  \node at (3,2.5) (dots1) {$\cdots$};
  \node at (3,-.5) (dots2) {$\cdots$};
  \node at (0,4) (dots3) {$\vdots$};
  \node at (2,4) (dots4) {$\vdots$};
  \node at (-2,4) (dots5) {$\vdots$};
  \node at (-2, -.5) (Ca0) {$C_a^*$};
  \node at (-2,2.5) (Ca1) {$C_a^*$};
  \draw[->] (Ca00) to (Ca10);
  \draw[->] (Cb00) to (Cb10);
  \draw[->, dashed] (Ca00) to (Cb10);
  \draw[->, dashed] (Cb00) to (Ca10);
  \draw[->] (Ca01) to (Ca11);
  \draw[->] (Cb01) to (Cb11);
  \draw[->, dashed] (Ca01) to (Cb11);
  \draw[->, dashed] (Cb01) to (Ca11);
  \draw[->, dotted, bend right=20] (Ca00) to (Cb01);
  \draw[->, dotted, bend right=20] (Cb00) to (Ca01);   
  \draw[->, bend left=20] (Ca00) to (Ca01);
  \draw[->, bend left=20] (Cb00) to (Cb01);
  \draw[->, dotted, bend right=20] (Ca10) to (Cb11);
  \draw[->, dotted, bend right=20] (Cb10) to (Ca11);   
  \draw[->, bend left=20] (Ca10) to (Ca11);
  \draw[->, bend left=20] (Cb10) to (Cb11);
  \draw[->] (Ca0) to node[left]{\lab{\Id+\tau^{\#}\circ G_{\beta}^*}} (Ca1);
  \draw[->] (Ca0) to node[above]{\lab{\Id}} (Ca00);
  \draw[->, dashed] (Ca0) to node[below]{\lab{G_\beta^{*}}} (Cb00);
  \draw[->] (Ca1) to node[above]{\lab{\Id}} (Ca01);
  \draw[->, dashed] (Ca1) to node[below]{\lab{G_\beta^{*}}} (Cb01);
  \end{tikzpicture}
\]
To see that this map is an isomorphism, filter by the $y$-coordinate and note that the map induces an isomorphism on the $E^1$-page of the associated spectral sequence.
The quasi-isomorphism intertwines the actions of $\Field[\theta]$ and the filtrations, proving the result.
\end{proof}

\subsection{Invariance}\label{subsec:invariance}
\begin{proposition}\label{prop:indep-of-cx-str}
  Up to quasi-isomorphism of complexes over $\Field[\ZZ/2]$, the freed Floer complex $\ECF(L_0,L_1)$ is 
  independent of the choice of $\tau$-equivariant sufficiently
  generic homotopy coherent diagram in $\ol{\JSpace}$ used to define
  it. 
\end{proposition}
As discussed in Section~\ref{sec:background}, it follows that
$\eCF(L_0,L_1)$ is well-defined up to filtered quasi-isomorphism, and
that $\eHF(L_0,L_1)$ is well-defined up to isomorphism over
$\Field[\theta]$.
\begin{proof}
  Let $\ICat$ be the category with two objects, $0$ and $1$, and
  a single morphism $f_{0,1}\co 0\to 1$:
  \[
  \ICat = \xymatrix{ 0 \ar@(dl, ul)^{\Id} \ar[r]^{f_{0,1}}& 1\ar@(dr, ur)_{\Id}}.
  \]
  Given homotopy coherent $\Cat$-diagrams $G_0$ and $G_1$ in
  $\ol{\JSpace}$ (respectively $\Complexes$), a \emph{morphism} from
  $G_0$ to $G_1$ is a homotopy coherent
  $\ICat\times\Cat$-diagram $H$ in $\ol{\JSpace}$ (respectively
  $\Complexes$) such that $H|_{\{i\}\times \Cat}=G_i$ for $i=0,1$. For
  a morphism $H$ of diagrams in $\Complexes$, we will call $H$ a
  \emph{quasi-isomorphism} if
  \[
  H(f_{0,1}\times \Id_x)\co G_0(x)\to G_1(x)
  \]
  is a quasi-isomorphism. It is not hard to show that morphisms of
  diagrams induce morphisms of homotopy colimits and that if $H$ is a
  quasi-isomorphism then the induced map $\hocolim H\co \hocolim
  G_0\to \hocolim G_1$ is a quasi-isomorphism. (See
  Vogt~\cite[Proposition (4.6) and Theorem (5.12)]{Vogt73:hocolim} for
  the analogous result in the topological case and
  Cordier-Porter~\cite{CordierPorter86:Vogt} for a setting which
  includes chain complexes.)

  Now, fix sufficiently generic, equivariant $\ECat\ZZ/2$-diagrams
  $F_0$ and $F_1$ in $\ol{\JSpace}$, and let $G_0,G_1\co
  \ECat\ZZ/2\to\Complexes$ be the maps given by
  Construction~\ref{const:F-to-G}.  By the same argument as in the
  proof of Lemma~\ref{lem:exist-equi-diagram}, we can extend $F_0$ and
  $F_1$ to a diagram $H\co \ICat\times \ECat\ZZ/2\to
  \ol{\JSpace}$ which is $\tau$-equivariant (where $\tau$ acts
  trivially on $\ICat$). Applying Construction~\ref{const:F-to-G}
  to $H$ gives a diagram $K\co
  \ICat\times\ECat\ZZ/2\to\Complexes$.

  In any diagram $G\co \Cat\to\Complexes$ obtained from
  Construction~\ref{const:F-to-G}, all edges map to
  quasi-isomorphisms: an edge is sent to a Floer continuation map
  associated to a (sufficiently generic) path of almost complex
  structures. In particular, the diagram $K$ is a quasi-isomorphism
  from $G_0$ to $G_1$. So, as above, $K$ induces a quasi-isomorphism
  \[
  \hocolim K\co \hocolim G_0\to \hocolim G_1
  \]
  and, since $H$ is $\ZZ/2$-equivariant, $\hocolim K$
  is a map of $\Field[\ZZ/2]$-modules. 
\end{proof}

We turn next to isotopy invariance. First, recall that we fixed a
$\ZZ/2$-invariant collection $\eta$ of homotopy classes of paths from
$L_0$ to $L_1$, and $\ECF(L_0,L_1)$ is defined using only intersection
points that belong to these homotopy classes. Given an isotopy from
$L_0$ (respectively $L_1$) to another Lagrangian $L'_0$ (respectively
$L'_1$), the homotopy classes $\eta$ induces a collection of homotopy
classes of paths $\eta'$ from $L_0'$ to $L_1'$. If $\eta$ is preserved
by the $\ZZ/2$-action and the isotopies are $\ZZ/2$-equivariant then
$\eta'$ is preserved by the $\ZZ/2$-action, as well.

\begin{proposition}\label{prop:indep-of-Ham-isotopy}
  If $L_0'$ and $L_1'$ are isotopic to $L_0$ and $L_1$ via
  compactly-supported, Hamiltonian isotopies through $\tau$-invariant
  Lagrangians then there is a quasi-isomorphism of complexes
  $\ECF(L_0,L_1)\simeq \ECF(L'_0,L'_1)$ over $\Field[\ZZ/2]$. (Here, the $\ECF(L_0,L_1)$ and $\ECF(L'_0,L'_1)$ are defined using corresponding collections of homotopy classes of paths $\eta$ and $\eta'$, respectively.) 
\end{proposition}
\begin{proof}
  The proof is essentially the same as the proof of
  Proposition~\ref{prop:indep-of-cx-str}, but using the Hamiltonian
  isotopy instead of a path of almost complex structures to define the
  map between diagrams. 

  If one of the triple intersections of $L_0$,
  $L_1$, $L'_0$, $L'_1$ is non-empty (i.e., the Lagrangians are not
  triple-wise transverse) there is an additional step, so assume first
  that all triple intersections are empty.
  Consider the topological category $\Dat$ where:
  \begin{itemize}
  \item The objects of $\Dat$ are $\ob(\ol{\JSpace})\amalg
    \ob(\ol{\JSpace})=\{0,1\}\times\ob(\ol{\JSpace})$.
  \item For $(i,J),(i,J')\in\{i\}\times \ob(\ol{\JSpace})$,
    $\Hom_{\Dat}((i,J),(i,J'))=\Hom_{\ol{\JSpace}}(J,J')$ is the space of
    sequences of eventually cylindrical almost complex structures
    for $J$ to $J'$.
  \item $\Hom_\Dat((1,J),(0,J'))=\emptyset$.
  \item $\Hom_\Dat((0,J),(1,J'))$ is given by sequences
    $(\wt{J}_{-i},\dots,\wt{J}_{-1},\wt{J}_0,\wt{J}_1,\dots,\wt{J}_j)$
    of eventually cylindrical almost complex structures with one
    distinguished level, $\wt{J}_0$ where we do \emph{not} quotient by
    $\RR$-translation. (The level $\wt{J}_0$ is where the Hamiltonian
    isotopy will occur.) There is a quotient map
    $\Hom_{\Dat}((0,J),(1,J'))\to \Hom_{\ol{\JSpace}}(J,J')$, by
    modding out by $\RR$ on the distinguished level.
  \end{itemize}
  Given a homotopy coherent diagram $\Cat\to\Dat$, we can apply Floer
  theory to obtain a diagram $G\co\Cat\to\Complexes$ as follows. Given
  an object $x$ of $\Cat$, with $F(x)=(i,J)$, let
  \[
  G(x)=
  \begin{cases}
    \CF(L_0,L_1;J) & i=0\\
    \CF(L'_0,L'_1;J) & i=1.
  \end{cases}
  \]
  Given a morphism $f$ in $\Cat$, if $F(f)=\wt{J}$ is contained in
  $\{0\}\times \ol{\JSpace}$ (respectively $\{1\}\times \ol{\JSpace}$)
  then define $G(f)$ to be the continuation map on $\CF(L_0,L_1)$
  (respectively $\CF(L'_0,L'_1)$) associated to $\wt{J}$. If
  $F(f)=(\wt{J}_{-i},\dots,\wt{J}_{-1},\wt{J}_0,\wt{J}_1,\dots,\wt{J}_j)$
  maps from $(0,J)$ to $(1,J')$ then define $G(f)$ to be the
  composition of the continuation maps associated to the $J_{k}$,
  $k=-i,\dots,-1$, on $\CF(L_0,L_1)$, followed by the continuation map
  associated to the Hamiltonian isotopy from $(L_0,L_1)$ to
  $(L'_0,L'_1)$ computed with respect to the almost complex structure
  $\wt{J}_0$, and then followed by the continuation maps associated to the
  $\wt{J}_{k}$, $k=1,\dots,j$, on $\CF(L'_0,L'_1)$. More generally, given a
  sequence of morphisms $f_1,\dots,f_n$, define $G(f_n,\dots,f_1)$ to
  be as in Construction~\ref{const:F-to-G} if all of the $f_i$ are in
  $\{0\}\times\ol{\JSpace}$ or $\{1\}\times\ol{\JSpace}$, and using
  the Hamiltonian isotopy at the $0$-level together with the family of
  almost complex structures $F(f_n,\dots,f_1)$ to obtain a map
  $\CF(L_0,L_1)\to \CF(L'_0,L'_1)$ if $f_n\circ\dots\circ f_1$ maps
  from $\{0\}\times\ol{\JSpace}$ to $\{1\}\times\ol{\JSpace}$. It is
  clear that for generically chosen functors $F\co \Cat\to\Dat$, $G$
  defines a homotopy coherent diagram of complexes.

  Observe that the $\ZZ/2$-action on $(M,\omega)$ induces a
  $\ZZ/2$-action on $\Dat$. Given $\tau$-equivariant homotopy coherent
  functors $F,F'\co \ECat\ZZ/2\to\ol{\JSpace}$ used to define
  $\ECF(L_0,L_1)$ and $\ECF(L'_0,L'_1)$, respectively, there is no
  obstruction to extending $\{0\}\times F\amalg \{1\}\times F'$ to a
  $\tau$-equivariant homotopy coherent functor $F''\co
  \ICat\times\ECat\ZZ/2\to \Dat$. The construction above then gives a
  homotopy coherent diagram $G''\co \ICat\times\ECat\ZZ/2\to
  \Complexes$. On each edge $f_{0,1}\times\Id_x$, $G''$ is induced by
  a continuation map, and hence is a quasi-isomorphism. So, as in the
  proof of Proposition~\ref{prop:indep-of-cx-str}, $G''$ induces a
  quasi-isomorphism $\ECF(L_0,L_1)=\hocolim G\to \hocolim
  G'=\ECF(L'_0,L'_1)$ over $\Field[\ZZ/2]$.

  This proves the result, under the assumption of no triple
  intersections.  In the case of triple intersections there may be
  constant bigons contained in $L_0\cap L'_0\cap L_1$, say, which are
  not transversely cut out by the $\overline{\partial}$-operator for
  any almost complex structure. (By considering different twistings of
  the diagrams in Section~\ref{sec:new-HFa}, it is easy to construct
  Lagrangians in which such disks exist.)  
  We discuss how to handle this case in more detail in
  Section~\ref{sec:non-transverse}. Briefly,
  extend the
  category $\Dat$ above to include small Hamiltonian perturbations of
  $L_1$ along $\{1\}\times\RR\subset\bdy\bD^2$ in
  the morphisms from $(0,J)$ to $(1,J')$.  With this larger space of
  perturbations we are able to achieve transversality for all homotopy
  classes of triangles, and the proof proceeds as before.
\end{proof}

In fact, under certain circumstances we do not need the isotopy to be
through equivariant Lagrangians. Recall that objects $L$ and $L'$ in
an $\Ainf$-category (such as the Fukaya category) are \emph{homotopy
  equivalent} if there are degree-zero cycles $f\in\Hom(L,L')$ and
$g\in\Hom(L',L)$ such that the compositions $\circ_2(g,f)\in\Hom(L,L)$
and $\circ_2(f,g)\in\Hom(L',L')$ are homotopic to the identity maps of
$L$ and $L'$, respectively.
\begin{proposition}\label{prop:non-equi-invar}
  Suppose that $L_0$, $L'_0$ and $L_1$ are $\tau$-invariant
  Lagrangians, which are pairwise transverse, and that $L_0$ is
  Hamiltonian isotopic to $L'_0$ or, more generally, so that $L_0$ is
  homotopy equivalent to $L_1$ in the Fukaya category. Suppose
  further that:
  \begin{enumerate}
  \item The collection of homotopy classes $\eta'$ induced by $\eta$
    and the Hamiltonian isotopies is preserved by the
    $\ZZ/2$-action. (Of course, $\eta$ is also required to be
    preserved by the $\ZZ/2$-action, per
    Hypothesis~\ref{hyp:Floer-defined}.)
  \item There is a $\tau$-invariant, $\omega$-compatible almost
    complex structure $J$ on $M$ which achieves transversality for all
    moduli spaces of holomorphic bigons with boundary on $(L_0,L'_0)$
    of Maslov index $\leq 1$.
  \item The top class in $\HF(L_0,L'_0)$ (i.e., the isomorphism
    between $L_0$ and $L'_0$ in the Fukaya category) is represented by
    a cycle $1\in\CF(L_0,L'_0)$ which is $\tau$-invariant, $\tau_{\#}(1)=1$. 
  \end{enumerate}
  Then there is a quasi-isomorphism
  $\ECF(L_0,L_1)\simeq \ECF(L'_0,L_1)$ of complexes over
  $\Field[\ZZ/2]$.
\end{proposition}
\begin{proof}
  Assume first that the triple intersection $L_0\cap L'_0\cap L_1$ is
  empty. In this case,
  the proof is similar to the proof of
  Proposition~\ref{prop:indep-of-Ham-isotopy}, but with the category
  $\Dat$ replaced by a different target category. Specifically, let
  $\Delta$ be a disk with three boundary punctures, labeled $p_1$,
  $p_2$, $p_3$ in clockwise order (i.e., a triangle with
  corners $p_1$, $p_2$ and $p_3$).  Fix an identification of a small,
  closed neighborhood of $p_i$ with $[0,\infty)\times[0,1]$. Fix also
  a $\tau$-invariant cylindrical almost complex structure $J$ on $M$
  achieving transversality for $(L_0,L'_0)$. Let $\JSpace_\Delta$ denote
  the space of almost complex structures on $M$ parameterized by
  $\Delta$ which are translation-invariant near the punctures of
  $\Delta$, and which agree with $J$ near $p_1$. That is,
  $\JSpace_\Delta$ is the space of maps $\wt{J}_\Delta$ from $\Delta$ to
  the space of almost complex structures on $M$ compatible with
  $\omega$ so that for each of the punctures of $\Delta$ there is an
  $N$ with $\wt{J}_\Delta|_{[N,\infty)\times[0,1]}$
  translation-invariant in the first factor, and so that on the
  specified neighborhood $[0,\infty)\times[0,1]$ of $p_1$,
  $\wt{J}_\Delta=J$. Define a category $\Dat$ with:
  \begin{itemize}
  \item $\ob(\Dat)=\ob(\ol{\JSpace})\amalg
    \ob(\ol{\JSpace})=\{0,1\}\times\ob(\ol{\JSpace})$.
  \item For $(i,J),(i,J')\in\{i\}\times \ob(\ol{\JSpace})$,
    $\Hom_{\Dat}((i,J),(i,J'))=\Hom_{\ol{\JSpace}}(J,J')$ is the space
    of sequences of eventually cylindrical almost complex
    structures for $J$ to $J'$.
  \item $\Hom((0,J),(1,J'))$ is given by sequences
    $(\wt{J}_{-i},\dots,\wt{J}_{-1},\wt{J}_0,\wt{J}_1,\dots,\wt{J}_j)$
    where:
    \begin{itemize}
    \item For $k\neq 0$, each $\wt{J}_k\in\JSpace(J_k,J_{k+1})$ is a
      eventually cylindrical almost complex structure. Here,
      $J_{-i},\dots,J_{j+1}$ is some sequence of cylindrical almost complex structures.
    \item $J_{-i}=J$ and $J_{j+1}=J'$.
    \item $\wt{J}_0\in\JSpace_\Delta$ is a family of almost complex
      structures parameterized by $\Delta$ which agree with $J_0$ on
      some cylindrical neighborhood $[n,\infty)\times[0,1]$ of $p_2$,
      $J_1$ on some cylindrical neighborhood $[n,\infty)\times[0,1]$
      of $p_3$ and, of course, $J$ on some cylindrical neighborhood
      $[n,\infty)\times[0,1]$ of $p_1$.
    \end{itemize}
  \end{itemize}
  Like $\ol{\JSpace}$, $\Dat$ is a topological category. 

  Let
  $1\in\CF(L_0,L'_0;J)$ be the $\tau$-invariant homotopy equivalence
  guaranteed by the proposition's hypotheses, and write $1$ as a
  sum of intersection points between $L_0$ and $L'_0$, $1=\sum_i 1_i$
  with $1_i\in L_0\cap L'_0$. Fix a complex structure $j_\Delta$ on
  $\Delta$ so that the identifications of the neighborhoods of the
  punctures with $[0,\infty)\times[0,1]$ are
  $j_\Delta$-holomorphic. Given a complex structure $\wt{J}\in
  \JSpace_\Delta$ and points $x\in L_0\cap L_1$ and $y\in L_0'\cap
  L_1$ we can consider the moduli space $\cM(x,y,1_i;\wt{J})$ of
  holomorphic sections of the bundle $(M\times \Delta, \wt{J}\times
  j_\Delta)\to (\Delta, j_\Delta)$ which send the three edges of the
  triangle to $L'_0$, $L_0$, and $L_1$ and are asymptotic to $1_i$ at
  $p_1$, $x$ at $p_2$, and $y$ at $p_3$.  The space
  $\cM(x,y,1_i;\wt{J})$ decomposes according to homotopy classes of triangles,
  $\cM(x,y,1_i;\wt{J})=\amalg_{\phi\in\pi_2(x,y,1_i)}\cM(\phi;\wt{J})$.
  More generally, given a family of complex structures $\wt{J}\co
  [0,1]^\ell\to \JSpace_\Delta$ and a homotopy class of triangles $\phi$
  we can consider the moduli space
  $\cM(\phi;\wt{J})=\cup_{\vec{t}\in[0,1]^\ell}\cM(\phi;\wt{J}(\vec{t}))$. For
  generic families $\wt{J}$, these moduli spaces are transversely cut
  out. Given a generic family $\wt{J}$ there is a corresponding map
  $G(\wt{J})\co \CF(L_0,L_1;J_0)\to\CF(L_0',L_1;J_1)$ defined by
  \begin{equation}\label{eq:triangle-map}
  G(x)=\sum_{y\in L_0'\cap L_1}\sum_{1_i\in 1}\sum_{\substack{\phi\in\pi_2(x,y,1_i)\\\mu(\phi)=-\ell}}\#\cM(\phi;\wt{J}) y.
  \end{equation}
  If $\ell=0$ then $G(\wt{J})$ is a chain map and, in fact, a chain
  homotopy equivalence.
  
  Given a sufficiently generic homotopy coherent diagram $F\co
  \Cat\to\Dat$, combining Formula~\eqref{eq:triangle-map} and
  Construction~\ref{const:F-to-G} gives a homotopy coherent diagram
  $G\co \Cat\to \Complexes$.

  The $\ZZ/2$-action on $M$ induces a $\ZZ/2$-action on $\Dat$.  Given
  sufficiently generic, $\tau$-equivariant homotopy coherent diagrams
  $F,F'\co\ECat\ZZ/2\to\ol{\JSpace}$ used to define $\ECF(L_0,L_1)$
  and $\ECF(L'_0,L_1)$, by the same argument as in the proof of
  Lemma~\ref{lem:exist-equi-diagram} we can extend $\{0\}\times F\cup
  \{1\}\times F'$ to a sufficiently generic, $\tau$-equivariant
  homotopy coherent diagram $F''\co\ICat\times \ECat\ZZ/2\to
  \Dat$. Applying Floer theory as above gives a homotopy coherent
  diagram $G''\co \ICat\times\ECat\ZZ/2\to \Complexes$ extending
  $G\amalg G'$. As noted after Formula~\eqref{eq:triangle-map}, for
  each $x\in\ob(\ECat)$, $G''(f_{0,1}\times \Id_x)$ is a homotopy
  equivalence. Hence, $\hocolim(G'')\co \hocolim G\to\hocolim G'$ is a
  quasi-isomorphism over $\Field[\ZZ/2]$, as desired.

  Finally, in the case that $L_0\cap L'_0\cap L_1$ is nonempty, there
  may be constant triangles contained in $L_0\cap L'_0\cap L_1$ which
  are not transversely cut out by the $\overline{\partial}$-operator
  for any almost complex structure. (For example, constant triangles
  may have negative Maslov index; see the proof of
  Lemma~\ref{lem:hat-cob-maps} for an example in which this occurs.)
  These are treated similarly to the discussion in Section~\ref{sec:non-transverse}.
  That is, extend the category $\Dat$ above to include small
  Hamiltonian perturbations of $L_1$ along the corresponding edge of
  the triangle, in the morphisms from $(0,J)$ to $(1,J')$. With this
  larger space of perturbations we are able to achieve transversality
  for all homotopy classes of triangles, and the proof proceeds as
  before.
\end{proof}

\subsection{Non-transverse intersections}\label{sec:non-transverse}
Suppose next that $L_0$ and $L_1$ are $\ZZ/2$-equivariant Lagrangians
which do not intersect transversely. By incorporating the space of
Hamiltonian perturbations of $L_0$, say, into the parameter space,
along with the space of almost complex structures, we can still define
an equivariant Floer complex. Fix an open neighborhood $U\ni L_0$
which is preserved by the $\ZZ/2$-action, let $C^\infty_0(U)$ denote
the space of compactly-supported functions $H\co U\to\RR$, and let
$\HSpace$ denote the space of smooth maps $\RR\to C^\infty_0(U)$ which
are translation-invariant near $\pm\infty$. Let $\ol{\HSpace}$ denote
the category with objects elements of $C^\infty_0(U)$ and
$\Hom(K_0,K_1)$ the space of sequences $(\wt{H}^1,\dots,\wt{H}^n)$ of
elements of $\HSpace$ so that $\wt{H}^1$ agrees with $K_0$ at
$-\infty$, $\wt{H}^n$ agrees with $K_1$ at $+\infty$, and for large
enough $T$, the restriction of $\wt{H}_{i+1}$ to $(-\infty,-T)$ agrees
with the restriction of $\wt{H}_{i}$ to $(T,\infty)$. Like
$\ol{\JSpace}$, the space $\HomO{\ol{\HSpace}}(K_0,K_1)$ has a natural
topology, making $\ol{\HSpace}$ into a topological category. The
action of $\ZZ/2$ on $M$ induces an action of $\ZZ/2$ on
$\ol{\HSpace}$.

Given a generic $K\in C^\infty_0(U)$, the time-1 Hamiltonian flow
$\phi_K$ of $K$ takes $L_0$ to a Lagrangian which is transverse to
$L_1$ so, if we fix a generic almost complex structure $J$, we get a
Floer complex $\CF(L_0,L_1;K;J)\coloneqq
(\CF(\phi_K(L_0),L_1),\bdy_J)$. Similarly, given a generic path of
Hamiltonians $H\in\HomO{\ol{\HSpace}}(K_0,K_1)$ and a path
$J\in\HomO{\ol{\JSpace}}(J_0,J_1)$ there is a corresponding Floer
continuation map
$\CF(L_0,L_1;K_0;J_0)\to\CF(L_0,L_1;K_1;J_1)$. More generally, given a map
$[0,1]^m\to \HomO{\ol{\HSpace}}(K_0,K_1)\times
\HomO{\ol{\JSpace}}(J_0,J_1)$ there is a corresponding map
\[
\CF(L_0,L_1;K_0;J_0)\to\CF(L_0,L_1;K_1;J_1)
\]
of degree $m$, by counting holomorphic bigons for any point in the
$[0,1]^m$-parameter family of (time-dependent) Hamiltonians and
(time-dependent) almost complex structures. So, given a sufficiently
generic homotopy coherent diagram $F\co
\Cat\to\ol{\HSpace}\times\ol{\JSpace}$ there is a corresponding
homotopy coherent diagram $G\co \Cat\to\Complexes$.

Choose a sufficiently generic, $\ZZ/2$-equivariant homotopy
coherent diagram $F\co \ECat\ZZ/2\to
\ol{\HSpace}\times\ol{\JSpace}$. (Since the space $C^\infty_0(U)$ is
contractible, it is easy to see that such a diagram exists.)  Let
$G\co\ECat\ZZ/2\to\Complexes$ be the corresponding $\ZZ/2$-equivariant
homotopy coherent diagram of complexes. We defined the freed Floer
complex of $L_0$ and $L_1$ to be $\ECF(L_0,L_1)\coloneqq \hocolim G$,
which is a module over $\Field[\ZZ/2]$, and the equivariant Floer
cohomology to be $\eHF(L_0,L_1)\coloneqq\ExtO{\Field[\ZZ/2]}(\hocolim
G,\Field)$. By the same argument as in
Section~\ref{subsec:invariance}, the freed Floer complex is
well-defined up to quasi-isomorphism over $\Field[\ZZ/2]$, and is
invariant, up to quasi-isomorphism over $\Field[\ZZ/2]$, under
equivariant Hamiltonian isotopies of $L_0$ and $L_1$.

\begin{lemma}
  If $L_0$ and $L_1$ intersect transversely then the new definition of
  the freed Floer complex agrees with the old one, up to
  quasi-isomorphism over $\Field[\ZZ/2]$.
\end{lemma}
\begin{proof}
  We can simply take all of the Hamiltonians to be the constant function $0$.
\end{proof}

\subsection{Other groups}\label{sec:2-groups}
One can imitate the construction above with $\ZZ/2$ replaced by any
finite group $K$ to obtain a $K$-equivariant Floer homology, though we
will continue to work with coefficients in $\Field$ to avoid
discussing orientations of moduli spaces.

So, fix a finite group $K$ and an action of $K$ on a symplectic manifold
$M$ by symplectomorphisms preserving a pair of Lagrangians $L_0$ and
$L_1$. Assume that $(M,L_0,L_1)$ satisfy
Hypothesis~\ref{hyp:Floer-defined}. 
%
Let $\ECat$ be any small category so that there is a unique morphism
$a\to b$ for any pair of objects $a,b$ in $\ECat$. Further assume that
we have a free right action by $K$ on the objects of $\ECat$ (which
therefore extends uniquely to an $K$-action on $\ECat$).  For
concreteness, we can choose $\ECat$ to be the category $\ECat K$ which
has one object $\ul{g}$ for each element $g\in K$ and a single
morphism $\ul{g}\to\ul{h}$ for each pair of objects $\ul{g}$ and
$\ul{h}$. (So, every object in $\ECat K$ is initial and terminal and,
in fact, $\ECat K$ depends only on the order of $K$.) There is a free
right action of $K$ on $\ECat K$, by $\ul{g}\cdot h=\ul{gh}$.

By the same inductive argument as in
Lemma~\ref{lem:exist-equi-diagram}, we can find a sufficiently
generic, $K$-equivariant homotopy coherent diagram $F\co \ECat \to
\ol{\JSpace}$.  Construction~\ref{const:F-to-G} then gives a homotopy
coherent functor $G\co \ECat \to \Complexes$. Define the \emph{freed
  Floer complex} to be the complex $\ECF(L_0,L_1)\coloneqq \hocolim
G$, which is a chain complex over $\Field[K]$. (If more than one group
is in play, we may denote $\ECF(L_0,L_1)$ by $\ECF[K](L_0,L_1)$ to
emphasize which group we are considering.)  The \emph{$K$-equivariant
  Floer complex} is $\eCF[K](L_0,L_1)=\HomO{\Field[K]}(\hocolim G,
\Field)$ and the \emph{$K$-equivariant Floer homology}
$\eHF[K](L_0,L_1)$ is the homology of $\eCF[K](L_0,L_1)$. Since
$\ECF(L_0,L_1)$ is a bounded-below complex of free
$\Field[K]$-modules, we have
$\eHF[K](L_0,L_1)\cong\ExtO{\Field[K]}(\ECF(L_0,L_1),\Field)$, so in
particular there is an action (and, by homological perturbation
theory, even an $\Ainf$-action) of $H^*(K)$
on $\eHF[K](L_0,L_1)$, and a spectral sequence
$H^*(K,\HF^*(L_0,L_1))\Rightarrow \eHF(L_0,L_1)$, both of which are
invariants of the quasi-isomorphism type of $\ECF(L_0,L_1)$ over
$\Field[K]$.

\begin{proposition}\label{prop:invariance-gen-group}
  Up to quasi-isomorphism over $\Field[K]$, the complex $\ECF(L_0,L_1)$
  is independent of the choice of $K$-equivariant sufficiently
  generic homotopy coherent diagram in $\JSpace$ used to define it.
  If $L_0'$ and $L_1'$ are isotopic to $L_0$ and $L_1$ via
  compactly-supported, Hamiltonian isotopies through $K$-invariant
  Lagrangians then there is a quasi-isomorphism of $\Field[K]$-modules
  $\ECF(L_0,L_1)\simeq \ECF(L'_0,L'_1)$.
\end{proposition}
The proof follows along the same lines as the proofs of
Propositions~\ref{prop:indep-of-cx-str}
and~\ref{prop:indep-of-Ham-isotopy}, and is left to the reader. (As in
Proposition~\ref{prop:indep-of-Ham-isotopy} the complexes
$\ECF(L_0,L_1)$ and $\ECF(L'_0,L'_1)$ are defined with respect to
corresponding $K$-invariant collections of homotopy classes of paths,
$\eta$ and $\eta'$.)

We also have an analogue of Proposition~\ref{prop:non-equi-invar}:
\begin{proposition}\label{prop:non-equi-invar-gen-group}
  Suppose that $L_0$, $L'_0$ and $L_1$ are $K$-invariant Lagrangians,
  which are pairwise transverse, and that $L_0$ is Hamiltonian
  isotopic to $L'_0$. Suppose further that:
  \begin{enumerate}
  \item The collection of homotopy classes $\eta'$ induced by $\eta$
    and the Hamiltonian isotopies is preserved by the $K$-action. (Of
    course, $\eta$ is also required to be preserved by the $K$-action,
    per Hypothesis~\ref{hyp:Floer-defined}.)
  \item There is a $K$-invariant, $\omega$-compatible almost complex
    structure $J$ on $M$ which achieves transversality for all moduli
    spaces of holomorphic bigons with boundary on $(L_0,L'_0)$ of
    Maslov index $\leq 1$.
  \item The top class in $\HF(L_0,L'_0)$ (i.e., the isomorphism
    between $L_0$ and $L'_0$ in the Fukaya category) is represented by
    a cycle $1\in\CF(L_0,L'_0)$ which is fixed by the $K$-action.
  \end{enumerate}
  Then there is a quasi-isomorphism
  $\eCF[K](L_0,L_1)\simeq \eCF[K](L'_0,L_1)$ of chain complexes over
  $\Field[K]$.
\end{proposition}
Again, we leave the generalization of the proof to the reader.


\section{Equivariant Floer homology via equivariant complex structures}\label{sec:equi-equi}
In this section, we show that the equivariant Floer homology from
Section~\ref{sec:equi-complex} agrees with the equivariant Floer
homology computed from an equivariant almost complex structure (as
in~\cite[Sections 2b and 3]{SeidelSmith10:localization}) essentially
whenever the latter is defined.

\subsection{The case of involutions}
We start by briefly recalling the equivariant construction of
equivariant Floer homology, in the case $H=\ZZ/2$. Fix $M$, $L_0$, $L_1$, $\eta$, and an action of $\ZZ/2=\{1,\tau\}$ as in
Section~\ref{sec:equi-complex}, and in particular satisfying
Hypothesis~\ref{hyp:Floer-defined}.
Consider almost complex structures $J$
which are \emph{$\tau$-invariant}, i.e., such that $d\tau\circ J\circ
d\tau=J$. This is an infinite-codimension condition, and it may not be
possible to find a $\tau$-invariant almost complex structure which achieves
transversality for the $\dbar$ operator. We hypothesize away this difficulty:
\begin{hypothesis}\label{hyp:equivariant-transversality}
  Assume that there is a one-parameter family of $\tau$-invariant 
  almost complex structures $J=J(t)$, $t\in[0,1]$ on $M$ compatible with
  $\omega$ and so that for any homotopy class of Whitney disks $\phi$
  between $L_0$ and $L_1$ with Maslov index $\mu(\phi)\leq 1$, the
  moduli space of $J$-holomorphic disks $\cM(\phi)$ in the homotopy
  class $\phi$ is transversely cut out.
\end{hypothesis}

Under this hypothesis, we can define the equivariant Floer complex as
follows. The Floer complex $(\CF(L_0,L_1),\bdy_J)$, computed with
respect to $J$, satisfies $\bdy_J^2=0$: even though the index $2$
moduli spaces for $J$ may not be cut out transversally, there will be
a (non-invariant) $J'$ arbitrarily close to $J$ for which the index
$2$ moduli spaces are transversally cut out, so $\bdy_{J'}^2=0$, and
for $J'$ close enough to $J$, $\bdy_J=\bdy_{J'}$. Further, since $J$
is $\tau$-invariant, the map $\tau_\#\co \CF(L_0,L_1)\to\CF(L_0,L_1)$
induced by $\tau\co L_0\cap L_1\to L_0\cap L_1$ is a chain map, so
$\CF(L_0,L_1)$ becomes a chain complex over $\Field[\ZZ/2]$. The
equivariant Floer cohomology is the $\ZZ/2$-equivariant cohomology of
$\CF(L_0,L_1)$, i.e., 
$\ewtHF(L_0,L_1)=\ExtO{\Field[\ZZ/2]}(\CF(L_0,L_1),\Field)$. Explicitly,
this is the cohomology of the bicomplex
\begin{equation}\label{eq:equiequi}
\ewtCF(L_0,L_1)=\left(0\rightarrow (\CF^*(L_0,L_1),d_J) 
\stackrel{\Id+\tau^\#}{\longrightarrow} (\CF^*(L_0,L_1),d_J)
\stackrel{\Id+\tau^\#}{\longrightarrow} \cdots\right)
\end{equation}
where $\CF^*(L_0,L_1)$ is the Floer cochain complex, the dual (over
$\Field$) of $\CF(L_0,L_1)$.

\begin{proposition}\label{prop:equi-is-equi}
  Under Hypothesis~\ref{hyp:equivariant-transversality}, there is a
  quasi-isomorphism of $\Field[\ZZ/2]$-modules 
  \[
  \CF(L_0,L_1)\simeq \ECF(L_0,L_1)
  \]
  between the Floer complex of $L_0$ and $L_1$ and the freed Floer
  complex (Definition~\ref{def:equi-Floer}).
\end{proposition}
In particular, there is a filtered homotopy equivalence
$\ewtCF(L_0,L_1)\simeq \eCF(L_0,L_1)$, and the induced isomorphism
$\ewtHF(L_0,L_1)\cong \eHF(L_0,L_1)$ respects the
$\Field[\theta]$-module structure.
\begin{proof}
  Since $J$ achieves transversality for the $\dbar$-operator, there is
  a small neighborhood $U$ of $J$ in the space of almost complex
  structures $\JSpace$ so that for any cylindrical at infinity almost complex structure $\wt{J}\co\RR\to U$ contained in $U$ and any
  homotopy class $\phi$ of Whitney disks there is a homeomorphism
  $\cM(\phi;\wt{J})\cong \cM(\phi;J)$. In particular, if $\mu(\phi)<0$
  then $\cM(\phi;\wt{J})=\emptyset$ and if $\mu(\phi)=0$ then
  $\cM(\phi;\wt{J})=\emptyset$ unless $\phi$ is the constant disk from
  an intersection point $x$ to itself, in which case
  $\cM(\phi;\wt{J})$ consists of a single point.

  In Lemma~\ref{lem:exist-equi-diagram}, we can choose the diagram $F$
  so that each eventually cylindrical almost complex structure is
  contained in $U$. It follows that:
  \begin{itemize}
  \item $\CF(L_0,L_1;F(a))\cong\CF(L_0,L_1;F(b))\cong \CF(L_0,L_1;J)$,
    as chain complexes, via the identity map $L_0\cap L_1\to L_0\cap
    L_1$.
  \item $G_{f_1,\dots,f_n}=0$ for any sequence of composable morphisms
    $\{f_i\}\subset\{\alpha,\beta\}$ of length $n>1$.
  \item $G_\alpha$ and $G_\beta$ are induced by the identity maps $L_0\cap L_1\to L_0\cap L_1$. 
  \end{itemize}
  It follows from Equations~\eqref{eq:d-of-alpha}
  and~\eqref{eq:d-of-beta} that 
  \begin{align*}\label{eq:equis-agree}
     \bdy(\alpha_n\otimes x)=\alpha_n \otimes(\bdy x )+\alpha_{n-1}\otimes x+\beta_{n-1}\otimes x \\
     \bdy (\beta_n \otimes x) = \beta_n \otimes (\bdy x) + \beta_{n-1} \otimes x + \alpha_{n-1} \otimes x
  \end{align*}
  
  We will take the quotient of $\ECF(L_0,L_1)$ by an acyclic complex preserved by $\tau$, and show that the result  is exactly $\CF(L_0,L_1)$ as an $\Field[\ZZ/2]$-module. Consider the following change of basis: replace $\beta_n \otimes x$ with 
$$  \xi_n(x) = \beta_n \otimes x + \alpha_{n} \otimes x + \alpha_{n+1} \otimes \bdy(x).$$
Then our differential becomes
\begin{align*}
&\bdy(\xi_n(x))=0 \\
&\bdy(\alpha_n \otimes x) = \xi_{n-1}(x) \ \ \ \text{ for } n \geq 1.\end{align*} 
Let $C$ be the subcomplex of $\ECF(L_0,L_1)$ spanned by $\{a_n \otimes x, \xi_m(x): x \in L_0 \cap L_1, \ n \geq 1, \ m \geq 0\}$, or equivalently spanned by $\{\alpha_n \otimes x, \beta_n \otimes x, \alpha_0\otimes x + \beta_0 \otimes x: x \in L_0 \cap L_1, \ n \geq 1\}$. From the second description it is clear $C$ is preserved by the $\ZZ/2$-action. Furthermore, the quotient of $\CF(L_0,L_1)$ by $C$ is spanned by elements $\{\alpha_0 \otimes x: x \in L_0 \cap L_1\}$ with $\bdy(\alpha_0 \otimes x) = \alpha_0 \otimes (\bdy x)$ and $\ZZ/2$ action 
\[
\tau_{\#}(\alpha_0 \otimes x) = \beta_0 \otimes \tau_{\#}x \equiv \alpha_0 \otimes \tau_{\#} x.
\]  
This is exactly $\CF(L_0,L_1)$.
\end{proof}

It is not entirely clear how restrictive
Hypothesis~\ref{hyp:equivariant-transversality} is, but it is
satisfied in at least two important cases:
\begin{enumerate}[label=(ET-\arabic*)]
\item\label{item:ET:Seidel-Smith} Let $NM^\fix$, $NL_0^\fix$ and $NL_1^\fix$ denote the normal
  bundles to the $\tau$-fixed parts $M^\fix$, $L_0^\fix$, and
  $L_1^\fix$ in $M$, $L_0$, and $L_1$. Suppose that there is an
  equivariant trivialization $NM^\fix\cong M^\fix\times \CC^k$, where
  $\tau$ acts on $\CC^k$ by $\tau(\vec{z})=-\vec{z}$, which sends
  $NL_0^\fix$ and $NL_1^\fix$ to $L_0^\fix\times \RR^k$ and
  $L_1^\fix\times i\RR^k$, respectively. Then Seidel-Smith observe
  that there is a one-parameter family of $\tau$-invariant almost complex structures $J$ on $M$ achieving transversality for all
  moduli spaces of $J$-holomorphic disks~\cite[Lemma
  19]{SeidelSmith10:localization}.
\item\label{item:ET:index-double} Let $i\co (M^\fix,L_0^\fix,L_1^\fix)\to
  (M,L_0,L_1)$ denote inclusion. Suppose that for any homotopy class
  of Whitney disks $\phi\co (\bD^2,\bdy\bD^2)\to (M^\fix,L_0\cup
  L_1)$, $\mu(i_*(\phi))=2\mu(\phi)$ (where $i_*(\phi)$ denotes the
  homotopy class of disks in $(M,L_0\cup L_1)$ corresponding to
  $\phi$). Then Hypothesis~\ref{hyp:equivariant-transversality} is
  satisfied. To see this, choose a one-parameter family of almost complex structures $J_{\fix}$ on $M^\fix$ which achieves
  transversality for holomorphic disks in $M^\fix$. 
  (Note that the Hypothesis~\ref{hyp:Floer-defined} for $M$ and the
  condition $\mu(i_*(\phi))=2\mu(\phi)$ imply
  Hypothesis~\ref{hyp:Floer-defined} for $M^\fix$.) As discussed
  in~\cite[Section 5c]{KhS02:BraidGpAction}, a generic choice of
  $\ZZ/2$-invariant extension of $J_{\fix}$ to an almost complex
  structure $J$ on $M$ will achieve transversality for all holomorphic
  disks not contained entirely in $M^\fix$. But since
  $\mu(i_*(\phi))=2\mu(\phi)$ and $J_{\fix}$ achieves transversality
  in $M^\fix$, there are no nontrivial holomorphic curves in $M^\fix$
  with index $\leq 1$.
\end{enumerate}
Seidel-Smith show that if $(M,L_0,L_1,\tau)$ satisfy a weaker (but
still restrictive) condition, \emph{stable normal triviality}, then
after an isotopy one can arrange for
Condition~\ref{item:ET:Seidel-Smith} to be satisfied.  As was already
observed in~\cite{LeeLipshitz08:gradings} and further exploited
in~\cite{Hendricks12:dcov-localization,Hendricks:periodic-localization},
Condition~\ref{item:ET:index-double} arises in $\ZZ/2$-covers of
Heegaard diagrams. Both points will be used in
Section~\ref{sec:HF-applications}.

\begin{proof}[Proof of Theorem~\ref{thm:SS-invt}]
  Seidel-Smith's equivariant Floer complex is
  $\ewtCF(L_0,L_1)$, which is well-defined in their setting
  because of Point~\ref{item:ET:Seidel-Smith}. It follows from
  Proposition~\ref{prop:equi-is-equi} that their filtered complex
  agrees with $\eCF(L_0,L_1)$, and by
  Propositions~\ref{prop:indep-of-cx-str}
  and~\ref{prop:indep-of-Ham-isotopy} that the filtered homotopy type
  of $\eCF(L_0,L_1)$ is independent of auxiliary choices and
  invariant under equivariant Hamiltonian isotopies.
\end{proof}

\subsection{General finite groups}
Finally, we turn briefly to general finite groups $K$. With notation as in
Section~\ref{sec:2-groups}, we have the following generalization of
Hypothesis~\ref{hyp:equivariant-transversality}:
\begin{hypothesis}\label{hyp:equivariant-transversality-K}
  Assume that there is a one-parameter family of $K$-invariant 
  almost complex structures $J=J(t)$, $t\in[0,1]$ on $M$ compatible with
  $\omega$ and so that for any homotopy class of Whitney disks $\phi$
  between $L_0$ and $L_1$ with Maslov index $\mu(\phi)\leq 1$, the
  moduli space of $J$-holomorphic disks $\cM(\phi)$ in the homotopy
  class $\phi$ is transversely cut out.
\end{hypothesis}
In this case, $\CF(L_0,L_1)$ becomes a module over $\Field[K]$. To
define the equivariant Floer complex, choose a projective resolution
$R_*$ of $\CF(L_0,L_1)$ over $\Field[K]$, and define
$\ewtCF(L_0,L_1)=\HomO{\Field[K]}(R_*,\Field)$ and
$\ewtHF(L_0,L_1)=H_*(\ewtCF(L_0,L_1))=\ExtO{\Field[\ZZ/2]}(\CF(L_0,L_1),\Field)$. (To
give an explicit definition of $\ewtCF(L_0,L_1)$ we can take $R_*$ to
be the bar resolution.)  The analogue of
Proposition~\ref{prop:equi-is-equi} is:
\begin{proposition}\label{prop:equi-is-equi-K}
  Under Hypothesis~\ref{hyp:equivariant-transversality-K}, there is a
  quasi-isomorphism of $\Field[K]$-modules 
  \[
  \CF(L_0,L_1)\simeq \ECF(L_0,L_1)
  \]
  between the Floer complex of $L_0$ and $L_1$ and the freed Floer
  complex (Definition~\ref{def:equi-Floer}).
\end{proposition}
\begin{proof}
  Fix a one-parameter family of almost complex structures as in
  Hypothesis~\ref{hyp:equivariant-transversality-K} and induce a
  homotopy coherent diagram $F\co \ECat K\to \ol{\JSpace}$ as in the
  proof of Proposition~\ref{prop:equi-is-equi}. If
  $G\co \ECat K\to\Complexes$ is the homotopy coherent diagram induced
  by $F$ and Floer theory then recall that
  \[
    \hocolim G = \bigoplus_{n\geq 0}\bigoplus_{g_0,\dots,g_n\in K}I_*^{\otimes n}\otimes G(g_0)/\sim.
  \]
  There is a map $\hocolim G\to \CF(L_0,L_1)$ which is:
  \begin{itemize}
  \item The identity map $I_*^{\otimes 0}\otimes G(g_0)=\CF(L_0,L_1)\to\CF(L_0,L_1)$ for $n=0$.
  \item Vanishes on $\alpha\otimes G(g_0)$ if $\alpha\in I_*^{\otimes n}$ has grading $>0$.
  \end{itemize}
  It is straightforward to verify that this map is $K$-equivariant,
  and it follows from Lemma~\ref{lem:ECF-is-CF} (or rather, its
  extension to arbitrary groups $K$) that this map is a
  quasi-isomorphism. The result follows.
\end{proof}

We will not formulate the analogue of Condition~\ref{item:ET:Seidel-Smith}, but the analogue of Condition~\ref{item:ET:index-double} is:
\begin{enumerate}[label=(ET-2${}^\prime$)]
\item\label{item:et2p}Let
  $i\co (M^\fix,L_0^\fix,L_1^\fix)\to (M,L_0,L_1)$ denote
  inclusion. Suppose that for any homotopy class of Whitney disks
  $\phi\co (\bD^2,\bdy\bD^2)\to (M^\fix,L_0\cup L_1)$,
  $\mu(i_*(\phi))=c\mu(\phi)$ for some $c\neq 1$ a divisor of $|K|$.
\end{enumerate}
Condition~\ref{item:et2p} implies
Hypothesis~\ref{hyp:equivariant-transversality-K}, since it rules out
index $1$ holomorphic curves fixed by nontrivial elements of $K$.


\section{Spectral sequences for periodic knots and branched double 
covers}\label{sec:HF-applications}
\subsection{First invariance statements} \label{sec:first-invariance}

As a warm-up, we prove Corollaries~\ref{cor:Hen-dcov-invt}
and~\ref{cor:Hen-periodic-invt}, that Hendricks's spectral sequences
are knot invariants. We start by recalling how the spectral sequences
are constructed. For the spectral sequence~\eqref{eq:Sig-to-K}, start
with a multi-pointed, genus-0 Heegaard diagram
$\HD=(S^2,\alphas,\betas,\zs,\ws)$ for $(S^3,K)$. Let
$\wt{\HD}=(\wt{\Sigma},\wt{\alphas},\wt{\betas},\wt{\zs},\wt{\ws})$ be
the double cover of $\HD$ branched along $\zs\cup\ws$ so that
$\wt{\HD}$ represents $(\Sigma(K),K)$. Let $n=|\alphas|$ denote the
number of $\alpha$-circles, so $|\betas|=n$ and
$|\zs|=|\ws|=|\wt{\zs}|=|\wt{\ws}|=n+1$, and
$|\wt{\alphas}|=|\wt{\betas}|=2n$. Consider the symmetric products
$\Sym^n(S^2\setminus (\zs\cup\ws))$ and
$\Sym^{2n}(\wt{\Sigma}\setminus (\wt{\zs}\cup\wt{\ws}))$, and the
submanifolds $T_\alpha=\prod\alphas,\
T_\beta=\prod\betas\subset\Sym^n(S^2\setminus(\zs\cup\ws))$ and
$\wt{T}_\alpha=\prod\wt{\alphas},\
\wt{T}_\beta=\prod\wt{\betas}\subset\Sym^{2n}(\wt{\Sigma}\setminus(\wt{\zs}\cup\wt{\ws}))$.
There is an involution $\tau\co \wt{\HD}\to\wt{\HD}$, which induces a
symplectic involution $\tau\co \Sym^{2n}(\wt{\Sigma})\to
\Sym^{2n}(\wt{\Sigma})$ (with respect to an appropriate symplectic
form), which preserves $\wt{T}_\alpha$ and $\wt{T}_\beta$. The fixed
set of $\tau$ is identified with
$(\Sym^n(S^2),T_\alpha,T_\beta)$. Using the fact that we are
considering a genus-0 Heegaard diagram for $K$, one can show that
Seidel-Smith's stable normal triviality conditions are satisfied on
the symmetric products with the basepoints deleted. The spectral
sequence~\eqref{eq:Sig-to-K} is then obtained by performing an
equivariant Hamiltonian isotopy to get a new pair of Lagrangians
$(\wt{T}_\alpha',\wt{T}_\beta')$ and then considering the equivariant
Floer complex $\ewtCF(\wt{T}_\alpha',\wt{T}_\beta')$ from
Section~\ref{sec:equi-equi}.

The spectral sequence~\eqref{eq:HFL-periodic} is obtained similarly,
except that one starts with a Heegaard diagram for the link $K\cup A$,
so that exactly one $z$ and one $w$ basepoint correspond to $A$, and
takes the branched double cover only over the two basepoints
corresponding to $A$. The construction of the spectral
sequence~\eqref{eq:HFK-periodic} is similar to the construction
of~\eqref{eq:HFL-periodic} except that one does not delete the
basepoint $\wt{z}$ corresponding to the axis $\wt{A}$.

Any two multi-pointed Heegaard diagrams for $K$ with a fixed number of
$w$-basepoints, say $n-1$, can be connected by a sequence of isotopies
of the $\alpha$- and $\beta$-circles, handleslides among the circles,
and index 1-2 stabilizations and destabilizations (connected sums with
the standard genus $1$ Heegaard diagram for $S^3$). Almost all of
these moves are easily seen to induce Hamiltonian isotopies of the
Lagrangians. There is one case that is of particular note, which we
explain here.

  \begin{lemma}\label{lemma:handleslides} 
    Let $\HD=(\Sigma,\alphas,\betas,\zs,\ws)$ be a multi-pointed
    Heegaard diagram for $(S^3, K)$ where $\Sigma$ has genus $g$ and $|\ws|=n+1$. If
    $\HD'=(\Sigma,\alphas',\betas,\zs,\ws)$ is obtained from $\HD$ via
    a handleslide, then for an appropriate choice of symplectic form
    on $\Sym^{g+n}(\Sigma\setminus \ws)$, the tori $T_{\alpha}$ and
    $T_{\alpha'}$ are Hamiltonian isotopic in $\Sym^{g+n}(\Sigma\setminus
    \ws)$. Furthermore, the isotopy can be taken to occur in the complement of 
    the divisor $\zs \times \Sym^{g+n-1}(\Sigma \setminus \ws)$. An 
    analogous result is true for a
    multi-pointed Heegaard diagram for a three-manifold.
\end{lemma}

\begin{proof} For a Heegaard diagram with a single $w$ basepoint, this was 
shown by Perutz~\cite[Theorem 1.2]{Perutz07:HamHand}. In order to deduce the 
result for a Heegaard diagram with $|\ws|=n$, we attach $n$ one-handles to 
$\Sigma$ with feet near the basepoints $w_i$ and $w_{i+1}$, creating 
a genus $g+n$ surface $\Sigma'$ with one basepoint. By Perutz's 
result~\cite[Theorem 1.2]{Perutz07:HamHand}, there is a symplectic form 
$\omega$ on $\Sym^{g+n}(\Sigma')$ with respect to which $T_{\alpha}$ and 
$T_{\alpha'}$ are Hamiltonian isotopic. Further, if we let $\alphas = 
(\alpha_1, \dots, \alpha_{g+n})$ and $\alphas' = (\alpha_1', \alpha_2, \dots, 
\alpha_{g+n})$, where $\alpha_1'$ is obtained from $\alpha_1$ by 
handlesliding over $\alpha_2$, then this isotopy takes place in $\alpha_3 
\times \cdots \times \alpha_{g+n} \times \Sym^2(U)$, where $U$ is a small 
pair-of-pants neighborhood of the handleslide region. In particular, we can 
choose $U$ to lie in $\Sigma$ and contain no basepoints (either $z$ or $w$), so 
this product is contained in $\Sym^{g+n}(\Sigma \setminus {\ws})$. Therefore, 
restricting $\omega$ to the submanifold $\Sym^{g+n}(\Sigma \setminus {\ws})$ 
gives the desired result. The argument for three-manifolds is identical, except 
that there is no need to avoid $z$ basepoints. 
\end{proof}

\begin{proof}[Proof of Corollary~\ref{cor:Hen-dcov-invt}]
  By Proposition~\ref{prop:equi-is-equi}, the spectral sequence is the
  same as the spectral sequence induced by
  $\CF_{\ZZ/2}(\wt{T}_\alpha,\wt{T}_\beta)$, the equivariant complex
  from non-invariant complex structures
  (Section~\ref{sec:equi-complex}). By
  Propositions~\ref{prop:indep-of-cx-str}
  and~\ref{prop:indep-of-Ham-isotopy}, this spectral sequence is
  independent of the choices in its construction and is invariant
  under equivariant Hamiltonian isotopies of the Lagrangians.

  As observed by Ozsv\'ath-Szab\'o~\cite[Section
  7.3]{OS04:HolomorphicDisks}, isotopies of the $\alpha$- and
  $\beta$-circles can be realized through deformations of the almost
  complex structures and Hamiltonian isotopies of the $\alpha$- and
  $\beta$-tori.  By Lemma~\ref{lemma:handleslides}, handleslides can
  also be realized as Hamiltonian deformations of the $\alpha$- and
  $\beta$-tori. The Hamiltonian
  isotopies of the Lagrangians $T_{\alpha}$ and $T_{\beta}$ in
  $\Sym^n(\Sigma)$ lift to equivariant Hamiltonian isotopies of the
  Lagrangians $\wt{T}_{\alpha}$ and $\wt{T}_{\beta}$ in
  $\Sym^{2n}(\wt{\Sigma})$.

  Finally, we may assume that index 1-2 stabilizations occur near the
  basepoints $z\in\Sigma$ and so, if we choose our complex structures to
  be split near the basepoints, these stabilizations have no effect at
  all on the equivariant Floer complex (cf.~\cite[Proof of Theorem
  10.1]{OS04:HolomorphicDisks}).
\end{proof}

Of course, we have actually proved a little more: the
quasi-isomorphism type of the complex $\CFKa(\Sigma,\alphas,\betas,\zs,\ws)$, over
$\Field[\ZZ/2]$, is an invariant of $(K,n)$.

\begin{proof}[Proof of Corollary~\ref{cor:Hen-periodic-invt}]
  This is similar to the proof of Corollary~\ref{cor:Hen-dcov-invt},
  and is left to the reader.
\end{proof}

\subsection{Computing spectral sequences from equivariant 
diagrams}\label{sec:nice-diags}
Now that we know the spectral sequences are knot invariants, it would
be nice to be able to compute them. Corollary~\ref{cor:Hen-from-diag}
says that in many cases we can compute the spectral
sequences~\eqref{eq:Sig-to-K} and~\eqref{eq:HFL-periodic} directly
from a diagram; Corollary~\ref{cor:Hen-nice} says that we can always find 
diagrams for which the spectral
sequences~\eqref{eq:Sig-to-K},~\eqref{eq:HFL-periodic}, 
and~\eqref{eq:HFK-periodic} are algorithmically computable.

While proving Corollary~\ref{cor:Hen-nice} for the spectral 
sequences~\eqref{eq:Sig-to-K} and~\eqref{eq:HFL-periodic}, we will be able to use any 
genus-0 nice Heegaard diagram for $K$ or $K\cup A$, as appropriate. However, 
to show the same result for the spectral sequence~\eqref{eq:HFK-periodic}, we 
will need to restrict to a class of specially adapted nice diagrams for $K \cup 
A$, here called \emph{desirable diagrams}. Beginning with a
two-periodic knot diagram 
$\wt{D}$ for $\wt{K}$, we construct a planar grid diagram similar to Beliakova's diagrams~\cite{Beliakova10:grid} as follows. We 
begin by deforming $\wt{D}$ equivariantly to a diagram $\wt{D}'$ made up of 
horizontal and vertical line segments, such that all overcrossing strands in $\wt{D}'$ are 
vertical and all undercrossing strands are horizontal. (This differs from 
the standard algorithm of placing a knot on a grid diagram only in the 
equivariance requirement.) Because of the requirement that we do this 
equivariantly, there will be even numbers of vertical and horizontal strands, 
say $2n$ each. We circle each horizontal strand with an $\alpha$ curve and each 
vertical strand with a $\beta$ curve such that each $\alpha$ curve intersects 
each $\beta$ curve in four points, and we place a $w_i$ at the beginning of each 
vertical strand and a $z_i$ at the end. We place a single $w_i$ outside the grid 
and a single $z_i$ in the center of the grid; these are the basepoints on the 
axis $\wt{A}$. The quotient of this diagram by the $\ZZ/2$ action is a 
Heegaard diagram for $K \cup A$ with $n$ $\alpha$-circles, such that each 
$\alpha$ curve intersects each $\beta$ curve eight times. The $z$ basepoint  
belonging to the axis, in the center of the diagram, is contained in a bigon. A 
diagram for the trefoil as a two-periodic knot with quotient 
a desirable diagram for the unknot are shown in Figure~\ref{fig:gridy-diagram}.

\begin{figure}
  \centering
  \includegraphics[width=\textwidth]{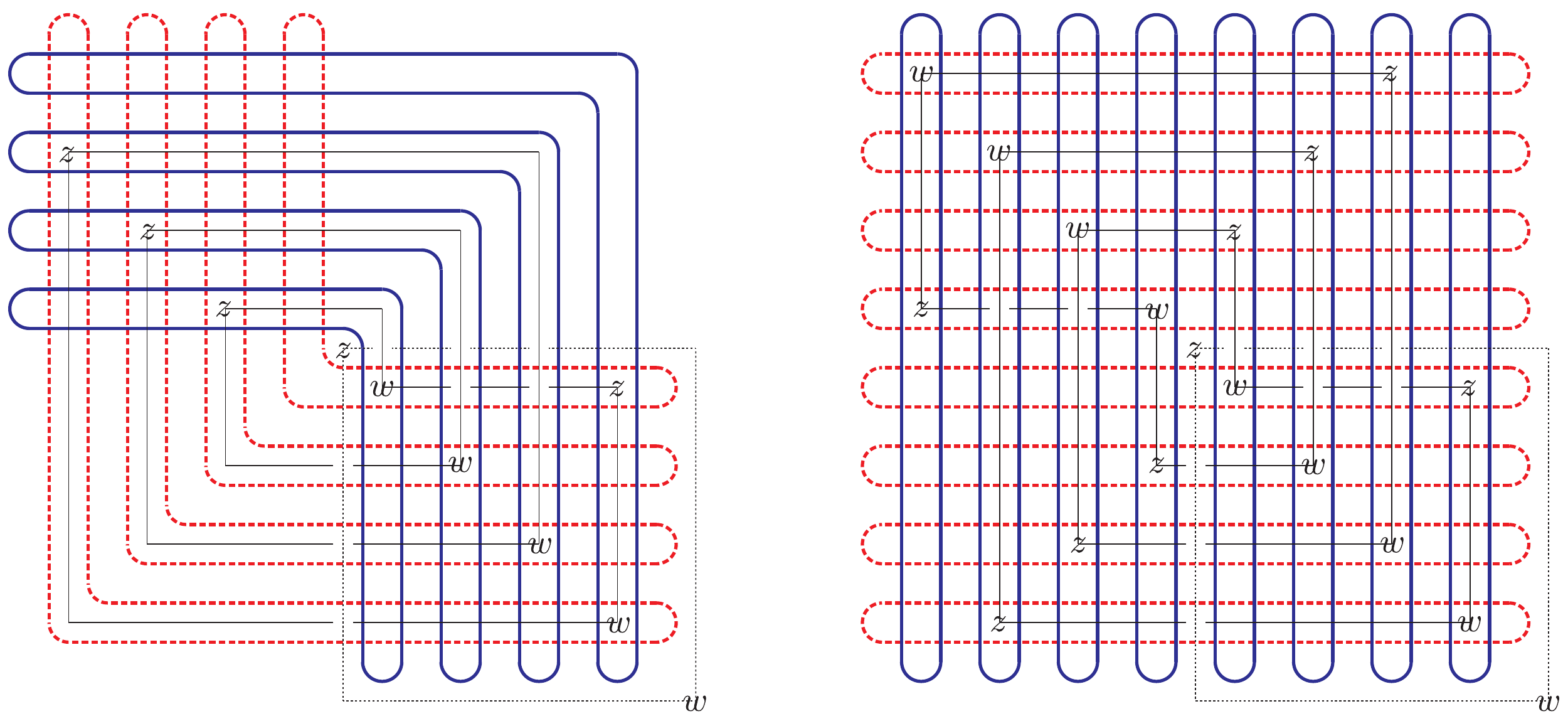}
  \caption{\textbf{Desirable diagrams for links.} Left: a desirable diagram for
    $K\cup A$, where $K$ is the unknot. Right: the corresponding
    2-periodic diagram for $\wt{K}\cup\wt{A}$, where $\wt{K}$ is the
    trefoil. $\textcolor{red}{\alpha}$'s are red and dashed, $\textcolor{blue}{\beta}$'s 
    are blue and solid, axes thin, black and
    dashed and the knots thin-black, too.}
  \label{fig:gridy-diagram}
\end{figure}

\begin{proof}[Proof of Corollary~\ref{cor:Hen-from-diag}]
  We explain the argument in the case of branched double covers; the
  case of periodic knots is similar.  Given a 1-parameter family of
  complex structures $j$ on $S^2$ there is an induced 1-parameter
  family of almost complex structures $\Sym^n(j)$ on
  $\Sym^n(S^2)$. Given a generic choice of $j$ the complex structures
  $\Sym^n(j)$ achieve transversality for all moduli spaces of
  holomorphic disks in $(\Sym^n(S^2\setminus(\zs\cup\ws)),T_\alpha\cup
  T_\beta)$~\cite[Proposition
  13.5]{Lipshitz06:CylindricalHF}. Further, the pullback $\wt{\jmath}$ of
  $j$ to $\wt{\Sigma}$ induces a one-parameter family of almost complex structures $\Sym^{2n}(\wt{\jmath})$ on
  $\Sym^{2n}(\wt{\Sigma}\setminus(\wt{\zs}\cup\wt{\ws}))$. The same
  argument which yields~\cite[Proposition
  13.5]{Lipshitz06:CylindricalHF} implies that, for a generic choice
  of $j$, $\Sym^{2n}(\wt{\jmath})$ achieves transversality for all
  homotopy classes of Whitney disks $\phi$ in
  $\Sym^{2n}(\wt{\Sigma}\setminus(\wt{\zs}\cup\wt{\ws}))$ which are not
  $\tau$-invariant, i.e., for which $\phi\neq \tau_*(\phi)$.

  Any $\tau$-invariant homotopy class $\phi$ has Maslov index
  $\mu(\phi)$ even. Thus, as discussed in
  Point~\ref{item:ET:index-double} of Section~\ref{sec:equi-equi},
  Hypothesis~\ref{hyp:equivariant-transversality} is satisfied. So, by
  Proposition~\ref{prop:equi-is-equi}, the equivariant complex
  $\ewtCF(\wt{T}_\alpha,\wt{T}_\beta)$ computed with respect
  to $\Sym^{2n}(\wt{\jmath})$ agrees with the equivariant complex
  $\CF_{\ZZ/2}(\wt{T}_\alpha,\wt{T}_\beta)$ computed using
  non-equivariant almost complex structures. As in the proof of
  Corollary~\ref{cor:Hen-dcov-invt},
  $\CF_{\ZZ/2}(\wt{T}_\alpha,\wt{T}_\beta)$ is isomorphic, in turn, to
  the equivariant Floer homology as defined by Seidel-Smith.
\end{proof}

\begin{proof}[Proof of Corollary~\ref{cor:Hen-nice}]
  For the spectral sequences~\eqref{eq:Sig-to-K}
  and\eqref{eq:HFL-periodic}, this is immediate from
  Corollary~\ref{cor:Hen-from-diag}, the fact that the double cover of
  a nice Heegaard diagram branched over the basepoints is nice
  (obvious), and the fact that curve counts in nice diagrams are
  combinatorial~\cite{SarkarWang07:ComputingHFhat}. For Corollary
  \eqref{eq:HFK-periodic}, curves are allowed to pass over the $z$
  basepoint on the axis, and therefore we need the additional fact
  that in our desirable diagram for $K\cup A$, the $z$ basepoint for
  $A$ is contained in a bigon, which ensures that in the diagram for
  $\wt{K}\cup\wt{A}$, the $z$ basepoint for $\wt{A}$ is contained in a
  rectangle.
\end{proof}

\subsection{Equivariantly destabilizing basepoints and invariance of the 
spectral sequences}
In this section, we prove that one can remove the extra copies of $V$
from Hendricks's spectral
sequences~\eqref{eq:Sig-to-K},~\eqref{eq:HFL-periodic},
and~\eqref{eq:HFK-periodic}. We start by reviewing why extra
basepoints give copies of $V$ in the first
place~\cite{OS05:HFL,MOS06:CombinatorialDescrip}.

Any two diagrams for a link $L$ with the same number of basepoints per
component of $L$ are related by a sequence of isotopies, handleslides,
and index 1-2 stabilizations/destabilizations, none of which affect
the Floer homology $\HFLa(Y,L)$. Thus, it suffices to consider the
effect of replacing a single basepoint $z$ with the $z$-$w$-$z$ triple
on the left of Figure~\ref{fig:add-basepoints}. Let $\HD$ denote the
original diagram and $\HD'$ the new diagram, so
$\HD'=\HD\#(S^2,\alpha_0,\beta_0,\{z_0,z_0'\},\{w_0\})$, and the
connected sum occurs near one of the basepoints $z$ of $\HD$ and a
point $p$ in the new $S^2$. Let $\{x_0,y_0\}=\alpha_0\cap \beta_0$,
labeled so that there is a bigon in $S^2\setminus\{z_0,z'_0,w_0\}$
from $x_0$ to $y_0$.

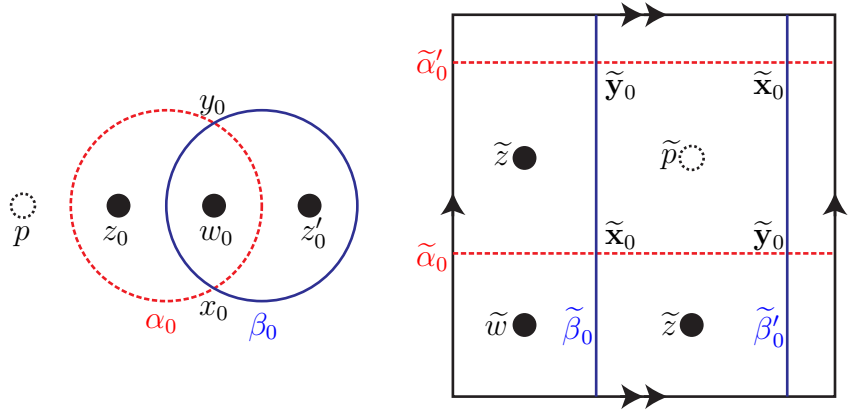
\begin{figure}
  \centering
  \begin{overpic}[tics=10]{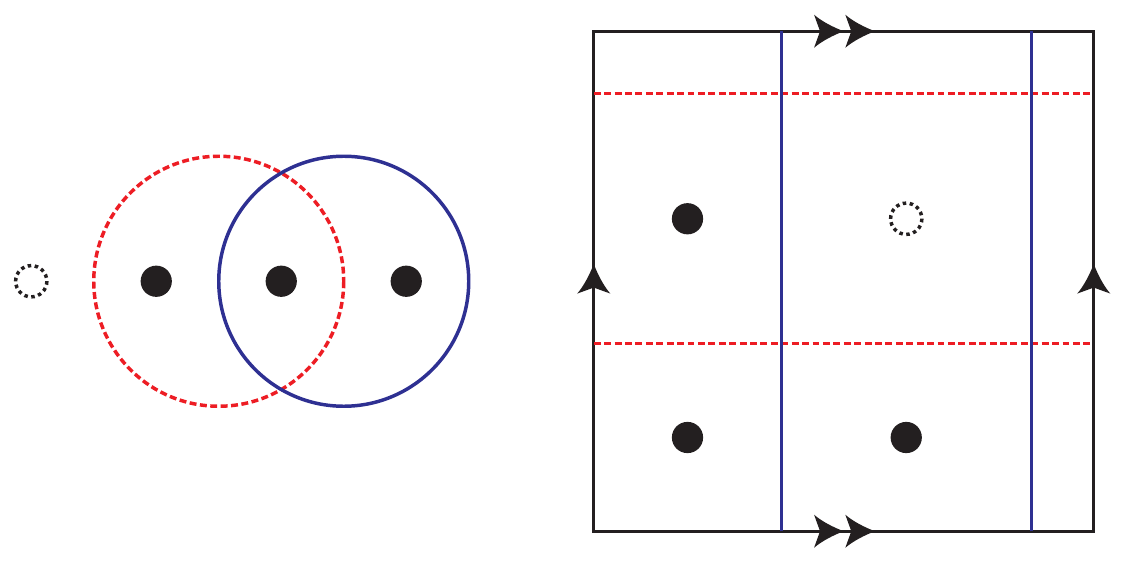}
    \put(1.75,21){$p$}
    \put(12,21){$z_0$}
    \put(23.25,21){$w_0$}
    \put(35,21){$z_0'$}
    \put(23.25,12.25){$x_0$}
    \put(23.25,36.25){$y_0$}
    \put(17,10.75){\textcolor{red}{$\alpha_0$}}
    \put(29,10){\textcolor{blue}{$\beta_0$}}
    \put(48.5,18){\textcolor{red}{$\wt{\alpha}_0$}}
    \put(48.5,40.5){\textcolor{red}{$\wt{\alpha}'_0$}}
    \put(65.5,9.5){\textcolor{blue}{$\wt{\beta}_0$}}
    \put(87.75,9.5){\textcolor{blue}{$\wt{\beta}'_0$}}
    \put(56.5,10){$\wt{w}$}
    \put(77,10){$\wt{z}$}
    \put(57.5,29.75){$\wt{z}$}
    \put(77,30){$\wt{p}$}
    \put(70.5,20.5){$\wt{\x}_0$}
    \put(70.5,38) {$\wt{\y}_0$}
    \put(87.75,20.5){$\wt{\y}_0$}
    \put(87.75,38){$\wt{\x}_0$}
  \end{overpic}
  \caption{\textbf{Stabilizations of type 0-1-2-3.} Left: a stabilization 
  replacing a $z$ basepoint by a $(z,w,z)$ triple. Right: the branched double 
  cover of this stabilization.}
  \label{fig:add-basepoints}
\end{figure}

There is an obvious correspondence between generators of $\CFKa(\HD')$
and $\CFKa(\HD)\otimes\Field\langle x_0,y_0\rangle$. We claim that for
appropriate choices of complex structure, this correspondence is a
chain map. Note first that for any homotopy class
$\phi\in\pi_2(\x,\y)$ of Whitney disks in $\HD$ with multiplicity $0$
at the basepoints there are corresponding
homotopy classes of Whitney disks
$\phi_x\in\pi_2(\x\cup\{x_0\},\y\cup\{x_0\})$ and
$\phi_y\in\pi_2(\x\cup\{y_0\},\y\cup\{y_0\})$ in $\HD'$. Assuming we
work with complex structures of the form $\Sym^n(j)$ and
$\Sym^{n'}(j')$ so that $j'$ agrees with $j$ on $\Sigma$ away from the
connected sum region, there are obvious bijections $\cM(\phi)\cong
\cM(\phi_x)\cong \cM(\phi_y)$.

It remains to show there are no other holomorphic disks for
$\HD'$. The only other homotopy classes of Whitney disks have
multiplicity $1$ in the connected sum region, and lie in
$\pi_2(\x\cup\{x_0\},\y\cup\{y_0\})$ for some
$\x,\y\in\CFKa(\HD)$. Fix such a homotopy class $\phi'$, and consider
the result of stretching the neck in the connected sum region. The
moduli space  $\cM(\phi')$ in $\HD'$ decomposes as a fibered product 
\[
\cM(\phi')\cong \cM(\phi)\times_{\bD^2}\cM(\phi_0)
\]
where $\phi$ (respectively $\phi_0$) is the domain in $\HD$
(respectively $S^2$) induced by $\phi'$. The domain $\phi$
(respectively $\phi_0$) has multiplicity $1$ at $z$ (respectively
$p$). The map $\cM(\phi)\to\bD^2$ (respectively $\cM(\phi_0)\to\bD^2$)
involved in the fiber product sends $u$ to $u^{-1}(\{z\}\times 
\Sym^{n-1}(\Sigma))\in\bD^2$
(respectively $u_0$ to $u_0^{-1}(p)$).

Now, leave $z$ alone and make the connected sum point $p$ approach
$\alpha_0$. As $p$ approaches $\alpha_0$, the point
$u_0^{-1}(p)\in\bD^2$ approaches the boundary of $\bD^2$. Hence, if
the fibered product is non-empty, $u^{-1}(z)$ also approaches $\bdy
\bD^2$.  But $u(\bdy \bD^2)\subset T_\alpha\cup T_\beta$ and $z\not\in
T_\alpha\cup T_\beta$, a contradiction. It follows that the fibered
product is empty, and so, for large neck length and $p$ close to
$\alpha_0$, $\cM(\phi')$ is empty, as well.

With this understanding, we are ready to prove
Theorems~\ref{thm:Hen1-small} and~\ref{thm:Hen2-small}, and Corollaries 
\ref{cor:rank-inequality-refined} and \ref{cor:dim-periodic-link}.

\begin{proof}[Proof of Theorem~\ref{thm:Hen1-small}]
  By induction, it suffices to consider a single 0-1-2-3
  destabilization (i.e., removing a single pair of basepoints).
  By Corollary~\ref{cor:Hen-from-diag}, we can use a generic
  one-parameter family of almost complex structures $j$ on $S^2$ and
  the induced $\ZZ/2$-invariant almost complex structure $\wt{\jmath}$ on
  $\Sigma$ to compute the spectral sequence. 

  The double cover of
  $(S^2,\alpha_0,\beta_0,\{z_0,z_0'\},\{w_0\})$, branched along
  $\{z_0,z_0',w_0,p\}$, is shown on the right of
  Figure~\ref{fig:add-basepoints}. There is an obvious correspondence
  between generators of $\CFKa(\wt{\HD}')$ and generators of
  $\CFKa(\wt{\HD})\otimes\Field\langle \wt{\x}_0,\wt{\y}_0\rangle$, and
  both $\wt{\x}_0$ and $\wt{\y}_0$ are fixed by the
  $\ZZ/2$-action. Moduli spaces in $\wt{\HD}'$ with multiplicity $0$
  in the connected sum region correspond to moduli spaces in $\wt{\HD}$, as
  before; it remains to show that we can choose an equivariant complex
  structure so that the moduli spaces in $\wt{\HD}'$ with multiplicity $1$
  in the connected sum region are empty. To this end, stretch the neck
  in the connected sum region. For large neck length, there is a
  fibered product description
  \[
  \cM(\phi')\cong \cM(\phi)\times_{\bD^2}\cM(\phi_0)
  \]
  just as before. Letting $p$ approach one of the $\alpha$-circles is not
  $\ZZ/2$-equivariant. However, we can pinch the rectangle so that
  points on both $\wt{\alpha}_0$ and $\wt{\alpha}_0'$ approach $p$ and
  so that the involution still exchanges $\wt{\alpha}_0$ and
  $\wt{\alpha}_0'$. (This is the preimage of letting $p$ approach a point on
  $\alpha_0\subset S^2$.) This still has the effect that
  $\wt{u}_0^{-1}(p)\to \bdy\bD^2$, and so still implies that, after
  sufficient pinching, $\cM(\phi')$ is empty.

  Thus, we have shown that for appropriate choice of
  $\ZZ/2$-equivariant almost complex structures we have an isomorphism
  of chain complexes over $\Field[\ZZ/2]$
  \begin{equation} \label{eq:cfk-splits}
  \CFKa(\wt{\HD}')\simeq \CFKa(\wt{\HD})\otimes V,
  \end{equation}
  where $\ZZ/2$ acts trivially on $V$. 

Now we turn our attention to the $E_{\infty}$ page of the spectral sequence coming from $\CFKa(\wt{\HD'})$. For notational simplicity, let $\HFKtDual(\wt{\HD'}) = H^*(\CFKa(\wt{\HD'}))$, and similarly for $\HFKtDual(\wt{\HD})$. The existence of the $\Field[\ZZ/2]$ quasi-isomorphism (\ref{eq:cfk-splits}) implies that the spectral sequence
\[
\HFKtDual(\wt{\HD'}) \otimes \Field[\theta, \theta^{-1}] \Rightarrow \theta^{-1} \eHF(\wt{T}_{\alpha}', \wt{T}_{\beta}')
\]
splits as a direct sum of two identical spectral sequences, 
\[
\HFKtDual(\wt{\HD}) \otimes \Field[\theta, \theta^{-1}] \Rightarrow \theta^{-1}\eHF(\wt{T}_{\alpha}, \wt{T}_{\beta})
\]

In particular,
$\theta^{-1}\eHF(\wt{T}_{\alpha}',
\wt{T}_{\beta}')\cong \theta^{-1}\eHF(\wt{T}_{\alpha},
\wt{T}_{\beta})\otimes V$ via a
map of $\Field[\theta,\theta^{-1}]$-modules that preserves Alexander
gradings. 

However, any Heegaard diagram $\HD$ for $(S^3,K)$ is equivalent, after sufficiently many ordinary Heegaard moves and stabilizations of type 0-1-2-3, to a genus-0 Heegaard diagram $\HD''$ admitting a localization isomorphism
$\theta^{-1}\eHF(\wt{T}_{\alpha}'', \wt{T}_{\beta}'') \cong
\HFKtDual(T_{\alpha}'', T_{\beta}'')\otimes
\Field[\theta,\theta^{-1}]$. This localization
isomorphism also preserves Alexander 
gradings~\cite[page~2144]{Hendricks12:dcov-localization}. Using the previous argument to inductively split off copies of $V$, we conclude that the spectral sequence coming from $\CFKa(\HD)$ is
\[
\HFKt(\wt{\HD}) \otimes \Field[\theta,\theta^{-1}] \Rightarrow \eHF(T_{\alpha},T_{\beta}) \cong \HFKt(\HD)\otimes \Field[\theta,\theta^{-1}]
\]
and in particular, if $\HD$ has a single pair of basepoints, is exactly
\[
\HFKaDual(\Sigma(K),\wt{K})\otimes \Field[\theta,\theta^{-1}] \Rightarrow \theta^{-1}\eHF(\wt{T}_{\alpha}, \wt{T}_{\beta})\cong \HFKaDual(S^3,K) \otimes \Field[\theta,\theta^{-1}],
\]
where the last isomorphism preserves the absolute Alexander grading. \end{proof}

\begin{proof}[Proof of Corollary~\ref{cor:rank-inequality-refined}]  
  Let $\HD$ be a Heegaard diagram for $(S^3,K)$ with a single pair of
  basepoints and let $\wt{\HD}$ be its double branched cover over the
  basepoints so that it is a Heegaard diagram for $(\Sigma(K),
  \wt{K})$. The involution $\tau$ on $\Sigma(K)$ induces an action on
  $\Spinc(\Sigma(K))$ which takes each $\SpinC$-structure to its
  conjugate and fixes only the unique spin structure $\spinc_0$
  \cite[page 1378]{Grigsby06:cyclic-covers}, \cite[Remark
  3.4]{Levine08:cycliccovers}. Moreover, $\tau$ preserves the absolute
  Alexander grading \cite[Proposition
  3.4]{Levine08:cycliccovers}. Thus, the equivariant cochain complex
  $\eCF(\wt{T}_{\alpha}, \wt{T}_{\beta})$ splits along Alexander
  gradings and orbits of $\SpinC$-structures. Ergo the spectral
  sequence (\ref{eq:Sig-to-K}) restricts to spectral sequences
\[
\HFKa(\Sigma(K), \widetilde{K}, \spinc_0, i) \otimes \Field[\theta,\theta^{-1}] 
\Rightarrow \theta^{-1}\eHF(\wt{T}_{\alpha}, \wt{T}_{\beta}, \spinc_0, i)
\]
so
\[
\dim(\HFKa(\Sigma(K),\wt{K}, \spinc_0, i)) \geq 
\rank(\theta^{-1}\eHF(\wt{T}_{\alpha}, \wt{T}_{\beta}, \spinc_0, i)),
\] 
where the latter rank is as an $\Field[\theta, \theta^{-1}]$ module. It remains 
to check that 
\[\rank(\theta^{-1}\eHF(\wt{T}_{\alpha}, \wt{T}_{\beta}, \spinc_0, i))= 
\dim(\HFKa(S^3,K,i)).
\] 
Grigsby showed that the absolute Alexander grading of an equivariant
lift $\wt{\x} \in \wt{T}_{\alpha} \cap \wt{T}_{\beta}$ of a generator
$\x \in T_{\alpha} \cap T_{\beta}$ is the same as the absolute
Alexander grading of $\x$ \cite[Lemma 4.7]{Grigsby06:cyclic-covers}
(see also \cite[Proposition 3.4]{Levine08:cycliccovers}). Furthermore,
localization isomorphisms for $(\Sigma(K),\wt{K})$ preserve absolute
Alexander gradings and $\SpinC$ structures
~\cite[page~2144]{Hendricks12:dcov-localization}, which in light of
the proof of Proposition Theorem~\ref{thm:Hen1-small} implies that the
isomorphism $\theta^{-1}\eHF(\wt{T}_{\alpha}, \wt{T}_{\beta}) \simeq
\HFKa(S^3,K)\otimes \Field[\theta,\theta^{-1}]$ does as well. The
conclusion follows. \end{proof}

\begin{proof}[Proof of Theorem~\ref{thm:Hen2-small}]
  The chain-level version of Theorem~\ref{thm:Hen2-small} is
  considerably simpler. Let $\HD$ be a multi-pointed Heegaard diagram
  for $(S^3,K\cup A)$ so that exactly two basepoints $z_A$ and $w_A$
  correspond to the axis $A$.  Let $\wt{\HD}$ be the double cover of
  $\HD$ branched over $\{z_A,w_A\}$, $\HD'$ the result of doing an
  index 0-1-2-3 stabilization as on the left of
  Figure~\ref{fig:add-basepoints} to $\HD$ at a basepoint $z$ on $K$
  (\emph{not} $z_A$), and $\wt{\HD}'$ the double cover of $\HD'$
  branched along $\{z_A,w_A\}$. We will show that for appropriate
  choice of almost complex structures there is a $\ZZ/2$-equivariant
  chain isomorphism
  \begin{equation}\label{eq:CFL-per-destab}
  \CFLa(\wt{\HD}')\simeq \CFLa(\wt{\HD})\otimes X
  \end{equation}
  where $X=\{xx,xy,yx,yy\}$ with action
  \[
  \tau(xx)=xx\qquad \tau(xy)=yx\qquad \tau(yx)=xy\qquad \tau(yy)=yy.
  \]
  The theorem is a homology-level reinterpretation of this fact (and induction).

  Since we are branching over $A$, not $K$, the diagram $\wt{\HD}'$
  has two regions which look like the left of
  Figure~\ref{fig:add-basepoints}, which are exchanged by
  $\tau$. Stretching the neck around both such regions decomposes $\wt{\HD}'$ 
  as 
  $\wt{\HD}\cup (S^2,\alpha_0,\beta_0,\{z_0,z'_0\},\{w_0\})
  \cup (S^2,\alpha_0,\beta_0,\{z_0,z'_0\},\{w_0\})$.  This induces the 
  desired isomorphism of abelian groups~\eqref{eq:CFL-per-destab}
  where, say, the $xy$ summand corresponds to choosing the generator
  $x$ in the first copy of $S^2$ and $y$ in the second copy of
  $S^2$. The same proof as in the non-equivariant case implies that
  this is a chain isomorphism; to ensure it is $\ZZ/2$-equivariant we
  must move both basepoints $p$ close to the $\alpha$-circles, but
  this does not interfere with the argument. This proves the $\HFLa$ case.

  For the case of $\HFKa$, there is an additional complication that
  Condition~\ref{item:ET:index-double} does not apply, so we do not
  know that we can compute $\HFKa$ using complex structures which are
  $\ZZ/2$-equivariant. Thus, we are forced to work with the more
  complicated definition of the equivariant complex from
  Section~\ref{sec:equi-complex}. We may, however, assume that all of
  the complex structures and paths of complex structures are split
  over the stabilization regions (the two $S^2$'s in the discussion
  above). Under this assumption, we still have a fiber product
  description of the moduli spaces, and taking the same $p\to \alpha$
  limit implies that the differential on the equivariant complex for
  $\wt{\HD}'$ is again of the form $\CFLa_{\ZZ/2}(\wt{\HD})\otimes
  X$. Further details are left to the reader.
\end{proof}

\begin{proof}[Proof of Corollary \ref{cor:dim-periodic-link}]The existence of 
the spectral sequence (\ref{eq:Hen2-small-seq}) immediately implies that
\[
2\dim(\HFLa(S^3, \wt{K}\cup\wt{A})) \geq \dim(\HFLa(S^3, K \cup A)).
\]
Our goal is to remove the factor of two in this inequality by proving that at 
least half the rank of $\HFLa(S^3, \wt{K} \cup \wt{A})\otimes V_1$ dies in the 
spectral sequence.

Recall that if $\HD$ is a Heegaard diagram for a link $L=L_1 \cup \cdots \cup 
L_n$ then the complex $\CFLa(\HD)$ has an Alexander multi-grading, $(A_{L_1}, 
\cdots, A_{L_n})$, valued in integers or half-integers according to the parity 
of the linking numbers of certain combinations of components. The differential on $\CFLa(\HD)$ preserves this multigrading.

With this in mind, let $\HD$ be a nice diagram for $K \cup A$ with one
pair of basepoints on each component of the link. Let $\wt{\HD}$ be
the branched cover of $\HD$ branched over the basepoints on $A$, so
that $\wt{\HD}$ is a nice Heegaard diagram for $\wt{K} \cup \wt{A}$
with two pairs of basepoints on $\wt{K}$ and one on $\wt{A}$.  Let
$\HFLt(\wt{\HD})$ denote the Heegaard Floer homology of the
multi-pointed Heegaard diagram $\wt{\HD}$, so
$\HFLt(\wt{\HD})\cong \HFLa(S^3,\wt{K}\cup\wt{A})\otimes V_1$.

We will concern ourselves only with the Alexander grading $A_{\wt{K}}$
associated to the component $\wt{K}$ of $\wt{K}\cup\wt{A}$, not the
Alexander grading coming from the axis.  The grading $A_{\wt{K}}$ is
always a half-integer, say $\frac{2k-1}{2}$. Write
$\HFLt(\wt{\HD},\frac{2k-1}{2})$ for the summand of $\HFLt(\HD)$
spanned by generators with $A_{\wt{K}}$-grading $\frac{2k-1}{2}$.
The $A_{\wt{K}}$-gradings of the two generators
of $V_1$ are $0$ and $-1$, so
\[
\bigoplus_{k \text{ even}} \HFLt\Bigl(\wt{\HD}, 
\frac{2k-1}{2}\Bigr) \simeq \bigoplus_{k \text{ odd}} 
\HFLt\Bigl(\wt{\HD}, \frac{2k-1}{2}\Bigr).
\]

The Heegaard diagram $\wt{\HD}$ admits an equivariant complex
structure achieving transversality, which by Proposition
\ref{prop:equi-is-equi} may be used to construct
$\eHF(\wt{T}_{\alpha}, \wt{T}_{\beta})$. Moreover, since the double
branched cover involution $\tau_{\#}$ preserves the absolute Alexander
multi-grading on $\CFLa(\wt{\HD})$ \cite[Lemma
3.1]{Hendricks:periodic-localization}, we see the spectral sequence
(\ref{eq:Hen2-small-seq}) splits along the Alexander grading
$A_{\wt{K}}$:
\[
\HFLtDual\Bigl(\wt{\HD}, \frac{2k-1}{2}\Bigr) \otimes \Field[\theta, \theta^{-1}] \Rightarrow 
\theta^{-1}HF_{\ZZ/2}\Bigl(\wt{T}_{\alpha}, \wt{T}_{\beta}, \frac{2k-1}{2}\Bigr)
\]

However, if we let $\x \in T_{\alpha} \cap T_{\beta}$ be a generator in $\CF(T_{\alpha},T_{\beta})$ and $\wt{\x}$ be the corresponding equivariant 
generator in $\wt{T}_{\alpha}\cap \wt{T}_{\beta}$, then $A_{\wt{K}}(\wt{\x}) = 
2A_K(\x) - \frac{1}{2}= \frac{4A_k(\x) -1}{2}$~\cite[Lemma~3.7]{Hendricks:periodic-localization}. In particular, all equivariant 
generators lie in $A_{\wt{K}}$ gradings $i=\frac{2k+1}{2}$ for even $k$, implying that $\theta^{-1}HF_{\ZZ/2}(\wt{T}_{\alpha}, \wt{T}_{\beta}, \frac{2k-1}{2}) = 0$ when $k$ is odd. We conclude that half the rank of $\HFKa(S^3, 
\wt{K} \cup \wt{A})\otimes V \otimes \Field[\theta, \theta^{-1}]$ is lost in 
the spectral sequence (\ref{eq:Hen2-small-seq}), as desired.
\end{proof}

\subsection{Equivalence of the Hendricks and Lipshitz-Treumann spectral 
sequences}
Given a surface $F$, together with a handle decomposition of $F$ with
a single $0$-handle and a single $2$-handle, and $2g$ $1$-handles,
bordered Floer homology~\cite{LOT1, LOT2} associates a differential
algebra $\Alg(F)=\bigoplus_{i=-g}^{g}\Alg(F,i)$. Given a 3-dimensional
cobordism $Y$ from $F_1$ to $F_2$, together with a framed arc $\gamma$
in $Y$ connecting the $0$-handle in $F_1$ and the $0$-handle in $F_2$,
bordered Floer homology associates a \dg
$(\Alg(F_1),\Alg(F_2))$-bimodule $\CFDAa(Y)$, well-defined up to
quasi-isomorphism. (The bimodule $\CFDAa(Y)$ as defined in~\cite{LOT2}
may, more generally, be an $\Ainf$-bimodule, but any $\Ainf$-bimodule
is quasi-isomorphic to a \dg bimodule; alternatively, if one computes
$\CFDAa(Y)$ from a nice Heegaard diagram then the higher operations on
$\CFDAa(Y)$ vanish.)  We will call $(Y,\gamma)$ an \emph{arced
  cobordism}.  If $(Y',\gamma')$ is an arced cobordism from $F_2$ to
$F_3$ then $\CFDAa(Y\cup_{F_2}Y')\simeq
\CFDAa(Y_1)\DTP_{\Alg(F_2)}\CFDAa(Y_2)$, where $\DTP$ is the derived
tensor product~\cite[Theorem 12]{LOT2}. If $F_1=F_2$ then we can
self-glue $(Y,\gamma)$ to get a closed $3$-manifold with a framed
knot. Doing framed surgery on the knot gives another 3-manifold
$Y^\circ$, and the core of the surgery solid torus is a knot $K^\circ$
in $Y$. Then $\HFKa(Y^\circ, K^\circ)\cong
\HH(\CFDAa(Y))$~\cite[Theorem 14]{LOT2}.

Given a \dg $(A,A)$-bimodule $M$, one can define a $\ZZ/2$-equivariant
Hochschild homology of $M$ as follows. Fix a biprojective resolution
$R$ of $A$.  Then $\HH(M)$ is the homology of $\HC(M) = M\otimes
R/\sim$ where for any $m\in M$, $a\in A$, and $r\in R$,
\[
ma\otimes r\sim m\otimes ar \qquad\text{and}\qquad  am\otimes r\sim m\otimes ra.
\]
Similarly, $\HH(M\DTP M)$ is the homology of 
$\HC(M\DTP M)=M\otimes R\otimes M\otimes R/\sim$ where 
\[
a_1 ma_2\otimes r\otimes a_3m'a_4\otimes r'\sim 
m\otimes a_2ra_3\otimes m'\otimes a_4r'a_1
\]
There is an action of $\ZZ/2$ on $\HC(M\DTP M)$ by 
\[
\tau(m\otimes r\otimes m'\otimes r')=m'\otimes r'\otimes m\otimes r.
\]
Define the $\ZZ/2$-equivariant Hochschild complex and homology
by
\begin{align*}
\HC_{\ZZ/2}(M)&=\RHomO{\Field[\ZZ/2]}(\HC(M\DTP M),\Field)\\
\HH_{\ZZ/2}(M)&=\ExtO{\Field[\ZZ/2]}(\HC(M\DTP M),\Field)=H(\HC_{\ZZ/2}(M)).
\end{align*}
From its definition as an $\Ext$ group, $\HH_{\ZZ/2}(M)$ inherits an
action of
$\ExtO{\Field[\ZZ/2]}(\Field,\Field)\cong\Field[\theta]$. 
There is also a filtration on $\HC_{\ZZ/2}(M)$: explicitly, we can write 
\[
\HC_{\ZZ/2}(M) = \{0\to \HC(M\DTP M)^*\stackrel{1+\tau}{\longrightarrow}
\HC(M\DTP M)^* \stackrel{1+\tau}{\longrightarrow}\cdots\},
\]
and then the horizontal filtration on this bicomplex gives a
filtration on $\HC_{\ZZ/2}(M)$. The induced spectral sequence has
$E_1$-page given by $\HH(M\DTP M)^*\otimes\Field[\theta]$ (In the
language of~\cite{LT:hoch-loc}, this is the $\lsup{vh}E$ spectral
sequence.)

\begin{lemma}\label{lem:LT-well-defined}
  The filtered quasi-isomorphism type of $\HC_{\ZZ/2}(M)$ depends only
  on the quasi-isomorphism type of the \dg bimodule $M$ and is
  independent of the choice of resolution $R$.
\end{lemma}
\begin{proof}
  Given a quasi-isomorphism $f\co M\to N$ there is an induced
  quasi-isomorphism
  \[
  f\otimes\Id\otimes f\otimes\Id\co (M\otimes R\otimes M\otimes R)/\sim \to 
  (N\otimes R\otimes N\otimes R)/\sim,
  \]
  and this induced map is a map of
  $\Field[\ZZ/2]$-modules. Further, given two resolutions $R$ and $R'$
  of $A$ and a quasi-isomorphism $g\co R\to R'$ there is an induced
  quasi-isomorphism
  \[
  \Id\otimes g\otimes\Id\otimes g\co (M\otimes R\otimes M\otimes R)/\sim \to 
  (M\otimes R'\otimes M\otimes R')/\sim,
  \]
  which again is a map of $\Field[\ZZ/2]$-modules.  The result
  follows.
\end{proof}

\begin{proof}[Proof of Theorem~\ref{thm:LT-is-Hen}]
  Fix a nice Heegaard diagram $\HD_Y$ for $Y$. Recall the bordered Heegaard 
  diagrams $\Denis(-F)$ and $\MirrorDenis(F)$ associated to $F$~\cite[Section 
  4]{LOTHomPair} which have the properties that:
  \begin{itemize}
  \item the result $\HD_{\Id}=\Denis(-F)\cup\MirrorDenis(F)$ of gluing
    $\Denis(-F)$ to $\MirrorDenis(F)$ along one boundary component is a
    diagram for the identity cobordism of $F$~\cite[Corollary 4.5]{LOTHomPair}, 
    and
  \item the bimodule $\CFDAa(\HD_{\Id})$ is a biprojective resolution of the
    diagonal $(\Alg(F),\Alg(F))$-bimodule $\Alg(F)$~\cite[Proof of Proposition 
    4.1]{LT:hoch-loc}.
  \end{itemize}
  Further, the proofs of the pairing theorem and self-pairing theorem
  via nice diagrams~\cite[Section 7]{LOT2} give an isomorphism
  \[
  \CFKa(\HD_{Y}\cup\HD_{\Id}\cup\HD_Y\cup\HD_{\Id}/\sim)\cong 
  \CFDAa(\HD_Y)\otimes\CFDAa(\HD_\Id)\otimes\CFDAa(\HD_Y)\otimes\CFDAa(\HD_\Id)/\sim
  \]
  where the equivalence relation $\sim$ on the left hand side
  identifies corresponding boundary components, circularly; and this
  isomorphism respects the $\ZZ/2$-actions. So, the result follows
  from Lemma~\ref{lem:LT-well-defined}, which says that we can use the
  right hand side to compute the equivariant Hochschild homology, and
  Corollary~\ref{cor:Hen-nice}, which says that we can use the left
  hand side to compute the equivariant Floer cohomology.
\end{proof}

\subsection{Invariants of covering spaces}\label{sec:cov-spaces}
Consider a normal covering space $\pi\co \wt{Y}\to Y$ of degree
$N<\infty$, with deck group $H$. Given a singly-pointed,
genus $g$ Heegaard diagram $\HD$ for $Y$, there is a corresponding
$N$-pointed Heegaard diagram
$\wt{\HD}=(\wt{\Sigma},\wt{\alphas},\wt{\betas},\wt{z})$ for $\wt{Y}$
and a projection map $\wt{\HD}\to \HD$~\cite[Section 
2.2]{LeeLipshitz08:gradings}: 
$\wt{\HD}$ is simply the total preimage,
under $\pi$, of $\HD$. The group $H$ acts on $\wt{\HD}$ and hence on
$\Sym^{Ng}(\wt{\Sigma}\setminus \wt{z})$. Given any nontrivial $h\in
H$, generators $\x,\y\in \wt{T}_\alpha\cap\wt{T}_\beta\subset
\Sym^{Ng}(\Sigma)$, and homotopy class $\phi\in\pi_2(\x,\y)$, if
$\phi=h_*(\phi)$ then $\mu(\phi)$ is divisible by the order of $h$. So, 
Condition~\ref{item:et2p} is satisfied, and hence 
a generic $1$-parameter family of $H$-equivariant almost complex structures satisfies
Hypothesis~\ref{hyp:equivariant-transversality-K}. So, by Proposition~\ref{prop:equi-is-equi-K}
we can define the Floer complex 
$\CFa(\wt{\HD})=\CF(T_\alpha,T_\beta)$ as a module over
$\Field[H]$. The group $H$ also acts on $\Spinc(Y)$, and the complex
$\CFa(\wt{\HD})=\CF(T_\alpha,T_\beta)$, over $\Field[H]$, decomposes
according to $H$-orbits of $\SpinC$-structures,
\[
\CFa(\wt{\HD})=\bigoplus_{[\spinc_0]\in 
\Spinc(\wt{Y})/H}\CFa(\wt{\HD},[\spinc_0])
\qquad\qquad\text{where}\qquad\qquad
\CFa(\wt{\HD},[\spinc_0])=\bigoplus_{\spinc\in H\spinc_0}\CFa(\wt{\HD},\spinc).
\]
Given a torsion $\SpinC$-structure $\spinc_0$ on $\wt{Y}$, 
$\CFa(\wt{\HD},[\spinc_0])$ is a relatively $\ZZ$-graded chain complex over 
$\Field[H]$, so we define
\[
\eHFa[H](\wt{\HD},[\spinc_0])=\ExtO{\Field[H]}(\CFa(\wt{\HD},[\spinc_0]),\Field).
\]
The equivariant homology $\eHFa[H](\wt{\HD},[\spinc_0])$ inherits an action of 
the
group cohomology $H^*(H)=\ExtO{\Field[H]}(\Field,\Field)$ of $H$.

\begin{theorem}\label{thm:cov-eq-invt}
  The quasi-isomorphism type of $\CFa(\wt{\HD},[\spinc_0])$ over $\Field[H]$ 
  and,
  in particular, the equivariant Heegaard Floer cohomology 
  $\eHFa[H](\wt{\HD},[\spinc_0])$, as a
  module over $H^*(H)$, is an invariant of the covering space $\pi$ and orbit 
  of $\SpinC$-structures $[\spinc_0]$.
\end{theorem}
\begin{proof}
  By 
  Proposition~\ref{prop:equi-is-equi-K}, $\CFa(\wt{\HD},[\spinc_0])$ is
  quasi-isomorphic to the freed Floer complex $\ECF(T_\alpha,T_\beta)$
  (where the collection of homotopy classes of paths $\eta$ from $T_\alpha$ to
  $T_\beta$ is given by $[\spinc_0]$).
  So, it follows from
  Proposition~\ref{prop:invariance-gen-group} that 
  $\CFa(\wt{\HD},[\spinc_0])$ is independent of the choices of almost complex
  structures used to define it and is invariant under equivariant Hamiltonian
  isotopies of $T_\alpha$ and $T_\beta$, up to quasi-isomorphism over 
  $\Field[H]$. It follows that the
  quasi-isomorphism type of
  $\CF(\wt{\HD},[\spinc_0])$ over $\Field[H]$ is invariant under isotopies and 
  handleslides
  of the $\alpha$- and $\beta$-curves in $\HD$ (see Lemma~\ref{lemma:handleslides} 
  and the proof of
  Corollary~\ref{cor:Hen-dcov-invt}). Finally, it is trivially true
  that $\CF(\wt{\HD},[\spinc_0])$ is invariant under stabilizations near the
  basepoint (again, see the proof of
  Corollary~\ref{cor:Hen-dcov-invt}). Since any two Heegaard diagrams
  for $Y$ are related by a sequence of isotopies, handleslides, and
  stabilizations near the basepoint, the result follows.
\end{proof}

The invariant $\CFa(\wt{\HD},[\spinc_0])$ can, of course, be computed
combinatorially starting from any nice diagram $\HD$ for $Y$.

\begin{proof}[Proof of Theorem~\ref{thm:LT-invt}]
  The fact that the equivariant Floer homology is well-defined is
  Theorem~\ref{thm:cov-eq-invt}.  It follows from Corollary~\ref{cor:Hen-nice}
  that the equivariant Floer complex can be computed from a nice
  diagram. (We should note that the whole equivariant complex
  decomposes along the action of $\ZZ/2$ on $\Spinc(\wt{Y})$.) If
  $\wt{Y}\to Y$ is induced by a $\ZZ$-cover,~\cite[Section~4.5]{LT:hoch-loc} also gives an equivariant Floer homology and a
  spectral sequence $\HFa(\wt{Y},\pi^*(\spinc))\otimes
  V\otimes\Field[\theta]\Rightarrow
  \eHFa(\wt{Y},\pi^*(\spinc))$, via Hochschild homology and
  bordered Floer homology. Essentially the same argument as used to
  prove Theorem~\ref{thm:LT-is-Hen} implies that the spectral sequence
  from~\cite[Section~4.5]{LT:hoch-loc} agrees with the spectral
  sequence from Theorem~\ref{thm:cov-eq-invt}.
\end{proof}

\begin{remark}
  For $n$-fold cyclic covers induced by $\ZZ$ covers with $n$ arbitrary, one can
  also use bordered Floer homology to compute the spectral
  sequence~\eqref{thm:cov-eq-invt}, by essentially the same technique
  as in~\cite[Section 4.5]{LT:hoch-loc}.
\end{remark}

\section{New spectral sequences from the branched double 
cover}\label{sec:new-dcov}

\subsection{A concordance homomorphism from the \texorpdfstring{$\ZZ/2$}{Z/2}-action on
  \texorpdfstring{$\CFa(\Sigma(K))$}{CF-hat of the branched double cover}}\label{sec:new-HFa}
Fix a link $L\subset S^3$. Following
Manolescu,\cite{Manolescu06:nilpotent}, we can define a Heegaard
diagram for the branched double cover $\Sigma(L)$ in terms of a bridge
diagram for $L$. Specifically, choose a link diagram $D$ for $L$ and a
decomposition of $D$ as a union of embedded arcs $A_i, B_i\subset
S^3$, $i=1,\dots, n$, so that the $A_i$ (respectively $B_i$) arcs are
pairwise disjoint and the $A_i$ arcs pass under the $B_i$ arcs. Let
$\{p_1,\dots,p_{2n}\}=\bdy \left(\bigcup_{i=1}^n A_i\right)=\bdy
\left(\bigcup_{i=1}^nB_i\right)$. Order the $A_i$ arcs so that $A_n$ and $B_n$ 
share the
endpoint $p_{2n}$.  Let $\pi\co \Sigma_0\to S^2$ be the double cover
of $S^2$ branched along $\{p_1,\dots,p_{2n}\}$, and let
$\alpha_i=\pi^{-1}(A_i)$ and $\beta_i=\pi^{-1}(B_i)$. Then
\[
\HD=(\Sigma,\alpha_1,\dots,\alpha_{n-1},\beta_1,\dots,\beta_{n-1},z=p_{2n})
\]
 is a pointed Heegaard diagram for $\Sigma(L)$, so $\HFa(\HD)=\HFa(\Sigma(L))$.

 There is a $\ZZ/2$-action on $\Sigma$, exchanging the two sheets of
 the branched cover, and this action preserves each $\alpha_i$, each
 $\beta_i$, and $z$. There is an induced $\ZZ/2$-action $\tau\co
 \Sym^{n-1}(\Sigma\setminus \{z\})\to \Sym^{n-1}(\Sigma\setminus \{z\})$
 preserving the Heegaard tori $T_\alpha$ and $T_\beta$. The fixed set
 of $\tau$ has several non-homeomorphic connected components.
 For example, if $n$ is odd, there is a component of
 $\Fix(\tau)$ containing $\pi^{-1}\{q_1,\dots,q_{(n-1)/2}\}$ for any
 points $q_1,\dots,q_{(n-1)/2}\in (S^2\setminus
 \{p_1,\dots,p_{2n}\})$; this component also contains, for instance,
 $p_1,p_1,q_3,\dots,q_{(n-1)/2}$. Given a subset $S\subset
 \{p_1,\dots,p_{2n}\}$ consisting of $n-1$ (distinct) points,
 $\pi^{-1}(S)$ also has cardinality $n-1$, and is an isolated point in
 $\Fix(\tau)$. Let
\[
D = \{\pi^{-1}(S)\mid S\subset \{p_1,\dots,p_{2n}\},\ |S|=n-1\}\subset \Fix(\tau)
\]
be the set of points of this form; we will call $D$ the \emph{discrete
  part} of $\Fix(\tau)$.

\begin{lemma}\label{lem:fixed-discrete}
  The $\tau$-fixed parts of $T_\alpha$ and $T_\beta$ satisfy
  $T_\alpha^\fix,T_\beta^\fix\subset D$ so, in particular,
  $T_\alpha^\fix$ and $T_\beta^\fix$ are finite sets of
  points. Further, $|T_\alpha^\fix\cap T_\beta^\fix|=2^{|L|-1}$, where
  $|L|$ denotes the number of components of $L$.
\end{lemma}
\begin{proof}
  Since $\tau$ preserves each $\alpha$-circle set-wise, if
  $\x=\{x_1,\dots,x_{n-1}\}\in T_\alpha$ is fixed by $\tau$ then each
  $x_i$ must be fixed by $\tau$, so $\x$ is in $D$. The same argument
  applies to $T_\beta$. For the second statement, consider the graph
  $G$ with one vertex for each $p_i$ and an edge for each $A_i$ and
  $B_i$, connecting the arc's endpoints. Let $G'$ be the result of
  deleting the edges $A_n$ and $B_n$ from $G$. Then $G$ is
  homeomorphic to $L$, and so consists of $|L|$ cycles, and $G'$
  consists of an interval, a point, and $|L|-1$ cycles. An element of
  $T_\alpha^\fix\cap T_\beta^\fix$ is a subset of the vertices of $G'$
  of size $n-1$ containing exactly one endpoint of each $A_i$ and each
  $B_i$, $1\leq i\leq n-1$. There is one way of choosing such a subset
  for the interval, and two ways for each cycle, giving $2^{|L|-1}$
  choices in all.
\end{proof}

\begin{convention}\label{conv:torsion}
  For the rest of this section, let $\eta$ denote the set of torsion
  $\SpinC$-structures on $\Sigma(L)$, viewed as a collection of
  homotopy classes of paths from $T_\alpha$ to $T_\beta$ (see
  Section~\ref{sec:hyps-and-statement}). By $\CFa(\Sigma(L))$ and
  $\CF^-(\Sigma(L))$ we mean
  $\bigoplus_{\spinc\in\eta}\CFa(\Sigma(L),\spinc)$ and
  $\bigoplus_{\spinc\in\eta}\CF^-(\Sigma(L),\spinc)$, respectively.
\end{convention}

 The tuple $(\Sym^{n-1}(\Sigma \setminus
\{z\}),T_\alpha,T_\beta,\tau)$ satisfies Hypothesis~\ref{hyp:Floer-defined}, 
so we can define the freed Floer
complex $\wt{\CF}(\Sigma(K))\coloneqq\wt{\CF}(T_\alpha,T_\beta)$ and
the equivariant Floer homology $\eHFa(\Sigma(K))=
\ExtO{\Field[\ZZ/2]}(\wt{\CF}(\Sigma(K)),\Field)$. Further, by
Lemma~\ref{lem:fixed-discrete}, there are no non-constant holomorphic
disks contained in the fixed set (since such disks would have to be
contained in the discrete set $D$), so
Hypothesis~\ref{hyp:equivariant-transversality} is satisfied and
hence, by Proposition~\ref{prop:equi-is-equi}, we can use a generic
1-parameter family of $\ZZ/2$-equivariant almost complex structures to
define $\CFa(\Sigma(K))$ over $\Field[\ZZ/2]$, and compute
$\eHFa(\Sigma(K))$ as
$\ExtO{\Field[\ZZ/2]}(\CFa(\Sigma(K)),\Field)$. Of course,
$\wt{\CF}(\Sigma(K))$ and $\CFa(\Sigma(K))$ are quasi-isomorphic over
$\Field[\ZZ/2]$; so for computations, it is more convenient to use
$\CFa(\Sigma(K))$, although, in the some of the proofs, it is
necessary to use $\wt{\CF}(\Sigma(K))$. Let $\eCFa(\Sigma(K))$ denote
$\RHomO{\Field[\ZZ/2]}(\CFa(\Sigma(K)),\Field)$; its homology is
$\eHFa(\Sigma(K))$.

Next we identify the localized equivariant cohomology:
\begin{proposition}\label{prop:localized-equivariant}
  The localized equivariant Floer cohomology
  $\theta^{-1}\eHFa(\Sigma(L))$ is isomorphic to
  $(\Field\oplus\Field)^{|L|-1}\otimes\Field[\theta,\theta^{-1}]$. In
  particular, there is a spectral sequence
  \begin{equation} \label{eq:new-dbc-sequence}
    \HFaDual(\Sigma(L))\otimes\Field[\theta,\theta^{-1}]\Rightarrow
    (\Field\oplus\Field)^{|L|-1}\otimes\Field[\theta,\theta^{-1}].
  \end{equation}
\end{proposition}
\begin{proof}
  Roughly, the result follows from Seidel-Smith's localization
  theorem~\cite{SeidelSmith10:localization}. More precisely, it
  follows from the fact that the proof of their localization 
  theorem~\cite[Theorem
  20]{SeidelSmith10:localization} only uses stable normal
  triviality on the components of the fixed set which intersect the
  Lagrangian submanifolds. By Lemma~\ref{lem:fixed-discrete}, these
  components are points, for which stable normal triviality
  is obvious.
\end{proof}

The spectral sequence~\eqref{eq:new-dbc-sequence} is the spectral
sequence~\eqref{eq:sar-dcov} mentioned in the introduction. As
promised in the introduction, it is an invariant of based
links:
\begin{proof}[Proof of Theorem~\ref{thm:sar-dcov-invt}]
  As usual, we will actually show that the quasi-isomorphism type of
  $\wt{\CF}(\Sigma(L))$ over $\Field[\ZZ/2]$, or equivalently the
  quasi-isomorphism type of $\CFa(\Sigma(L))$ over $\Field[\ZZ/2]$, is
  an invariant of based links. Note that if two
  based links $L$ and $L'$ containing the same basepoint $p$ are
  isotopic, then we may choose the entire isotopy to fix the basepoint
  $p$. Therefore, we do not need to prove invariance under changing
  which point $p_{2n}$ is distinguished (or, equivalently, which pair
  of arcs $A_n$, $B_n$ is deleted). Thus, we must prove invariance
  under:
  \begin{itemize}
  \item Changes of almost complex structure; and
  \item Bridge moves, i.e., isotopies of the $A_i$ and $B_i$,
    handleslides among the $A_i$ and $B_i$, and stabilizations
    (Figure~\ref{fig:bridge-moves}).
  \end{itemize}
  For stabilizations, we will work with $\CFa(\Sigma(L))$; for the
  other moves, we will work with $\wt{\CF}(\Sigma(L))$.

  Changes of almost complex structure induce quasi-isomorphisms over
  $\Field[\ZZ/2]$ by Proposition~\ref{prop:indep-of-cx-str}. Isotopies
  of the $A_i$ and $B_i$ give equivariant isotopies of the
  Lagrangians, and hence induce quasi-isomorphisms over
  $\Field[\ZZ/2]$ by
  Proposition~\ref{prop:indep-of-Ham-isotopy}. Handleslides among the
  $A_1,\dots,A_{n-1}$ (respectively $B_1,\dots,B_{n-1}$) give
  Hamiltonian isotopic Heegaard tori, by Perutz's
  work~\cite{Perutz07:HamHand}. Further, since the fixed set of the
  $\ZZ/2$ action is discrete, no nontrivial holomorphic disks are
  contained in the fixed set, so there are equivariant almost complex
  structures achieving equivariant transversality; see
  Point~\ref{item:ET:Seidel-Smith} in
  Section~\ref{sec:equi-equi}. Note that the relevant top class is
  represented by fixed generators. Thus, by
  Proposition~\ref{prop:non-equi-invar}, handleslides induce
  equivariant quasi-isomorphisms. Handlesliding $A_{n}$ (respectively
  $B_n$) over one of the other $A_i$ (respectively $B_i$) has no
  effect on the complex at all. Since $S^2\setminus A_n$ is connected,
  we may trade handleslides of $A_i$ over $A_n$ for handlesliding
  $A_i$ over all the other $A$-arcs, and similarly for handlesliding
  $B_i$ over $B_n$.

  To prove stabilization invariance, consider the diagrams shown in
  Figure~\ref{fig:skein-stab}. That is, let $\HD$ be the branched
  double cover of a bridge diagram and let $\HD_1$ be the branched
  double cover of the result of stabilizing the bridge diagram. The
  stabilization occurs near one of the points $p_i$. Since we have
  already verified handleslide invariance, we may assume that $p_i$ is
  in the same region as $p_{2n}$, i.e., there is an arc from $p_i$ to
  $p_{2n}$ which is disjoint from the $A_i$ and $B_i$. Let $\HD_0$ be
  the connected sum of $\HD$ and a 1-bridge diagram for the unknot,
  where the connected sum occurs in the region adjacent to $p_i$ and
  $p_{2n}$. (So, $\HD_0$ is a diagram for a link with a small unknot
  component.) Let $\HD'$ be the result of applying a Reidemeister I
  move to the new pair of arcs in $\HD_1$. The motivation for this
  notation is that the underlying links for $\HD_0$ and $\HD_1$ are
  the zero- and one-resolutions of the crossing introduced
  by the Reidemeister I move in $\HD'$. (Here, we are using the
  convention of zero- and one-resolutions from~\cite{BrDCov}, which is
  \emph{opposite} to Khovanov's~\cite{Khovanov00:CatJones}.)
  There is a skein triangle

  \begin{equation}\label{eq:dcov-skein-seq}
  \dots\to \HFa(\HD')\stackrel{f}{\longrightarrow} 
  \HFa(\HD_0)\stackrel{g}{\longrightarrow} \HFa(\HD_1)\to \HFa(\HD')\to\dots
\end{equation}

  \begin{figure}
    \centering
    \begin{overpic}[tics=10]{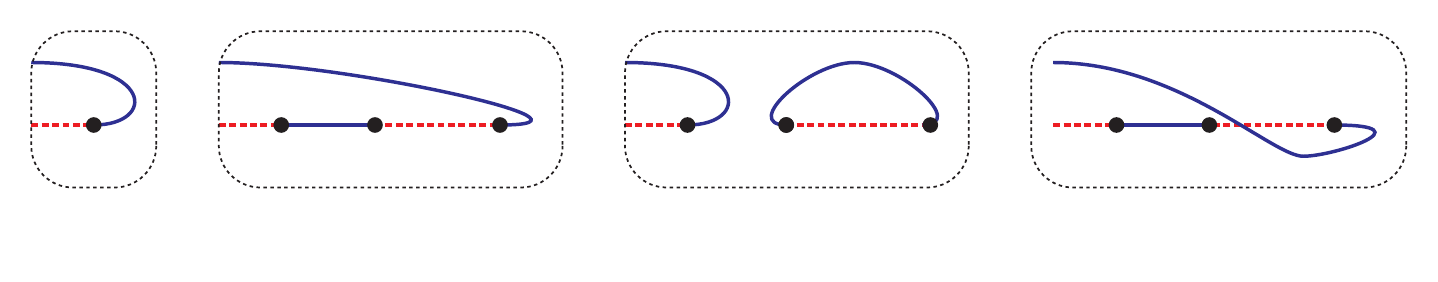}
      \put(5,2){(a)}
      \put(25,2){(b)}
      \put(54,2){(c)}
      \put(83,2){(d)}
    \end{overpic}
    \caption{\textbf{Stabilization and a skein sequence for the
        branched double cover.} (a) A piece of a bridge diagram. $A_i$
      is red and dashed and $B_i$ is blue and solid. The double cover
      of this diagram is $\HD$. (b) A stabilization of the bridge
      diagram. The double cover of this diagram is $\HD_1$. (c) The
      disjoint union of the bridge diagram with an unknot. The double
      cover of this diagram is $\HD_0$. (d) A Reidemeister I move
      applied to the stabilization. The double cover of this diagram
      is $\HD'$. Diagrams (b) and (c) are obtained from (d) by
      unoriented skein moves.}
    \label{fig:skein-stab}
  \end{figure}
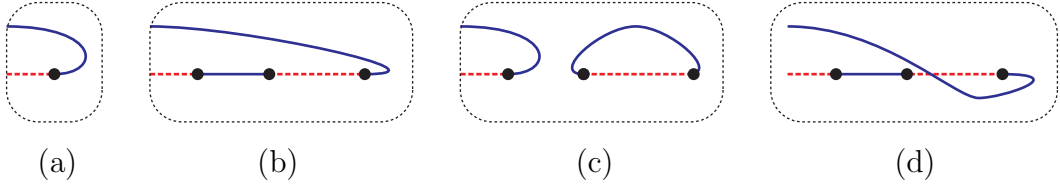

  Note that:
  \begin{enumerate}
  \item The diagram $\HD_0$ represents $\Sigma(L)\#(S^2\times S^1)$; the 
  diagrams $\HD$, $\HD'$, and $\HD_1$ all represent $\Sigma(L)$.
  \item\label{item:inv-con-sum} There are obvious, $\ZZ/2$-equivariant 
  isomorphisms
    \[
    \CFa(\HD_0)\cong \CFa(\HD)\otimes\CFa(S^2\times S^1)
	 \cong \CFa(\HD)\otimes \Field\langle \ttop,\tbot\rangle
    \]
    where the $\ttop$ (respectively $\tbot$) is the top-graded
    (respectively bottom-graded)
    generator of $\HFa(S^2\times S^1)$. (In particular, the
    $H_1/\tors$-action on $\HFa(S^2\times S^1)$ sends $\ttop$ to
    $\tbot$ and sends $\tbot$ to $0$.)
  \item We already know that the quasi-isomorphism type of $\CFa$ over
    $\Field$ is an invariant of $\Sigma(L)$; we are trying to prove
    that these maps respect the $\Field[\ZZ/2]$-module
    structure. Non-equivariant invariance implies that the long exact
    sequence~\eqref{eq:dcov-skein-seq} splits as
    \[
    0\to \HFa(\HD')\to \HFa(\HD_0)\to \HFa(\HD_1)\to 0.
    \]
  \item The map $g$ sends the kernel of the $H_1(S^2\times S^1)$-action to $0$, 
  and $\HFa(\HD)\otimes\Field\langle\ttop\rangle$ isomorphically to 
  $\HFa(\HD_1)$.
  \end{enumerate}
  The map $g$ is given by counting holomorphic triangles, and hence
  respects the $\Field[\ZZ/2]$-module structure. Thus, the composition 
  \[
  \CFa(\HD)\to \CFa(\HD_0)\to \CFa(\HD_1),
  \]
  where the first map sends $\CFa(\HD)$ to $\CFa(\HD)\otimes \Field\langle 
  \ttop\rangle$, is an $\Field[\ZZ/2]$-equivariant
  quasi-isomorphism.
\end{proof}

In the special case that $K$ is a knot, we extract a new numerical
invariant. Before localizing we have the spectral sequence
\begin{align}
\HFaDual(\Sigma(K))\otimes\Field[\theta]\Rightarrow \eHFa(\Sigma(K)).
\end{align}
Since $\theta^{-1}\eHFa(\Sigma(K)) \cong \Field[\theta,
\theta^{-1}]$, there is a (non-canonical) submodule of
$\eHFa(\Sigma(K))$ isomorphic to $\Field[\theta]$. Also,
$\eHFa(\Sigma(K))$ inherits an absolute $\QQ$-grading from
$\CFa(\HD)$. We use the convention that multiplication by $\theta$
raises the Maslov grading by one (so that the equivariant differential
on $\eCFa(\Sigma(K))$ raises the Maslov grading by one), and for $x\in
\CFaDual(\Sigma(K))$, $\gr(x\otimes \theta^0)=\gr(x)$. Under this
convention, we let $q_{\tau}(K)$ be twice the minimum
grading of an element of the tower $\Field[\theta]$. More precisely we
set:
\begin{align*}
q_{\tau}(K) = 2\min\{q\mid \exists \ [x] \in \eHFa(\Sigma(K)), \ 
\gr([x]) = q,\ \theta^n[x]\neq 0 \ \forall \ n\}
\end{align*}

 As a trivial example, if $K$ is the unknot then $q_{\tau}(K)=0$. Similarly, 
if $K$ is alternating then $\Sigma(K)$ is an $L$-space, implying that $\HFa(\Sigma(K),\s_0)\simeq \Field$ and therefore $\HFaDual(\Sigma(K),\s_0)\otimes \Field[\theta]\simeq \Field[\theta]$. The induced map $\tau^*$ on $\HFaDual(\Sigma(K),\s_0)$ must be the identity, so in the spectral sequence $\HFaDual(\Sigma(K),\s_0)\Rightarrow \eHFa(\Sigma(K))$, we see that the differential $d_1=\theta(1+\tau^*)$ is zero. All further induced differentials are trivial for grading reasons, so $\HFaDual(\Sigma(K),\s_0)\otimes \Field[\theta]\simeq \eHFa(\Sigma(K))$, implying that $q_{\tau}(K) = 2d(\Sigma(K),
\spinc_0)=\delta(K)$. 

Indeed, because $\delta(K)$ is always an integer,  \cite[Section 2]{ManolescuOwens07:delta}, we see that $q_{\tau}(K)$ is as well.

We have the following graded refinement of Theorem~\ref{thm:sar-dcov-invt}
\begin{proposition}\label{prop:dbc-hat-grading-invariance}
  The action of $\tau$ and the quasi-isomorphisms from
  Theorem~\ref{thm:sar-dcov-invt} respect the absolute Maslov grading
  on $\CFa(\Sigma(K))$. In particular, the number $q_{\tau}(K)$ is a
  knot invariant.
\end{proposition}
\begin{proof}
  First, to see that
  $\tau_{\#}$ preserves the absolute Maslov grading, note that the
  mapping cylinder of $\tau$ is a rational homology cobordism from
  $\Sigma(K)$ to itself with the property that the cobordism map
  $\CFa(\Sigma(K))\rightarrow \CFa(\Sigma(K))$ is $\tau_{\#}$. Such
  maps preserve the absolute Maslov grading~\cite[Theorem
  7.1]{OS06:HolDiskFour}.  (Alternatively, see the proof of
  Proposition~\ref{prop:D2m-grading}.)

  The quasi-isomorphisms
  induced by isotopies of the $A_i$ and $B_i$ and by handleslides can
  be thought of as maps induced by Heegaard moves on the branched
  double cover of the Heegaard diagram, which preserve the absolute
  Maslov grading; similarly, changes of complex structure induce
  grading-preserving isomorphisms.  For stabilization invariance,
  recall that since only $\SpinC$-structures $\spinc$ with
  $c_1(\spinc)=0$ persist past the first page of the spectral 
  sequence~\eqref{eq:new-dbc-sequence}, it suffices to prove that stabilization
  induces grading-preserving maps for those
  $\SpinC$-structures. Recall that the stabilization map is a
  composition
  \begin{equation}\label{eq:stab-on-hmlgy}
    \HFa(\Sigma(L))\to \HFa(\Sigma(L) \# (S^1\times S^2))\xrightarrow{g} \HFa(\Sigma(L)),
  \end{equation}
  where the first map is inclusion of the summand
  $\HFa(\Sigma(L))\otimes \Field\langle \ttop \rangle \subset
  \HFa(\Sigma(L))\otimes \Field\langle \ttop,\tbot\rangle$ and the
  second map is from the surgery exact sequence, and in particular is
  induced by a two-handle cobordism. By construction the first map of
  Formula~\eqref{eq:stab-on-hmlgy} raises the absolute Maslov grading
  by $\frac{1}{2}$. For the second map, let $W$ be the
  two-handle surgery cobordism from $\Sigma(L) \# (S^1 \times S^2)$
  to $\Sigma(L)$. The cobordism map $g=F_W$ breaks up as a
  sum of maps $\Sigma_{\spinct \in \Spinc(W)} F_{W,\spinct}$. Let
  $\spinc$ be a $\SpinC$-structure on $\Sigma(L)$, and let $\spinc_0$
  be the torsion $\SpinC$-structure on $S^1 \times S^2$. Let $\spinct$
  be the unique $\SpinC$-structure on $W$ restricting to $\spinc \#
  \spinc_0$. Then
  \[
  F_{W,\spinct} \co \HFa(\Sigma(L) \# (S^1 \times S^2), \spinc \#
  \spinc_0) \to \HFa(\Sigma(L), \spinc)
  \]
  is the only component of $F_W$ landing in $\HFa(\Sigma(L),
  \spinc)$. Since $c_1(\spinc)=0$, using the absolute Maslov grading
  formula for cobordism maps we see that $F_{W,\spinct}$ lowers the
  absolute Maslov grading by $\frac{1}{2}$~\cite[Theorem
  7.1]{OS06:HolDiskFour}.
  Thus overall our stabilization maps preserve the absolute Maslov
  grading.
\end{proof}

More is true:
\begin{proposition}\label{prop:dbc-hat-grading-concordance}
  The number $q_{\tau}(K)$ is a smooth concordance invariant.
\end{proposition}

Before proving Proposition~\ref{prop:dbc-hat-grading-concordance} we must
develop the cobordism maps somewhat.

\begin{lemma}\label{lem:hat-cob-maps}
  Given an oriented cobordism $T$ from a link $L_0$ to a link $L_1$
  there is a map of equivariant Floer complexes
  $\eCFa(\Sigma(L_1))\to\eCFa(\Sigma(L_0))$ so that:
  \begin{itemize}
  \item The induced map on the $E^1$-page of the spectral sequence
    associated to the $\theta$-power filtration is
    \[
    (\hat{F}_{\Sigma(T)}^*\otimes\Id)\co \HFaDual(\Sigma(L_1))\otimes\Field[\theta]\to\HFaDual(\Sigma(L_0))\otimes\Field[\theta]
    \]
    where $\hat{F}_{\Sigma(T)}$ is the Ozsv\'ath-Szab\'o cobordism map associated to $\Sigma(T)$, and
  \item The map $\eCFa(\Sigma(L_1))\to\eCFa(\Sigma(L_0))$ has the same effect on gradings as the dual of the Ozsv\'ath-Szab\'o cobordism map $\hat{F}_{\Sigma(T)}$.
  \end{itemize}
\end{lemma}
\begin{proof}
  \begin{figure}
    \centering
    \includegraphics{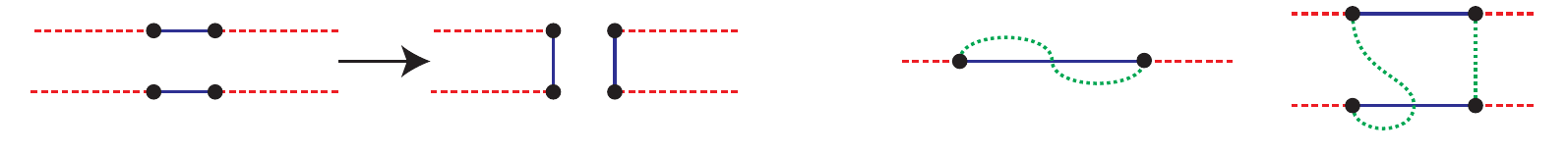}
    \caption{\textbf{A saddle cobordism.} Left: the cobordism
      itself. Right: the perturbations used in the triple diagram for
      the triangle map.}
    \label{fig:saddle}
  \end{figure}
  Any cobordism of links $T\subset [0,1]\times S^3$ from $L_0$ to $L_1$ can be viewed as a
  sequence of:
  \begin{itemize}
  \item Bridge moves,
  \item Births / deaths, i.e., creation or annihilation of a 1-bridge
    unknot disjoint from the rest of the bridge diagram, and
  \item Saddles, as shown in Figure~\ref{fig:saddle}.
  \end{itemize}
  (See also~\cite{Waldron:KhSympMaps}.) 
  We can also assume that the births and deaths occur in a region
  adjacent to the basepoint $p_{2n}$.

  As we saw in Theorem~\ref{thm:sar-dcov-invt}, bridge moves induce
  quasi-isomorphisms of $\wt{\CF}(\Sigma(L))$, or equivalently
  $\CFa(\Sigma(L))$, over $\Field[\ZZ/2]$; consequently, they induce maps
  of $\eCFa(\Sigma(L))$ over $\Field[\theta]$, and it is
  straightforward to check that these maps satisfy the conditions of
  the lemma.
  
  In the
  branched double cover, births correspond to $1$-handles, deaths to
  $3$-handles, and saddles to $2$-handles. The Ozsv\'ath-Szab\'o
  cobordism map associated to a 1-handle attachment is $x\mapsto
  x\otimes\ttop$, where $\ttop$ is the top-graded generator of
  $\CFa(S^2\times S^1)$ and the tensor product refers to the K\"unneth
  isomorphism $\CFa(\Sigma(K\amalg U))\cong
  \CFa(\Sigma(K))\otimes_{\Field}\CFa(S^2\times S^1)$~\cite[Section
  4.3]{OS06:HolDiskFour}. (Here, for the chain-level statement, it is
  important that the new unknot component is disjoint from the other
  components of the bridge diagram, and it is convenient that it is adjacent to $p_{2n}$.) Similarly, the Ozsv\'ath-Szab\'o
  cobordism map associated to a death is $x\otimes \ttop\mapsto 0$,
  $x\otimes \tbot\mapsto x$, where $\tbot$ is the bottom-graded
  generator of $\CFa(S^2\times S^1)$~\cite[Section
  4.3]{OS06:HolDiskFour}. Both of these maps are obviously
  $\ZZ/2$-equivariant, and so give maps to equivariant complexes.

  For saddles, there is a corresponding knot $C\subset \Sigma(L_0)$
  and the Ozsv\'ath-Szab\'o cobordism map is defined by considering a
  Heegaard triple diagram $(\Sigma,\alphas,\betas,\gammas)$ subordinate
  to $C$ and counting holomorphic triangles with one input the
  top-graded generator of $\HFa(\Sigma,\betas,\gammas)\cong
  (\Field\oplus\Field)^{\otimes (n-2)}$~\cite[Section
  4.1]{OS06:HolDiskFour}.

  Consider the bridge triple diagram for the saddle shown on the right
  of Figure~\ref{fig:saddle}. This diagram consists of three sets of
  arcs $\{A_i\}$, $\{B_i\}$, and $\{C_i\}$, so that $(\{A_i\},\{B_i\})$ is
  a diagram for $L_0$, $(\{A_i\},\{C_i\})$ is a diagram for $L_1$, the
  result of the saddle move, and $(\{B_i\},\{C_i\})$ is a diagram for
  an $(n-1)$-component unlink. If $B_i$ is not involved in the saddle
  move then we choose $C_i$ to be a small perturbation of $B_i$
  intersecting $B_i$ at the two endpoints and one interior point, so
  that there is a bigon from each of the two endpoints to the middle
  point. For the two $B_i$ involved in the saddle move, we do a
  perturbation as shown, so that, in particular, one of the
  corresponding $C_i$ intersects one of the $B_i$ at an interior
  point. Let $(\Sigma,\alphas,\betas,\gammas,z)$ denote the Heegaard
  triple diagram which is the branched double cover of this bridge
  triple diagram. The perturbations are chosen so that in the branched
  double cover, the top-graded generator for
  $\HFa(\Sigma,\betas,\gammas)$ is represented by the sum of all
  $2^{n-2}$ fixed generators. In particular, all of the fixed
  generators for $(\Sigma,\beta,\gamma)$ have the same grading.  By
  stabilization invariance, we may also assume that none of the arcs
  involved in the saddle is distinguished.

  Now, as in the proof of
  Proposition~\ref{prop:non-equi-invar}, 
  we can define a
  $\ZZ/2\ZZ$-equivariant triangle map betweeen the freed Floer complexes
  $\wt{\CF}(\Sigma(L_0))\to\wt{\CF}(\Sigma(L_1))$.  We claim that, for suitable
  choices of almost complex structures, this map is induced by a map
  of ordinary (not freed) Floer complexes
  $\CFa(\Sigma(L_0))\to\CFa(\Sigma(L_1))$. The argument is similar to
  the proof of Proposition~\ref{prop:equi-is-equi}. As in
  Proposition~\ref{prop:equi-is-equi}, to define the freed Floer
  complexes for $(\Sigma,\alphas,\betas)$ and
  $(\Sigma,\alphas,\gammas)$ we may choose all of the complex
  structures to agree with a given $\ZZ/2$-equivariant one, in which
  case the freed Floer complexes are just the ordinary Floer complexes
  tensored with the standard resolution for $\Field$ over
  $\Field[\ZZ/2]$.  Recall from Proposition~\ref{prop:non-equi-invar}
  that to define the map of freed Floer complexes we choose a
  collection of families of perturbed almost complex structures over
  the triangle, as well as a collection of perturbations of
  $T_\alpha$. Consider the limit as these families of complex
  structures approach the constant family, and the perturbations
  become constant. By Gromov compactness, the holomorphic triangles
  converge to a holomorphic triangle with sequences of bigons at the
  three corners.  Since the Lagrangians intersect only the discrete
  components of the fixed set, a generic $\ZZ/2$-equivariant almost
  complex structure achieves transversality for all moduli spaces of
  non-constant triangles and bigons. Hence, the only limits with more
  than one component and Maslov index $0$ must have a triangle with
  negative Maslov index contained in the fixed set. However, our
  choice of perturbations ensured that there are no negative-index
  holomorphic triangles contained in the fixed set. The index $0$,
  constant holomorphic triangles are transversely cut out. (To see
  this, note that the constant triangles correspond to small,
  non-constant holomorphic triangles in a nearby, non-equivariant
  diagram; this uses the particular form of the perturbation in
  Figure~\ref{fig:saddle}.) It follows that counting index $0$
  holomorphic triangles with respect to a generic, $\ZZ/2$-equivariant
  almost complex structure and unperturbed $T_\alpha$ gives the same
  map between freed Floer complexes as using a generic, nearby
  collection of perturbed almost complex structures and perturbations
  of $T_\alpha$. This map is induced by the count of index $0$
  holomorphic triangles between ordinary Floer complexes which agrees with the Ozsv\'ath-Szab\'o cobordism map.

  \begin{figure}
    \centering
    \includegraphics{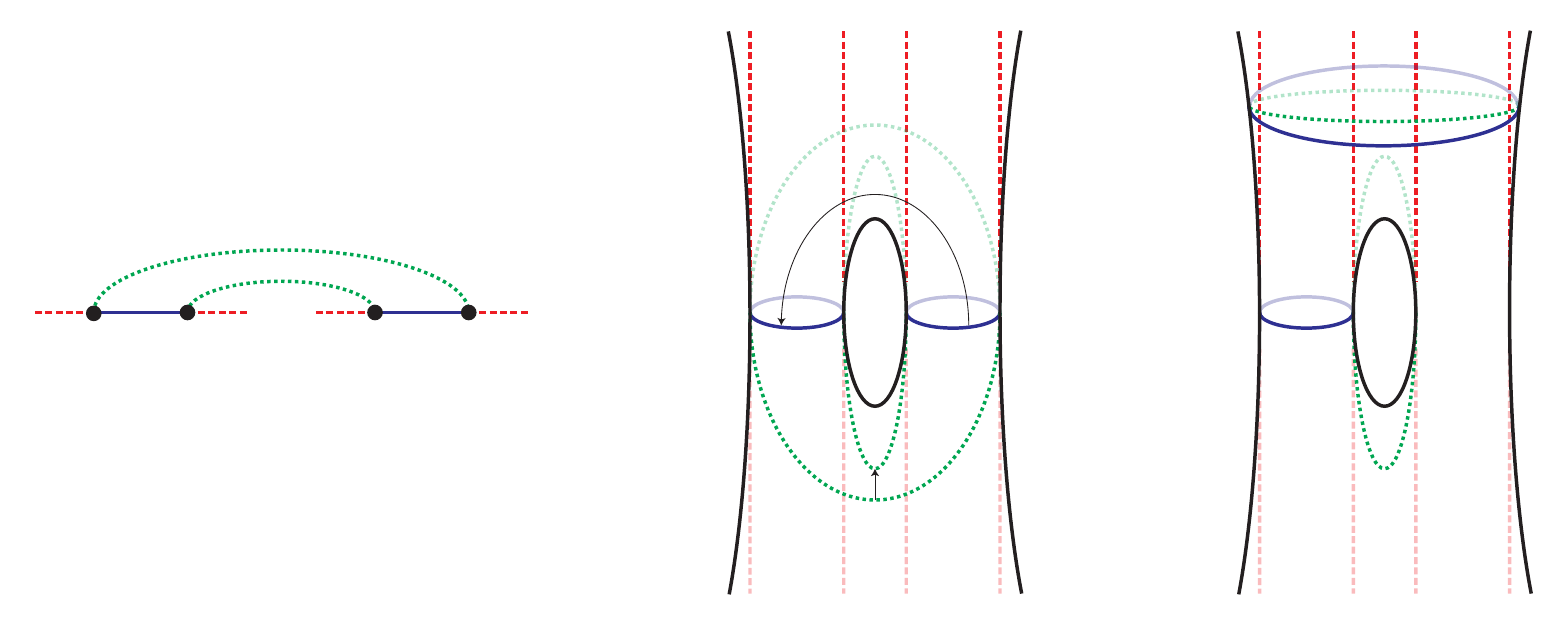}
    \caption{\textbf{Strong equivalence to a subordinate diagram.}
      Left: part of the bridge triple diagram where a saddle
      occurs. Center: the branched double cover of this
      diagram. Right: after performing handleslides according to the
      two small arrows in the center diagram, the resulting diagram is
      subordinate to a link.}
    \label{fig:subordinate}
  \end{figure}

  There is one last technical point for the saddle case.  The
  cobordism map on $\HFa$, which in this case corresponds to a
  2-handle attachment along a framed knot $C$, is defined using a
  Heegaard triple diagram subordinate to $C$.  The Heegaard triple
  diagram $(\Sigma,\alphas,\betas,\gammas)$ is not subordinate to $C$,
  but it is strongly equivalent to a Heegaard triple subordinate to
  $C$ (see Figure~\ref{fig:subordinate}), and so counting triangles in
  $(\Sigma,\alphas,\betas,\gammas)$ gives the Ozsv\'ath-Szab\'o cobordism
  map on $\HFa$~\cite[Proof of Proposition 4.6]{OS06:HolDiskFour}.
  This completes the argument for saddle cobordisms, and the proof.
\end{proof}

\begin{figure}
  \centering
  \includegraphics[scale=1.5]{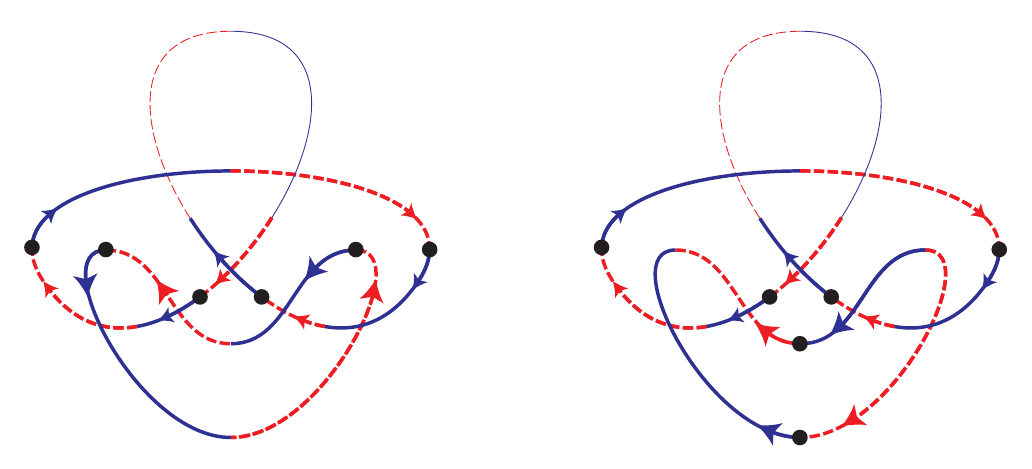}
  \caption{\textbf{Fixed generators and orientations.} The
    distinguished arcs on the based component are thin. Both choices
    of generators, corresponding to the two orientations of the
    unmarked component, are shown.}
  \label{fig:fixed-or}
\end{figure}

Next we study the interaction of the cobordism maps with
$\theta^{-1}\eHFa(\Sigma(L))$.  Recall that
$\theta^{-1}\eHFa(\Sigma(L))$ is identified with the Floer homology of
the $0$-dimensional Lagrangians $T_\alpha^\fix=T_\alpha\cap \Fix(\tau)$ and
$T_\beta^\fix=T_\beta\cap \Fix(\tau)$, tensored with $\Field[\theta,\theta^{-1}]$
(Proposition~\ref{prop:localized-equivariant}). To specify this
identification, it helps to work with a particular class of diagrams.
Call a diagram
\emph{quite twisty} if
\begin{itemize}
\item the fixed generators all lie in the same grading, and
\item any non-fixed generator has grading strictly lower than any
  fixed generator.
\end{itemize}
Any bridge diagram can be made quite twisty by performing
Reidemeister I moves at the endpoints of the arcs, which we will call
\emph{twisting} the diagram; see
Figure~\ref{fig:winding}. (A similar result in the context of symplectic Khovanov homology is Lemma~\ref{lem:ss-fixed-pts-subcx}; see that argument for some more details.) Figure~\ref{fig:or-cob} shows a quite
twisty diagram for the unknot.

Given a quite twisty diagram $(\{A_i\},\{B_i\})$ for $L$, the
fixed generators span a quotient complex of $\CFa(\Sigma(L))$, and
hence there is a chain inclusion
$\CFaDual(T_\alpha^\fix,T_\beta^\fix)\to \CFaDual(\Sigma(L))$ (over
$\Field[\ZZ/2]$). The induced map
\[
\theta^{-1}\eHFa(T_\alpha^\fix,T_\beta^\fix)=\HFaDual(T_\alpha^\fix,T_\beta^\fix)\otimes\Field[\theta,\theta^{-1}]
\to \theta^{-1}\eHFa(\Sigma(L))
\]
is an isomorphism; see Lemma~\ref{lem:trivial-localization}. The Floer
cohomology in the (discrete) fixed set is, of course, trivial, so each
fixed generator gives a homology class in the localized equivariant
Floer cohomology. The standard basis (of intersection points) for the
Floer homology of the fixed set corresponds to the set of orientations
of all of the unmarked components of $L$; see
Figure~\ref{fig:fixed-or}. That is, if $L_\bullet$ is the marked
component of $L$ then
\[
\theta^{-1}\eHFa(\Sigma(L))\cong \Field[\theta,\theta^{-1}]\langle \{\text{orientations of }L\setminus L_\bullet\}\rangle.
\]
(Note that this is the same as the Lee
homology of $L\setminus L_\bullet$~\cite[Theorem
4.2]{Lee05:Khovanov}.) Since the two arcs on the marked
component $L_\bullet$ are adjacent, $L_\bullet$ also
inherits an orientation.

By a \emph{based cobordism} between links $L_0$ and $L_1$ we mean a
cobordism $T\subset [0,1]\times S^3$ so that the projection $\pi\co T\to
[0,1]$ is a Morse function, together with a section $\gamma\co
[0,1]\to T$ of $\pi$ so that
$\Image(\gamma)$ is disjoint from $\Crit(\pi)$. In particular, every
regular slice of a based cobordism is a based link.

\begin{lemma}\label{lem:cob-or}
  Given any connected, oriented, based cobordism $T$ between knots
  $K_0,K_1\subset S^3$, for suitable choices of auxiliary data the map
  \[
  \hat{f}_{\Sigma(T)}^*\co \theta^{-1}\eHFa(\Sigma(K_1))\to\theta^{-1}\eHFa(\Sigma(K_0))
  \]
  induced by the map from Lemma~\ref{lem:hat-cob-maps} is an isomorphism.
\end{lemma}
(We will not show that $\hat{f}_{\Sigma(T)}$ or the induced map on
equivariant cohomology is independent of the choices in its
construction.)
\begin{proof}
  \begin{figure}
    \centering
    \includegraphics{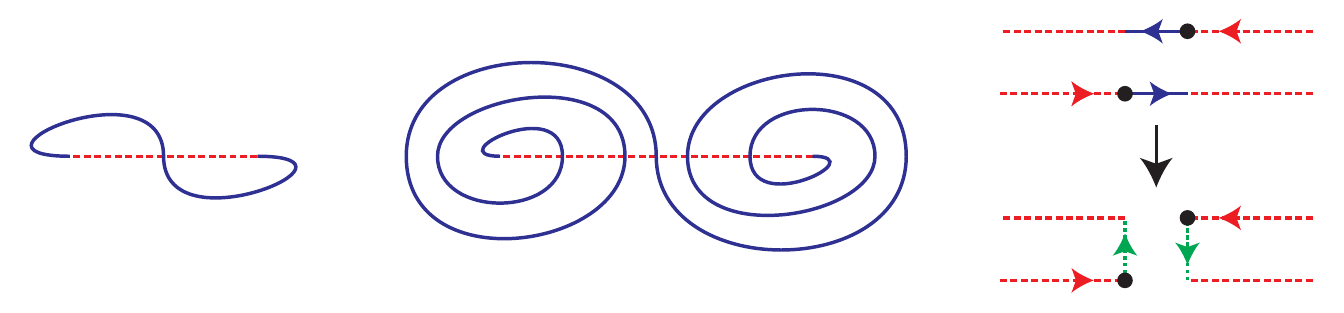}
    \caption{\textbf{Orientations and cobordisms.} Left: a
      quite twisty diagram for an unknot component. Center: an even more twisty unknot. Right:
      orientations and generators for an oriented saddle.}
    \label{fig:or-cob}
  \end{figure}
  
  We will prove the following:
  \begin{enumerate}
  \item\label{item:twisty-1} Given a diagram $L_0$ which has been
    twisted sufficiently and an isotopic link $L_1$, for some diagram
    of $L_1$ the Heegaard Floer continuation map $\CFa(\Sigma(L_0))\to
    \CFa(\Sigma(L_1))$ associated to the isotopy is such that the
    induced map
    \begin{multline*}
    \Field[\theta,\theta^{-1}]\langle \{\text{orientations of }L_1\setminus L_{1,\bullet}\}\rangle \cong \theta^{-1}\eHFa(\Sigma(L_1))\\
    \to
    \theta^{-1}\eHFa(\Sigma(L_0))\cong \Field[\theta,\theta^{-1}]\langle \{\text{orientations of }L_0\setminus L_{0,\bullet}\}\rangle 
    \end{multline*}
    agrees with the identification of sets of orientations induced by
    the isotopy. Further, by twisting $L_0$ the corresponding diagram
    $L_1$ can be made arbitrarily twisted.
  \item If $L_1$ is obtained from $L_0$ by a birth, $L_0$ is quite
    twisty, the new unknot component in $L_1$ is as on the left of
    Figure~\ref{fig:or-cob}, and $o_1,o_2$ are the two orientations of
    the new unknot component of $L_1$, then the cobordism map
    \[
    \theta^{-1}\eHFa(\Sigma(L_1))\cong \theta^{-1}\eHFa(\Sigma(L_0))\otimes\Field\langle o_1,o_2\rangle \to \theta^{-1}\eHFa(\Sigma(L_0))
    \]
    from Lemma~\ref{lem:hat-cob-maps}
    is $x\otimes o_1\mapsto x$, $x\otimes o_2\mapsto x$. More
    generally, for $U$ as in the middle of Figure~\ref{fig:or-cob} the
    cobordism map is $x\otimes o_1\mapsto \theta^{-m} x$, $x\otimes
    o_2\mapsto \theta^{-m} x$ for some $m$.
  \item If $L_1$ is obtained from $L_0$ by a death, and $L_0$ and
    $L_1$ are quite twisty, and $o_1,o_2$ are the two
    orientations of the dying unknot component of $L_0$, then the cobordism map 
    \[
    \theta^{-1}\eHFa(\Sigma(L_1))\to
    \theta^{-1}\eHFa(\Sigma(L_1))\otimes\Field\langle
    o_1,o_2\rangle\cong \theta^{-1}\eHFa(\Sigma(L_0))
    \]
    from Lemma~\ref{lem:hat-cob-maps}
    is $x\mapsto \theta^{-m}x\otimes(o_1+o_2)$ for some $m$.
  \item\label{item:twisty-4} If $S$ is a saddle cobordism from $L_0$
    to $L_1$ compatible with an orientation $o_0$ for $L_0$ and $o_1$
    for $L_1$ then for quite twisty diagrams for $L_0$ and
    $L_1$ as in Figure~\ref{fig:saddle}, the induced map $\theta^{-1}\eHFa(\Sigma(L_1))\to
    \theta^{-1}\eHFa(\Sigma(L_0))$ from Lemma~\ref{lem:hat-cob-maps} sends $o_1$ to $o_0$.
  \end{enumerate}
  If we start with a diagram which has been twisted enough, and the
  births we use are also twisted enough, then we can arrange that all intermediate
  diagrams during the cobordism 
  satisfy the
  conditions in
  Points~(\ref{item:twisty-1})--(\ref{item:twisty-4})
  during the corresponding moves.  The result then follows as in
  the case of Lee homology~\cite[Corollary 4.2]{Rasmussen10:s}.

  We start with Point~(\ref{item:twisty-1}), about bridge moves.  For
  stabilizations near $p_{2n}$, Point~(\ref{item:twisty-1}) is
  obvious. Next, consider an isotopy or handleslide; for concreteness,
  in the case of a handleslide assume it is of a $B$-arc over another
  $B$-arc.  Choose a triple diagram $(\{A_i\},\{B_i\},\{C_i\})$ for
  the isotopy or handleslide so that $(\{A_i\},\{C_i\})$ and
  $(\{B_i\},\{C_i\})$ are also quite twisty, the top-graded generator
  of $\HFa(T_\beta,T_\gamma)$ is in the same grading as the fixed
  generators $T^\fix_\beta\cap T^\fix_\gamma$, and there is a Maslov
  index zero triangle connecting some fixed generators
  $\x_{\alpha,\beta}\in T^\fix_\alpha\cap T^\fix_\beta$,
  $\x_{\beta,\gamma}\in T^\fix_\beta\cap T^\fix_\gamma$, and
  $\x_{\alpha,\gamma}\in T^\fix_\alpha\cap T^\fix_\gamma$.  This can
  be arranged by twisting the diagram $(\{A_i\},\{B_i\})$ sufficiently
  and then choosing $C_i$ to intersect the corresponding $B_i$ in one
  interior point (and the two endpoints). (In visualizing the
  handleslide, it may be helpful to visualize the $B$-arcs as straight
  and the $A$-arcs as twisty.)

 Note that each of $T_\alpha^\fix$, $T_\beta^\fix$ and
  $T_\gamma^\fix$ consists of $2^{n-1}$ points, and by the hypothesis,
  the top-graded generator of $\CFa(T_\beta,T_\gamma)$ is
  $\x_{\mathit{top}}\coloneqq \sum_{\x\in T_\beta^\fix\cap
    T_\gamma^\fix}\x$.

  Under the correspondences $\theta^{-1}\eHFa(\Sigma(L_1))\cong
  \HFaDual(T_\alpha^\fix,T_\gamma^\fix)$ and
  $\theta^{-1}\eHFa(\Sigma(L_0))\cong
  \HFaDual(T_\alpha^\fix,T_\beta^\fix)$ between the localized equivariant
  Floer cohomology and the Floer cohomology in the fixed set, we claim
  that the map $\hat{f}_{\Sigma(T)}^*$ corresponds to the map which
  counts triangles in the fixed set. To see this, note that the
  triangle map is graded and, via small triangles, sends some fixed
  generators to fixed generators. Thus, the triangle map sends the
  grading containing the fixed generators to the grading containing
  the fixed generators, and hence fixed generators of
  $\CF(T_\alpha,T_\beta)$ map only to fixed generators of
  $\CF(T_\alpha,T_\gamma)$. Any triangles not entirely contained in
  the fixed set connecting fixed generators to fixed generators occur
  in pairs (via the $\ZZ/2$-action), and hence do not contribute to
  the count with $\Field$ coefficients. Thus, the map on fixed
  generators counts only triangles in the fixed set. It follows from
  the form of $\x_{\mathit{top}}$ that this map is the
  identity on $T_\alpha^\fix\cap T_\beta^\fix=T_\alpha^\fix\cap
  T_\gamma^\fix$. Tracing through the proof of
  Lemma~\ref{lem:trivial-localization} then proves the first statement.

  The proof of the fourth statement, about saddle moves, is similar.
  We use the triple diagram from the proof of
  Lemma~\ref{lem:hat-cob-maps} (Figure~\ref{fig:saddle}), with the
  further requirement that both $(\{A_i\},\{B_i\})$ and
  $(\{A_i\},\{C_i\})$ are quite twisty; $\{B_i,C_i\}$ is
  quite twisty by construction.  Again, the map
  $\CF(T_\alpha^\fix,T_\beta^\fix)\to
  \CF(T_\alpha^\fix,T_\gamma^\fix)$ on the fixed generators
  corresponds to counting triangles in the fixed set where the
  $\beta$-$\gamma$ corner maps to any point in $T_\beta^\fix\cap
  T_\gamma^\fix$. As illustrated in Figure~\ref{fig:or-cob}, it
  follows that the map respects the identification with orientations.

  Finally, we turn to births and deaths. First, for births, suppose
  that $L_1=L_0\amalg U$ where $U$ is the quite twisty diagram for the
  unknot shown on the left in Figure~\ref{fig:or-cob}, and the
  disjoint union takes place in a region adjacent to the basepoint
  $p_{2n}$. The map $\CFa(\Sigma(L_0))\to \CFa(\Sigma(L_1))$ sends $x$
  to $x\otimes\ttop$. Note that $\ttop$ is the sum of the two fixed
  generators for $U$. Dualizing, the map $\CFaDual(\Sigma(L_1))\to
  \CFaDual(\Sigma(L_0))$ sends $x\otimes o_i$ to $x$.

  More generally, in order to ensure that later diagrams are quite
  twisty, we may need to create the unknot $U'$ shown in the center of
  Figure~\ref{fig:or-cob}, rather than the left.  For this diagram,
  the complex $\CFa(\Sigma(U'))$ has the form:
  \[
  \xymatrix{
    & a\ar[dl]\ar[dr] & & & b\ar[dl]\ar[dr] & \\
    u_{m-1}\ar[d]\ar[drr] & & v_{m-1}\ar[d]\ar[dll] & x_{m-1}\ar[d]\ar[drr] & & y_{m-1}\ar[d]\ar[dll]\\
    u_{m-2}\ar[d]\ar[drr] & & v_{m-2}\ar[d]\ar[dll] & x_{m-2}\ar[d]\ar[drr] & & y_{m-2}\ar[d]\ar[dll]\\
    \vdots\ar[d]\ar[drr] & & \vdots\ar[d]\ar[dll] & \vdots\ar[d]\ar[drr] & & \vdots\ar[d]\ar[dll]\\
    u_1\ar[dr]\ar[drrrr] & & v_1\ar[dl]\ar[drr] & x_1\ar[dr]\ar[dll] & & y_1\ar[dl]\ar[dllll]\\
    & s & & & t & 
  }
  \]  
  The map $\tau$ exchanges $u_i\leftrightarrow v_i$,
  $x_i\leftrightarrow y_i$, and $s\leftrightarrow t$.  The fixed
  generators are the top two, $a$ and $b$, but the homology is
  supported in the bottom two gradings.  
  
  We are interested in the composition of the birth map with the
  twisting isotopy from the slightly twisty unknot $U$ to the
  arbitrarily twisty unknot $U'$. Call this composition $\phi\co
  \eCFa(\Sigma(L_1))=\eCFa(\Sigma(L_0))\otimes_{\Field[\theta]}\eCFa(\Sigma(U'))\to
  \eCFa(\Sigma(L_0))$. (Here, we are using the standard equivariant
  cochain complex, Formula~(\ref{eq:standard-res-rhom}).) Since $U'$
  is adjacent to the basepoint, $\phi$ must be induced by a projection
  map $\pi\co \eCFa(\Sigma(U'))\to \Field[\theta,\theta^{-1}]$.  Since
  on the $E^1$-page of the spectral sequence associated to the
  $\theta$-power filtration, $\pi$ agrees with the dual of the
  Heegaard Floer cobordism map (see Lemma~\ref{lem:hat-cob-maps}),
  $\pi$ must send $x\otimes(u_1^*+v_1^*)$ and $x\otimes (x_1^*+y_1^*)$
  to $x$.  In the localized equivariant complex, there are homologies
  \begin{align*}
    u_i^*+v_i^*&\sim \theta^{-1}(u_{i+1}^*+v_{i+1}^*) & 
    x_i^*+y_i^*&\sim \theta^{-1}(x_{i+1}^*+y_{i+1}^*) \\
    u_{m-1}^*+v_{m-1}^*&\sim \theta^{-1}a^* & x_{m-1}^*+y_{m-1}^*&\sim \theta^{-1}b^*.
  \end{align*}
  The generators $a^*$ and $b^*$ represent the two orientations, so the
  birth map must send $x\otimes o_i$ to $\theta^{1-m}x$, as desired.

  Deaths are similar. Suppose $L_0=L_1\amalg U$ where $U$ is a quite
  twisty diagram for the unknot. Again, we may assume that $U$ has the
  form shown in the middle of Figure~\ref{fig:or-cob}, and is adjacent
  to the basepoint $p_{2n}$.  The Ozsv\'ath-Szab\'o map
  $\CFa(\Sigma(L_0))\to \CFa(\Sigma(L_1))$ sends $x\otimes \tbot$ to
  $x$ and other generators to $0$.  Thus, the cobordism map
  $\eCFa(\Sigma(L_1))\to \eCFa(\Sigma(L_0))\cong
  \eCFa(\Sigma(L_1))\otimes_{\Field[\theta]}\eCFa(\Sigma(U'))$ must
  send $x$ to $x\otimes(s^*+t^*)$. In the localized equivariant
  complex, $s^*+t^*$ is homologous to
  $\theta^{-1}(u_1^*+v_1^*+x_1^*+y_1^*)$ and hence to
  $\theta^{-m}(a^*+b^*)$, so the map sends $x$ to $x\otimes
  \theta^{-m}(a^*+b^*)=\theta^{-m}x\otimes(o_1+o_2)$, as desired.
\end{proof}

\begin{proof}[Proof of Proposition~\ref{prop:dbc-hat-grading-concordance}]
  With Lemmas~\ref{lem:hat-cob-maps} and~\ref{lem:cob-or}, the proof
  follows the same basic structure as the proof that $s$ is a
  concordance invariant~\cite[Theorem 1]{Rasmussen10:s}. Recall
  that
  \[
  q_\tau(K)=2\min\{\gr(x)\mid x\in \eHFa(\Sigma(K))/\tors\},
  \]
  where $\tors$ denotes the subspace annihilated by $\theta^m$ for
  some $m$. (Note that $\tors$ is the sum of its homogeneous pieces,
  since the equation $\theta^mx=d(y)$ implies that each graded piece
  of $x$ is also a boundary, so $\eHFa(\Sigma(K))/\tors$ inherits a
  grading from $\eHFa(\Sigma(K))$.)

  Suppose that $T$ is a concordance between knots $K_0$ and $K_1$. By
  Lemma~\ref{lem:hat-cob-maps}, there is a corresponding map $g_{\Sigma(T)}^{\ZZ/2}\co
  \eHFa(\Sigma(K_1))\to \eHFa(\Sigma(K_0))$ which by Lemma~\ref{lem:cob-or} induces an injective map 
  \[
  \Field[\theta]\cong \eHFa(\Sigma(K_1))/\tors\to
  \eHFa(\Sigma(K_0))/\tors\cong \Field[\theta].
  \]
  Since $\Sigma(T)$ is a rational homology cobordism between rational homology spheres,
  the map $g_{\Sigma(T)}$ preserves the absolute Maslov grading. Now, if $x$ is a
  non-zero element of the module $\eHFa(\Sigma(K_1))/\tors$ in minimal grading
  then its image $g_{\Sigma(T)}^{\ZZ/2}(x)$ is a non-zero element of
  $\eHFa(\Sigma(K_0))/\tors$, and so:
  \[
  q_\tau(K_1)=2\gr(x)=2\gr(g_{\Sigma(T)}(x))\geq q_\tau(K_0).
  \]
  Applying the same argument with the roles of $K_0$ and $K_1$
  exchanged gives the opposite inequality, so
  $q_\tau(K_0)=q_\tau(K_1)$, as desired.
\end{proof}

\begin{proposition}\label{prop:q-homo}
  The number $q_\tau$ satisfies
  $q_\tau(K_1\#K_2)=q_\tau(K_1)+q_\tau(K_2)$
  and $q_\tau(m(K))=-q_\tau(K)$.
\end{proposition}
\begin{proof}
  We start with connected sums.  Fix bridge diagrams
  $(\{A_i\},\{B_i\})$ and $(\{A'_i\},\{B'_i\})$ for $K_1$ and $K_2$,
  respectively. Let $A_n$ and $B_n$ (respectively $A'_1$ and $B'_1$)
  be the arcs incident to a marked point $p_{2n}$ (respectively
  $p'_{1}$). Taking the connected sum of the bridge diagrams at
  $p_{2n}$ and $p'_1$ gives a bridge diagram for $K_1\# K_2$, with
  arcs $A_1,\dots,A_{n-1},A_n\#_bA'_1,A'_2,\dots,A'_m$ and
  $B_1,\dots,B_{n-1},B_n\#_bB'_1,B'_2,\dots,B'_m$. (See
  Figure~\ref{fig:con-sum-bridge}; $A_n\#_bA'_1$ and $B_n\#_bB'_1$ are
  boundary connected sums with respect to boundary points $p_{2n}$ and
  $p'_1$.) If we choose $A_n$ and $B_n$
  as the deleted arcs for $K_1$, $A'_1$ and $B'_1$ as the deleted arcs
  for $K_2$, and $A_n\#_bA'_1$ and $B_n\#_bB'_1$ as the deleted arcs for
  $K_1\#K_2$ then for appropriate choices of almost complex structures
  there is an isomorphism of complexes over $\Field[\ZZ/2]$
  \[
  \CFa(\Sigma(K_1))\otimes_\Field\CFa(\Sigma(K_2))\cong
  \CFa(\Sigma(K_1\#K_2)),
  \]
  with the diagonal action of $\ZZ/2$ on the left hand side. Further,
  this isomorphism respects the absolute $\QQ$-grading (an immediate
  consequence of~\cite[Theorem 4.3]{AbsGraded}).
	 
  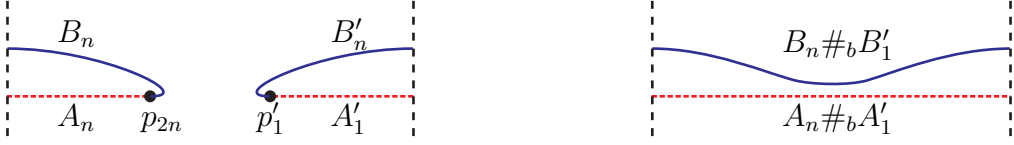
\begin{figure}
    \begin{overpic}[tics=10, scale=1.0]{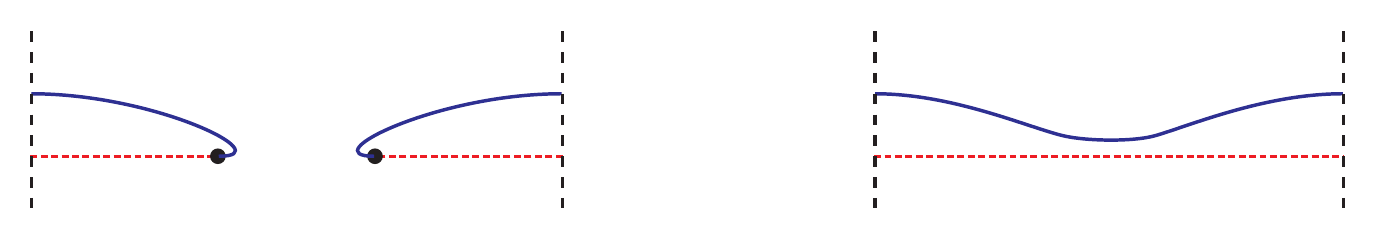}
      \put(7,4){$A_n$}
      \put(7,12){$B_n$}
      \put(33,4){$A'_1$}
      \put(33,12){$B'_n$}
      \put(15,4){$p_{2n}$}
      \put(26,4){$p'_1$}
      \put(76,4){$A_n\#_bA'_1$}
      \put(76,11){$B_n\#_bB'_1$}
    \end{overpic}
    \caption{\textbf{A bridge diagram for the connected 
        sum.\label{fig:con-sum-bridge}}}
  \end{figure}

  Thus, $\eCFa(\Sigma(K_1\#K_2))$ is isomorphic to
  $\eCFa(\Sigma(K_1))\otimes_{\Field[\theta]}\eCFa(\Sigma(K_2))$, but
  the differentials are superficially different: the differential on
  $\eCFa(\Sigma(K_1\#K_2))$ is $d(\theta^ix\otimes
  y)=\theta^i\left[d_{\CFa}(x)\otimes y+x\otimes d_{\CFa}(y)\right] +
  \theta^{i+1}\left[x\otimes y + \tau(x)\otimes\tau(y)\right]$, while
  the tensor product differential is $d(\theta^ix\otimes
  y)=\theta^i\left[d_{\CFa}(x)\otimes y+x\otimes d_{\CFa}(y)\right] +
  \theta^{i+1}\left[\tau(x)\otimes y + x\otimes \tau(y)\right]$. The
  map $\phi\co
  \eCFa(\Sigma(K_1))\otimes_{\Field[\theta]}\eCFa(\Sigma(K_2))\to
  \eCFa(\Sigma(K_1\#K_2))$, $\phi(\theta^ix\otimes y)=\theta^i
  x\otimes \tau^i(y)$ intertwines the two differentials. This map does
  not respect the $\Field[\theta]$-module structures, but does respect
  the action of $\Field[\theta^2]$.

  By the universal coefficient theorem,
  $H_*(\eCFa(\Sigma(K_1))\otimes_{\Field[\theta]}\eCFa(\Sigma(K_2)))/\tors$
  is isomorphic to $\eHFa(\Sigma(K_1))/\tors \otimes
  \eHFa(\Sigma(K_2))/\tors$.  (Here, $\tors$ stands for the submodule
  annihilated by $\theta^m$ for some $m$.)  Since the map $\phi$
  respects the action of $\Field[\theta^2]$, it also takes
  (non-)torsion classes to (non-)torsion classes. Thus, the
  composition
  \begin{align*}
    \eHFa(\Sigma(K_1))/\tors \otimes \eHFa(\Sigma(K_2))/\tors &\longrightarrow 
    H_*(\eCFa(\Sigma(K_1))\otimes_{\Field[\theta]}\eCFa(\Sigma(K_2)))/\tors\\
    &\stackrel{\phi}{\longrightarrow} \eHFa(\Sigma(K_1\#K_2))/\tors
  \end{align*}
  is an isomorphism.
  It follows that $q_\tau(K_1\#K_2)=q_\tau(K_1)+q_\tau(K_2)$.
  
  The argument for mirrors is easier. Mirroring a knot corresponds to
  exchanging the $A_i$ and $B_i$ arcs, which has the effect of
  replacing $\CFa(\Sigma(K))$ with
  $\HomO{\Field}(\CFa(\Sigma(K)),\Field)$. The behavior on gradings is
  that if $\{x_i\}$ is a homogeneous basis for $\CFa(\Sigma(K))$ with
  dual basis $\{x_i^*\}$ then the grading of $x_i^*$, viewed as an
  element in $\CFa(\Sigma(m(K)))$, is minus the grading of
  $x_i$~\cite[Proposition 4.2]{AbsGraded}. This isomorphism holds
  $\ZZ/2$-equivariantly. It follows that
  $\eCFa(\Sigma(m(K)))\cong\HomO{\Field[\theta]}(\eCFa(\Sigma(K)),\Field[\theta])$.
  Applying the universal coefficient theorem,
  \[
  \eHFa(\Sigma(m(K)))/\tors\cong
  \HomO{\Field[\theta]}(\eHFa(\Sigma(K))/\tors,\Field[\theta]),
  \]
  and it follows 
  that $q_\tau(m(K))=-q_\tau(K)$.
\end{proof}

\begin{proof}[Proof of Theorem~\ref{thm:q-tau-conc-homo}]
  This is simply a restatement of
  Propositions~\ref{prop:dbc-hat-grading-concordance}
  and~\ref{prop:q-homo}.
\end{proof}

\subsection{Many concordance invariants from the \texorpdfstring{$\ZZ/2$}{Z/2}-action on \texorpdfstring{$\CF^-(\Sigma(K))$}{CF-minus of the branched double cover}}
For this section, for convenience, we will restrict our attention to
the case that $K$ is a knot, not a link.

Consider the freed complex $\ECFm(\Sigma(K))$ for the
Heegaard Floer invariant $\HF^-(\Sigma(K))$. Unlike $\CFa(\Sigma(K))$,
the complex $\CF^-(\Sigma(K))$ counts curves in the closed symplectic
manifold $\Sym^g(\Sigma)$, which does not satisfy Item~\ref{item:J-1}
of Hypothesis~\ref{hyp:Floer-defined}. The construction of the freed
Floer complex used this part of the hypothesis for two reasons: for
energy bounds allowing us to work with $\Field$ (or $\Field[U]$) coefficients 
instead
of Novikov coefficients, and for a relative $\ZZ$-grading. To work with
$\Field$ coefficients, the following weaker condition suffices:
\begin{enumerate}[label=(J$'$-\arabic*)]
\item There is a non-negative $\tau\in\RR$ so that for any loop of
  paths
  \[
  v\co \bigl([0,1]\times S^1,\{0\}\times S^1,\{1\}\times S^1\bigr)\to 
  (M,L_0,L_1)
  \]
  such that each $[v|_{[0,1]\times \{\pt\}}]$ is homotopic to a
  constant path from $L_0$ to $L_1$, the area of $v$ (with respect to
  $\omega$) is $\tau$ times the Maslov index of $v$. (Again, compare,
  for example,~\cite[Theorem 1.0.1]{WW12:compose-correspond}.)
\end{enumerate}
For $\CF^-(\Sigma(K))$, one actually works with coefficients in
$\Field[U]$, where the power of $U$ keeps track of the intersection
number with the divisor $V_z$ in the symmetric product corresponding
to some basepoint $z$ on the Heegaard surface.
The resulting complex is a relatively $\ZZ$-graded complex
over $\Field[U]$, where $U$ has degree $-2$~\cite[Section
4.2]{OS04:HolomorphicDisks}. With these observations, the freed Floer
complex is defined as in
Section~\ref{sec:build-equi-cx}, with $U$-powers always tracking
intersections numbers with $V_z$. The equivariant Floer complex is 
\[
\eCF^-(\Sigma(K))=\HomO{\Field[U][\ZZ/2]}(\ECFm(\Sigma(K)),\Field[U])
\]
and the equivariant Floer cohomology $\eHF^-(\Sigma(K))$ is the
homology of $\eCF^-$. Both the equivariant complex $\eCF^-(\Sigma(K))$
and its homology $\eHF^-(\Sigma(K))$
are modules over $\Field[U,\theta]$. As discussed earlier in the case of $\CFa$, there are no
non-constant holomorphic disks contained in the fixed set, so by
Proposition~\ref{prop:equi-is-equi} we can compute $\eHF^-(\Sigma(K))$
using a generic 1-parameter family of $\ZZ/2$-equivariant almost complex structures.

In what follows, we will be particularly interested in the summand
$\HF^-(\Sigma(K),\spinc_0)$, where $\spinc_0$ is the unique spin
structure on $\Sigma(K)$, also called the \emph{central
  $\SpinC$-structure}. The involution
$\tau \co \Sigma(K) \rightarrow \Sigma(K)$ acts on $\Spinc (Y)$ by
conjugation and fixes only ${\spinc}_0$ \cite[page
1378]{Grigsby06:cyclic-covers}, \cite[Remark
3.4]{Levine08:cycliccovers}. Therefore, $\eCF^-(\Sigma(K))$ splits
along orbits of $\SpinC$-structures. 

\begin{theorem}\label{thm:new-dcov-minus}
  There are isomorphisms between the completed, localized equivariant Floer 
  homologies
  \begin{align}
     U^{-1}\eHF^-(\Sigma(K))\otimes_{\Field[\theta]}\Field[[\theta]]&\cong 
     \Field[U,U^{-1}][[\theta]]\label{eq:minus-U-loc}\\
    \theta^{-1}\eHF^-(\Sigma(K))\otimes_{\Field[U]}\Field[[U]]&\cong 
    \Field[\theta,\theta^{-1}][[U]]\label{eq:minus-theta-loc} \\
   (U\theta)^{-1}\eHF^-(\Sigma(K))\otimes_{\Field[U,\theta]}\Field[[U,\theta]]&\cong 
   \Field[[U,\theta]][U^{-1},\theta^{-1}]\label{eq:minus-both-loc}.
 \end{align}
\end{theorem}
\begin{proof}
	We start with the proof of Equation~\eqref{eq:minus-U-loc}. Consider the 
	filtration of $\eCF^-(T_\alpha,T_\beta)$ by the $\theta$-power, i.e., so 
	that, using the notation of Observation \ref{obs:concrete-hocolim}, $\alpha_n^*\otimes x^*$ lies in filtration $n$. The $E_1$-page of the 
	associated spectral sequence is isomorphic to 
	$\HF^{-,*}(\Sigma(K))\otimes\Field[\theta]$.
        Each 
	$\HF^{-,*}(\Sigma(K))\otimes\{\theta^i\}$ decomposes (non-canonically) as 
	$\Field[U]\oplus \HF^{\red,*}(\Sigma(K))$ where
        $\HF^{\red,*}(\Sigma(K))$ is 
	$U^N$-torsion for some $N$~\cite[Theorem 10.1]{OS04:HolDiskProperties}, so 
	the $E_1$-page itself is $\Field[U,\theta]\oplus 
	\HF^{\red,*}(\Sigma(K))\otimes\Field[\theta]$. Let $x$ be a generator of the 
	$\Field[U,\theta]$ submodule. It follows from the fact that the induced 
	differentials $d_r$ on the $r$th page of the spectral sequence commute with the 
	$\Field[U,\theta]$-module structure and $(d_r)^2=0$ that $d_r(U^nx)$ must be 
	$U$-torsion. It then follows from finite-dimensionality of 
	$\HF^{\red,*}(\Sigma(K))$ and the $\Field[U,\theta]$-module structure that 
	$d_r=0$ for $r$ sufficiently large. Thus, the $E_\infty$-page of the 
	spectral sequence is isomorphic to $\Field[U,\theta]\langle 
	U^nx\rangle\oplus T$, where $T$ is $U^N$-torsion (for some $N$).
	
	Now, choose a cycle $x_n\in \eCF^-(T_\alpha,T_\beta)$ whose image in the 
	associated graded complex $E_0$ represents $U^nx$. There is a corresponding 
	map $\Field[U,\theta]\to \eCF^-(T_\alpha,T_\beta)$ given by 
	$(U^i\theta^j)\mapsto U^i\theta^jx_n$. The induced map 
	$\Field[U,U^{-1},\theta]\to U^{-1}\eCF^-(T_\alpha,T_\beta)$ is by 
	definition a map of $\Field[U,U^{-1},\theta]$-modules. If we endow 
	$\Field[U,U^{-1},\theta]$ with the filtration $\Filt(U^i\theta^j)=j$ then 
	the map respects the filtrations and induces an 
	isomorphism on the $E_1$-page of the spectral sequence, and hence after 
	completing with respect to $\theta$ the map itself is a 
	filtered quasi-isomorphism~\cite[Corollary 3.15]{McCleary01:sseq}.

	The proof of the Equation~\eqref{eq:minus-theta-loc} is the same as the 
	proof of Equation~\eqref{eq:minus-U-loc}, with the roles of $U$ and $\theta$ 
	exchanged, and using Proposition~\ref{prop:localized-equivariant} in place of 
	Ozsv\'ath-Szab\'o's computation of $\HF^\infty$~\cite[Theorem 
	10.1]{OS04:HolDiskProperties}
	to analyze the $E_1$-page of the spectral sequence. 
	Equation~\eqref{eq:minus-both-loc} is immediate from either of 
	Equation~\eqref{eq:minus-U-loc} or~\eqref{eq:minus-theta-loc}.
\end{proof}

Before we discuss how to extract numerical invariants from $\eCF^-(\Sigma(K),\spinc_0)$, we pause for a remark about dualizing the chain complex $\CF^-(\Sigma(K), \spinc_0)$.

\begin{remark} 
  Below, we will often use the Heegaard Floer cohomology $\HF^{-,*}(\Sigma(K), \spinc_0)$, which we take to be the homology of the dual complex 
  $
    \Hom_{\Field[U]}(\CFm(\Sigma(K),\spinc_0),\Field[U]).
  $ 
  This is a different Floer cohomology from
  Ozsv\'ath-Szab\'o~\cite[Section 2.2]{OS04:HolDiskProperties}: we have
  dualized over $\Field[U]$ rather than over $\Field$. Consequently,
  $\HF^{-,*}(\Sigma(K),\spinc_0)$ is isomorphic to
  $\HF^-(-\Sigma(K),\spinc_0)$. If this isomorphism sends $x$ to $x'$,
  then $\gr(x)=-\gr(x')-4$. 
\end{remark}

We can now discuss how to extract numerical invariants from $\eCF^-(\Sigma(K),
\spinc_0)$. Consider the filtration of $\eCF^-(\Sigma(K),
\spinc_0)$ by the $\theta$-power. The $E_0$-page of the associated
spectral sequence is, in each filtration level, a copy of the dual
complex $\Hom_{\Field[U]}(\CFm(\Sigma(K),\spinc_0),\Field[U])$. The
homology of this complex is the Heegaard Floer cohomology
$\HF^{-,*}(\Sigma(K), \spinc_0)$, so the $E_1$-page of the spectral
sequence is $\HF^{-,*}(\Sigma(K), \spinc_0)\otimes \Field[\theta]$.
Let $\gr(x)$ denote the cohomological grading on the spectral
sequence, so the differential raises $\gr$ by $1$. With respect to the
grading $\gr$, the variables $U$ and $\theta$ have degrees $2$ and $1$
respectively.

As an $\Field[U]$ module, $\HF^{-,*}(\Sigma(K), \spinc_0)$ decomposes into the direct sum of
two submodules. One is a canonical finite dimensional $\Field$-vector
space $R_1=\HF^{\red,*}(\Sigma(K),\spinc_0)$ consisting of elements which are $U^n$ torsion for some
$n$. The other is a (non-canonical) copy of $\Field[U]$. The minimal grading of a non-$U$-torsion homogeneous element of
$\HF^{-,*}(\Sigma(K),\spinc_0)$ is an invariant of the pair
$(\Sigma(K),\spinc_0)$ which is equal to the
Ozsv{\'a}th-Szab{\'o} correction term $d(\Sigma(K), \spinc_0)$ minus
two~\cite[Section 4]{AbsGraded}. Let $x$ be a homogeneous generator of
the $\Field[U]$ submodule.  Since the first page of the spectral sequence 
associated to the $\theta$-filtration of $\eCF^{-}(\Sigma(K), \spinc_0)$ is equal to
$\HF^{-,*}(\Sigma(K),\spinc_0)\otimes \Field[\theta]$, it must decompose as 
$\Field[U,\theta] \oplus (R_1\otimes
\Field[\theta])$, where $\Field[U, \theta]$ is generated by $x$.

Now consider the differential $d_1$ on this complex, which raises the
$\theta$-power by one. This differential is equivariant with respect
to the actions of both $U$ and $\theta$. Indeed, because $(d_1)^2 =
0$, we see that $d_1$ is identically zero on $U^n E_1\cong \Field[U,
\theta]$ for sufficiently high $n$. Let
$d_1(x) = \theta y$; $y$ is forced to be $U$-torsion, so let $m$ be
the smallest integer so that $U^my=0$.  The $E_2$-page of the spectral
sequence is given by $\Field[U, \theta] \oplus R_2$, where $R_2$ is a
vector space consisting of elements that are $U^n$ torsion for some
$n$ (and may or may not be $\theta$-torsion) and $\Field[U,\theta]$ is
generated by $[U^mx]$. A model of such a spectral sequence is drawn in
Figure~\ref{fig:minusspectralsequence}.

\begin{figure}
  \centering
  \includegraphics{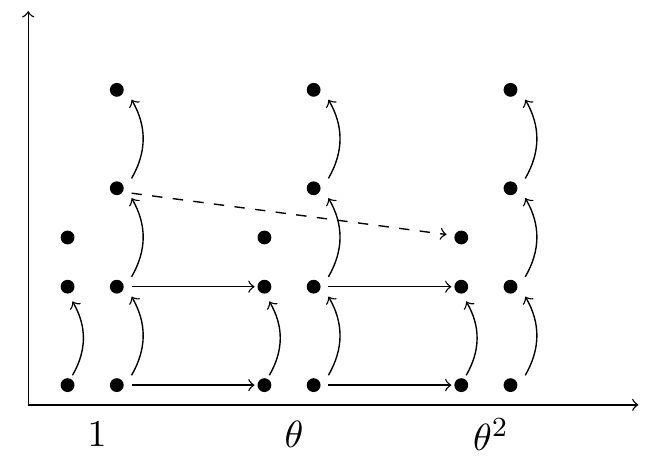}
  \caption{\textbf{A model for $HF^-$ spectral sequence.} This is a sample 
  $E_1$-page for the spectral sequence associated to the filtration of 
  $\eHF(\Sigma(K),\spinc_0)$ by the $\theta$-power. In each filtration level we 
  have a copy of $\HF^{-,*}(\Sigma(K), \spinc_0)$; the vertical axis is the homological grading in this complex. Curved vertical arrows denote 
  the action of $U$; horizontal arrows show a possible nontrivial $d_1$ 
  differential. The dashed arrow is a possible $d_2$ differential on the next 
  page of the spectral sequence. }
  \label{fig:minusspectralsequence}
\end{figure}

We see a similar story on subsequent pages: on page $E_s$, the
differential $d_s$ is trivial on $U^nE_s$ for $n \gg 0$, but there may
be nontrivial differentials from elements which are not $U^n$-torsion
to elements that lie in low $U$-gradings. So, define
\[
d_{\tau}(K,s) = \min\{i \mid \exists x \in E_s, \ \gr(x)=i, \ U^n x \neq 0\ \forall 
n\} + 2.
\]
Notice that $d_{\tau}(K,1) = d(\Sigma(K), \spinc_0)$, which is, up to
scaling, the concordance homomorphism $\delta$ studied by
Manolescu-Owens~\cite{ManolescuOwens07:delta}:
$d_\tau(K,1)=\delta(K)/2$.  Also, $d_{\tau}(K,s) \leq d_{\tau}(K,t)$
whenever $s < t$, and finiteness of the vector space $R_1$ implies
that $\{d_{\tau}(K,s)\}$ assumes at most finitely many values over all
natural numbers $s$.

\begin{theorem} \label{thm:minus-invariance} The quasi-isomorphism
  class of $\eCF^-(\Sigma(K))$, as a $\gr$-graded chain complex over
  $\Field[U, \theta]$, is an invariant of the knot $K$.
\end{theorem}
\begin{proof}
  The proof is the same as the proofs of
  Theorem~\ref{thm:sar-dcov-invt} and
  Proposition~\ref{prop:dbc-hat-grading-invariance}, with two caveats. First,
  for the surgery exact triangle to hold for $\HF^-$, we must restrict
  to the torsion $\SpinC$-structure on the new $S^2\times S^1$-summand
  of $\HD_0$; compare Convention~\ref{conv:torsion}. Second, the proof
  that $\CF^-(\HD_0)\cong \CF^-(\HD)\oplus\CF^-(\HD)$
  (point~\eqref{item:inv-con-sum} from the proof of
  Theorem~\ref{thm:minus-invariance}) is somewhat harder, requiring a
  degeneration argument~\cite[Proposition
  6.4]{OS04:HolDiskProperties}, but for a suitably pinched almost complex structure the isomorphism is still the obvious one and, in
  particular, is still $\ZZ/2$-equivariant (cf.\ proof of Theorem~\ref{thm:Hen1-small}).
\end{proof}

\begin{corollary} \label{cor:minus-grading-invariance} The numbers 
$d_{\tau}(K,m)$ are knot invariants.
\end{corollary}
\begin{proof}
  This is immediate from Theorem~\ref{thm:minus-invariance}.
\end{proof}

Next we show that the invariants $d_\tau(K,i)$ are concordance invariants.
\begin{lemma}\label{lem:minus-cob-maps}
  Given an oriented cobordism $T$ from $L_0$ to $L_1$ there is a map of equivariant Floer complexes $\eCF^-(\Sigma(L_1))\to \eCF^-(\Sigma(L_0))$ so that:
  \begin{itemize}
  \item The induced map on the $E^1$-page of the
    spectral sequence induced by the $\theta$-power filtration is
    \[
    ({F}_{\Sigma(T)}^*\otimes\Id)\co \HF^-(\Sigma(L_1))\otimes\Field[\theta]\to\HF^-(\Sigma(L_0))\otimes\Field[\theta]
    \]
    where ${F}_{\Sigma(T)}$ is the Ozsv\'ath-Szab\'o cobordism map associated to $\Sigma(T)$, and
  \item The map $\eCF^-(\Sigma(L_1))\to\eCF^-(\Sigma(L_0))$ has the same effect on gradings as the dual of the Ozsv\'ath-Szab\'o cobordism map ${F}_{\Sigma(T)}$.
  \end{itemize}
\end{lemma}
\begin{proof}
  Replace all hats by minuses in the proof of
  Lemma~\ref{lem:hat-cob-maps}, and when dealing with births and
  deaths, work with a suitably pinched almost complex structure (as in
  the proof of Theorem~\ref{thm:minus-invariance}).
\end{proof}

\begin{theorem} \label{thm:d-tau-concordance} 
  The numbers $d_{\tau}(K,m)$ are concordance invariants.
\end{theorem}
\begin{proof}
  Fix a concordance $T$ from $K_0$ to $K_1$. Because $\Sigma(T)$ is a
  rational homology cobordism, by the grading shift
  formula~\cite[Theorem 7.1]{OS06:HolDiskFour} the map
  $f_{\Sigma(T)}\co \CF^-(\Sigma(K_0), \spinc_0)\to\CF^-(\Sigma(K_1), \spinc_0)$ induced
  by the unique spin structure on $\Sigma(T)$ preserves the absolute Maslov
  grading. Further, by
  Lemma~\ref{lem:minus-cob-maps}, the map $f_{\Sigma(T)}$ induces a
  map $f_{\Sigma(T)}^*\co \eCF^{-}(\Sigma(K_1),\spinc_0)\to
  \eCF^{-}(\Sigma(K_0),\spinc_0)$ on equivariant Floer cochain
  complexes. If $E_s(\Sigma(K_i))$ denotes the $s\th$ page in the
  spectral sequence associated to the filtration on
  $\eCF^-(\Sigma(K_i),\spinc_0)$ by $\theta$-powers, then the fact
  that $f_{\Sigma(T)}$ induces an isomorphism on $\HF^\infty$ implies
  that
  \[
  f_{\Sigma(T)}^*\co E_1(\Sigma(K_1))/\text{$U$-torsion}\cong \Field[U,\theta]\to \Field[U,\theta]\cong E_1(\Sigma(K_0))/\text{$U$-torsion}
  \]
  is injective (or, equivalently, non-zero). (Here, $U$-torsion means
  the submodule spanned by elements annihilated by $U^n$ for some $n$.) From the
  $\Field[U,\theta]$-module structure and $d_1^2=0$, it follows that
  the image of the differential $d_1$ is contained in the $U$-torsion
  submodule of $E_1$, so in particular 
  \[
  E_2(\Sigma(K_i))/\text{$U$-torsion}=U^n\Field[U,\theta]\subseteq \Field[U,\theta]\cong E_1(\Sigma(K_i))
  \]
  and so the induced map $f_{\Sigma(T)}^*$ on $E_2/\text{$U$-torsion}$
  is injective. Repeating this argument, the induced map on
  $E_s/\text{$U$-torsion}$ is injective for all $s$.

  Note that
  \[
  d_\tau(K,s)=\min\{\gr(x)\mid x\in E_s(\Sigma(K))/\text{$U$-torsion}\}+2.
  \]

  Since $f_{\Sigma(T)}^*$ is injective on $E_s/\text{$U$-torsion}$, we have
  \[
  \min\{\gr(x)\mid x\in E_s(\Sigma(K_1))/\text{$U$-torsion}\}
  \geq \min\{\gr(x)\mid x\in E_s(\Sigma(K_0))/\text{$U$-torsion}\},
  \]
  so $d_\tau(K_1,s)\geq d_\tau(K_0,s)$. But since the property of
  being concordant is symmetric, we also have $d_\tau(K_0,s)\geq
  d_\tau(K_1,s)$, implying the result.
\end{proof}

\begin{remark} Recall that the spectral sequence (\ref{eq:Sig-to-S3}) mentioned 
in the introduction has the form
\[
\HFaDual(\Sigma(K))\otimes \Field[\theta, \theta^{-1}]\otimes V^{\otimes n} 
\Rightarrow V^{\otimes n} \otimes \Field[\theta,\theta^{-1}]
\]
for knots (the version for links is slightly more complicated). We do not know 
if the spectral sequences (\ref{eq:new-dbc-sequence}) and (\ref{eq:Sig-to-S3}) 
are identified up to tensoring with $V^{\otimes n}$. 
\end{remark}

\begin{remark}
  One can also define a family of integer invariants $q_\tau(K,m)$ analogous 
  to $d_\tau(K,m)$ but using the spectral sequence associated to the $U$-power 
  filtration instead of the 
  $\theta$-power filtration on $\CF^-_{\ZZ/2}(\Sigma(K),\spinc_0)$. These invariants satisfy 
  $q_\tau(K)=q_\tau(K,1)\leq q_\tau(K,2)\leq\cdots$.
\end{remark}

\subsection{Computations}\label{sec:computations}

We will concentrate on the homology sphere $\Sigma(2,3,7)$ which is:
\begin{enumerate}[leftmargin=*]
\item The double branched cover of the torus knot $T(3,7)$ (a positive
  knot).
\item The double branched cover of the Montesinos knot
  $M(-1;(-2,1),(-3,1),(-7,1))$, which is also the pretzel knot
  $P(2,-3,-7)$ (a negative knot).
\item The $(-1)$-surgery of the positive trefoil $T(2,3)$.
\end{enumerate}
$\HFa(\Sigma(2,3,7))$ is three-dimensional: two-dimensional in grading
$0$ and one-dimensional in grading $-1$; and $\HFm(\Sigma(2,3,7))$ is
of the form $\Field[U]\langle \alpha\rangle\oplus \Field\langle \beta\rangle$
with $\gr(\alpha)=\gr(\beta)=-2$~\cite[Equation~(25)]{AbsGraded}.

\begin{lemma}\label{lem:q-tau-in-simple-cases}
  Assume $K\subset S^3$ is a knot so that $\HFa(\Sigma(K))$ is
  three-dimensional and is supported in two adjacent gradings:
  two-dimensional in grading $Q$ and one-dimensional in grading
  $Q-1$. If the involution on the double branched cover $\Sigma(K)$
  induces the identity map on $\HFa(\Sigma(K))$ then $q_{\tau}(K)=2Q$,
  $d_{\tau}(K,2)=Q$, and $d_{\tau}(m(K),2)=-Q$. Otherwise,
  $q_{\tau}(K)=2(Q-1)$, $d_{\tau}(K,2)=Q$, and
  $d_{\tau}(m(K),2)=-Q+2$. 
\end{lemma}
(As before, $m(K)$ denotes the mirror of $K$.)
\begin{proof} 
We begin with the computation of $q_{\tau}$. Let $\HFa(\Sigma(K))$ be generated over $\Field$ by two elements $\eta$ and $\nu$ in degree $Q$ and an element $\zeta$ in degree $Q-1$. Even though the isomorphism $\HFa(K)\simeq \HFaDual(K)$ is not canonical, by looking at a particular Heegaard diagram and dualizing the chain complex $\CFa(\Sigma(K))$ obtained from it, we see that the involution $\tau^*$ on $\HFaDual(\Sigma(K))$ is the identity if and only if the involution $\tau_*$ on $\HFa(K)$ is the identity.

First we suppose that $\tau_*$ is not the identity on
$\HFa(\Sigma(K))$, so that $\tau^*$ is not the identity on
$\HFaDual(K)$. Up to change of basis, we may assume that
$\tau^*(\nu) = \nu+\eta$ and $\tau^*(\zeta)=\zeta$. Then in the
spectral sequence
$\HFaDual(\Sigma(K)) \otimes \Field[\theta] \Rightarrow
\eHFa(\Sigma(K))$,
the differential $d_1 = \theta(1+ \tau^*)$ on the first page maps
$\nu\otimes \theta^k$ to $\eta \otimes \theta^{k+1}$ for all $k$. This
implies that the $E_2$ page of the spectral sequence decomposes as
$\Field[\theta]\langle [\zeta] \rangle \oplus \langle [\nu] \rangle$.
It follows from grading considerations that the spectral sequence
collapses at the $E_2$ page, implying that
$q_{\tau}(K) = 2 \deg([\zeta]\otimes \theta^0)=2(Q-1)$.

Now suppose that $\tau_*$ is the identity on $\HFa(K)$. Then $\tau^*$ is also the identity on $\HFaDual(K)$, implying that in the spectral sequence $\HFaDual(\Sigma(K)) \otimes \Field[\theta] \Rightarrow \eHFa(\Sigma(K))$, the differential $d_1 = \theta(1+ \tau^*)$ on the first page is identically zero. Therefore the $E_2$ page of this spectral sequence is isomorphic to $\HFaDual(\Sigma(K)) \otimes \Field[\theta]$. Furthermore, we see that for grading reasons, the induced differentials $d_r$ on the $E_r$ page of the spectral sequence for $r\geq 3$ must be zero. Since we know that $\theta^{-1}\eHFa(\Sigma(K))\simeq \Field[\theta,\theta^{-1}]$, this implies that there must be a nonzero differential on the $E_2$ page of the spectral sequence. Possibly after a change of basis, we may assume that $d_2([\eta]\otimes \theta^k) = [\zeta]\otimes \theta^{k+2}$. This implies that the $E_3$ page of the spectral sequence decomposes as the direct sum $\Field[\theta]\langle [\nu]\rangle \oplus (\Field[\theta]/(\theta^2))\langle [\zeta] \rangle$. The spectral sequence must collapse at this point, so $q_{\tau}(K) = 2\deg([\nu]\otimes \theta^0) = 2Q$.

Now we turn our attention to $d_{\tau}(K,2)$ and $d_{\tau}(m(K),2)$. First,
recall that there is a long exact sequence
\begin{equation} \label{eq:minus-exact}
\cdots \rightarrow \HF^-(\Sigma(K)) \xrightarrow{\cdot U} \HF^-(\Sigma(K)) \rightarrow \HFa(\Sigma(K)) \rightarrow \HF^-(\Sigma(K))\rightarrow \cdots 
\end{equation}
such that the map $\HF^-(\Sigma(K))\to \HFa(\Sigma(K))$ increases the
grading by $2$ and the map $\HFa(\Sigma(K))\to\HF^-(\Sigma(K))$
decreases the grading by $1$ \cite[Proposition
2.1]{OS04:HolDiskProperties}. It follows from this long exact sequence
and the gradings for the isomorphism $\HFa(-Y)\cong \HFaDual(Y)$ that
there are non-canonical isomorphisms:
\begin{itemize}
\item $\HF^-(\Sigma(K)) \cong \Field[U] \langle \alpha \rangle \oplus \Field \langle \beta \rangle$, where both $\alpha$ and $\beta$ lie in grading $Q-2$.
\item $\HFa(\Sigma(m(K)))$ is two-dimensional in grading $-Q$ and one-dimensional in grading $-Q+1$.
\item $\HF^-(\Sigma(m(K))) \cong \Field[U] \langle \gamma \rangle \oplus \Field \langle \epsilon \rangle$, where $\gamma$ lies in grading $-Q-2$ and $\epsilon$ lies in grading $-Q-1$.
\end{itemize}

In particular, $\HF^-(\Sigma(m(K)))$ is one-dimensional in each homological grading, implying that $\tau_*$ is necessarily the identity on $\HF^-(\Sigma(m(K)))$. From this we conclude that $\tau^*$ is the identity map on $\HF^{-,*}(\Sigma(K))$, and therefore that $d_{\tau}(K,2) = d(\Sigma(K),\spinc_0) = Q$.

We now turn to $d_{\tau}(m(K), 2)$, which requires a closer look at
the long exact sequence~(\ref{eq:minus-exact}). From the grading
shifts, we see that summand of $\HFa(\Sigma(K))$ in grading $Q$ is
exactly the image of the summand of $\HF^-(\Sigma(K))$ in grading
$Q-2$, which is spanned as a vector space by $\alpha$ and
$\beta$. Since the long exact sequence~(\ref{eq:minus-exact}) respects
the action $\tau_*$, the involution on $\HFa(\Sigma(K))$ is determined
by the involution on $\HF^-(\Sigma(K))$. There are exactly two
$U$-equivariant involutions on $\HF^-(\Sigma(K))$: the identity and
the involution $\tau_*(\alpha) = \alpha + \beta$ and
$\tau_*(\beta)=\beta$. The first of these induces the identity
involution on $\HFa(\Sigma(K))$, and the second induces the unique
nontrivial involution on $\HFa(\Sigma(K)$.

This leaves us with two cases. First, suppose that $\tau_*$ is the
identity map on $\HFa(\Sigma(K))$. Then $\tau_*$ is also the identity map on
$\HF^-(\Sigma(K))$, implying that $\tau^*$ is the identity on
$\HF^{-,*}(\Sigma(m(K))$. We conclude that $d_{\tau}(m(K),2) = d(\Sigma(m(K)),\spinc_0) = -Q$.

Finally, suppose that $\tau_*$ is nontrivial on
$\HFa(\Sigma(K))$. Then on $\HF^-(\Sigma(K))$, $\tau_*(\alpha) =
\alpha + \beta$ and $\tau_*(\beta) = \beta$. This implies that
$\HF^{-,*}(\Sigma(m(K)))$ has the form $\Field[U]\langle \alpha'
\rangle \oplus \langle \beta' \rangle$ such that $\gr(\alpha') =
\gr(\beta') = -(Q-2)-4 = -Q-2$, and the involution is $\tau^*(\alpha')
= \alpha'+\beta'$ and $\tau^*(\beta')=\beta'$. In particular,
$(1+\tau_*)(\alpha') = \beta'$, so on the first page of the spectral
sequence of Figure~\ref{fig:minusspectralsequence}, $\theta^k \alpha'$
cancels with $\theta^{k+1}\beta'$ for all $k \geq 0$. This implies
that the second page of the spectral sequence of
Figure~\ref{fig:minusspectralsequence} is the direct sum of a free
$\Field[U,\theta]$-summand generated by $[U\alpha']$ and a single $\Field$-summand generated by $[\beta']$, which lies in grading $(-Q-2)+2 =
-Q$. Thus $d_{\tau}(m(K), 2) = -Q+2$.
\end{proof}

In view of Lemma~\ref{lem:q-tau-in-simple-cases}, in order to compute
these invariants for $T(3,7)$ or $P(2,-3,-7)$, it is enough to compute the
corresponding $\ZZ/2$-actions on $\HFa(\Sigma(2,3,7))$. Thanks
to~\cite{JT:Naturality}, we are free to choose any Heegaard diagram,
which is invariant under this $\ZZ/2$-action in order to
compute the induced map on $\HFa$.
We carry out the computations for the two $\ZZ/2$-actions with two
$\ZZ/2$-equivariant Heegaard diagrams in following two propositions.

\captionsetup[subfloat]{width=7cm}
\newlength{\thinred}
\setlength{\thinred}{0.15pt}
\begin{figure}
  \subfloat[The knot $T(3,7)$, drawn with two basepoints on the
  standard genus one Heegaard diagram of $S^3$.]{\label{subfig:t-3-7-standard}
\begin{tikzpicture}
\draw [gray] (0,0) -- (7,0);
\draw [gray] (0,3) -- (7,3);
\draw [blue] (0,0) -- (0,3);
\draw [blue] (7,0) -- (7,3);
\draw (0,0) -- (3,3);
\draw (1,0) -- (4,3);
\draw (2,0) -- (5,3);
\draw (3,0) -- (6,3);
\draw (4,0) -- (7,3);
\draw (5,0) -- (7,2);
\draw (6,0) -- (7,1);
\draw (0,1) -- (2,3);
\draw (0,2) -- (1,3);
\draw [red] (0,2.5)--(7,2.5);
\draw [fill, thick] (1,1) circle [radius=0.1];
\draw [fill=white, thick] (2,2) circle [radius=0.1];
\end{tikzpicture}
}
 \hspace{2ex} \subfloat[A doubly pointed Heegaard diagram for
$T(3,7)$ obtained from the previous diagram by doing a finger move
along the knot (drawn with train-tracks).\label{subfig:t-3-7-heegaard}]{
\begin{tikzpicture}[every node/.style={inner sep=0,outer sep=0,fill=white}]
\draw [gray] (0,0) -- (7,0);
\draw [gray] (0,3) -- (7,3);
\draw [blue] (0,0) -- (0,3);
\draw [blue] (7,0) -- (7,3);
\draw [fill, thick] (1,1) circle [radius=0.1];
\draw [fill=white, thick] (2,2) circle [radius=0.1];
\draw [red, line width=\thinred] (2,2.5) to [out=0,in=225] (3,3);
\draw [red, line width=\thinred] (3,2.5) to [out=180,in=225] (3,3);
\draw [red, line width=2\thinred] (3,0) -- (6,3) node[midway] {\tiny 2}; 
\draw [red, line width=\thinred] (5,2.5) to [out=0,in=225] (6,3);
\draw [red, line width=\thinred] (6,2.5) to [out=180,in=225] (6,3);
\draw [red, line width=4\thinred] (6,0) -- (7,1) node[midway] {\tiny 4}; 
\draw [red, line width=4\thinred] (0,1) -- (2,3) ;
\draw [red, line width=\thinred] (1,2.5) to [out=0,in=225] (2,3);
\draw [red, line width=\thinred] (2,2.5) to [out=180,in=225] (2,3);
\draw [red, line width=6\thinred] (2,0) -- (5,3) node[midway] {\tiny 6}; 
\draw [red, line width=\thinred] (4,2.5) to [out=0,in=225] (5,3);
\draw [red, line width=\thinred] (5,2.5) to [out=180,in=225] (5,3);
\draw [red, line width=8\thinred] (5,0) -- (7,2) node[midway] {\tiny 8}; 
\draw [red, line width=8\thinred] (0,2) -- (1,3) ; 
\draw [red, line width=\thinred] (0,2.5) to [out=0,in=225] (1,3);
\draw [red, line width=\thinred] (1,2.5) to [out=180,in=225] (1,3);
\draw [red, line width=10\thinred] (1,0) -- (4,3) node[midway] {\tiny 10}; 
\draw [red, line width=\thinred] (3,2.5) to [out=0,in=225] (4,3);
\draw [red, line width=\thinred] (4,2.5) to [out=180,in=225] (4,3);
\draw [red, line width=12\thinred] (4,0) -- (7,3) node[midway] {\tiny 12}; 
\draw [red, line width=\thinred] (6,2.5) to [out=0,in=225] (7,3);
\draw [red, line width=\thinred] (7,2.5) to [out=180,in=225] (7,3);
\draw [red, line width=14\thinred] (0,0) -- (0.5,0.5) node[midway] {\tiny 14}; 
\draw [red, line width=7\thinred] (0.5cm + 2.47*\the\thinred,0.5cm - 2.47*\the\thinred) to [out=45,in=315] (1.2,1.2) to [out=135, in=45] (0.5cm - 2.47*\the\thinred,0.5cm + 2.47*\the\thinred);
\end{tikzpicture}
}\\
\subfloat[Making the previous diagram nice by doing one more finger move.\label{subfig:t-3-7-nice}]{
\begin{tikzpicture}[every node/.style={inner sep=0,outer sep=0,fill=white}]
\draw [gray] (0,0) -- (7,0);
\draw [gray] (0,3) -- (7,3);
\draw [blue] (0,0) -- (0,3);
\draw [blue] (7,0) -- (7,3);
\draw [fill, thick] (1,1) circle [radius=0.1];
\draw [fill=white, thick] (2,2) circle [radius=0.1];
\draw [red, line width=\thinred] (2,2.5) to [out=0,in=225] (3,3);
\draw [red, line width=\thinred] (3,2.5) to [out=180,in=225] (3,3);
\draw [red, line width=2\thinred] (3,0) -- (6,3) node[midway] {\tiny 2}; 
\draw [red, line width=\thinred] (5,2.5) to [out=0,in=225] (6,3);
\draw [red, line width=\thinred] (6,2.5) to [out=180,in=225] (6,3);
\draw [red, line width=4\thinred] (6,0) -- (7,1) node[midway] {\tiny 4}; 
\draw [red, line width=4\thinred] (0,1) -- (2,3) ; 
\draw [red, line width=\thinred] (0.5,2) to [out=45,in=225] (2,3);
\draw [red, line width=\thinred] (2,2.5) to [out=180,in=225] (2,3);
\draw [red, line width=6\thinred] (2,0) -- (5,3) node[midway] {\tiny 6}; 
\draw [red, line width=\thinred] (3.5,2) to [out=45,in=225] (5,3);
\draw [red, line width=\thinred] (5,2.5) to [out=180,in=225] (5,3);
\draw [red, line width=8\thinred] (5,0) -- (7,2) node[midway] {\tiny 8}; 
\draw [red, line width=8\thinred] (0,2) -- (1,3) ; 
\draw [red, line width=\thinred] (0,2.3) to [out=30,in=225] (1,3);
\draw [red, line width=\thinred] (6.5,2) to [out=45,in=210] (7,2.3);
\draw [red, line width=\thinred] (0.5,2) to [out=45,in=225] (1,3);
\draw [red, line width=10\thinred] (1,0) -- (4,3) node[midway] {\tiny 10}; 
\draw [red, line width=\thinred] (3,2.5) to [out=0,in=225] (4,3);
\draw [red, line width=\thinred] (3.5,2) to [out=45,in=225] (4,3);
\draw [red, line width=12\thinred] (4,0) -- (7,3) node[midway] {\tiny 12}; 
\draw [red, line width=\thinred] (6,2.5) to [out=0,in=225] (7,3);
\draw [red, line width=\thinred] (6.5,2) to [out=45,in=225] (7,3);
\draw [red, line width=14\thinred] (0,0) -- (0.5,0.5) node[midway] {\tiny 14}; 
\draw [red, line width=7\thinred] (0.5cm + 2.47*\the\thinred,0.5cm - 2.47*\the\thinred) to [out=45,in=315] (1.2,1.2) to [out=135, in=45] (0.5cm - 2.47*\the\thinred,0.5cm + 2.47*\the\thinred);
\draw [red, line width=2\thinred] (6.5,2) -- (4.5,0) node[midway] {\tiny 2};
\draw [red, line width=2\thinred] (4.5,3) -- (3.5,2);
\draw [red, line width=4\thinred] (3.5,2) -- (1.5,0) node[midway] {\tiny 4};
\draw [red, line width=4\thinred] (1.5,3) -- (0.5,2);
\draw [red, line width=6\thinred] (0.5,2) -- (0,1.5) node[midway] {\tiny 6};
\draw [red, line width=6\thinred] (7,1.5) -- (6.5,1) node[midway] {\tiny 6};
\draw [red, line width=3\thinred] (6.5cm - 1.06*\the\thinred,1cm + 1.06*\the\thinred) to [out=225,in=135] (6,0.5) to [out=315,in=225] (6.5cm + 1.06*\the\thinred,1cm - 1.06*\the\thinred);
\end{tikzpicture}
} \hspace{2ex} \subfloat[The same Heegaard diagram, straightened out,
and drawn without train-tracks.\label{subfig:t-3-7-nice-straight}
]{
\begin{tikzpicture}
\draw [gray] (0,0) -- (7,0);
\draw [gray] (0,3) -- (7,3);
\draw [blue] (0,0) -- (0,3);
\draw [blue] (7,0) -- (7,3);
\draw [red, rounded corners] (0,1*3/34+19*3/34) -- (2-1*3/34,1*3/34+19*3/34) -- (2-1*3/34,14*3/34+19*3/34) -- (0,14*3/34+19*3/34);
\draw [red, rounded corners] (0,2*3/34+19*3/34) -- (2-2*3/34,2*3/34+19*3/34) -- (2-2*3/34,13*3/34+19*3/34) -- (0,13*3/34+19*3/34);
\draw [red, rounded corners] (0,3*3/34+19*3/34) -- (2-3*3/34,3*3/34+19*3/34) -- (2-3*3/34,12*3/34+19*3/34) -- (0,12*3/34+19*3/34);
\draw [red, rounded corners] (0,4*3/34+19*3/34) -- (2-4*3/34,4*3/34+19*3/34) -- (2-4*3/34,11*3/34+19*3/34) -- (0,11*3/34+19*3/34);
\draw [red, rounded corners] (0,5*3/34+19*3/34) -- (2-5*3/34,5*3/34+19*3/34) -- (2-5*3/34,10*3/34+19*3/34) -- (0,10*3/34+19*3/34);
\draw [red, rounded corners] (0,6*3/34+19*3/34) -- (2-6*3/34,6*3/34+19*3/34) -- (2-6*3/34,9*3/34+19*3/34) -- (0,9*3/34+19*3/34);
\draw [red] (0,7*3/34+19*3/34) -- (2-7*3/34,7*3/34+19*3/34) -- (2-7*3/34,8*3/34+19*3/34) -- (0,8*3/34+19*3/34);
\draw [dashed] (1,7.5*3/34+19*3/34) to [out=0,in=180] (7.5*3/34+19*3/34,7.5*3/34+19*3/34) to [out=0,in=270] (3,3);
\draw [dashed] (3,0) -- (3,7.5*3/34+19*3/34);
\node at (7.5*3/34+19*3/34,7.5*3/34+19*3/34) {\footnotesize{\ding{34}}};
\draw [fill, thick] (1,7.5*3/34+19*3/34) circle [radius=0.05];
\draw [fill=white, thick] (3,7.5*3/34+19*3/34) circle [radius=0.05];
\draw [red] (7,1*3/34) -- (5+2*3/34,1*3/34) -- (5+2*3/34,0);
\draw [red] (7,3-1*3/34) -- (5+2*3/34,3-1*3/34) -- (5+2*3/34,3);
\draw [red, rounded corners] (7,2*3/34) -- (5+1*3/34,2*3/34) -- (5+1*3/34,0);
\draw [red, rounded corners] (7,3-2*3/34) -- (5+1*3/34,3-2*3/34) -- (5+1*3/34,3);
\draw [red, rounded corners] (7,3*3/34) -- (5+1*3/34,3*3/34) -- (5+1*3/34,12*3/34) -- (7,12*3/34);
\draw [red, rounded corners] (7,4*3/34) -- (5+2*3/34,4*3/34) -- (5+2*3/34,11*3/34) -- (7,11*3/34);
\draw [red, rounded corners] (7,5*3/34) -- (5+3*3/34,5*3/34) -- (5+3*3/34,10*3/34) -- (7,10*3/34);
\draw [red, rounded corners] (7,6*3/34) -- (5+4*3/34,6*3/34) -- (5+4*3/34,9*3/34) -- (7,9*3/34);
\draw [red] (7,7*3/34) -- (5+6*3/34,7*3/34) -- (5+6*3/34,8*3/34) -- (7,8*3/34);
\draw [red] (0,-36/34+13*3/34) to [out=0,in=180] (2,-36/34+13*3/34) to [out=0,in=180] (5,13*3/34) to [out=0,in=180] (7,13*3/34);
\draw [red] (0,-36/34+14*3/34) to [out=0,in=180] (2,-36/34+14*3/34) to [out=0,in=180] (5,14*3/34) to [out=0,in=180] (7,14*3/34);
\draw [red] (0,-36/34+15*3/34) to [out=0,in=180] (2,-36/34+15*3/34) to [out=0,in=180] (5,15*3/34) to [out=0,in=180] (7,15*3/34);
\draw [red] (0,-36/34+16*3/34) to [out=0,in=180] (2,-36/34+16*3/34) to [out=0,in=180] (5,16*3/34) to [out=0,in=180] (7,16*3/34);
\draw [red] (0,-36/34+17*3/34) to [out=0,in=180] (2,-36/34+17*3/34) to [out=0,in=180] (5,17*3/34) to [out=0,in=180] (7,17*3/34);
\draw [red] (0,-36/34+18*3/34) to [out=0,in=180] (2,-36/34+18*3/34) to [out=0,in=180] (5,18*3/34) to [out=0,in=180] (7,18*3/34);
\draw [red] (0,-36/34+19*3/34) to [out=0,in=180] (2,-36/34+19*3/34) to [out=0,in=180] (5,19*3/34) to [out=0,in=180] (7,19*3/34);
\draw [red] (0,-36/34+20*3/34) to [out=0,in=180] (2,-36/34+20*3/34) to [out=0,in=180] (5,20*3/34) to [out=0,in=180] (7,20*3/34);
\draw [red] (0,-36/34+21*3/34) to [out=0,in=180] (2,-36/34+21*3/34) to [out=0,in=180] (5,21*3/34) to [out=0,in=180] (7,21*3/34);
\draw [red] (0,-36/34+22*3/34) to [out=0,in=180] (2,-36/34+22*3/34) to [out=0,in=180] (5,22*3/34) to [out=0,in=180] (7,22*3/34);
\draw [red] (0,-36/34+23*3/34) to [out=0,in=180] (2,-36/34+23*3/34) to [out=0,in=180] (5,23*3/34) to [out=0,in=180] (7,23*3/34);
\draw [red] (0,-36/34+24*3/34) to [out=0,in=180] (2,-36/34+24*3/34) to [out=0,in=180] (5,24*3/34) to [out=0,in=180] (7,24*3/34);
\draw [red] (0,-36/34+25*3/34) to [out=0,in=180] (2,-36/34+25*3/34) to [out=0,in=180] (5,25*3/34) to [out=0,in=180] (7,25*3/34);
\draw [red] (0,-36/34+26*3/34) to [out=0,in=180] (2,-36/34+26*3/34) to [out=0,in=180] (5,26*3/34) to [out=0,in=180] (7,26*3/34);
\draw [red] (0,-36/34+27*3/34) to [out=0,in=180] (2,-36/34+27*3/34) to [out=0,in=180] (5,27*3/34) to [out=0,in=180] (7,27*3/34);
\draw [red] (0,-36/34+28*3/34) to [out=0,in=180] (2,-36/34+28*3/34) to [out=0,in=180] (5,28*3/34) to [out=0,in=180] (7,28*3/34);
\draw [red] (0,-36/34+29*3/34) to [out=0,in=180] (2,-36/34+29*3/34) to [out=0,in=180] (5,29*3/34) to [out=0,in=180] (7,29*3/34);
\draw [red] (0,-36/34+30*3/34) to [out=0,in=180] (2,-36/34+30*3/34) to [out=0,in=180] (5,30*3/34) to [out=0,in=180] (7,30*3/34);
\draw [red] (0,-36/34+31*3/34) to [out=0,in=180] (2,-36/34+31*3/34) to [out=0,in=180] (5,31*3/34) to [out=0,in=180] (7,31*3/34);
\end{tikzpicture}
}
\caption{A nice Heegaard diagrams for the $(3,7)$ torus knot. In each
  case, the left and right edges of the rectangle are the
  $\beta$-circle(s). To construct a nice Heegaard diagram for the double
  branched cover $\Sigma(2,3,7)$, cut along the dashed line, take two
  copies, and then join the two copies along the preimages of the
  dashed line.}\label{fig:computation-t-3-7}
\end{figure}

\begin{proposition}
  The involution of $\HFa(\Sigma(2,3,7))$ that comes from viewing
  $\Sigma(2,3,7)$ as the double branched cover of $T(3,7)$ is the
  identity map.  Hence, $q_{\tau}(T(3,7))=d_{\tau}(T(3,7),2)=d_{\tau}(-T(3,7),2)=0$.
\end{proposition}

\begin{proof}
  For this computation, we resort to the technique of nice diagrams. We first
  choose a doubly pointed Heegaard diagram for the knot $T(3,7)\subset
  S^3$, satisfying the following two conditions:
  \begin{itemize}
  \item The diagram is nice in the sense
    of~\cite{SarkarWang07:ComputingHFhat}; i.e., every region that does
    not contain a $z$ or a $w$ basepoint is either a bigon or a
    rectangle.
  \item The region containing the basepoint $z$ is a bigon.
  \end{itemize}
  Given such a \emph{doubly nice} Heegaard diagram, the induced
  $\ZZ/2$-equivariant Heegaard diagram for
  $\Sigma(T(3,7))=\Sigma(2,3,7)$ (see
  Section~\ref{sec:first-invariance}), with basepoint the preimage of
  $w$, is nice.
  Thus, it is a combinatorial exercise to compute the chain complex
  $\CFa(\Sigma(2,3,7))$, and the $\ZZ/2$-action on it from such a
  diagram (compare Section~\ref{sec:nice-diags}).

  One can construct a doubly nice Heegaard diagram for any knot $K$ as
  follows. Start with any doubly pointed Heegaard diagram for
  $K$. Ensure that $z$ and $w$ are separated by a single
  $\alpha$-circle by adding a handle between the regions containing
  $z$ and $w$ if necessary. Then run the Sarkar-Wang algorithm to make
  this Heegaard diagram nice with respect to the basepoint $w$. Recall
  that one may run the algorithm without needing to stabilize or
  destabilize the Heegaard diagram, and by only isotoping and
  handlesliding the $\beta$-circles (and keeping the $\alpha$-circles
  fixed). Therefore, for the resulting nice diagram, we may still
  assume that the $z$ and the $w$ basepoint are separated by a single
  $\alpha$-circle. Let $R$ be the region in this new Heegaard diagram
  that contains $z$. If $R$ is a bigon, we are already
  done. Otherwise, if $R$ is a rectangle, draw an arc from $z$ to $w$
  in the complement of the $\alpha$-circles intersecting the
  $\beta$-multicurve minimally. Then perform a finger move with the
  $\alpha$-arc that separates $z$ and $w$ along this arc; this cuts
  $R$ into a rectangle and a bigon, and we may, trivially, ensure that
  $z$ ends up in the bigon.

  In our case, we do not have to resort to the above general algorithm,
  which tends to produce unwieldy diagrams. Instead, we
  use the strategy outlined in Figure~\ref{fig:computation-t-3-7}. We
  first draw $T(3,7)$ on the standard genus-one Heegaard diagram for
  $S^3$ as a union of two embedded arcs joining two basepoints, so
  that the first arc lies in the complement of the $\beta$-circle; see
  Figure~\ref{subfig:t-3-7-standard} (as usual, red is $\alpha$ and
  blue is $\beta$). Then we perform a finger move on the $\alpha$
  circle along the other arc, producing a doubly pointed Heegaard
  diagram for $T(3,7)$; see Figure~\ref{subfig:t-3-7-heegaard}. This
  is not a nice Heegaard diagram: it has a single bad region which is
  a hexagon. Performing one more finger move makes this diagram doubly
  nice; see
  Figures~\ref{subfig:t-3-7-nice}--\ref{subfig:t-3-7-nice-straight}.

  As per the above discussion, if $\mathcal{H}$ denotes the double
  branched cover of this doubly nice diagram, we can 
  compute $\CFa(\mathcal{H})$ and the $\ZZ/2$-action combinatorially.  We use
  the computer program at \verb!https://github.com/sucharit/hf-hat! to
  study $\mathcal{H}$. The complex $\CFa(\mathcal{H})$ has $545$
  generators, and its homology
  $H_*(\CFa(\mathcal{H}))=\HFa(\Sigma(2,3,7))$ is three-dimensional. To
  compute the $\ZZ/2$-action, say $\tau$, on $\HFa(\Sigma(2,3,7))$, we
  compute the homology of the mapping cone of $(\Id+\tau_{\#})$, where
  $\tau_{\#}\from\CFa(\mathcal{H})\to\CFa(\mathcal{H})$ is the
  chain-level $\ZZ/2$-action. The homology of
  $\mathrm{Cone}(\Id+\tau_{\#})$ is six-dimensional, which implies
  that $\tau\from\HFa(\Sigma(2,3,7))\to\HFa(\Sigma(2,3,7))$ is the
  identity map.
\end{proof}

\def\ctrllen{0.25}
\def\trefy{1.2}
\def\trefz{0.1}
\begin{figure}
  \subfloat[A surgery presentation of $\Sigma(2,3,7)$ that exhibits a
  $\ZZ/2$
  symmetry.]{\label{subfig:link-surgery}\begin{tikzpicture}[every
      node/.style={inner sep=0,outer sep=0}]
\begin{scope}[zxplane=0]
  \draw[thick,dashed] (0,2) arc (90:180:2cm);
  \draw[thick] (-2,0) arc (-180:90:2cm) ;
  \draw (-3,0) -- (3,0);
\end{scope}
\begin{scope}[yzplane=-3]
  \draw[->] (0,-0.7) arc (-90:180:0.7cm); 
\end{scope}
\draw[thick] (2,-0.4) arc (-90:180:0.4cm);
\draw[thick] (2,-0.8) arc (-90:180:0.8cm);
\draw[thick] (2,-1.2) arc (-90:180:1.2cm);
\draw[thick,dashed] (2-0.4,0) arc (180:270:0.4cm);
\draw[thick,dashed] (2-0.8,0) arc (180:270:0.8cm);
\draw[thick,dashed] (2-1.2,0) arc (180:270:1.2cm);
\draw (0,-1.5) -- (0,1.5);
\draw (-3.5,0) -- (3.5,0);
\node [label=above:\tiny{-2}] at (2.15+0.4,0) {};
\node [label=above:\tiny{-3}] at (2.15+0.8,0) {};
\node [label=above:\tiny{-7}] at (2.15+1.2,0) {};
\node [label=above:\tiny{-1}] at (-2.15,0) {};
\end{tikzpicture}}\hspace{2ex} \subfloat[Using $\ZZ/2$-equivariant
Kirby moves, it is also the $(-1)$-surgery on the positive trefoil
(the portion of the trefoil in the positive octant is emphasized.)]{
\begin{tikzpicture}[every node/.style={inner sep=0,outer sep=0}]
  \draw (0,-1.5) -- (0,1.5);
  \draw (-3.5,0) -- (3.5,0);
\begin{scope}[zxplane=0]
  \draw (-3,0) -- (3,0);
  \fill[gray,opacity=0.3] (0,0) rectangle (3,3.5);
  \fill (0.5cm+2pt,0) arc (0:180:2pt);
  \fill (0,2cm-2pt) arc (-90:90:2pt);
\end{scope}
\begin{scope}[yzplane=-3]
  \draw[->] (0,-0.7) arc (-90:180:0.7cm); 
\end{scope}
\begin{scope}[yzplane=0]
  \fill[gray,opacity=0.35] (0,0) rectangle (1.5,3);
  \fill (1*\trefy,0.5) circle (2pt);
  \fill (0,0.5cm-2pt) arc (-90:90:2pt);
\end{scope}
\fill (0.5,0.5*\trefy,0) circle (2pt);
\fill (2cm+2pt,0) arc (0:180:2pt);
\fill[gray,opacity=0.2] (0,0) rectangle (3.5,1.5);
\draw[thick,dashed] (0,0*\trefy,-0.5) .. controls (0-\ctrllen,0*\trefy+\ctrllen,-0.5) and (-0.5,0.5*\trefy-\ctrllen,0-\ctrllen) .. (-0.5,0.5*\trefy,0);
\draw[thick] (-0.5,0.5*\trefy,0) .. controls
(-0.5,0.5*\trefy+\ctrllen,\ctrllen) and (0-2*\ctrllen,1*\trefy,0.5)
.. (0,1*\trefy,0.5);
\draw[very thick] (0,1*\trefy,0.5) .. controls (0+5*\ctrllen,1*\trefy,0.5) and (2,0*\trefy+3*\ctrllen,0) .. (2,0*\trefy,0);
\draw[thick,dashed] (2,0*\trefy,0) .. controls (2,0*\trefy-3*\ctrllen,0) and (0+5*\ctrllen,-1*\trefy,-0.5) .. (0,-1*\trefy,-0.5);
\draw[thick] (0,-1*\trefy,-0.5) .. controls (0-2*\ctrllen,-1*\trefy,-0.5) and (-0.5,-0.5*\trefy-\ctrllen,0-\ctrllen) .. (-0.5,-0.5*\trefy,0);
\draw[thick] (-0.5,-0.5*\trefy,0) .. controls
(-0.5,-0.5*\trefy+\ctrllen,0+\ctrllen) and (0-\ctrllen,0*\trefy-\ctrllen,0.5)
.. (0,0*\trefy,0.5);
\draw[very thick] (0,0*\trefy,0.5) .. controls (0+\ctrllen,0*\trefy+\ctrllen,0.5) and (0.5,0.5*\trefy-\ctrllen,0+\ctrllen) .. (0.5,0.5*\trefy,0);
\draw[thick,dashed] (0.5,0.5*\trefy,0) .. controls (0.5,0.5*\trefy+\ctrllen,-\ctrllen) and (0+2*\ctrllen,1*\trefy,-0.5) .. (0,1*\trefy,-0.5);
\draw[thick] (0,1*\trefy,-0.5) .. controls (0-5*\ctrllen,1*\trefy,-0.5) and (-2,0*\trefy+3*\ctrllen,0) .. (-2,0*\trefy,0);
\draw[thick] (-2,0*\trefy,0) .. controls (-2,0*\trefy-3*\ctrllen,0) and (0-5*\ctrllen,-1*\trefy,0.5) .. (0,-1*\trefy,0.5);
\draw[ thick] (0,-1*\trefy,0.5) .. controls (0+2*\ctrllen,-1*\trefy,0.5)
and (0.5,-0.5*\trefy-\ctrllen,0+\ctrllen) .. (0.5,-0.5*\trefy,0);
\draw[thick,dashed] (0.5,-0.5*\trefy,0) .. controls (0.5,-0.5*\trefy+\ctrllen,-\ctrllen) and (0+\ctrllen,0*\trefy-\ctrllen,-0.5) .. (0,0*\trefy,-0.5);
\node [label=above:\tiny{-1}] at (-2.15,0) {};
\end{tikzpicture}}\\
  \subfloat[The fixed set in the quotient $S^3$ is the Montesinos knot $M(-1;(-2,1),(-3,1),(-7,1))$.]{\label{subfig:montesinos}\begin{tikzpicture}[every node/.style={inner sep=1pt,outer sep=0,draw=black,text=black,fill=white,thick}]
\begin{scope}[zxplane=0]
  \draw[thick] (0,-2-0.1) arc (-90:90:2cm+0.1cm) ;
  \draw[thick] (0,-2+0.1) arc (-90:90:2cm-0.1cm) ;
  \node at (2,0) {\tiny -1};
\end{scope}
\draw[thick] (2+0.4+0.08,0) arc (0:180:0.4cm+0.08cm);
\draw[thick] (2+0.8+0.08,0) arc (0:180:0.8cm+0.08cm);
\draw[thick] (2+1.2+0.08,0) arc (0:180:1.2cm+0.08cm);
\draw[thick] (2+0.4-0.08,0) arc (0:180:0.4cm-0.08cm);
\draw[thick] (2+0.8-0.08,0) arc (0:180:0.8cm-0.08cm);
\draw[thick] (2+1.2-0.08,0) arc (0:180:1.2cm-0.08cm);
\path (0,-2.1) -- (0,2.1);
\draw[thick] (-3.5,0) -- (-2-0.1,0);
\draw[thick] (-2+0.1,0) -- (2-1.2-0.08,0);
\draw[thick] (2-1.2+0.08,0) -- (2-0.8-0.08,0);
\draw[thick] (2-0.8+0.08,0) -- (2-0.4-0.08,0);
\draw[thick] (2-0.4+0.08,0) -- (2-0.1,0);
\draw[thick] (2+0.1,0) -- (2+0.4-0.08,0);
\draw[thick] (2+0.4+0.08,0) -- (2+0.8-0.08,0);
\draw[thick] (2+0.8+0.08,0) -- (2+1.2-0.08,0);
\draw[thick] (2+1.2+0.08,0) -- (3.5,0);

\node at (2,0.4) {\tiny -2};
\node at (2,0.8) {\tiny -3};
\node at (2,1.2) {\tiny -7};
\end{tikzpicture}}\hspace{2ex} \subfloat[A Heegaard diagram for the
$(-1)$-surgery of the trefoil, along with the $\ZZ/2$
symmetry. The small planar circles are $\alpha$, and the large
non-planar ones are $\beta$.]{\label{subfig:-1-surgery}\begin{tikzpicture}[every node/.style={inner sep=0pt,outer sep=0,text=black,fill=white}]
\path (0,-2.1) -- (0,2.1);
\begin{scope}[yzplane=-3]
  \draw[->] (0,-0.7) arc (-90:180:0.7cm); 
\end{scope}
\draw (-3.5,0)--(-2.5,0);\draw (-1.8,0)--(-0.8,0);
\draw (3.5,0)--(2.5,0);\draw (1.8,0)--(0.8,0);
\draw[dashed]  (-2.5,0)--(-1.8,0);
\draw[dashed]  (2.5,0)--(1.8,0);
\draw[dashed] (-0.8,0)--(0.8,0);

\draw[blue] (0,-0.5,0) .. controls (0+\ctrllen,-0.5,0+\trefz) and (0.8,0-\ctrllen,0+\trefz) .. (0.8,0,0);
\draw[blue,dashed] (0.8,0,0) .. controls (0.8,0+\ctrllen,0-\trefz) and (0+\ctrllen,0.5,0-\trefz) .. (0,0.5,0);
\draw[blue] (0,0.5,0) .. controls (0-\ctrllen,0.5,0+\trefz) and (-0.8,0+\ctrllen,0+\trefz) .. (-0.8,0,0);
\draw[blue,dashed] (-0.8,0,0) .. controls (-0.8,0-\ctrllen,0-\trefz) and (0-\ctrllen,-0.5,0-\trefz) .. (0,-0.5,0);

\draw[blue] (0,1.1,0) .. controls (0+5*\ctrllen,1.1,0+\trefz) and (1.3,0.3,0+2*\trefz) .. (1.3,0.3,0);
\draw[blue,dashed] (1.3,0.3,0) .. controls (1.3,0.3,0-2*\trefz) and (0+3*\ctrllen,2,0-\trefz) .. (0,2,0);
\draw[blue] (0,2,0) .. controls (0-3*\ctrllen,2,0+\trefz) and (-1.3,0.3,0+2*\trefz) .. (-1.3,0.3,0);
\draw[blue,dashed] (-1.3,0.3,0) .. controls (-1.3,0.3,0-2*\trefz) and (0-5*\ctrllen,1.1,0-\trefz) .. (0,1.1,0);

\draw[blue,dashed] (0,-1.1,0) .. controls (0+5*\ctrllen,-1.1,0-\trefz) and (1.3,-0.3,0-2*\trefz) .. (1.3,-0.3,0);
\draw[blue] (1.3,-0.3,0) .. controls (1.3,-0.3,0+2*\trefz) and (0+3*\ctrllen,-2,0+\trefz) .. (0,-2,0);
\draw[blue,dashed] (0,-2,0) .. controls (0-3*\ctrllen,-2,0-\trefz) and (-1.3,-0.3,0-2*\trefz) .. (-1.3,-0.3,0);
\draw[blue] (-1.3,-0.3,0) .. controls (-1.3,-0.3,0+2*\trefz) and (0-5*\ctrllen,-1.1,0+\trefz) .. (0,-1.1,0);

\draw[blue,dashed] (0.5,-0.8,0) .. controls (0.5,-0.8+3*\ctrllen,0-\trefz) and (-0.5,0.8-3*\ctrllen,0-\trefz) .. (-0.5,0.8,0);
\draw[blue] (-0.5,0.8,0) .. controls (-0.5,0.8+5*\ctrllen,0+\trefz) and (1.8,0+4*\ctrllen,0+\trefz) .. (1.8,0,0);
\draw[blue,dashed] (1.8,0,0) .. controls (1.8,0-4*\ctrllen,0-\trefz) and (-0.5,-0.8-5*\ctrllen,0-\trefz) .. (-0.5,-0.8,0);
\draw[blue] (-0.5,-0.8,0) .. controls (-0.5,-0.8+3*\ctrllen,0+\trefz) and (0.5,0.8-3*\ctrllen,0+\trefz) .. (0.5,0.8,0);

\draw[blue,dashed] (0.5,0.8,0) ..controls (0.5,0.8+5*\ctrllen,0-\trefz) and (-2.1+\ctrllen,0.5+\ctrllen,0-\trefz).. (-2.1,0.5,0);
\draw[blue] (-2.1,0.5,0) ..controls (-2.1-0.5*\ctrllen,0.5-0.5*\ctrllen,0+\trefz) and (-1.7,0.15+0.5*\ctrllen,0+\trefz).. (-1.7,0.15,0);
\draw[blue,dashed] (-1.7,0.15,0) ..controls (-1.7,0.15-0.5*\ctrllen,0-\trefz) and (-2.5,0,0-2*\trefz).. (-2.5,0,0);
\draw[blue] (-2.5,0,0) ..controls (-2.5,0,0+2*\trefz) and (-1.7,-0.15+0.5*\ctrllen,0+\trefz).. (-1.7,-0.15,0);
\draw[blue,dashed] (-1.7,-0.15,0) ..controls (-1.7,-0.15-0.5*\ctrllen,0-\trefz) and (-2.1-0.5*\ctrllen,-0.5+0.5*\ctrllen,0-\trefz).. (-2.1,-0.5,0);
\draw[blue] (-2.1,-0.5,0) ..controls (-2.1+\ctrllen,-0.5-\ctrllen,0+\trefz) and (0.5,-0.8-5*\ctrllen,0+\trefz).. (0.5,-0.8,0);

\draw[red,thick] plot [smooth cycle] coordinates {(0,0.5) (0.5,0.8) (0,1.1) (-0.5,0.8)};
\draw[red,thick] plot [smooth cycle] coordinates {(0,-0.5) (0.5,-0.8) (0,-1.1) (-0.5,-0.8)};
\draw[red,thick] plot [smooth cycle] coordinates {(0.8,0) (1.3,0.3) (1.8,0) (1.3,-0.3)};
\draw[red,thick] plot [smooth cycle] coordinates {(-0.8,0) (-1.3,0.3) (-1.7,0.15) (-1.8,0) (-1.7,-0.15) (-1.3,-0.3)};
\draw plot [smooth cycle] coordinates {(0,2) (0.6,1.8) (2.1,0.5) (2.5,0) (2.1,-0.5) (0.6,-1.8) (0,-2) (-0.6,-1.8) (-2.1,-0.5) (-2.5,0) (-2.1,0.5) (-0.6,1.8)};
\draw [fill, thick] (2.5,0) circle [radius=0.07];
\node at (-1.3,0.3) {\tiny a};
\node at (-1.7,0.15) {\tiny b};
\node at (-1.7,-0.15) {\tiny c};
\node at (-1.3,-0.3) {\tiny d};
\node at (-0.8,0) {\tiny e};
\node at (0,1.1) {\tiny f};
\node at (-0.5,0.8) {\tiny g};
\node at (0,0.5) {\tiny h};
\node at (0.5,0.8) {\tiny i};
\node at (0,-0.5) {\tiny j};
\node at (-0.5,-0.8) {\tiny k};
\node at (0,-1.1) {\tiny l};
\node at (0.5,-0.8) {\tiny m};
\node at (1.3,0.3) {\tiny n};
\node at (0.8,0) {\tiny o};
\node at (1.3,-0.3) {\tiny p};
\node at (1.8,0) {\tiny q};
\end{tikzpicture}}
\caption{A Heegaard diagram for $\Sigma(2,3,7)$, viewed as a
  $(-1)$-surgery of the positive trefoil. The diagram reflects the
  $\ZZ/2$ symmetry that comes from viewing $\Sigma(2,3,7)$ as the
  double branched cover of the Montesinos knot
  $M(-1;(-2,1),(-3,1),(-7,1))$.}
\end{figure}

\begin{proposition}
  The involution of $\HFa(\Sigma(2,3,7))$ that comes from viewing
  $\Sigma(2,3,7)$ as the double branched cover of $P(2,-3,-7)$ is not
  the identity map. Hence, $q_{\tau}(P(2,-3,-7))=-2$, $d_{\tau}(P(2,-3,-7),2)=0$, and $d_{\tau}(P(-2,3,7),2)=2$. 
\end{proposition}

\begin{proof}
  We carry out this computation directly by hand. It is
  well-known that the following $\ZZ/2$-actions
  on $\Sigma(2,3,7)$ are the same:
  \begin{itemize}
  \item The $\ZZ/2$-action coming from viewing $\Sigma(2,3,7)$ as $\Sigma(P(2,-3,-7))$.
  \item The $\ZZ/2$-action coming from the description of
    $\Sigma(2,3,7)$ as $(-1)$-surgery on the trefoil $T(2,3)$ and applying the unique
    strong inversion (orientation-reversing involution) of $T(2,3)$.
  \end{itemize}
  See, for instance, Watson~\cite{Watson10:remark}; see also
  Figures~\ref{subfig:link-surgery}--\ref{subfig:montesinos}.

  Therefore, we start with a $\ZZ/2$-equivariant Heegaard diagram
  $\HD$ for
  the $(-1)$-surgery on $T(2,3)$ in Figure~\ref{subfig:-1-surgery},
  where as usual, red denotes $\alpha$ and blue denotes
  $\beta$. (Note that the Heegaard surface is oriented as the boundary of
  the $\alpha$-handlebody, so it is oriented as the boundary of the
  `outside'.) We have labeled the seventeen intersection points
  between the $\alpha$-circles and the $\beta$-circles
  a, b, \dots, p, q. As in Point~\ref{item:ET:index-double} of
  Section~\ref{sec:equi-equi}, the Lagrangians $T_\alpha,T_\beta\subset\Sym^4(\Sigma)$
  satisfy Hypothesis~\ref{hyp:equivariant-transversality}, so we may
  compute the equivariant Floer complex as $\RHomO{\Field[\ZZ/2]}(\CFa(\HD),\Field)$,
  where $\CFa(\HD)$ is computed with respect to a generic
  $\ZZ/2$-equivariant path of almost complex structures on
  $\Sym^g(\Sigma)$. With respect to such a path, the chain complex
  $\CFa(\HD)$ has the following form:
  \begin{center}
    \begin{tikzpicture}[xscale=2,every node/.style={inner sep=1pt,outer sep=0,text=black,fill=white,thick}]

      \coordinate (5) at (-3,1);
      \coordinate (8) at (-1.5,1);
      \coordinate (7) at (1.5,1);
      \coordinate (16) at (3,1);
      \coordinate (1) at (-3,0);
      \coordinate (9) at (-2,0);
      \coordinate (19) at (-1,0);
      \coordinate (22) at (0,0);
      \coordinate (2) at (1,0);
      \coordinate (12) at (2,0);
      \coordinate (13) at (3,0);
      \coordinate (18) at (-3,-1);
      \coordinate (0) at (-2.1,-1);
      \coordinate (6) at (-1.3,-1);
      \coordinate (24) at (-0.5,-1);
      \coordinate (20) at (0.5,-1);
      \coordinate (15) at (1.3,-1);
      \coordinate (14) at (2.1,-1);
      \coordinate (3) at (3,-1);
      \coordinate (23) at (-3,-2);
      \coordinate (11) at (-2.1,-2);
      \coordinate (4) at (-1.3,-2);
      \coordinate (17) at (1.3,-2);
      \coordinate (10) at (2.1,-2);
      \coordinate (21) at (3,-2);

      \draw (8) -- (9);
      \draw (8) -- (19);
      \draw[ultra thick] (5) -- (1); 
      \draw[ultra thick] (16) -- (13);
      \draw (7) -- (2);
      \draw (7) -- (12);
      \draw (9) -- (24);
      \draw (19) -- (24);
      \draw (1) -- (24);
      \draw[ultra thick] (22) -- (24);
      \draw[ultra thick] (22) -- (20);
      \draw (13) -- (20);
      \draw (2) -- (20);
      \draw (12) -- (20);
      \draw[ultra thick] (18) -- (23);
      \draw (0) -- (23);
      \draw[ultra thick] (0) -- (11);
      \draw (0) -- (4);
      \draw[ultra thick] (6) -- (4);
      \draw[ultra thick] (15) -- (17);
      \draw (14) -- (17);
      \draw[ultra thick] (14) -- (10);
      \draw (14) -- (21);
      \draw[ultra thick] (3) -- (21);

      \node at (4,1) {$1$};
      \node at (4,0) {$0$};
      \node at (4,-1) {$(-1)$};
      \node at (4,-2) {$(-2)$};

\node at (0) {\small ailo};
\node at (1) {\small aglo};
\node at (2) {\small aijp};
\node at (3) {\small agjp};
\node at (4) {\small ahkp};
\node at (5) {\small ahmp};
\node at (6) {\small ahlq};
\node at (7) {\small bflo};
\node at (8) {\small cflo};
\node at (9) {\small bfjp};
\node at (10) {\small cfjp};
\node at (11) {\small bhln};
\node at (12) {\small chln};
\node at (13) {\small dfko};
\node at (14) {\small dfmo};
\node at (15) {\small dfjq};
\node at (16) {\small dijn};
\node at (17) {\small dgjn};
\node at (18) {\small dhkn};
\node at (19) {\small dhmn};
\node at (20) {\small efkp};
\node at (21) {\small efmp};
\node at (22) {\small eflq};
\node at (23) {\small eiln};
\node at (24) {\small egln};

    \end{tikzpicture}
  \end{center}
  Here, a line (thick or thin) indicates the presence of a positive
  Maslov index-one domain, the only domains which may contribute to
  the differential. A thick line indicates that the domain is a
  $2n$-gon, and therefore, necessarily contributes $1$ to the $\CFa$
  differential. The $\ZZ/2$-action is reflection along the central
  vertical line.
  
  The fact that $\bdy^2=0$ implies that egln does not occur in
  $\bdy(\text{aglo})$ and efkp does not appear in $\bdy(\text{dfko})$. So,
  this complex decomposes into two $\ZZ/2$-equivariant summands, so
  that one of them is the following, and other one is acyclic.
  \begin{center}
    \begin{tikzpicture}[xscale=2,every node/.style={inner sep=1pt,outer sep=0,text=black,fill=white,thick}]

      \coordinate (8) at (-1.5,1);
      \coordinate (7) at (1.5,1);
      \coordinate (9) at (-2,0);
      \coordinate (19) at (-1,0);
      \coordinate (22) at (0,0);
      \coordinate (2) at (1,0);
      \coordinate (12) at (2,0);
      \coordinate (24) at (-0.5,-1);
      \coordinate (20) at (0.5,-1);

      \draw (8) -- (9);
      \draw (8) -- (19);
      \draw (7) -- (2);
      \draw (7) -- (12);
      \draw (9) -- (24);
      \draw (19) -- (24);
      \draw[ultra thick] (22) -- (24);
      \draw[ultra thick] (22) -- (20);
      \draw (2) -- (20);
      \draw (12) -- (20);

      \node at (3,1) {$1$};
      \node at (3,0) {$0$};
      \node at (3,-1) {$(-1)$};

\node at (2) {\small aijp};
\node at (7) {\small bflo};
\node at (8) {\small cflo};
\node at (9) {\small bfjp};
\node at (12) {\small chln};
\node at (19) {\small dhmn};
\node at (20) {\small efkp};
\node at (22) {\small eflq};
\node at (24) {\small egln};

    \end{tikzpicture}
  \end{center}
  The gradings on the right are \textit{a priori} just relative
  gradings, well-defined up to some additive constant. However, from
  the form of the homology $\HFa(\Sigma(2,3,7))$, we see that these
  gradings are actually the correct gradings.

  Since $\HFa(\Sigma(2,3,7))$ is zero-dimensional in grading $1$, we
  must have both $\bdy(\text{cflo})$ and $\bdy(\text{bflo})$ non-zero. And
  since $\HFa(\Sigma(2,3,7))$ is one-dimensional in grading $(-1)$, we
  must have
  $\bdy(\text{bfjp})=\bdy(\text{dhmn})=\bdy(\text{aijp})=\bdy(\text{chln})=0$. Therefore,
  the two-dimensional homology in grading $0$ is generated by
  $\langle\text{bfjp},\text{dhmn}\rangle/\bdy(\text{cflo})$ and
  $\langle\text{aijp},\text{chln}\rangle/\bdy(\text{bflo})$. The
  $\ZZ/2$-action interchanges these two generators, and is therefore,
  non-trivial.
\end{proof}

Our results summarize as follows: 
\begin{align*}
&q_{\tau}(T(3,7))=0 && d_{\tau}(T(3,7),2)=0 \\
&q_{\tau}(T(-3,7))=0 && d_{\tau}(T(-3,7),2)=0 \\
&q_{\tau}(P(2,-3,-7))=-2 && d_{\tau}(P(2,-3,-7),2)=0\\
&q_{\tau}(P(-2,3,7))=2&&d_{\tau}(P(-2,3,7),2)=2
\end{align*}

We deduce a number of corollaries.

\begin{corollary}
  The actions $\tau_*$ on $\HFa(\Sigma(K))$, $\HF^{-}(\Sigma(K))$, and
  $\HF^+(\Sigma(K))$ coming from the double-branched cover involution
  on $\Sigma(K)$ are not determined by the 3-manifold $\Sigma(K)$.
\end{corollary}

\begin{corollary}
  The knot invariant $q_\tau(K)$ is not determined by 3-manifold $\Sigma(K)$.
\end{corollary}

Furthermore, our examples show that $q_{\tau}(K)$ is not the same as three other major concordance invariants arising from knot homology theories.

\begin{corollary} The knot invariant $q_{\tau}(K)$ is not identified,
  even after scaling, with any of the Ozsv{\'a}th-Szab{\'o}
  concordance invariant $\tau(K)$
  \cite{OS03:4BallGenus,Rasmussen03:Knots}, the Rasmussen invariant
  $s(K)$ \cite{Rasmussen10:s}, or the Manolescu-Owens invariant
  $\delta(K)$ \cite{ManolescuOwens07:delta}.
\end{corollary}

\begin{proof} We have $s(T(3,7))=2\tau(T(3,7))=12$ and
  $s(P(2,-3,-7))=2\tau(P(2,-3,-7)=-10$ \cite[Theorem
  4]{Rasmussen10:s}. Furthermore, $\delta(T(3,7)) =
  \delta(P(2,-3,-7))=2d(\Sigma(2,3,7))=0$ \cite[Subsection 8.1]{AbsGraded}. On the other hand, for the
  right-handed trefoil,
  $q_{\tau}(T(2,3))=\tau(T(2,3))=s(T(2,3))/2=\delta(T(2,3))=1$.
\end{proof}

The Heegaard Floer homology of a three-manifold also carries a conjugation symmetry
\[
\iota_* \co \HFa(Y,\spinc) \rightarrow \HFa(Y, \overline{\spinc})
\]
induced by a sequence of Heegaard moves to a Heegaard diagram $\HD = (\Sigma, \alphas, \betas, z)$ for $Y$ from its conjugate $\overline{\HD}=(-\Sigma,\alphas,\betas,z)$ \cite[Theorem 2.4]{OS04:HolDiskProperties}. A similar action exists for the other variants of Heegaard Floer homology.

\begin{corollary} There are knots $K$ for which the actions $\tau_*$
  and $\iota_*$ on $\HFa(\Sigma(K))$ are different. The same holds
  with $\HFa(\Sigma(K))$ replaced by $\HF^+(\Sigma(K))$ or
  $\HF^-(\Sigma(K))$.
\end{corollary}

In \cite{HM:involutive}, Hendricks and Manolescu use the conjugation symmetry to associate to a three manifold $Y$ with spin structure $\spinc$ two homology cobordism invariants $\underline{d}(Y,s)$ and $\bar{d}(Y,s)$. The second of these is defined extremely formally similarly to $d_{\tau}(K,2)$.

\begin{corollary} The invariants $\bar{d}(\Sigma(K), \spinc_0)$ and $d_{\tau}(K,2)$ are not identified.
\end{corollary}

\begin{proof} We have $\bar{d}(-\Sigma(2,3,7))=2$ \cite[Section 6.8]{HM:involutive}.\end{proof}


\section{Spectral sequences from symplectic Khovanov homology}\label{sec:ssseq}
Khovanov homology admits a number of
deformations~\cite{Lee05:Khovanov,BarNatan05:Kh-tangle-cob,BrDCov,Szabo:geometric-ss,SSS:spectrals};
some of these are believed to correspond, at least philosophically, to
actions of $U(1)$ or $\ZZ/2$. On the Floer-theoretic side, it is easy
to see (and well known by experts) that the manifolds used to
construct symplectic Khovanov homology admits an $O(2)$-symmetry; see
Section~\ref{sec:ssss-review}. The action of a reflection
$\ZZ/2\subset O(2)$ on symplectic Khovanov homology was studied by
Seidel-Smith~\cite{SeidelSmith10:localization}; see also
Section~\ref{sec:ssss-conj}. The main goal of this section is to
develop some properties of the action of more general finite
subgroups of $O(2)$ in the symplectic setting. We start by reviewing the
definition of symplectic Khovanov homology and the group actions on it
in Section~\ref{sec:ssss-review} and some notations for the cohomology
of cyclic and dihedral groups in Section~\ref{sec:dihedral-groups}. 
We then prove a localization result for
some of these equivariant cohomologies in
Section~\ref{sec:ssseq-limits}, and prove that all of these
equivariant cohomologies are link invariants in
Section~\ref{sec:ssseq-invariance}. Section~\ref{sec:rKh} is a brief
digression, to prove that Manolescu's reduced symplectic Khovanov homology
is also an invariant of based links.  We conclude with some
speculations about the relationship between these constructions and
various combinatorial deformations in Section~\ref{sec:ssss-conj}.

As in the rest of this paper, all cohomology groups have coefficients
in $\Field$, unless otherwise noted.
\subsection{A brief review of symplectic Khovanov homology}
\label{sec:ssss-review}
We start by reviewing the construction of Seidel-Smith's symplectic Khovanov
homology~\cite{SeidelSmith6:Kh-symp}, largely from the perspective of
Manolescu's reformulation~\cite{Manolescu06:nilpotent}.

As in Section~\ref{sec:new-dcov}, fix a bridge diagram for a link $L$,
consisting of arcs $A_i$ and $B_i$ with endpoints
$\{p_1,\dots,p_{2n}\}$. Let $p(z)=(z-p_1)\cdot\cdots\cdot(z-p_{2n})$
and consider the affine algebraic surface
\begin{equation}\label{eq:surface-S}
S=\{(z,u,v)\in\CC^3\mid u^2+v^2+p(z)=0\}.
\end{equation}
The ring of regular functions on $S$ is, of course, 
\[
R=\CC[z,u,v]/(u^2+v^2+p(z)).
\]
Let $\Hilb^n(S)$ denote the Hilbert scheme of $n$-tuples of points in
$S$ or, more precisely, length $n$ closed subschemes of $S$. In particular,
a point in $\Hilb^n(S)$ is an ideal $I\subset R$ so that
$\dim_\CC(R/I)=n$. For example, given $n$ distinct points
$q_1=(z_1,u_1,v_1)$, \dots, $q_n=(z_n,u_n,v_n)$ of $S$, let
$I_j=(z-z_j,u-u_j,v-v_j)\subset R$ and let $I=I_1\cdot\cdots\cdot I_n$
be the product of these ideals. Then $I\in\Hilb^n(S)$. Thus, we have
an embedding $\Sym^n(S)\setminus\Delta\hookrightarrow \Hilb^n(S)$; it
turns out that the image is open. 
We call the complement of this
embedding the \emph{diagonal} in $\Hilb^n(S)$, and denote it by
$\Delta$.

Consider the map $i\co S\to \CC$, $(z,u,v)\mapsto z$, which
corresponds to the obvious map $\CC[z]\hookrightarrow \CC[z,u,v]\onto
R$.  Let $R_1$ denote the image of $\CC[z]$ in $R$. The
intersection of the ideal $(u^2+v^2+p(z))$ with $\CC[z]\subset
\CC[z,u,v]$ is trivial, so $R_1\cong\CC[z]$. Given an ideal
$I\in\Hilb^n(\CC)$, let $i(I)=I\cap R_1\subset R_1$. Let
\[
\ssspace{n}=\{I\in\Hilb^n(S)\mid i(I)\text{ has length $n$}\}.
\]
Manolescu shows that the space $\ssspace{n}$ is biholomorphic to the
space $\mathcal{Y}_{n,\tau}$ considered by
Seidel-Smith~\cite[Proposition 2.7]{Manolescu06:nilpotent}.

Each of the arcs $A_i$ and $B_i$ gives a totally real $S^2$ in $S$, by:
\begin{align*}
  \Sigma_{A_i}&=\{(z,u,v)\in S\mid z\in A_i,\ u,v\in \sqrt{-p(z)}\RR\}\\
  \Sigma_{B_i}&=\{(z,u,v)\in S\mid z\in B_i,\ u,v\in \sqrt{-p(z)}\RR\}.
\end{align*}
\cite[p.~2]{Manolescu06:nilpotent}.  These spheres give submanifolds
\begin{align*}
\CipLag_A &= \Sigma_{A_1}\times\cdots\times \Sigma_{A_n}\subset \left(\Sym^n(S)\setminus\Delta\right)\subset \Hilb^n(S)\\
\CipLag_B &= \Sigma_{B_1}\times\cdots\times \Sigma_{B_n}\subset \left(\Sym^n(S)\setminus\Delta\right)\subset \Hilb^n(S).
\end{align*}
There is a natural K\"ahler form on $\ssspace{n}$ for which
$\CipLag_A$ and $\CipLag_B$ are Lagrangian~\cite[Theorem
1.2]{Manolescu06:nilpotent}.

The Lagrangians $\CipLag_A$ and $\CipLag_B$ do not intersect
transversely. To define their Floer homology one needs to either
perturb one of them slightly, or work in the more general setting of
Lagrangians with clean intersections~\cite{Pozniak99:clean-Floer} (as
Manolescu does~\cite[Section 6.1]{Manolescu06:nilpotent}); the
approaches give isomorphic Floer homology groups. Manolescu shows that
the Floer cohomology $\HF(\CipLag_A,\CipLag_B)$ inside $\ssspace{n}$ is
the symplectic Khovanov homology $\KhSymp(L)$~\cite[Theorem
1.2]{Manolescu06:nilpotent}.

There is also a reduced version of the construction, giving a reduced
symplectic Khovanov homology $\rKhSymp$, where we consider the variety
\[
\rssspace{n}=\{I\in\Hilb^{n-1}(S)\mid i(I)\text{ has length $n-1$}\}.
\]
and consider the Floer homology of
$\rCipLag_A=\Sigma_{A_1}\times\cdots\times\Sigma_{A_{n-1}}$ and
$\rCipLag_B=\Sigma_{B_1}\times\cdots\times\Sigma_{B_{n-1}}$~\cite[Section
7.5]{Manolescu06:nilpotent}.  We will return to the reduced theory in
Section~\ref{sec:rKh}.

\begin{convention}
  The combinatorial Khovanov complex is of cohomological type, i.e.,
  the differential raises the homological grading by
  $1$. Correspondingly, Seidel-Smith constructed symplectic Khovanov
  homology as a Floer cohomology group. We will continue to work with
  the Floer chain (rather than cochain) complex of $\CipLag_A$ and
  $\CipLag_B$, so the cohomology of our complexes gives Seidel-Smith's
  symplectic Khovanov homology.
\end{convention}

\subsubsection{The group action}
There is an action of the orthogonal group $O(2,\CC)\subset
\GL(2,\CC)$ on $S$ induced from the action on $\CC\langle
u,v\rangle=\CC^2$, and a corresponding action of $O(2,\CC)$ on
$\Hilb^n(S)$.  The subgroup $O(2)\coloneqq O(2,\RR)\subset O(2,\CC)$
on $S$ preserves the spheres $\Sigma_{A_i}$ and $\Sigma_{B_i}$, and
hence the action of $O(2)$ on $\Hilb^n(S)$ preserves the Lagrangians
$\CipLag_A$ and $\CipLag_B$.

\begin{lemma}
  The $O(2)$-action on $\Hilb^n(S)$ preserves the open subset $\ssspace{n}$.
\end{lemma}
\begin{proof}
  Elements of $O(2)$ fix $R_1$, and hence do not change $i(I)$ or, in particular, the length of $i(I)$.
\end{proof}

Instead of focusing on the action of $O(2)$, we will restrict our
attention to the dihedral group of order $2^{m+1}$, $D_{2^m}=\langle
\sigma,\tau\mid \sigma^{2^m},\tau^2,\sigma\tau\sigma\tau\rangle$, acting on
$S$ as follows:
\begin{equation}\label{eq:dm-action-on-uv}
\sigma(z,u,v)=(z,u\cos\theta_m-v\sin\theta_m,u\sin\theta_m+v\cos\theta_m) \qquad \tau(z,u,v)=(z,u,-v),
\end{equation}
where $\theta_m=2\pi/{2^m}.$

\begin{lemma}\label{lem:perturb-Ciprian}
  The action of $D_{2^m}$ on $\ssspace{n}$ preserves $\CipLag_A$ and
  $\CipLag_B$. Moreover, one can choose a small Hamiltonian
  perturbation $\CipLag'_A$ of $\CipLag_A$ so that $\CipLag'_A$ is
  transverse to $\CipLag_B$ and is still preserved by the
  $D_{2^m}$-action.
\end{lemma}
\begin{proof}
  The first statement is immediate from the definitions of $\CipLag_A$
  and $\CipLag_B$. For the second statement, recall the map $i\co S\to
  \CC$. There are two kinds of intersections between the spheres $\Sigma_{A_j}$ and $\Sigma_{B_k}$:
  \begin{enumerate}[label=(i-\arabic*)]
  \item\label{item:first-int} The first kind consists of intersections which lie over a point
    $p_\ell\in\{1,\dots,p_{2n}\}$, i.e., intersections corresponding
    to endpoints of the arcs, i.e., points of $\Sigma_{A_j}\cap
    \Sigma_{B_k}\cap i^{-1}(p_\ell)$. These intersections are already
    transverse.
  \item\label{item:second-int} The second kind consists of intersections which
    lie over a point $q\not\in\{1,\dots,p_{2n}\}$, i.e., points of
    $(\Sigma_{A_j}\cap \Sigma_{B_k}\cap i^{-1}(q))\cong S^1$. The
    action of $D_{2^m}$ on each such circle of intersections is the
    usual action of $D_{2^m}\subset O(2)$ on $S^1\subset \CC$. Choose
    a symplectic identification of a neighborhood of this $S^1$ with
    $[-1,1]\times[-1,1]\times T^*S^1=\{(x,y,\theta,t)\}$ (with
    symplectic form $dx\wedge dy + d\theta\wedge dt$), so that
    $i(x,y,\theta,t)=x+iy$, $q$ is identified with the origin, the
    $O(2)$-action is on the $\theta$ coordinate, $\Sigma_{A_j}$ is
    $\{(x,0,\theta,0)\}$, $\Sigma_{B_k}$ is $\{(0,y,\theta,0)\}$, and
    in terms of $(z,u,v)$-coordinates the subspace
    $\{(x,y,\theta,0)\}$ is embedded as
    \[
    z=x+iy\qquad u=\gamma\cos\theta\qquad v=\gamma\sin\theta
    \]
    where $\gamma$ is chosen to be one of the two (distinct) square
    roots of $-p(z)$. Choose a bump function $\phi\co[-1,1]\to \RR$
    which is $1$ on $(-1/2,1/2)$ and $0$ outside
    $(-3/4,3/4)$. Consider the Hamiltonian
    \[
    H(x,y,\theta,t)=\phi(x)\phi(y)\phi(t)\cos(2^m\theta).
    \]
    Let $\Sigma'_{A_j}$ be the image of $\Sigma_{A_j}$ under a
    short-time flow of the Hamiltonian vector field associated to
    $H$. Then $\Sigma'_{A_j}$ is still $D_{2^m}$-invariant, intersects
    $\Sigma_{B_k}$ transversely (in $2^{m+1}$ points), and the projection $i(\Sigma'_{A_j})$ is still $A_j$.
  \end{enumerate}
  Choosing a perturbation as in Point~\ref{item:second-int} for each
  of the second kind of intersections and taking the product of the
  corresponding Lagrangians $\Sigma'_{A_j}$ gives the desired
  $\CipLag'_A$.
\end{proof}

We can define a freed Floer complex over $D_{2^m}$,
$\eKCSymp(L)\coloneqq \ECF[D_{2^m}](\CipLag'_A,\CipLag_B)$, by applying the construction of
Section~\ref{sec:2-groups} to the perturbation from
Lemma~\ref{lem:perturb-Ciprian}. We can also restrict to
$\ZZ/2^m\subset D_{2^m}$, giving a freed Floer complex
$\ECF[\ZZ/2^m](\CipLag'_A,\CipLag_B)$. In the special case $m=1$,
there are two nonequivalent $\ZZ/2$-actions: the action by rotation by
$\pi$, denoted $\sigma$, and the reflection, $\tau$; when we talk
about $\ECF[\ZZ/2](\CipLag'_A,\CipLag_B)$ we will make it clear in the
text which action is under consideration. We also obtain equivariant
Floer complexes and homologies
\begin{align*}
  \eCF[G](\CipLag'_A,\CipLag_B)&\coloneqq
  \RHomO{\Field[G]}(\ECF[G](\CipLag'_A,\CipLag_B),\Field)\\
  \eHF[G](\CipLag'_A,\CipLag_B)&\coloneqq H_*(\eCF[G](\CipLag'_A,\CipLag_B)),
\end{align*}
for $G=D_{2^m}$ or $\ZZ/2^m$, which are modules over $H^*(G)$. There
are spectral sequences from
$H^*(G;\HF(\CipLag'_A,\CipLag_B))\Rightarrow
\eHF[G](\CipLag'_A,\CipLag_B)$.  There is a quasi-isomorphism
$\ECF[\ZZ/2^m](\CipLag'_A,\CipLag_B)\cong\ECF[D_{2^m}](\CipLag'_A,\CipLag_B)$
over $\Field[\ZZ/2^m]$; furthermore, for any $m'\geq m$, there is a
quasi-isomorphism $\ECF[D_{2^m}](\CipLag'_A,\CipLag_B)\cong
\ECF[D_{2^{m'}}](\CipLag'_A,\CipLag_B)$ over $\Field[D_{2^m}]$.  It
follows from Theorem~\ref{thm:KC-equi-invt}, proved in
Section~\ref{sec:ssseq-invariance}, that all of this data is an
invariant of the link $L$.

\subsubsection{Homotopy classes and gradings}
Next, we see that the action of $D_{2^m}$ preserves the grading on the
symplectic Khovanov complex. Although there is a more direct argument
(see Lemma~\ref{lem:ss-fixed-pts-subcx}), we will use a lemma about
the path space which seems of independent interest.  Recall that the
Heegaard Floer homology groups decompose according to
$\SpinC$-structures, which correspond to the homotopy classes of paths
between the Lagrangians $T_\alpha$ and $T_\beta$.  There is no
corresponding decomposition of symplectic Khovanov homology:
\begin{lemma}\label{lem:Yn-1-conn}
  The space $\ssspace{n}$ is simply connected. In particular, the
  space of paths from $\CipLag_A$ to $\CipLag_B$ is path connected.
\end{lemma}
This lemma and its proof was communicated to us by Mohammed Abouzaid
and Ivan Smith.
\begin{proof}
  The proof of simple-connectivity is essentially the same as Seidel-Smith's proof that
  $H^1(\ssspace{n})=0$~\cite[Lemma 44]{SeidelSmith6:Kh-symp}: as
  explained there (see also~\cite[Section 2.3]{SeidelSmith6:Kh-symp}),
  $\ssspace{n}$ is diffeomorphic to the $(n,n)$ Springer variety $\mathfrak{B}_{n,n}$.
  Russell-Tymoczko~\cite[Theorem 5.5]{RussellTymoczko11:Springer} and
  Wehrli~\cite[Theorem 1.2]{Wehrli:Springer} showed that
  $\mathfrak{B}_{n,n}$ is homeomorphic to the space
  $\widetilde{S}\coloneqq \cup_{a}S_a\subset (\CC P^1)^{2n}$ where the
  union is over all isotopy classes $a$ of upper half-plane
  crossingless matchings $a$, viewed as involutions
  $a\co\{1,\dots,2n\}\to\{1,\dots,2n\}$ and
  \[
  S_a\coloneqq \{(l_1,\dots,l_{2n})\in (\CC P^1)^{2n}\mid \forall j\ l_j=l_{a(j)}\}.
  \]
  This is a union of simply connected sets, all of the pairwise
  intersections of which are connected, so by the Seifert-van Kampen
  theorem, $\widetilde{S}\cong \mathfrak{B}_{n,n}\cong\ssspace{n}$ is
  simply connected.  Now, the fact that the path space is connected
  follows from the path-loop fibration
  \[
  \Omega\ssspace{n}\to \mathcal{P}(\CipLag_A,\CipLag_B)\to \CipLag_A\times\CipLag_B.\qedhere
  \]
\end{proof}

\begin{proposition}\label{prop:D2m-grading}
  The action of $D_{2^m}$ on $(\ssspace{n},\CipLag'_A,\CipLag_B)$ preserves the absolute (homological) grading.
\end{proposition}
\begin{proof}
  By Lemma~\ref{lem:Yn-1-conn}, any two generators
  $\x,\y\in\CipLag'_A\cap\CipLag_B$ can be connected by a homotopy
  class of Whitney disks $\phi\in\pi_2(\x,\y)$. Since $D_{2^m}$ acts
  by symplectomorphisms, the Maslov index satisfies
  $\mu(\phi)=\mu(g\phi)$ for any $g\in D_{2^m}$, so the
  $D_{2^m}$-action preserves the relative Maslov gradings. Since there
  is a $D_{2^m}$-fixed generator (in fact, at least two fixed
  generators; see Section~\ref{sec:ssseq-limits} for more discussion),
  it follows that the $D_{2^m}$-action preserves the absolute grading,
  as well.
\end{proof}

\subsubsection{The cylindrical formulation}\label{sec:ss-cyl}
To visualize holomorphic disks in $(\ssspace{n},\CipLag'_A,\CipLag_B)$
it will be convenient to use the following six-dimensional
formulation; compare~\cite[Section 5.8]{AbouzaidSmith:arc-alg}.  First,
an intersection point between $\CipLag_A$
and $\CipLag_B$ corresponds to an $n$-tuple of points $\x=\{x_j\in
\Sigma_{A_j}\cap \Sigma_{B_{\sigma(j)}}\}_{j=1}^n$, where $\sigma\in
S_n$ is a permutation. Further, if $\Sigma'_{A_j}$ is a small
Hamiltonian perturbation of $\Sigma_{A_j}$ then each point in
$i(\Sigma'_{A_j}\cap \Sigma_{B_\ell})$ is close to an intersection
point of ${A_j}$ and ${B_\ell}$. Given $x\in \Sigma'_{A_j}\cap
\Sigma_{B_\ell}$ let $i(x)$ denote the corresponding intersection
point in $A_j\cap B_\ell$, and given a generator $\x=\{x_j\}$ let
$i_*(\x)=\sum_j i(x_j)$, viewed as a $0$-chain in $\CC$.  As in the
proof of Lemma~\ref{lem:perturb-Ciprian}, we can further arrange that
each intersection point of $A_j$ and $B_\ell$ of
type~\ref{item:first-int} (above) corresponds to a single intersection
point of $\Sigma'_{A_j}$ and $\Sigma_{B_\ell}$, while each
intersection point of type~\ref{item:second-int} corresponds to
$2^{m+1}$ intersection points of $\Sigma'_{A_j}$ and
$\Sigma_{B_\ell}$.

For the differential, note that a complex structure $j$ on $S$ induces
a complex structure $\Hilb^{n}(j)$ on $\Hilb^{n}(S)$. It is possible
to achieve transversality for holomorphic curves in $\Hilb^{n}(S)$
using $[0,1]$-families of complex structures of the form $\Hilb^{n}(j)$, and
we will suppose we are working with such complex structures. Then,
there is a tautological correspondence
between holomorphic disks
\begin{equation}\label{eq:holo-disk}
\phi\co (\RR\times [0,1],\RR\times\{0\},\RR\times\{1\})\to
(\Hilb^{n}(S),\CipLag'_A,\CipLag_B)
\end{equation}
connecting $\x$ to $\y$
and holomorphic maps
\begin{equation}\label{eq:cyl-ss}
\psi\co(X,\bdy X)\to \bigl(\RR\times[0,1]\times S,(\RR\times\{0\}\times (\Sigma'_{A_1}\cup\dots\cup \Sigma'_{A_n}))\cup(\RR\times\{1\}\times (\Sigma_{B_1}\cup\dots\cup \Sigma_{B_n}))\bigr),
\end{equation}
where $X$ is a Riemann surface with boundary and $2n$ boundary
punctures, $\psi$ is asymptotic to $\{-\infty\}\times[0,1]\times\x$
and $\{+\infty\}\times[0,1]\times \y$, and
$\pi_{\RR\times[0,1]}\circ\psi$ is an $n$-fold branched covering.  The
branch points of $\pi_{\RR\times[0,1]}\circ\psi$ correspond to
$\phi^{-1}(\Delta)$. Given a holomorphic disk in $\Hilb^{n}(S)$ there
is an induced homology class $D(\phi)=(\pi_S\circ\psi)_*([X,\bdy X])$
in
$H_2(S,\Sigma'_{A_1}\cup\cdots\cup\Sigma'_{A_n}\cup\Sigma_{B_1}\cup\cdots\cup\Sigma_{B_n})$
which, inspired by terminology in Heegaard Floer homology, we will
call the \emph{domain of $\phi$}. There is also a \emph{projected
  domain} $i_*D(\phi)\in H_2(\CC,A_1\cup \cdots\cup A_{n}\cup
B_1\cup\cdots \cup B_n)$.

If the map $i\co S\to \CC$ is holomorphic with respect to an almost complex structure $j_0$ on $S$ then all of the coefficients in the
projected domain of a $\Hilb^n(j_0)$-holomorphic curve are
non-negative. It follows that the same holds for almost complex
structures sufficiently close to $\Hilb^n(j_0)$, and we will assume
that we are always working with such almost complex structures.

Note that holomorphic disks which contribute to the symplectic
Khovanov differential must lie in $\ssspace{n}\subset \Hilb^n(S)$, so
not every curve of the form~\eqref{eq:cyl-ss} contributes to the
differential. 
Let $\nabla=\Hilb^n(S)\setminus \ssspace{n}$. Given a
map $\psi$ as in Formula~\eqref{eq:cyl-ss}, let $B(\psi)$ be the
preimage in $X$ of the branch points of
$\pi_{\RR\times[0,1]}\circ\psi$. (The points $B(\psi)$ correspond to
the diagonal in $\Hilb^n(S)$.)  
\begin{lemma}\label{lem:cyl-Yn}
  For a generic choice of one parameter families of almost complex
  structures $j$ on $S$, a rigid holomorphic curve $\psi$ as in
  Formula~\eqref{eq:cyl-ss} corresponds to a disk in
  $\ssspace{n}\subset\Hilb^n(S)$ if and only if
  \begin{enumerate}[label=(YC)]
  \item\label{item:YC} \ \ \  The map $
  (\Id\times i)\circ \psi \co (X\setminus B(\psi))\to \RR\times[0,1]\times\CC
  $
  is an embedding.
  \end{enumerate}
  Specifically, this cylindrical formulation~\ref{item:YC} of maps to $\ssspace{n}$
  is valid whenever all rigid holomorphic disks in
  $(\Hilb^n(S),\CipLag_A,\CipLag_B)$ lie in the complement of
  $\Delta\cap \nabla$. 
\end{lemma}
\begin{proof}
  If $a\in\Hilb^n(S)\setminus\Delta$, so $a$ corresponds to $n$ distinct
  points in $S$, then $a\in\ssspace{n}$ if and only if $i(a)$ consists
  of $n$ distinct points in $\CC$. Thus, if all rigid holomorphic
  disks miss $\Delta\cap\nabla$ then the condition of a disk $\phi$ missing
  $\nabla$ is equivalent to the condition that $i\circ \pi_S\circ
  \psi$ is injective on each regular fiber of
  $\pi_{\RR\times[0,1]}\circ \psi$. This is exactly the extra
  condition on $\psi$ in the statement of the lemma. 
  As $\Delta\cap \nabla$ is a real codimension $4$ subvariety of
  $\Hilb^n(S)$ disjoint from $\CipLag_A$ and $\CipLag_B$, for a
  generic choice of almost complex structure all rigid holomorphic
  disks miss this subspace.
\end{proof}

Suppose that a holomorphic curve $\psi$ as in
Formula~\eqref{eq:cyl-ss} satisfies:
\begin{itemize}
\item The map $\pi_{\RR\times[0,1]}\circ \psi$ has only order $2$ branch
  points and
\item If $p,q\in B(\psi)$ are distinct preimages of a branch
  point then $i(\pi_S(\psi(p)))\neq i(\pi_S(\psi(q)))$.
\end{itemize}
In this case, the map $\phi\co\RR\times[0,1]\to\Hilb^n(S)$
corresponding to $\psi$ lies the complement of
$\Delta\cap\nabla$. Thus, we can verify that the cylindrical formulation~\ref{item:YC}
of maps to $\ssspace{n}$ is valid without referencing the Hilbert
scheme: we need only check the two conditions above for all rigid
holomorphic curves $\psi$.

\subsection{A brief review of cyclic and dihedral groups}\label{sec:dihedral-groups}

Some good references for this section are
\cite[Chapter~VI,~Section~3]{FiedorowiczPriddy} and
\cite[Chapter~IV,~Section~2]{AdemMilgram}.

We start by recalling the cohomology of $\ZZ/2^m=\langle \sigma\mid
\sigma^{2^m}\rangle$, with coefficients in $\Field$. The group cohomology is:
\[
H^*(\ZZ/2^m)=
\begin{cases}
\Field[\alpha,w_2]/(\alpha^2+w_2)=\Field[\alpha]&\text{if $m=1$,}\\
\Field[\alpha,w_2]/(\alpha^2)&\text{if $m>1$;}
\end{cases}
\] 
see, for example,~\cite[Example 3.41]{Hatcher02:book}. The reason for the
notation $w_2$ is that, under the inclusion $\ZZ/2^m\hookrightarrow
\SO(2)$, $w_2$ is the pullback of the universal second Stiefel-Whitney
class in $H^2(\BSO(2);\Field)$. In the case $m=1$, the class $\alpha$
was denoted $\theta$ earlier in the paper; we have changed notation
because soon there will be two different $\ZZ/2$-actions. 

A particularly nice free resolution of $\Field$ over $\Field[\ZZ/2^m]$
is given by
\[
\xymatrix{
  \Field[\ZZ/2^m]&\Field[\ZZ/2^m]\ar[l]_-{1+\sigma}&&\Field[\ZZ/2^m]\ar[ll]_-{(1+\sigma)^{2^m-1}}&\Field[\ZZ/2^m]\ar[l]_-{1+\sigma}&&\cdots\ar[ll]_-{(1+\sigma)^{2^m-1}}.
}
\]
So, given a chain complex $C_*$ over $\Field[\ZZ/2^m]$ we can compute
the equivariant cohomology of $C_*$ via the complex
\[
\xymatrix{
  C^*\ar[r]^-{1+\sigma}&C^*\ar[rr]^-{(1+\sigma)^{2^m-1}}&& C^*\ar[r]^-{1+\sigma} & C^*\ar[rr]^-{(1+\sigma)^{2^m-1}}&&\cdots.
}
\]
The group cohomology $H^*(\ZZ/2^m)$ acts on the above complex as
follows: $w_2$ shifts the complex two units to the right; and $\alpha$
is the following map
\[
\xymatrix{
  C^*\ar[r]^-{1}&C^*\ar[rr]^-{(1+\sigma)^{2^m-2}}&& C^*\ar[r]^-{1} & C^*\ar[rr]^-{(1+\sigma)^{2^m-2}}&&\cdots.
}
\]
(Compare Lemma~\ref{lem:ring-str}.)

Next, we turn to the dihedral group of order $2^{m+1}$,
$D_{2^m}=\langle \sigma,\tau\mid
\sigma^{2^m},\tau^2,\sigma\tau\sigma\tau\rangle$, which
acts on $\RR^2$ and $\CC^2$ as
\[
  \sigma(u,v)=(u\cos\theta_m-v\sin\theta_m,u\sin\theta_m+v\cos\theta_m) \qquad
  \tau(u,v)=(u,-v),
\]
where $\theta_m=2\pi/{2^m}$, giving a subgroup inclusion
$D_{2^m}\into O(2)$.  Passing to group cohomologies (with coefficients
in $\Field$), we get a map $H^*(D_{2^m})\leftarrow\Field[w_1,w_2]$,
where $w_1\in H^1(BO(2))$ and $w_2\in H^2(BO(2))$ are the universal
first and second Stiefel-Whitney classes.  By an abuse of notation,
let $w_1$ and $w_2$ denote their images in $H^*(D_{2^m})$.

For any
finite group $G$ with abelianization $\Ab(G)$, $H^1(G)$ is just
$\Hom(\Ab(G),\Field)$. Each dihedral group $D_{2^m}$ has
abelianization isomorphic to $(\ZZ/2)^2$, generated by $\sigma$ and
$\tau$, and therefore, $H^1(D_{2^m})\cong\Field\oplus\Field$. One of
the non-zero elements is $w_1$, which with respect to the above
identification satisfies
\[
w_1(\sigma)=0\qquad w_1(\tau)=1\qquad w_1(\sigma+\tau)=1.
\]
Let $\alpha,\alpha+w_1$ be the other two non-zero elements of
$H^1(D_{2^m})$, labeled so that
\begin{align*}
\alpha(\sigma)&=1 & \alpha(\tau)&=0 &\alpha(\sigma+\tau)&=1\\
(\alpha+w_1)(\sigma)&=1 & (\alpha+w_1)(\tau)&=1& (\alpha+w_1)(\sigma+\tau)&=0.
\end{align*}
Then the group cohomology is given by
\begin{equation}\label{eq:dm-cohomology-ring}
H^*(D_{2^m})=
\begin{cases}
\Field[\alpha,w_1,w_2]/(\alpha(\alpha+w_1)+w_2)=\Field[\alpha,w_1]&\text{if $m=1$,}\\
\Field[\alpha,w_1,w_2]/(\alpha(\alpha+w_1))&\text{if $m>1$;}
\end{cases}
\end{equation}
compare~\cite[Proposition VI.3.1]{FiedorowiczPriddy},~\cite[Theorem
IV.2.7]{AdemMilgram}.

Fix a bounded-below chain complex $C_*$ with a left 
$D_{2^m}$-action. As in
Section~\ref{sec:background}, given any free resolution $R_*$ of
$\Field$ over $\Field[D_{2^m}]$, we obtain a free resolution
$C_*\otimes R_*$ of $C_*$ over $\Field[D_{2^m}]$, and let
$\RHomO{\Field[D_{2^m}]}(C_*,\Field)=\Hom_{\Field[D_{2^m}]}(C_*\otimes
R_*,\Field)$ and
$\eH[D_{2^m}]^*(C_*)\coloneqq \ExtO{\Field[D_{2^m}]}(C_*,\Field)=H_*(\RHomO{\Field[D_{2^m}]}(C_*,\Field))$. The
following is a reasonably simple
free resolution of $\Field$ over $\Field[D_{2^m}]$:
\begin{equation}\label{eq:Dm-free-res}
\vcenter{\xymatrix{
  \vdots\ar[d]^-{\tau+1}&\vdots\ar[d]^-{\sigma\tau+1}&&\vdots\ar[d]^-{\tau+1}&\vdots\ar[d]^-{\sigma\tau+1}\\
  \Field[D_{2^m}]\ar[d]^-{\tau+1}&\Field[D_{2^m}]\ar[l]_-{1+\sigma}\ar[d]^-{\sigma\tau+1}&&\Field[D_{2^m}]\ar[ll]_-{(1+\sigma)^{2^m-1}}\ar[d]^-{\tau+1}&\Field[D_{2^m}]\ar[l]_-{1+\sigma}\ar[d]^-{\sigma\tau+1}&&\cdots\ar[ll]_-{(1+\sigma)^{2^m-1}}\\
  \Field[D_{2^m}]&\Field[D_{2^m}]\ar[l]_-{1+\sigma}&&\Field[D_{2^m}]\ar[ll]_-{(1+\sigma)^{2^m-1}}&\Field[D_{2^m}]\ar[l]_-{1+\sigma}&&\cdots\ar[ll]_-{(1+\sigma)^{2^m-1}}
}};
\end{equation}
compare~\cite[Proposition IV.2.5]{AdemMilgram}.  (This is a diagram of
left modules, and the map
$\Field[D_{2^m}]\stackrel{x}{\longrightarrow}\Field[D_{2^m}]$ is
shorthand for right-multiplication by $x$.) The dual complex $C^*$ can
be equipped with a left $D_{2^m}$-action by defining the action of
$g\in D_{2^m}$ on $f\from C_*\to\Field$ to be $f\circ g^{-1}$.  With
this left-action, the complex $\RHomO{\Field[D_{2^m}]}(C_*,\Field)$
coming from the resolution~\eqref{eq:Dm-free-res} is:
\begin{equation}\label{eq:Dn-cx-res}
\vcenter{\xymatrix{
  \vdots\ar@{<-}[d]^-{\tau+1}&\vdots\ar@{<-}[d]^-{\sigma\tau+1}&&\vdots\ar@{<-}[d]^-{\tau+1}&\vdots\ar@{<-}[d]^-{\sigma\tau+1}\\
  C^*\ar@{<-}[d]^-{\tau+1}&C^*\ar@{<-}[l]_-{1+\sigma}\ar@{<-}[d]^-{\sigma\tau+1}&&C^*\ar@{<-}[ll]_-{(1+\sigma)^{2^m-1}}\ar@{<-}[d]^-{\tau+1}&C^*\ar@{<-}[l]_-{1+\sigma}\ar@{<-}[d]^-{\sigma\tau+1}&&\cdots\ar@{<-}[ll]_-{(1+\sigma)^{2^m-1}}\\
  C^*&C^*\ar@{<-}[l]_-{1+\sigma}&&C^*\ar@{<-}[ll]_-{(1+\sigma)^{2^m-1}}&C^*\ar@{<-}[l]_-{1+\sigma}&&\cdots\ar@{<-}[ll]_-{(1+\sigma)^{2^m-1}}
}}.
\end{equation}

This complex is $\ZZ\times\ZZ$-filtered.  The associated graded
complex with respect to the $\ZZ\times\ZZ$-filtration is isomorphic to $C^*\otimes
H^*(D_{2^m})$.
From such a doubly filtered chain complex one can construct many
different spectral sequences, each with $E_1$-page given by
$H^*(C_*)\otimes H^*(D_{2^m})$ and $E_\infty$-page given by
$\ExtO{\Field[D_{2^m}]}(C_*,\Field)$. For instance, consider the set of functions
\begin{equation}\label{eq:fun-with-double-complexes}
\big\{f\co\NN\times\NN\to \RR\mid f(a,b)>f(a',b')\text{ if }\big(a\geq
a',b\geq b',(a,b)\neq(a',b')\big)\big\}.
\end{equation}
Such a function produces an $\RR$-filtered complex by declaring the
$C^*$-summand that lives over the lattice point $(a,b)$ to be in
filtration $f(a,b)$. The associated graded object for this
singly filtered complex is still $C^*\otimes
H^*(D_{2^m})$. Restricting to functions $f$ with $\Image(f)=\NN$, we
get spectral sequences $H^*(C_*)\otimes
H^*(D_{2^m})\Rightarrow\ExtO{\Field[D_{2^m}]}(C_*,\Field)$.

\begin{lemma}\label{lem:ring-str}
  The action of $H^*(D_{2^m})$ on $\eH[D_{2^m}]^*(C_*)$ is induced from
  the following actions on the $\RHomO{\Field[D_{2^m}]}(C_*,\Field)$
  complex from~\eqref{eq:Dn-cx-res}:
  \begin{itemize}[leftmargin=*]
    \item $w_1$ shifts the complex one unit up;
    \item $w_2$ shifts the complex two units to the right; and
    \item $\alpha$ acts as follows (we have suppressed the
      differentials on the chain complex)
      \[
      \vcenter{\xymatrix{
        \vdots&\vdots\ar@{<-}[d]^-{1}&&\vdots&\vdots\ar@{<-}[d]^-{1}\\
        C^*&C^*\ar@{<-}[l]_-{1}\ar@{<-}[d]^-{1}&&C^*\ar@{<-}[ll]_-{(1+\sigma)^{2^m-2}}&C^*\ar@{<-}[l]_-{1}\ar@{<-}[d]^-{1}&&\cdots\ar@{<-}[ll]_-{(1+\sigma)^{2^m-2}}\\
        C^*&C^*\ar@{<-}[l]_-{\sigma}\ar@{<-}[d]^-{1}&&C^*\ar@{<-}[ll]_-{\sigma(1+\sigma)^{2^m-2}}&C^*\ar@{<-}[l]_-{\sigma}\ar@{<-}[d]^-{1}&&\cdots\ar@{<-}[ll]_-{\sigma(1+\sigma)^{2^m-2}}\\
        C^*&C^*\ar@{<-}[l]_-{1}&&C^*\ar@{<-}[ll]_-{(1+\sigma)^{2^m-2}}&C^*\ar@{<-}[l]_-{1}&&\cdots\ar@{<-}[ll]_-{(1+\sigma)^{2^m-2}}}}.
      \]
  \end{itemize}
\end{lemma}
\begin{proof}
  It is straightforward to check that these actions are chain
  maps. The actions by $w_1$ and $\alpha$ commute up to homotopy, and
  the action by $\alpha(\alpha+w_1)$ is null-homotopic (respectively,
  homotopic to the action by $w_2$) if $m>1$ (respectively,
  $m=1$). Furthermore, these actions are natural with respect to maps
  induced by $D_{2^m}$-equivariant chain maps $C_*\to C'_*$. One can
  check by direct computation that these actions induce the ring
  structure on $H^*(D_{2^m})$ when $C_*$ is one-dimensional, and
  therefore by naturality, these induce the module structure of
  $\eH[D_{2^m}]^*(C_*)$ over $H^*(D_{2^m})$ for general $C_*$.
\end{proof}
In particular, by ignoring the $\alpha$-action, the
$\RHomO{\Field[D_{2^m}]}(C_*,\Field)$ complex
from~\eqref{eq:Dn-cx-res} can be viewed as a strict dg-module over
$\Field[w_1,w_2]$. This also produces a double filtration on the
$\RHom$ complex by the $w_1$ and $w_2$ powers. However, this double
filtration is `twice' as coarse as as the previous one (and consequently,
carries `half' the information).

\subsection{A localization result}\label{sec:ssseq-limits}
Let $L$ be a link with $|L|$ components. The $D_{2^m}$-action on $(\ssspace{n},\CipLag_A',\CipLag_B)$ restricts to a $\ZZ/2^m$-action, for which we have the following localization result:
\begin{theorem}\label{thm:khsymp-loc}
  For each positive integer $m$, the localized equivariant Floer
  cohomology $w_2^{-1}\eHF[\ZZ/2^m](\CipLag_A',\CipLag_B)$ is
  isomorphic to a direct sum of $2^{|L|}$ copies of
  $w_2^{-1}H^*(\ZZ/2^m)$. In particular, for
  $m=1$, we have that $\alpha^{-1}\eHF(\CipLag_A',\CipLag_B)$ is
  isomorphic to a direct sum of $2^{|L|}$ copies of
  $\Field[\alpha,\alpha^{-1}]$.
\end{theorem}

The case $m=1$ follows easily from the Seidel-Smith localization
theorem~\cite[Theorem 20]{SeidelSmith10:localization}; see also
Proposition~\ref{prop:localized-equivariant}. The general case seems to need
an additional argument.

We start by discussing the fixed points of the action. The action of
$\ZZ/2^m$ on $\CipLag_A'\cap\CipLag_B$ is semi-free: each
$\ZZ/2^m$-orbit is either free or trivial. Further, there are exactly
$2^{|L|}$ fixed points of the $\ZZ/2^m$ action: in the notation of
Lemma~\ref{lem:perturb-Ciprian} the fixed points are $n$-tuples of
intersections of type~\ref{item:first-int}. (See also
Lemma~\ref{lem:fixed-discrete}.) We will refer to these fixed points
as \emph{$\ZZ/2^m$-fixed generators}.

\begin{lemma}\label{lem:ss-fixed-pts-subcx}
  For any link $L$, there is a bridge diagram
  $(\{p_i\},\{A_i\},\{B_i\})$ so that the $\ZZ/2^m$-fixed generators
  are the unique generators of maximal homological grading. In
  particular, the $\ZZ/2^m$-fixed generators span a subcomplex of the
  symplectic Khovanov cochain complex $\CF^*(\CipLag_A',\CipLag_B)$.
\end{lemma}
\begin{proof}
  Manolescu explains how to compute the homological grading on
  $\CF(\CipLag_A',\CipLag_B)$ in the non-equivariant setting, where
  each type~\ref{item:second-int} intersection corresponds to two
  intersections of $\Sigma'_A$ and 
  $\Sigma_B$~\cite[Section 6.2]{Manolescu06:nilpotent}. To compute the
  relative homological grading, first isotope the bridge diagram so
  that all of the $A_i$-arcs are horizontal, and then replace each
  $B_i$ in the bridge diagram by a figure $8$ (immersed circle with
  one transverse double point) $\gamma_i$ near $B_i$, so that
  $\gamma_i$ intersects each $A_j$ at right angles. The generators of
  $\CF(\CipLag_A',\CipLag_B)$ correspond to the intersections of
  $A_i$ with $\gamma_i$. Given two generators $\x=\{x_i\}$ and
  $\y=\{y_i\}$, choose a smooth path $\eta_i$ in $\gamma_i$ from
  $\x\cap \gamma_i$ to $\y\cap \gamma_i$, and let $\mu(\eta_i)$ denote
  the Maslov index of the loop of subspaces
  $T\eta_i\subset\RR^2$. Then the grading difference between $\x$ and
  $\y$ is $\sum_i\mu(\eta_i)$.

  This description carries over to a description of the gradings in
  the equivariant case, as follows. In the equivariant case, each
  type~\ref{item:second-int} intersection corresponds to $2^m$
  intersections of $\Sigma'_A$ and $\Sigma_B$, corresponding to the
  critical points of a $D_{2^m}$-equivariant Morse function on
  $S^1$. Like the indices of the critical points of a Morse function
  on $S^1$, these $2^m$ intersection points lie in two adjacent
  gradings---the same two gradings as the two intersections for the
  non-equivariant perturbation.

  Now, to achieve the conditions of the lemma, wind the $B$-arcs as in
  Figure~\ref{fig:winding}. That is, fix a small loop $\zeta_i$ around
  each $p_i$ and perform $N_i\gg 0$ Dehn twists on the $B$-arcs around
  each $\zeta_i$. From Manolescu's description of the gradings, it is
  clearly possible to choose the $N_i$ so that all of the
  type~\ref{item:first-int} intersection points have the same grading,
  and, if the $N_i$ are large, this grading is (strictly) higher than
  the grading of any type~\ref{item:second-int} intersection
  point. The result follows.
\end{proof}

\begin{figure}
  \centering
  \includegraphics{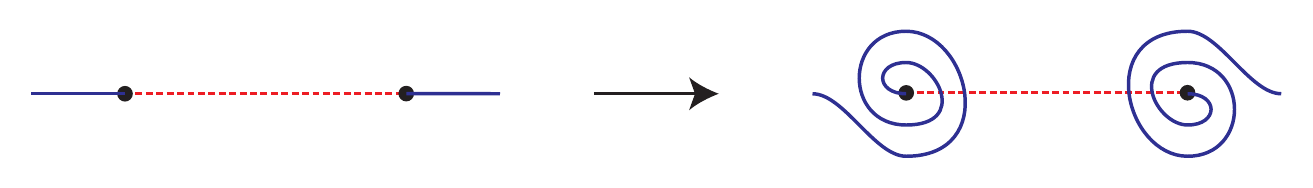}
  \caption{\textbf{Winding a bridge diagram.} Perform $N$ Reidemeister
    I moves near each endpoint, as shown, by winding the $B$-curve
    clockwise $N$ full turns. The case $N=2$ is shown.}
  \label{fig:winding}
\end{figure}

From Lemma~\ref{lem:ss-fixed-pts-subcx}, the localization result
follows from homological algebra, which we abstract as its own lemma:

\begin{lemma}\label{lem:trivial-localization}
  Suppose that $C_*$ is a chain complex over $\Field$, with
  $C_n=\Field\langle X_n\rangle$ freely generated by a finite set
  $X_n$, with $X_n=\emptyset$ for $|n|$ sufficiently large. Suppose
  further that $\ZZ/2^m$ acts on each $X_n$, and that this action
  induces an action by chain maps on $C_*$; that the action on
  each $X_n$ is semi-free, so $X_n=X_n^\sigma\amalg X_n^{\free}$ where
  $\sigma$ acts freely on $X_n^{\free}$ and fixes $X_n^\sigma$; and that the $X_n^{\free}$ generate a
  subcomplex $C_*^{\free}$ of $C_*$. In particular, the dual basis to
  $X_n^\sigma$ generates a subcomplex $C^*_\sigma$ of the dual complex
  $C^*=\HomO{\Field}(C_*,\Field)$. Let $H(C^*_\sigma)$ denote the
  homology of $C^*_\sigma$. Then the localized equivariant cohomology
  \[
  w_2^{-1}\eH[\ZZ/2^m](C_*)
  =w_2^{-1}\ExtO{\Field[\ZZ/2^m]}(C_*,\Field)
  \]
  is isomorphic to $w_2^{-1}H^*(\ZZ/2^m)\otimes_\Field
  H(C^*_\sigma)$. In particular, if $m=1$ then
  \[
  \alpha^{-1}\eH[\ZZ/2](C_*)
  =\alpha^{-1}\ExtO{\Field[\ZZ/2]}(C_*,\Field)
  \]
  is isomorphic to $\alpha^{-1}H^*(\ZZ/2)\otimes_\Field
  H(C^*_\sigma)=\Field[\alpha,\alpha^{-1}]\otimes_\Field
  H(C^*_\sigma)$.
\end{lemma}
\begin{proof}
  We can compute the localized equivariant cohomology using the
  complex
  \begin{equation}\label{eq:Tate}
    \xymatrix{
    \cdots \ar[rr]^-{(1+\sigma)^{2^m-1}}&& C^*\ar[r]^-{1+\sigma}&C^*\ar[rr]^-{(1+\sigma)^{2^m-1}}&& C^*\ar[r]^-{1+\sigma} & C^*\ar[rr]^-{(1+\sigma)^{2^m-1}}&&\cdots.
    }
  \end{equation}
  where $C^*$ denotes the dual complex, over $\Field$, to $C_*$.

  Inclusion of $C^*_\sigma$ gives a map $\iota$ of bicomplexes
    from 
  \begin{equation}\label{eq:Tate-sigma}
    \xymatrix{
    \cdots \ar[r]^-{0} & C^*_\sigma\ar[r]^-{0} & C^*_\sigma\ar[r]^-{0} & C^*_\sigma\ar[r]^-{0} & C^*_\sigma\ar[r]^-{0}&\cdots.
    }
  \end{equation}
  to~\eqref{eq:Tate}.
  If we filter both bicomplexes~\eqref{eq:Tate} and~\eqref{eq:Tate-sigma} by the
  degree $*$ at each vertex then $\iota$ induces an isomorphism on the
  $E^1$-page of the resulting spectral sequences.  Hence, by spectral
  sequence comparison~\cite[Corollary 3.15]{McCleary01:sseq},
  the map $\iota$ is a quasi-isomorphism between the
  complexes~\eqref{eq:Tate} and~\eqref{eq:Tate-sigma}. The result
  follows.
\end{proof}

\begin{proof}[Proof of Theorem~\ref{thm:khsymp-loc}]
  The special case for $m=1$ is immediate from Seidel-Smith's
  localization theorem~\cite[Theorem 20]{SeidelSmith10:localization}.
  For the general case, since there are no nontrivial disks contained
  in the fixed set of $\sigma$, we can achieve transversality using a
  $\ZZ/2^m$-equivariant almost complex structure $J$ (compare
  Section~\ref{sec:equi-equi}). Hence, by 
  Proposition~\ref{prop:equi-is-equi-K}, we may compute the equivariant
  symplectic Khovanov cohomology using $\CF(\CipLag'_A,\CipLag_B;J)$
  as the chain complex over $\Field[\ZZ/2^m]$.  By
  Lemma~\ref{lem:ss-fixed-pts-subcx}, we can arrange that the
  $\ZZ/2^m$-fixed generators span a subcomplex of
  $\CF^*(\CipLag'_A,\CipLag_B;J)$.  Then, via
  Proposition~\ref{prop:D2m-grading}, the hypotheses of
  Lemma~\ref{lem:trivial-localization} are satisfied. There are
  exactly $2^{|L|}$ $\ZZ/2^m$-fixed generators, all in the same
  homological grading, so by Lemma~\ref{lem:trivial-localization} the
  localized equivariant Floer cohomology is a direct sum of $2^{|L|}$
  copies of $w_2^{-1}H^*(\ZZ/2^m)$, as desired.
\end{proof}

\subsection{Equivariant invariance}\label{sec:ssseq-invariance}
We start by outlining the non-equivariant proof of invariance of
symplectic Khovanov homology. Invariance was originally proved by
Seidel-Smith in the case of braid closures~\cite[Theorem
1]{SeidelSmith6:Kh-symp}; invariance for general bridge diagrams was
proved by Waldron~\cite[Theorem 4.12]{Waldron:KhSympMaps}.  We will
mostly follow their arguments.

The first, easy step is independence of the choice of perturbation
$\CipLag'_A$ and of the almost complex structure used to define the
Floer homology groups. Both of these statements follow from the usual
continuation map arguments; in the equivariant case, they follow from
Proposition~\ref{prop:invariance-gen-group}. To keep notation simple,
we will suppress the perturbations $\CipLag'_A$ from the notation in
the rest of this section; instances of $\CipLag_A$ really mean
$\CipLag'_A$ for some choice of perturbation as in
Lemma~\ref{lem:perturb-Ciprian}.

The meat of the argument, then, is independence of the bridge diagrams. 
Any two bridge diagrams for the same link can be connected by a sequence of:
\begin{enumerate}[label=(B\arabic*)]
\item\label{item:reg-htpy} Isotopies of the arcs $A_i$ (or rather, $\cup_i 
A_i$) and the arcs $B_i$ (or rather, $\cup_i B_i$), rel endpoints.
\item\label{item:hand} Handleslides (``passing moves'' in Waldron's terminology).
\item Stabilizations.
\end{enumerate}
Note that move~\ref{item:reg-htpy} induces both Reidemeister I and II
moves of the knot diagram, and move~\ref{item:hand} induces Reidemeister III moves.  
The bridge moves are illustrated in Figure~\ref{fig:bridge-moves}.

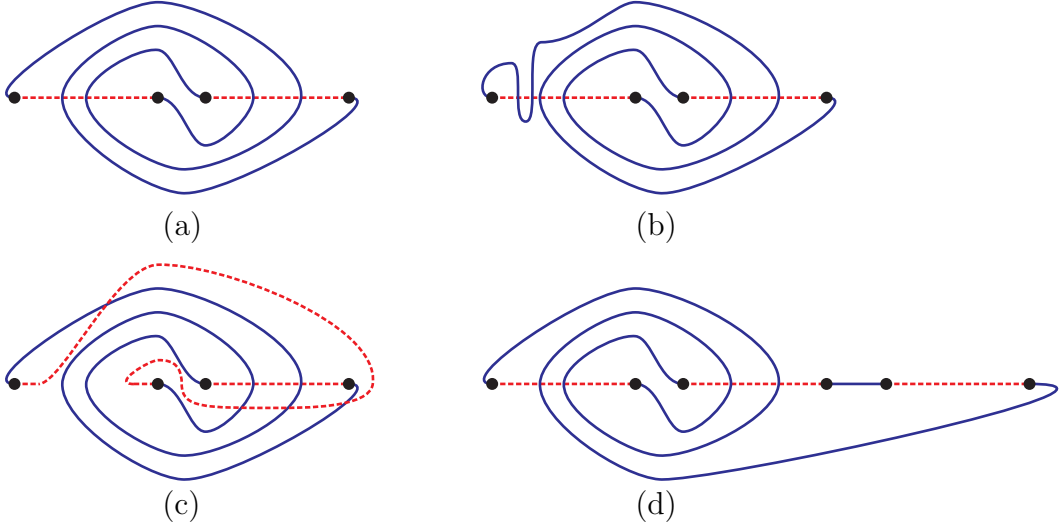
\begin{figure}
  \centering
  \begin{overpic}[tics=10]{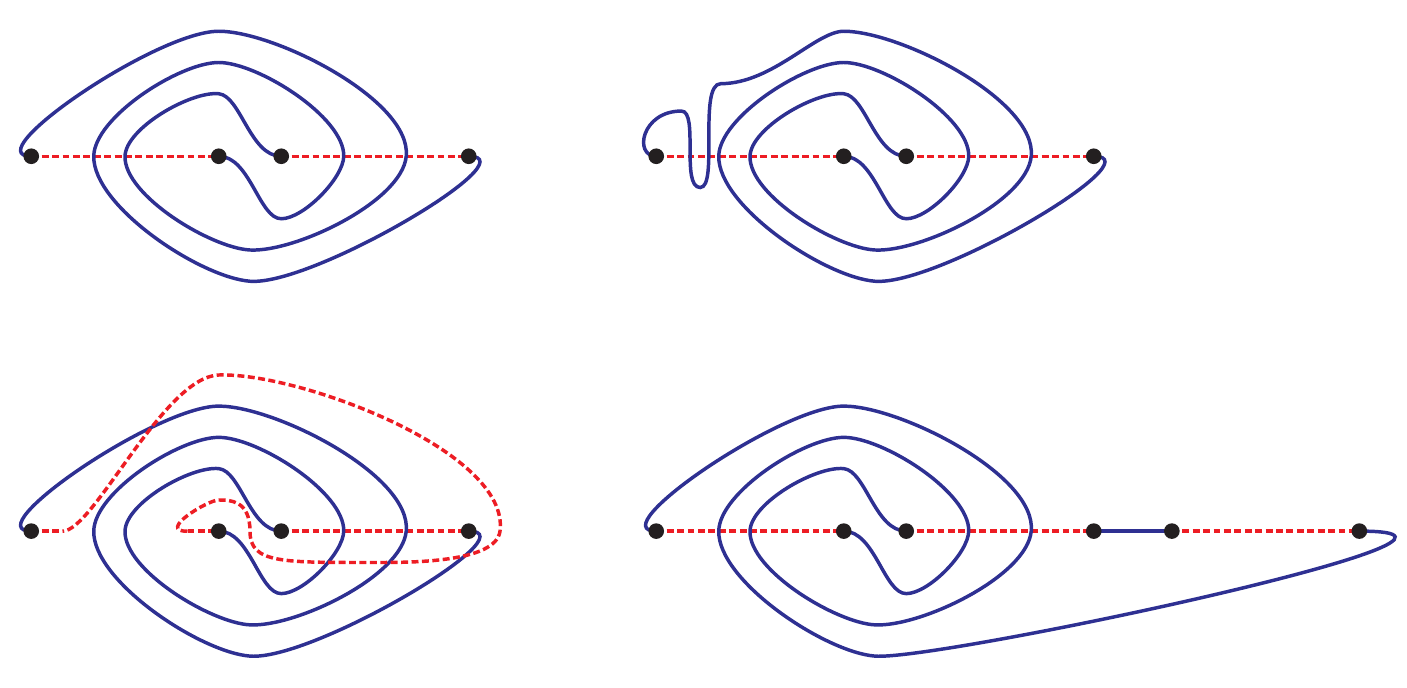}
  \put(16,25){(a)}
  \put(60,25){(b)}
  \put(16,-1){(c)}
  \put(60,-1){(d)}
  \end{overpic}
  \caption{\textbf{Moves for bridge diagrams.} (a) A bridge
    diagram. The $A_i$ are red and dashed and the $B_i$ are blue and
    solid. (b) The result of an isotopy of the $B_i$. (c) The result of a
    handleslide (or passing move). (d) The result of a stabilization.}
  \label{fig:bridge-moves}
\end{figure}

The proof of invariance has two steps:
\begin{enumerate}
\item Isotopies and handleslides induce Hamiltonian isotopies
  of $\CipLag_A$ and $\CipLag_B$. This is obvious for isotopies,
  from either Seidel-Smith's or Manolescu's
  formulation. For handleslides, see Seidel-Smith~\cite[Lemma
  48]{SeidelSmith6:Kh-symp}.
\item Stabilizations. This is proved by a degeneration
  argument~\cite[Section 5.4]{SeidelSmith6:Kh-symp}; the rest of this
  section gives a slightly different degeneration proof of
  stabilization invariance.
\end{enumerate}

\subsubsection{Non-equivariant stabilization invariance}\label{sec:ssseq-stab-invariance}
Because of handleslide invariance, it suffices to prove stabilization
invariance for a stabilization at a point $p\in B_n$ adjacent to
$\infty$, i.e., introducing a new arc $A'_{n+1}$ adjacent to the
unbounded region of the knot complement, and replacing $B_n$ by
$B'_{n}$ and $B'_{n+1}$. Let $S$ and $S'$ denote the
surface~\eqref{eq:surface-S} before and after the stabilization,
respectively. Let $R$ denote the bounded region in the stabilized
diagram $(\CC,A_1\cup\cdots\cup B'_{n+1})$ adjacent to $A'_{n+1}$, and
let $R_0$ denote the bounded region in unstabilized diagram
$(\CC,A_1\cup\cdots\cup B_n)$ adjacent to $p$. Let
$\{p_{2n+1},p_{2n+2}\}=A'_{n+1}\cap(B'_n\cup B'_{n+1})$, so that the
arc $A'_{n+1}$, oriented as part of the boundary of $R$, runs from
$p_{2n+1}$ to $p_{2n+2}$.

The correspondence between generators of the Floer complexes before
and after the stabilization is fairly clear. Each of
$\Sigma_{A'_{n+1}}\cap \Sigma_{B'_{n+1}}$ and
$\Sigma_{A'_{n+1}}\cap\Sigma_{B'_n}$ consists of a single point, and
any generator $\x$ for the stabilized diagram contains exactly one of
these two points, as well as one point from
$(\Sigma_{B'_n}\cup\Sigma_{B'_{n+1}})\cap(\Sigma_{A_1}\cup\cdots\cup\Sigma_{A_n})$
and $(n-1)$ other points in
$(\Sigma_{B_1}\cup\cdots\cup\Sigma_{B_{n-1}})\cap(\Sigma_{A_1}\cup\cdots\cup\Sigma_{A_n})$. The
points of
$(\Sigma_{B'_n}\cup\Sigma_{B'_{n+1}})\cap(\Sigma_{A_1}\cup\cdots\cup\Sigma_{A_n})$
for the stabilized diagram correspond to the points of
$\Sigma_{B_n}\cap(\Sigma_{A_1}\cup\cdots\cup\Sigma_{A_n})$; in
Figure~\ref{fig:kh-stab-invt}, this corresponds to the fact that
$B'_{n}\cup B'_{n+1}$ agrees with $B_n$ on one side of the line
$C$. Forgetting the point on
$(\Sigma_{B'_n}\cup\Sigma_{B'_{n+1}})\cap\Sigma_{A'_{n+1}}$, using the correspondence 
\[
\left[(\Sigma_{B'_{n+1}}\cup\Sigma_{B'_n})\cap(\Sigma_{A_1}\cup\cdots\cup\Sigma_{A_n})\right]\cong\left[\Sigma_{B_n}\cap(\Sigma_{A_1}\cup\cdots\cup\Sigma_{A_n})\right],
\]
and leaving the other $(n-1)$ points alone gives the corresponding generator for the un-stabilized diagram.

Before studying the holomorphic disks, we make an observation about their domains:
\begin{lemma}\label{lem:domain-mult-1}
  The projected domain of any holomorphic disk in
  $(\ssspace{n},\CipLag_{A},\CipLag_{B})$, respectively
  $(\ssspace{n+1},\CipLag_{A'},\CipLag_{B'})$, has multiplicity $0$ or $1$ at
  $R_0$, respectively $R$.
\end{lemma}
\begin{proof}
  Our conditions on the complex structures imply that the projected
  domain of any holomorphic disk has only non-negative coefficients
  (compare~\cite[Lemma 3.2]{OS04:HolomorphicDisks}). Further, given
  generators $\x$ and $\y$ of the Floer complex, the boundary of any
  projected domain $i_*D(\phi)$ of a curve from $\x$ to $\y$ satisfies
  $\bdy((\bdy i_*D(\phi))\cap A)=i_*(\y)-i_*(\x)$, as $0$-chains in
  $\CC$. Thus, the coefficients of the projected domain at two
  adjacent regions (separated by $A_i$ or $B_i$) can differ by at most
  one. Since the coefficient of the projected domain at the unbounded
  region is zero, the lemma follows. (For instance, in the stabilized
  diagram, $i_*D(\phi)$ has multiplicity $0$ at $R$ if and only if
  $\x\cap \{p_{2n+1},p_{2n+2}\}=\y\cap\{p_{2n+1},p_{2n+2}\}$,
  $i_*D(\phi)$ has multiplicity $1$ at $R$ if and only if
  $p_{2n+1}\in\x$ and $p_{2n+2}\in \y$, and $i_*D(\phi)$ cannot have
  multiplicity $2$ or higher at $R$.)
\end{proof}

We turn now to the identification of the differentials on the
stabilized and un-stabilized diagrams, using a degeneration
argument. This argument will be given in the cylindrical formulation
of Section~\ref{sec:ss-cyl}, and is similar to (a simple case of) the
neck stretching argument underlying bordered Heegaard Floer
homology~\cite[Section 9.1]{LOT1}.

Consider a vertical line $C$ which separates $A'_{n+1}$ from
$A_1\cup\cdots\cup A_n$, intersects $B'_n$ and $B'_{n+1}$ in one point
each, and is disjoint from $B_1\cup\cdots\cup B_{n-1}\cup
A_1\cup\cdots\cup A_n\cup A'_{n+1}$. (See
Figure~\ref{fig:kh-stab-invt}.)  Assume that $B'_n$ and $B'_{n+1}$ are
horizontal near $C$. Let $E$ denote the component of $\CC\setminus C$
containing $A_1,\dots,A_n$ and let $D$ denote the component of
$\CC\setminus C$ containing $A'_{n+1}$.  Choose a family of complex
structures $j_T^\CC$ on $\CC$ for $T$ in $\RR$ large with long necks
along $C$, so that the length of the neck tends to infinity as
$T\to\infty$, and almost complex structures $j'_T$ on $S'$ so that the projection $i\co
S'\to \CC$ is $(j'_T,j_T^\CC)$-holomorphic and $j'_T$ is induced from
$j_T^\CC$ via
Formula~(\ref{eq:surface-S}) over $D$. (As usual, by almost complex
structure we mean a $[0,1]$-parameter family of almost complex
structures, though we assume the family is constant over $D$. For
readability, we will continue to suppress the $[0,1]$ parameter from
the discussion.)  There are induced almost complex structures
$\Hilb^{n+1}(j'_T)$ on $\Hilb^{n+1}(S')$.

Let $\sigma_{n}'=\Sigma_{B'_{n}}\cap i^{-1}(C)$ and
$\sigma_{n+1}'=\Sigma_{B'_{n+1}}\cap i^{-1}(C)$, so $\sigma_n'$ and
$\sigma_{n+1}'$ are topologically circles. There is a vector field
$\vec{R}$ on $i^{-1}(C)$ corresponding to the stretching, which is a
lift of the vector field $\frac{\partial}{\partial y}$ on $C$ to
$S'$. The data $(i^{-1}(C),\vec{R})$ is adjusted to the K\"ahler form on $S'$
in the sense of~\cite[Section 3.4]{BEHWZ03:CompactnessInSFT}.  For
each point $q\in\sigma_n'$ there is a unique $\vec{R}$-chord
$\gamma_q$ starting at $q$ and ending on $\sigma_{n+1}'$. (In
particular, the flow of $\vec{R}$ is Morse-Bott with respect to
$\sigma_n'$ and $\sigma_{n+1}'$.) In fact, these chords lie in the
surfaces $i^{-1}(C)\cap \{v/u=\text{constant}\}$ (which are Heegaard
surfaces for the branched double cover of $L$~\cite[Section
7.3]{Manolescu06:nilpotent}; see Figure~\ref{fig:ssss-HD}).

\begin{lemma}
  For generic $j_T^\CC$ and $j'_T$ as above, the rigid moduli spaces of
  $\Hilb^{n+1}(j'_T)$-holomorphic disks, or equivalently the moduli
  spaces of maps as in Equation~\eqref{eq:cyl-ss}, are transversely
  cut out for all sufficiently large $T$.
\end{lemma}
\begin{proof}
  Given a non-constant holomorphic disk $u$ in
  $\ssspace{n+1}$ we can
  find a point $p\in\bD^2$ so that $i(u(p))\in\Hilb^{n+1}(\CC)$
  consists of $n+1$ distinct points. Writing
  $u(p)=\{q_1,\dots,q_{n+1}\}$ and $i(u(p))=\{z_1,\dots,z_{n+1}\}$,
  the tangent space at $u(p)$ to $\Hilb^{n+1}(S')$ decomposes as
  \[
  T_{u(p)}(\Hilb^{n+1}(S'))=
  (T_{q_1} i^{-1}(z_1))\oplus(T_{z_1}\CC)\oplus\cdots\oplus(T_{q_{n+1}} i^{-1}(z_{n+1}))\oplus(T_{z_{n+1}}\CC).
  \]
  Our conditions on the almost complex structures amount to choosing
  an almost complex structure on each summand separately. This is
  enough flexibility for standard transversality arguments
  (see, e.g.,~\cite[Section 3.2]{MS04:HolomorphicCurvesSymplecticTopology})
  to go through.
\end{proof}

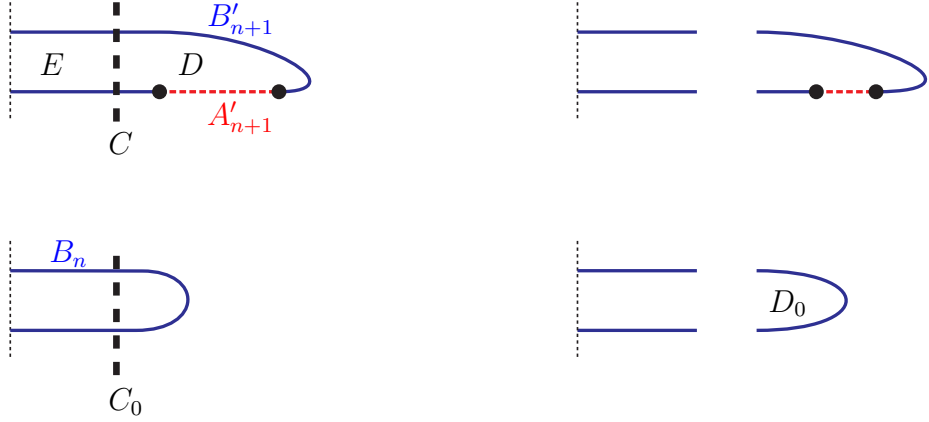
\begin{figure}
  \centering
  \begin{overpic}[scale=1.25, tics=10]{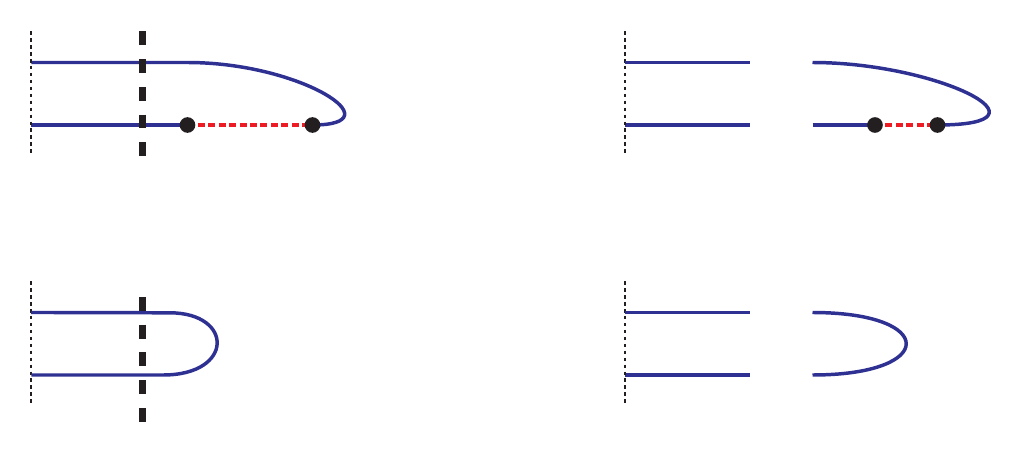}
    \put(23,30){\textcolor{red}{$A'_{n+1}$}}
    \put(23,40){\textcolor{blue}{$B'_{n+1}$}}
    \put(13,27){$C$}
    \put(6,35){$E$}
    \put(20,35){$D$}
    \put(7,16){\textcolor{blue}{$B_n$}}
    \put(80,11){$D_0$}
    \put(13,1){$C_0$}
  \end{overpic}
  \caption{\textbf{Stabilization invariance for symplectic
      Khovanov homology.} In all cases, only the stabilization region is
    shown; the rest of the knot is to the left.}
  \label{fig:kh-stab-invt}
\end{figure}

By Lemma~\ref{lem:domain-mult-1}, there are two cases: projected
domains with multiplicity $0$ at $D\cap R$ and projected domains with
multiplicity $1$ at $D\cap R$. The first case is simple: it is
analogous to the stabilization invariance proof for the Heegaard Floer
homology group $\HFa$~\cite[Section 10]{OS04:HolomorphicDisks}. So, we
will focus on the case of projected domains with multiplicity $1$ at
$D\cap R$.

As the neck stretching parameter $T$ goes to infinity,
$(j_{[0,1]\times\RR}\times j'_T)$-holomorphic maps of the
form~\eqref{eq:cyl-ss} converge to pairs of holomorphic maps
\begin{align}
v_\infty^D&\co (X_D,\bdy X_D)\to (\RR\times[0,1]\times i^{-1}(D),(\RR\times\{0\}\times \Sigma_{A'_{n+1}})\cup(\RR\times\{1\}\times\Sigma_{B_n^D}))\label{eq:vD-infty}\\
v_\infty^E&\co (X_E,\bdy X_E)\to \bigl(\RR\times[0,1]\times i^{-1}(E),((\RR\times\{0\}\times(\Sigma_{A_1}\cup\dots\cup\Sigma_{A_n}))\label{eq:vE-infty}\\
\nonumber&\hspace{2.25in}\cup(\RR\times\{1\}\times(\Sigma_{B_1}\cup\dots\cup\Sigma_{B_{n-1}}\cup \Sigma_{B_n^E})))\bigr)
\end{align}
so that 
\begin{itemize}
\item $X_D$ has three boundary punctures.
\item $X_E$ has $2n+1$ boundary punctures.
\item There is a boundary puncture $p_D$ of $X_D$ and $p_E$ of $X_E$
  so that at $p_D$ the map $\pi_{S'}\circ v_\infty^D\co X_D\to i^{-1}(D)$
  is asymptotic to an $\vec{R}$-chord $\gamma$ in $i^{-1}(C)$ between $\sigma'_n$
  and $\sigma'_{n+1}$ and at $p_E$ the map $\pi_{S'}\circ v_\infty^E\co
  X_E\to i^{-1}(E)$ is asymptotic to the same chord
  $\gamma$.
\item Filling in the puncture $p_D\circ v_\infty^D$, the map
  $\pi_{\RR\times[0,1]}\circ v_\infty^D\co X_D\to\RR\times[0,1]$
  extends to a degree $1$ branched cover (i.e., a diffeomorphism).
\item Filling in the puncture $p_E$, the map
  $\pi_{\RR\times[0,1]}\circ v_\infty^E\co X_E\to\RR\times[0,1]$
  extends to a degree $n$ branched cover.
\end{itemize}

(To deduce the compactness theorem we are using, fix a sequence $v_T$
of $j'_T$-holomorphic curves and consider the two
projections $\pi_{S'}\circ v_T$ and $\pi_{\RR\times[0,1]}\circ v_T$
separately, along the lines of~\cite[Section 5.4]{LOT1}. The gluing
result follows from standard arguments, along the lines
of~\cite[Proposition A.2]{Lipshitz06:CylindricalHF}, say, which in
turn follows arguments
from~\cite{MS04:HolomorphicCurvesSymplecticTopology}.)

We claim that for any curve $v^E_\infty$ as in
Formula~\eqref{eq:vE-infty} asymptotic to some chord $\gamma$ in
$i^{-1}(C)$ there is a unique curve $v^D_\infty$ as in
Formula~\eqref{eq:vD-infty} asymptotic to the same chord
$\gamma$. Thus, for large $T$, $(j_{[0,1]\times\RR}\times
j'_T)$-holomorphic curves as in Equation~\eqref{eq:cyl-ss} correspond
to $(j_{[0,1]\times\RR}\times j'_\infty)$-holomorphic curves
$v_\infty^E$ as in Formula~\eqref{eq:vE-infty}.

To see this, consider the Riemann surface $\mathcal{H}\coloneqq
i^{-1}(D)\cap \{u=0\}$ and its intersections with $\Sigma_{A'_{n+1}}$,
$\Sigma_{B'_n}^D$, and $\Sigma_{B'_{n+1}}^D$, as shown in
Figure~\ref{fig:ssss-HD}.  There are two chords between
$\Sigma_{B'_n}^D\cap \mathcal{H}$ and $\Sigma_{B'_{n+1}}^D\cap
\mathcal{H}$, labeled $\gamma$ and $\gamma'$ in the figure. With
respect to the standard complex structure on $S'$ induced from
Formula~(\ref{eq:surface-S}), there is a unique holomorphic disk as in
Formula~\eqref{eq:vD-infty} asymptotic to $\gamma$ (respectively
$\gamma'$). It follows from boundary injectivity~\cite[Proposition
3.9]{OS04:HolomorphicDisks} that these disks are transversally cut out
in $\mathcal{H}$. Further, the normal Maslov index to these disks is
$1$, so it follows from automatic transversality~\cite[Theorem
2]{HLS97:GenericityHoloCurves} that these holomorphic disks are
transversally cut out in $i^{-1}(D)$. Now, the $S^1$ action on $S'$
(or rather, $i^{-1}(D)$) implies that, with respect to the standard
complex structure, there are similar disks asymptotic to each chord
from $\Sigma_{B'_n}^D$ to $\Sigma_{B'_{n+1}}^D$. 

Next, we claim these are, algebraically, the only non-constant
holomorphic curves of the form $v_\infty^D$. To see that, note that
the count of holomorphic curves $v_\infty^E$ asymptotic to a generic
chord $\gamma$ is invariant under Lagrangian isotopies of $B'_n\cup
B'_{n+1}$.  So, consider instead the diagram in which $D\cap R$
consists of the upper half plane, $D\cap(B_{n}'\cup A'_{n+1}\cup
B'_{n+1})$ is the real axis, $\bdy A'_{n+1}=\{0,1\}$, and the
polynomial $p(z)$ is simply $p(z)=z(z-1)$. Consider the projection
$\pi_u$ to the $u$-coordinate. The boundary conditions $\Sigma_{A_{n+1}}$,
$\Sigma_{B'_n}\cap D$, and $\Sigma_{B'_{n+1}}\cap D$ project to $\RR$,
$i\RR$, and $i\RR$ respectively. Because of the $S^1$-symmetry, it
suffices to consider the case that $\gamma$ lies in the surface
$\mathcal{H}$. Fix a map $v_\infty^D\co X_D\to \RR\times[0,1]\times
i^{-1}(D)$ asymptotic to $\gamma$. The map $i\circ v_\infty^D$
identifies $X_D$ with the upper half plane. With respect to this
identification, $\pi_u\circ v_\infty^D$ is a
map from the upper
half plane to $\CC$ sending the boundary to a bounded subset of
$\RR\cup i\RR$ (because $\pi_u(\gamma)=0$). It follows from the open
mapping principle that $\pi_u\circ v_\infty^D$ is constant, so
$v_\infty^D$ is contained in $\mathcal{H}$ and so $v_\infty^D$ must be
one of the curves in the $S^1$-family previously constructed. In
particular, the count of curves $v_\infty^D$ asymptotic to $\gamma$ is
$1$.

Thus, the count of $j'_T$-holomorphic disks as in
Formula~\eqref{eq:cyl-ss}, for $T$ large, is the same as the count of
curves $u_\infty^E$ as in Formula~\eqref{eq:vD-infty}, with no
restriction on the chord $\gamma_q$. By Lemma~\ref{lem:domain-mult-1},
the condition that the curve with large $T$ lies in
$\ssspace{n+1}\subset \Hilb^{n+1}(S')$ corresponds to the condition
that the limiting curve lies in $\ssspace{n}\subset\Hilb^{n}(E)$; see
also Lemma~\ref{lem:cyl-Yn}.

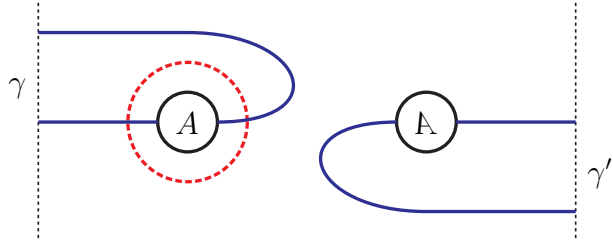
\begin{figure}
  \begin{overpic}[scale=1.25, tics=10]{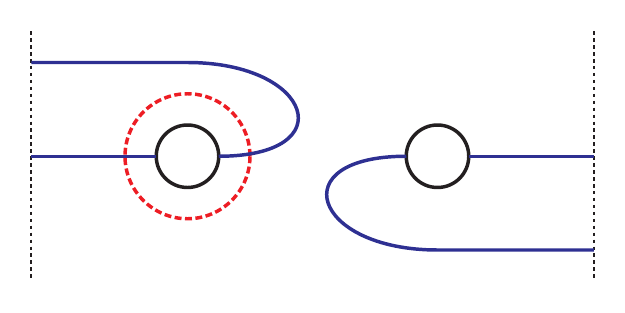}
    \put(28,23){$A$}
    \put(68,23){\reflectbox{$A$}}
    \put(0,30){$\gamma$}
    \put(97,15){$\gamma'$}
  \end{overpic}
  \caption{\textbf{The surface $i^{-1}(D)\cap \{u=0\}$.} This is the branched double cover of $D$ along the points $p_{2n+1}$ and $p_{2n+2}$. The circles labeled $A$ are connected by a handle; the branch points are at the intersections of the $\Sigma_{A'_{n+1}}\cap \mathcal{H}$ (dashed) and $(\Sigma_{B'_n}^E\cup\Sigma_{B'_{n+1}}^E)\cap \mathcal{H}$ (solid).}\label{fig:ssss-HD}
\end{figure}

Next, consider instead stretching the neck along a vertical line $C_0$
in the un-stabilized diagram $(\CC,A_1,\dots,A_n,B_1,\dots,B_n)$ which
intersects $B_n$ perpendicularly in two points and is disjoint from
the other arcs. (Again, see Figure~\ref{fig:kh-stab-invt}.)
This time, holomorphic curves of the
form~\eqref{eq:cyl-ss} converge to pairs $(v^D_\infty,v^E_\infty)$
where $v^E_\infty$ is of the form~\eqref{eq:vE-infty} and $v^D_\infty$
is a holomorphic disk with one puncture mapping to $\RR\times
[0,1]\times i^{-1}(D_0)$ which is constant in the $\RR\times[0,1]$-direction,
and 
the map $i\circ\pi_{S}\circ v_\infty^D$ is a diffeomorphism to
$D_0\cap R_0$. 
(This is similar to the
boundary degenerations in Heegaard Floer theory~\cite[Section
5]{OS05:HFL}.) Again, there is a unique disk $v^D_\infty$ asymptotic
to each chord $\gamma$, and in spite of the somewhat degenerate
situation, these curves are transversely cut out by the
$\dbar$-equations. So, for large $T$, holomorphic curves in the
un-stabilized diagram $S$ of the form~\eqref{eq:cyl-ss} correspond to
maps $v_\infty^E$ of the form~\eqref{eq:vE-infty}, and hence to curves
in the stabilized diagram. This completes the cylindrical proof of
stabilization invariance.

\subsubsection{Equivariant invariance}
\begin{proof}[Proof of Theorem~\ref{thm:KC-equi-invt}]
  As in the non-equivariant case, it suffices to check invariance under 
  isotopies and handleslides of the $A_j$ and $B_j$ and stabilizations of the 
  diagram. 
  
  It is clear that an isotopy of the $A_j$ is covered by an $O(2)$-equivariant 
  Lagrangian (and hence Hamiltonian---$H^1(\ssspace{n})$ vanishes) isotopy of 
  the Lagrangians $\CipLag_A$ and $\CipLag_B$; and this isotopy extends to an 
  isotopy between the $D_{2^m}$-equivariant perturbations of $\CipLag_A'$ from 
  Lemma~\ref{lem:perturb-Ciprian}. Thus, 
  Proposition~\ref{prop:invariance-gen-group} implies that the 
  quasi-isomorphism type of $\eCF[D_{2^m}](\CipLag_A,\CipLag_B)$, 
  over $\Field[D_{2^m}]$, is unchanged by isotopies.
  
  As noted earlier handleslides give Lagrangian isotopies in the
  Hilbert scheme~\cite[Lemma 48]{SeidelSmith6:Kh-symp}. Rather than
  checking that these isotopies can be made equivariant, we appeal to
  Proposition~\ref{prop:non-equi-invar-gen-group}. For notational
  simplicity, assume we are sliding $A_1$ over $A_2$. Let $\CipLag_A$
  and $\CipLag_{\wt{A}}$ be the Lagrangians before and after the
  handleslide.  By Proposition~\ref{prop:non-equi-invar-gen-group}, to
  prove handleslide invariance it suffices to verify that we can
  choose equivariant perturbations $\CipLag'_A$ and
  $\CipLag'_{\wt{A}}$ of $\CipLag_A$ and $\CipLag_{\wt{A}}$, and a
  $D_{2^m}$-equivariant almost complex structure $J$ on $\ssspace{n}$, so that
  all moduli spaces of holomorphic disks in
  $(\ssspace{n},\CipLag'_A,\CipLag'_{\wt{A}})$ with Maslov index $\leq 1$ are transversally cut
  out, and the fundamental class is fixed by the $D_{2^m}$-action. Let $\CipLag'_A=\CipLag_A$, i.e., do not do any perturbation at
  all. For the perturbation $\CipLag'_{\wt{A}}$ perturb the handleslid
  arcs $\wt{A}$ so that they intersect the original arcs $A$ only at
  endpoints and let $\CipLag'_{\wt{A}}$ be the associated
  Lagrangian. (So, both $\CipLag'_A$ and $\CipLag'_{\wt{A}}$ are
  $O(2)$-invariant.) See Figure~\ref{fig:sss-handleslides}. It follows
  that the Floer complex $\CF(\CipLag'_A,\CipLag'_{\wt{A}})$ has rank
  $2^{n}$; since the diagram $(A,\wt{A})$ represents an $n$-component
  unlink, $\KhSymp(A,\wt{A})$ also has rank $2^{n}$, and the
  differential on
  $\CF(\CipLag'_A,\CipLag'_{\wt{A}})$ must vanish. (The differential
  also vanishes for grading reasons: all generators lie in the same
  grading modulo $2$.) Note that all
  intersection points are fixed by
  the $D_{2^m}$-action, so in particular the top class is fixed.
  
  \begin{figure}
    \begin{overpic}[tics=10, scale=1]{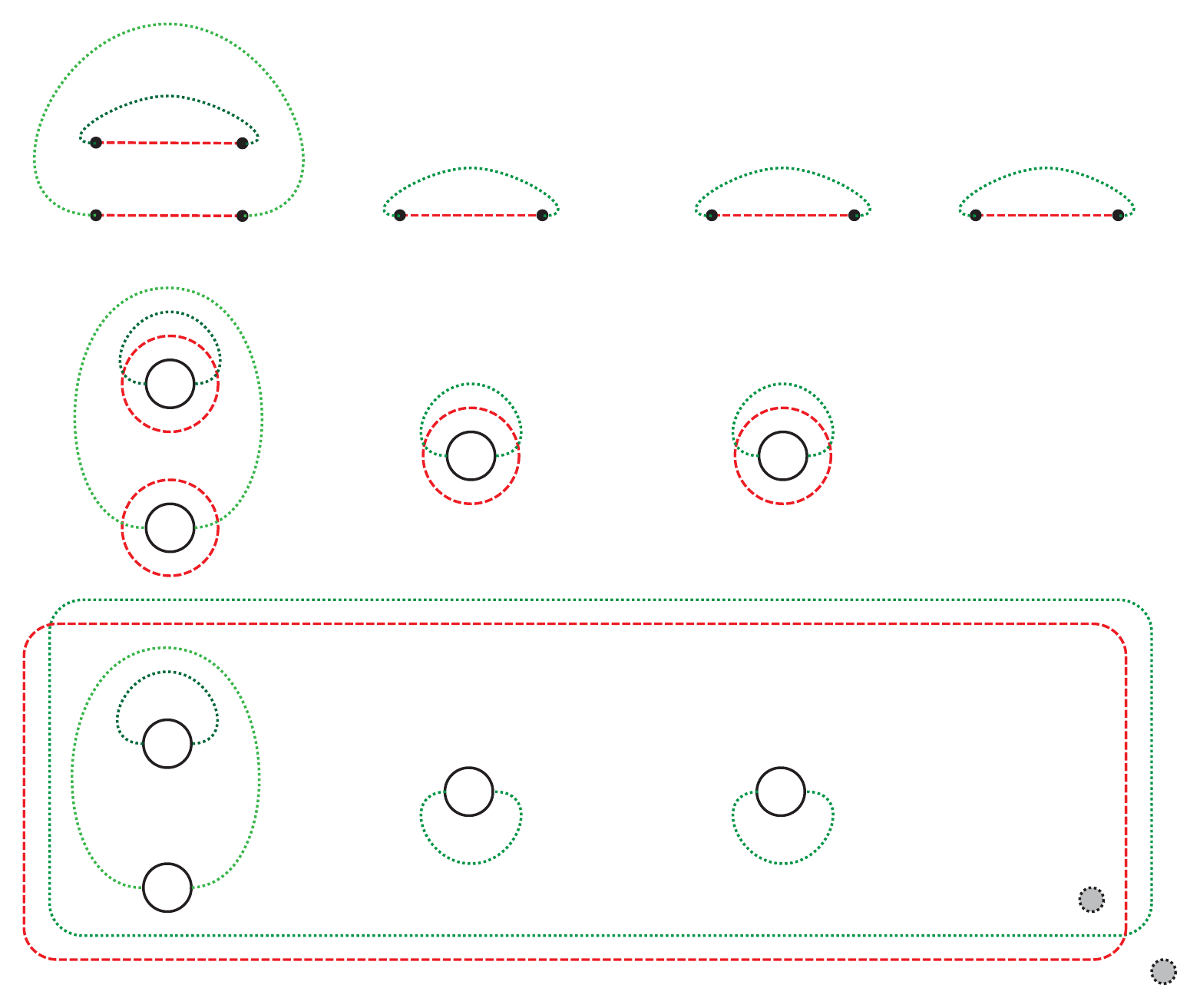}
      \put(13,51){$A$}
      \put(13,22.5){\rotatebox{180}{\reflectbox{$A$}}}
      \put(13,39){$B$}
      \put(13,10.5){\rotatebox{180}{\reflectbox{$B$}}}
      \put(38,45){$C$}
      \put(38,19){\rotatebox{180}{\reflectbox{$C$}}}
      \put(64,45){$Z$}
      \put(64,19){\rotatebox{180}{\reflectbox{$Z$}}}
      \put(50,18){$\cdots$}
      \put(50,45){$\cdots$}
      \put(50,66){$\cdots$}
      \put(13,63.5){\textcolor{red}{$A_1$}}
      \put(13,69.5){\textcolor{red}{$A_2$}}
      \put(38,63.5){\textcolor{red}{$A_3$}}
      \put(64,63.5){\textcolor{red}{$A_{n-1}$}}
      \put(86,63.5){\textcolor{red}{$A_n$}}
      \put(13,83){\textcolor{green}{$\wt{A}_1$}}
      \put(13,76.5){\textcolor{green}{$\wt{A}_2$}}
      \put(38,70.5){\textcolor{green}{$\wt{A}_3$}}
      \put(64,70.5){\textcolor{green}{$\wt{A}_{n-1}$}}
      \put(86,70.5){\textcolor{green}{$\wt{A}_n$}}
      \put(98,1){$\infty$}
      \put(87.25,10.25){$\infty$}
    \end{overpic}	
    \caption{\textbf{Equivariant handleslide invariance.} Top: Up to 
      diffeomorphism of $\CC$, the arcs $A$ and $\wt{A}$ in the proof of 
      handleslide invariance are as shown. Bottom: the branched double covers of 
      the pieces at the top, branched along the $p_i$. The small gray circles 
      indicate the two non-compact regions.\label{fig:sss-handleslides}}
  \end{figure}
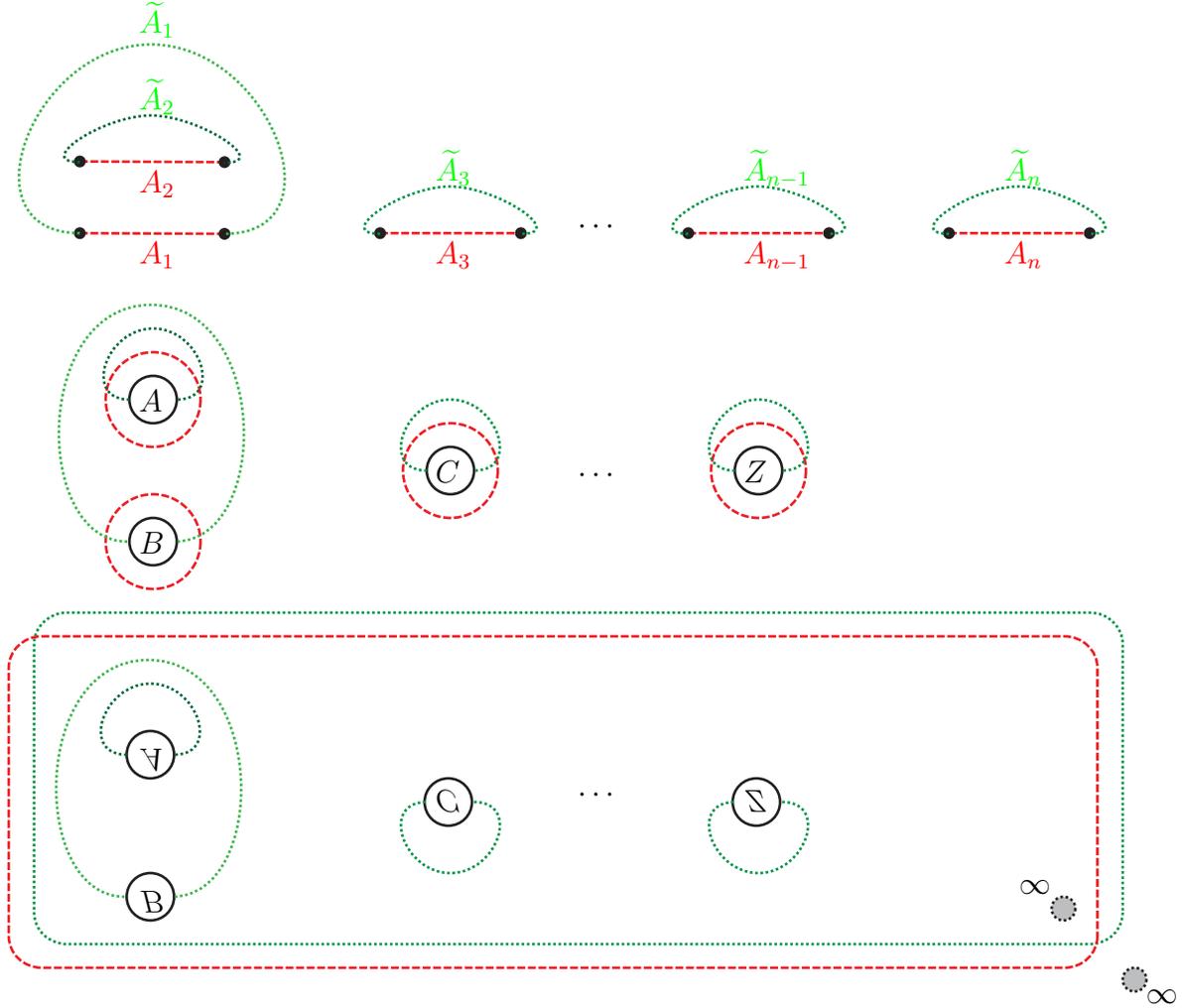  
  
  For a generic $D_{2^m}$-equivariant almost complex structure, the
  moduli spaces of holomorphic disks not entirely contained in the
  fixed set of some $g\neq 1\in D_{2^m}$ are transversally cut out;
  compare~\cite[Section 5c]{KhS02:BraidGpAction}. The components of
  the fixed sets of $\sigma^i$ intersecting $\CipLag'_A$ and
  $\CipLag'_{\wt{A}}$ are discrete, and hence contain no non-constant
  holomorphic disks. The fixed set of $\sigma^i\tau$ is the symmetric
  product of a Heegaard diagram for $\Sigma(K)$ with an extra $\alpha$
  and $\beta$ circle, as shown in the bottom of
  Figure~\ref{fig:sss-handleslides}. (More precisely, the fixed set is
  the complement of the anti-diagonal inside this symmetric product,
  since we are considering $\ssspace{n}$, not all of $\Hilb^n(S)$.)
  Positivity of domains implies that there are no holomorphic curves
  in this fixed set with negative Maslov index in $\ssspace{n}$ (or in
  $\Sym^n(\Sigma)$). For a generic perturbation of the $A_i$ and
  $\wt{A}_i$, the index $0$ domains in the Heegaard diagram have no
  holomorphic representatives (by boundary
  injectivity~\cite[Proposition 3.9]{OS04:HolomorphicDisks}), so in
  particular those moduli spaces are transversally cut out at as
  well. There are no homotopy classes of disks with index $1$. So, the
  hypotheses of Proposition~\ref{prop:non-equi-invar-gen-group} are
  satisfied, and handleslide invariance follows.
  
  Finally, for stabilization invariance, the proof described in
  Section~\ref{sec:ssseq-stab-invariance} never breaks the $D_{2^m}$
  symmetry, so the isomorphisms described there (for large $T$) give
  isomorphisms of freed Floer complexes. (Note that, in choosing the
  families of almost complex structures in the definition of the freed
  Floer complex, we may assume the families are constant over the
  region $D$ from Section~\ref{sec:ssseq-stab-invariance} and still
  achieve transversality.)
\end{proof}

\begin{proof}[Proof of Corollary~\ref{cor:SSss-invt}]
  This follows from Theorem~\ref{thm:KC-equi-invt}; see also the proof
  of Corollary~\ref{cor:Hen-dcov-invt}.
\end{proof}

\begin{proof}[Proof of Corollary~\ref{cor:SS-periodic-ss-invt}]
  The proof is similar to the proof of Theorem~\ref{thm:KC-equi-invt},
  but now working equivariantly with respect to a symmetry on the
  bridge diagram, rather than the $O(2)$ symmetry on $S$. Any two
  periodic bridge diagrams for the same knot and period can be related
  by a sequence of equivariant isotopies, pairs of handleslides, and
  pairs of stabilizations: find a sequence of bridge moves relating
  the two quotient knots and pull back the diagrams. Invariance under
  equivariant isotopies is immediate from
  Proposition~\ref{prop:invariance-gen-group}. Invariance under
  handleslides follows from
  Proposition~\ref{prop:non-equi-invar-gen-group} just as in the proof
  of Theorem~\ref{thm:KC-equi-invt}; the fact that we are doing pairs
  of handleslides at once leads to only notational
  changes. Equivariant stabilization invariance follows from the
  argument outlined in Section~\ref{sec:ssseq-stab-invariance}. Again, the fact
  that we are performing pairs of stabilizations simultaneously does
  not lead to additional complications: the neck stretching argument
  happens in two separate regions of the diagram.
\end{proof}

\subsection{Invariance of reduced symplectic Khovanov  homology}\label{sec:rKh}
As a last theorem of the paper, we fill a small gap in the literature
not related to equivariant Floer homology. As discussed in
Section~\ref{sec:ssss-review}, Manolescu defined a reduced symplectic
Khovanov homology, by deleting one $A_i$-arc and one $B_i$-arc and
working in the space $\rssspace{n}$. It seems that invariance of the
reduced symplectic Khovanov homology has not been verified previously.

In this section, we work with coefficients in $\ZZ$, rather than just $\ZZ/2$.

\begin{theorem}\label{thm:rKh-invt}
 The reduced symplectic Khovanov homology $\rKhSymp(K,p)$ is an invariant of the based link $(K,p)$.
\end{theorem}
\begin{proof}
 We must check invariance under:
 \begin{enumerate}[label=(\arabic*), ref=\arabic*]
 \item Regular homotopies of the bridge diagram.
 \item\label{item:good-handleslide} Handleslides of the $A_i$ and $B_i$, where we do not handleslide over the deleted curves $A_n$ or $B_n$.
 \item\label{item:bad-handleslide} Handleslides of an $A_i$ across $A_n$ or of a $B_i$ across $B_n$.
 \item Stabilizations of the bridge diagram.
 \end{enumerate}
 
 As usual, invariance under isotopies is easy: isotopies of the $A_i$
 and $B_i$ give isotopies of the Lagrangians $\rCipLag_A$ and
 $\rCipLag_B$.
 
 Handlesliding $A_n$ or $B_n$ over another curve has no effect at all
 on the Lagrangians. So, suppose that $A_i$ and $A'_i$ are related by
 a handleslide over $A_j$, where $i,j<n$: the case of handleslides
 among $B_1,\dots,B_{n-1}$ is symmetric. Choose $A'_i$ so that it
 intersects $A_i$ only at its endpoints (and transversally) and is
 disjoint from $A_k$ for $k\neq i$. For $k\neq i$ let $A'_k$ be a
 small pushoff of $A_k$ intersecting $A_k$ only at its endpoints (and
 transversally).  Let
 $\rCipLag'_A=\Sigma_{A'_1}\times\cdots\times\Sigma_{A'_i}\times\cdots\times
 \Sigma_{A'_{n-1}}$. (See the top row of Figure~\ref{fig:sss-handleslides}.) To prove that
 $\HF(\rCipLag_A,\rCipLag_B)\cong\HF(\rCipLag'_A,\rCipLag_B)$ one
 could verify that the construction of Seidel-Smith's Hamiltonian
 isotopy~\cite[Lemma 48]{SeidelSmith6:Kh-symp} restricts to give a
 Hamiltonian isotopy from $\rCipLag_A$ to $\rCipLag'_A$. Instead,
 however, we will adapt Ozsv\'ath-Szab\'o's handleslide invariance
 proof for Heegaard Floer homology~\cite[Section
 9]{OS04:HolomorphicDisks}, which boils down to showing that
 $\rCipLag_A$ and $\rCipLag'_A$ are isomorphic objects of the Fukaya
 category.

 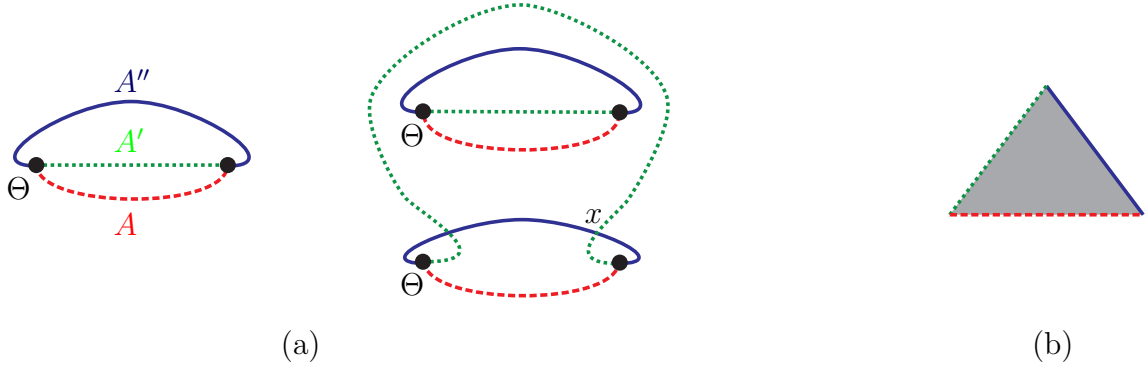
\begin{figure}
   \centering
   \begin{overpic}[tics=10, scale=1.35]{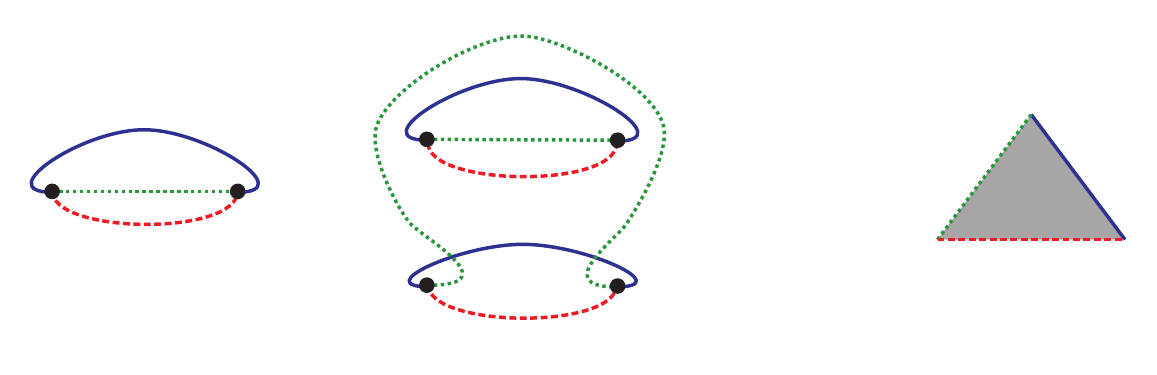}
     \put(25,0){(a)}
     \put(88,0){(b)}
     \put(11,10){\textcolor{red}{$A$}}
     \put(11,17){\textcolor{green}{$A'$}}
     \put(11,22){\textcolor{darkblue}{$A''$}}
     \put(2,13){$\Theta$}
     \put(35,5){$\Theta$}
     \put(50.5,11){$x$}
     \put(35,17.5){$\Theta$}
   \end{overpic}
   \caption{\textbf{A triple diagram for a handleslide.} (a) Part of the
     triple $(A,A',A'')$. $A$ is dashed (red), $A'$ is dotted (green),
     and $A''$ is solid (blue). The $\Theta$s denote $\Theta_{A,A'}$
     or $\Theta_{A,A''}$. (b) Orientation convention for
     triangle maps.}
   \label{fig:triangle-domains}
 \end{figure}

As a relatively graded abelian group, $\CF(\rCipLag_A,\rCipLag'_A)$ is
isomorphic to $H_*(S^2)^{\otimes (n-1)}$. In particular, the
differential on $\CF(\rCipLag_A,\rCipLag'_A)$ must vanish for grading
reasons. For $k\neq i$ let $A''_k$ be a small deformation of
$A_k$, intersecting each of $A_k$ and $A'_k$ only in the endpoints,
and so that $A'_k$ is between $A_k$ and $A''_k$. Choose
$A''_i$ similarly, but here $A''_i$ necessarily intersects $A'_i$ in
two interior points (as well as both endpoints); see Figure~\ref{fig:triangle-domains}.
Let $\rCipLag''_A$ be the corresponding Lagrangian.  The
top-graded generator $\Theta_{A,A''}\in\CF(\rCipLag_A,\rCipLag''_A)$
represents the identity map of $\rCipLag_A$ in the Fukaya
category. We
will show that if $\Theta_{A,A'}\in\CF(\rCipLag_A,\rCipLag'_A)$ is the
top-graded generator and
$\Theta_{A',A''}\in\CF(\rCipLag'_A,\rCipLag''_A)$ is a chain
representing the top-graded homology class then the composition (via triangle maps)
$\Theta_{A',A''}\circ \Theta_{A,A'}=\Theta_{A,A''}$. A symmetric
argument shows that 
the composition the other way is the
identity element of $\rCipLag'_A$, implying that $\Theta_{A,A'}$ and
$\Theta_{A',A''}$ are the desired isomorphisms between $\rCipLag_A$
and $\rCipLag'_A$. 

To see that
$\Theta_{A',A''}\circ \Theta_{A,A'}=\Theta_{A,A''}$, arrange $A$,
$A'$, and $A''$ as in Figure~\ref{fig:triangle-domains}.  In
particular, $\Theta_{A,A'}$ and $\Theta_{A,A''}$ are both the tuple of
left endpoints of the $n$ $A_k$-arcs. The generator $\Theta_{A',A''}$
is represented by a cycle which is the sum of this tuple and some
other tuples, each of which has at least one point projecting to the
intersection marked $x$ in the figure. (See the proof of
Lemma~\ref{lem:ss-fixed-pts-subcx} for a discussion of computing the
gradings of generators.)

The constant triangles at $\Theta_{A,A'}=\Theta_{A,A''}$ are, in fact,
transversely cut out and give one contribution of $\Theta_{A,A''}$ in
$\Theta_{A',A''}\circ \Theta_{A,A'}$. (One way to see this is to note
that the local picture near $\Theta_{A,A'}$ is the same as if $A'$
were a small Hamiltonian-isotopic copy of $A$, and the constant triangles
contribute in the case of a small Hamiltonian isotopy. Alternately,
lift $A$, $A'$, and $A''$ to the branched double cover, perturb away
the triple point, and note that there is a unique small triangle
connecting the three generators.) It remains to see that no other
triangles cancel these. However, we can read off the (projections to
$\CC$ of) corners of
holomorphic triangles from the projected domain, and there is no other
projected domain with non-negative multiplicities and corners at
$\Theta_{A,A'}$, $\Theta_{A,A''}$, and points in
$\Theta_{A',A''}$. Thus, $\Theta_{A',A''}\circ
\Theta_{A,A'}=\Theta_{A,A''}$, as desired.

The same argument works for handlesliding $A_i$ over $A_n$.

 Turning to stabilization invariance, recall that stabilizations of
 the arc diagram occur at one of the $p_i$. Since we have already
 verified isotopy and handleslide invariance, it suffices to
 prove stabilization invariance at one $p_i$ for each link
 component. So, if the link component containing $p_{2n}$ involves at
 least $2$ $A_i$-arcs then it suffices to prove stabilization at a
 $p_i$ which is not an endpoint of $A_n$ or $B_n$. In this case, we
 can perform a sequence of handleslides and isotopies so that
 the $p_i$ is adjacent to the non-compact region in $\RR^2$, and then
 the neck-stretching argument from Section~\ref{sec:ssseq-invariance}
 implies stabilization invariance. 

 \begin{figure}
   \centering
   \includegraphics{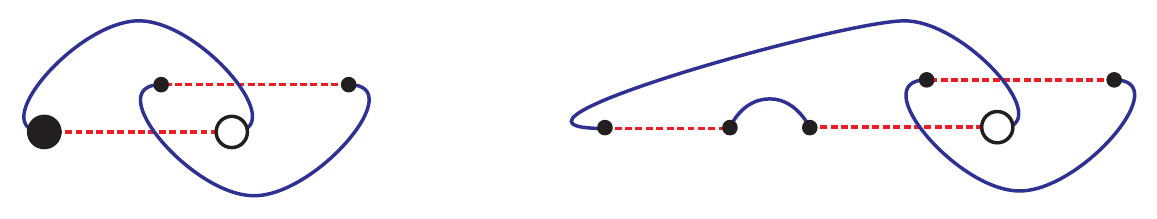}
   \caption{\textbf{Stabilizing when the basepoint is on a one-bridge
       unknot.} The case of the Hopf link is shown, but the un-marked
     components could, of course, be more complicated. The key
     features are that the marked component is a one-bridge unknot and
     the point $p_{2n-1}$, which is drawn as a large black dot, is
     adjacent to the non-compact region in $\CC$. The basepoint $p_{2n}$
     is drawn as a large white dot in both pictures.}
   \label{fig:bad-stab}
 \end{figure}

 The remaining case is that $A_n$ and $B_n$ together form an unknot
 component. Let $p_{2n-1}$ and $p_{2n}$ denote the endpoints of $A_n$
 (and $B_n$); $p_{2n}$ is the basepoint, and we will stabilize at
 $p_{2n-1}$. In this case, perform a sequence of handleslides and
 isotopies so that $p_{2n-1}$ is adjacent to the non-compact
 region and $A_n$ and $B_n$ only intersect at their endpoints (as in
 Figure~\ref{fig:bad-stab}). In the stabilized diagram, the unknot component has two
 $A_i$-arcs and two $B_i$-arcs. Of these four arcs, the two that do
 not contain the basepoint $p_{2n}$ are adjacent to the non-compact
 region. Therefore it easily follows that the chain complex after the
 stabilization is chain isomorphic to the chain complex before the
 stabilization. See Figure~\ref{fig:bad-stab}.
\end{proof}


\begin{remark}
  Abouzaid-Smith's
  techniques~\cite{AbouzaidSmith:arc-alg,AbouzaidSmith:KhSympKh}
  presumably also lead to a proof of Theorem~\ref{thm:rKh-invt}.
\end{remark}

\subsection{Some conjectures}\label{sec:ssss-conj}

We start by recalling some earlier conjectures. First, Seidel-Smith
conjectured a relationship between symplectic Khovanov homology and
combinatorial Khovanov homology:
\begin{conjecture}\cite[Conjecture 2]{SeidelSmith6:Kh-symp}
  $\Kh^{\mathit{symp},\ell}(L)\cong\oplus_{i-j=\ell}\Kh^{i,j}(L)$.
\end{conjecture}
(The Khovanov differential has bigrading $(1,0)$.) 
Abouzaid-Smith recently proved this conjecture in characteristic
$0$~\cite{AbouzaidSmith:KhSympKh}; at the time of writing, the positive
characteristic case remains open.

Perhaps inspired by Lee's deformation of the Khovanov
complex\cite{Lee05:Khovanov}, Bar-Natan
constructed~\cite{BarNatan05:Kh-tangle-cob} a bigraded chain complex
$\BNCx(L)$ freely generated over $\Field[H]$, with the differential in
bigrading $(1,0)$ and the formal variable $H$ in bigrading $(0,-2)$,
satisfying the following: the chain homotopy type of $\BNCx(L)$ over
$\Field[H]$ is a link invariant; $\BNCx(L)/\{H=0\}=\KhCx(L)$; and the
homology of the localized complex $H^{-1}\BNCx(L)$ is $2^{|L|}$ copies
of $\Field[H,H^{-1}]$. On the symplectic side, if we consider the
action of $\ZZ/2\subset \SO(2)$ on $\ssspace{n}$
then by Theorem~\ref{thm:khsymp-loc},
$\alpha^{-1}\eHF(\CipLag'_A,\CipLag_B)$ is isomorphic to
$2^{|L|}$ copies of $\Field[\alpha,\alpha^{-1}]$. Therefore, the
following seems to be a reasonable conjecture.
\begin{conjecture}
  There is an ungraded isomorphism of $\Field[H]\cong\Field[\alpha]$-modules
  \[
  \eHF(\CipLag'_A,\CipLag_B)\cong H_*(\BNCx(L)).
  \]
\end{conjecture}

\begin{question}
  Is there a refinement of the above conjecture relating the single
  grading of $\eHF(\CipLag'_A,\CipLag_B)$ and the
  bigrading of $\BNCx(L)$?
\end{question}

On a parallel front, Szab\'o constructed a bigraded chain complex
$\SzCx(L)$ freely generated over $\Field[W]$, with the differential in
bigrading $(1,0)$ and the formal variable $W$ in bigrading $(-1,-2)$,
so that $\KhCx(L)=\SzCx(L)/(W=0)$, and the chain homotopy type of
$\SzCx(L)$ over $\Field[W]$ is a link
invariant~\cite{Szabo:geometric-ss}. Moreover, he conjectured that the
chain complex $\SzCx(L)$ should be chain homotopy equivalent to a
similar looking chain complex constructed earlier by
Ozsv\'ath-Szab\'o~\cite{BrDCov} via holomorphic curves. On the
symplectic side, Seidel-Smith used the $\ZZ/2=\{1,\tau\}$ action
coming from $\tau(u,v)=(u,-v)$ (Equation~\eqref{eq:dm-action-on-uv})
to construct the $\ZZ/2$-equivariant symplectic Khovanov cohomology
$\HF_{\{1,\tau\}}(\CipLag'_A,\CipLag_B)$ as a module over $H^*(\ZZ/2)$
($=\Field[W]$ say), and propose~\cite[Section 4.5]{SeidelSmith10:localization} that the resulting spectral
sequence from $\KhSymp(L)$ may be isomorphic to the Ozsv\'ath-Szab\'o
spectral sequence
$\Kh(L)\Rightarrow \HFa\bigl(\Sigma(m(L))\#(S^2\times S^1)\bigr)$~\cite{BrDCov}. Combining the two
conjectures, we get:
\begin{conjecture}[Seidel-Smith-Szab\'o]
  As modules over $\Field[W]$,
  $\HF_{\{1,\tau\}}(\CipLag'_A,\CipLag_B)$ is isomorphic to the
  homology of $\SzCx(L)$ after collapsing the bigrading $(i,j)$ into a
  single grading $i-j$.
\end{conjecture}

Finally, Sarkar-Seed-Szab\'o combined $\BNCx(L)$ and $\SzCx(L)$ into
a single bigraded chain complex $\totCx(L)$, freely generated over a
polynomial ring $\Field[H,W]$, with the differential in
bigrading $(1,0)$, and the formal variables $H$ and $W$ in bigradings
$(0,-2)$ and $(-1,-2)$, respectively~\cite{SSS:spectrals}. The construction satisfies the
following properties (listed in \cite[Section~1]{SSS:spectrals}):
\begin{itemize}[leftmargin=*]
\item The chain homotopy type of $\totCx(L)$ over $\Field[H,W]$ is a
  link invariant.
\item $\BNCx(L)=\totCx(L)/(W=0)$ and $\SzCx(L)=\totCx(L)/(H=0)$;
  in particular, $\KhCx(L)=\totCx(L)/(H=W=0)$.
\item The homology of $H^{-1}\mathcal{C}_{\mathit{tot}}$ is
  $2^{|L|}$ copies of $\Field[H,H^{-1},W]$.  
\end{itemize}
On the symplectic side, there is an $O(2)$-action, and we may restrict
to dihedral group $D_{2^m}$-action, as described in
Equation~\eqref{eq:dm-action-on-uv}; in particular, we can look at the
$\ZZ/2\times\ZZ/2=\{1,\sigma,\tau,\sigma\tau\}$ action combining the
above two actions. 
\begin{question}
  Is the $D_{2^m}$-equivariant symplectic Khovanov cohomology
  $\HF_{D_{2^m}}(\CipLag'_A,\CipLag_B)$ (in particular, the
  $\ZZ/2\times\ZZ/2$-equivariant Floer cohomology) related to
  $\totCx(L)$? If so, how are the module structures (over
  $H^*(D_{2^m})$ and $\Field[H,W]$, respectively) related? How do
  the single grading on the symplectic side and the bigrading on the
  combinatorial side interact?
\end{question}


\appendix
\section{Instructive Examples}\label{app:examples}
The examples in this section are intended to help the reader
unfamiliar with homotopy coherence to come to terms with the notions
used in this paper. Vogt's original paper~\cite{Vogt73:hocolim} and
the Kamps-Porter's book~\cite{KampsPorter97:abstract} are also
pleasantly written, and the reader is encouraged to consult them for
more information.

We start with examples illustrating homotopy coherent diagrams for almost
complex structures and complexes,
Definitions~\ref{def:ho-coh-diag-strs} and~\ref{def:ho-coh-diag-cxs},
respectively.

\begin{example}\label{eg:coherent-diag}
  The first few terms of a $\ECat\ZZ/2$-diagram in $\ol{\JSpace}$ consist of:
	  \begin{enumerate}
  \item Two cylindrical almost complex structures $F(a)= J_{a}$ and $F(b)=J_{b}$.
  \item Two (equivalence classes of sequences of) eventually cylindrical complex structures $F(\alpha) = \wt{J}_{\alpha} \in \ol{\JSpace}(J_a,J_b)$, and $F(\beta) = \wt{J}_{\beta} \in \ol{\JSpace}(J_b, J_a)$, and two others associated to the identity maps and equal to $J_{a}$ and $J_{b}$.
  \item Two smooth families of eventually cylindrical complex structures indexed by $t \in [0,1]$:
  	\begin{align*}
  	F(\alpha, \beta)&= \widetilde{J}_{(\alpha, \beta, t)} \in \ol{\JSpace}(J_b, J_b) & 
  	F(\beta, \alpha) &= \widetilde{J}_{(\beta, \alpha, t)} \in \ol{\JSpace}(J_a, J_a)\\
        \shortintertext{which satisfy:}
	\wt{J}_{(\alpha, \beta, 1)} &= \wt{J}_{\alpha \circ \beta} = \wt{J}_{\Id_b} = J_b &
	\wt{J}_{(\alpha, \beta, 0)} &= \wt{J}_{\alpha} \circ \wt{J}_{\beta} = (\wt{J}_{\beta} , \wt{J}_{\alpha})\\
	\wt{J}_{(\beta, \alpha, 1)} &= \wt{J}_{\beta \circ \alpha} = \wt{J}_{\Id_a} = J_a &
	\wt{J}_{(\beta, \alpha, 0)} &= \wt{J}_{\beta} \circ \wt{J}_{\alpha}=(\wt{J}_{\alpha} , \wt{J}_{\beta}).
  	\end{align*}
        and six other $[0,1]$-families of almost complex structures determined by the data of the first two steps.	

\item Two smooth families of cylindrical-at-infinity complex structures indexed by $(t_1,t_2) \in [0,1]^2$:
	\begin{align*}
	F(\alpha, \beta, \alpha) &= \wt{J}_{(\alpha, \beta, \alpha, t_1, t_2)} \in \ol{\JSpace}(J_a, J_a) &
	F(\beta, \alpha, \beta) &= \wt{J}_{(\beta, \alpha, \beta, t_1, t_2)} \in \ol{\JSpace}(J_b, J_b)\\
        \shortintertext{which satisfy:}
	\wt{J}_{(\alpha, \beta, \alpha, 1,t_2)} &= \wt{J}_{(\alpha, \beta \circ \alpha, t_2)} = \wt{J}_{(\alpha, \Id_a, t_2)} = \wt{J}_{\alpha} &
	\wt{J}_{(\alpha, \beta, \alpha, t_1, 1)} &= \wt{J}_{(\alpha \circ \beta, \alpha, t_1)} = \wt{J}_{(\Id_b, \alpha, t_1)} = \wt{J}_{\alpha}\\
	\wt{J}_{(\alpha, \beta, \alpha, 0,t_2)} &= (\wt{J}_\alpha, \wt{J}_{(\alpha,\beta,t_2)}) &
	\wt{J}_{(\alpha, \beta, \alpha, t_1, 0)} &= (\wt{J}_{(\beta, \alpha, t_2)}, \wt{J}_\alpha) \\
	\wt{J}_{(\beta, \alpha, \beta, 1, t_2)} &= \wt{J}_{(\beta, \alpha \circ \beta, t_2)} = \wt{J}_{(\beta, \Id_b, t_2)} = \wt{J}_{\beta} &
	\wt{J}_{(\beta, \alpha, \beta, t_1, 1)}&=\wt{J}_{(\beta \circ \alpha, \beta, t_1)} = \wt{J}_{(\Id_a, \beta, t_1)} = \wt{J}_{\beta}\\
	\wt{J}_{(\beta, \alpha, \beta, 0, t_2)} &= (\wt{J}_{\beta}, \wt{J}_{(\beta, \alpha, t_2)}) &
	\wt{J}_{(\beta, \alpha, \beta, t_1, 0)} &= (\wt{J}_{(\alpha, \beta, t_2)}, \wt{J}_{\beta})
	\end{align*}
	and fourteen other $[0,1]^2$-families of almost complex structures determined by the data of the first three steps.
	\end{enumerate}
\end{example}

\begin{example}\label{eg:coherent-diag-complexes}
  The first few terms of a homotopy coherent $\ECat\ZZ/2$-diagram in $\Complexes$ consist of:
	\begin{enumerate} 
	\item Two chain complexes $G(a)$ and $G(b)$ over $\Field$.
	\item Two chain maps $G_{\alpha} \co G(a) \rightarrow G(b)$ and $G_{\beta}\co G(b) \rightarrow G(a).$
	\item Two maps
	\begin{align*}
	G_{\beta, \alpha}\co G(a) \rightarrow G(a) \ \ \ \ & \ \ \
	G_{\alpha, \beta}\co G(b) \rightarrow G(b)
  \shortintertext{which satisfy:}
  \partial \circ G_{\beta, \alpha} = G_{\beta, \alpha}\circ \partial + G_{\beta} \circ G_{\alpha} +\Id_{G(a)} \ \ \ \ & \ \ \
  \partial \circ G_{\alpha, \beta} = G_{\alpha, \beta}\circ \partial + G_{\alpha} \circ G_{\beta} + \Id_{G(b)}.\end{align*}

\item Two maps 
\begin{align*}
	G_{\alpha, \beta, \alpha} \co G(a) \rightarrow G(a) \ \ \ \ \ \
	G_{\beta, \alpha, \beta} \co G(b) \rightarrow G(b)
  \shortintertext{which satisfy:}
	\partial \circ G_{\alpha, \beta, \alpha} = G_{\alpha, \beta, \alpha} \circ \partial + G_{\alpha,\beta}\circ G_{\alpha} + G_{\alpha} \circ G_{\beta, \alpha}\\
	\partial \circ G_{\beta, \alpha, \beta} = G_{\beta, \alpha, \beta} \circ \partial + G_{\beta, \alpha}\circ G_{\beta} + G_{\beta} \circ G_{\alpha,\beta}.\end{align*}
\end{enumerate}  
\end{example}

The following example illustrates homotopy colimits,
Definition~\ref{def:hocolim}. In particular, we observe that a notion
of iterated mapping cones defined using homotopy colimits agrees with
a more direct definition in terms of complexes.
\begin{example}\label{eg:hocolim}
Let $\Cat$ be the category
\[
\Cat=\vcenter{\hbox{\xymatrix{\bullet\ar[d]_{f_2} & \bullet \ar[l]_{f_1}\ar[d]^{f_3}\\
  \bullet & \bullet\ar[l]^{f_4}}}}\big/(f_2 \circ f_1 = f_4 \circ f_3).
\]
A homotopy coherent $\Cat$-diagram $G\co\Cat\to\Complexes$
induces a diagram
\begin{equation}
\vcenter{\hbox{\xymatrix@=4em{
  C_{01}\ar[d]_{G(f_2)} & C_{11}\ar[l]_{G(f_1)}\ar[d]^{G(f_3)}\ar[dl]\\
  C_{00} & C_{10}\ar[l]^{G(f_4)}
}}}
\label{square}
\end{equation}
where the diagonal map is $G_{f_2,f_1} + G_{f_4,f_3}$. Since $G_{f_2,f_1}$
is a homotopy between $G(f_2)\circ G(f_1)$ and $G(f_2 \circ f_1)$ and
$G_{f_4, f_3}$ is a homotopy between $G(f_4 \circ f_3)$ and $G(f_4)
\circ G(f_3)$, the diagonal map is a homotopy between $G(f_2) \circ
G(f_1)$ and $G(f_4) \circ G(f_3)$.

Let $\Cat_+$ be the result of adding a single object $*$ to $\Cat$ and
a morphism from each non-terminal vertex of $\Cat$ to $0$. The
homotopy coherent diagram $G$ extends uniquely to a homotopy coherent
diagram $G_+\co \Cat_+\to \Complexes$ with $G_+(*)=0$ the zero
complex. We label the arrows as:
\begin{equation}
\label{squarecomplex3}
\vcenter{\hbox{\xymatrix@=2em{
  C_{01}\ar[dd]_{G(f_2)} \ar@/^4pc/[rrrrd]^{G(f_6)}& & C_{11}\ar[ll]_{G(f_1)}\ar[dd]^{G(f_3)}\ar[ddll]\ar[drr]^{G(f_5)}\\
  & & & & 0\\
  C_{00} & & C_{10}\ar[ll]^{G(f_4)} \ar[urr]_{G(f_7)} & & \ \ \ .
}}}
\end{equation}  

We will show that the homotopy colimit of $G_+$ is quasi-isomorphic to the
total mapping cone of the diagram~(\ref{square}).
  
First, let us consider the general form of some basic differentials in
the homotopy colimit. Let $f:a\rightarrow b$ be a morphism in the
original category $\Cat$. Let $x \in G(a)$. Then
\begin{align*}
\partial(G(f);\{0,1\};x)&= (G(f);\{0,1\}; \partial x) + (G(f);\{1\};x) + (G(f);\{0\};x)  \\
  			  &= (G(f); \{0,1\}; \partial x) + 	x + G(f)(x).
\end{align*}  
Note that if $G(b)=0$ then $G(f)(x)=0$ automatically, so we have $d(G(f);\{0,1\};x)= (G(f); \{0,1\}; \partial x) + x$. Schematically, we draw the following picture:
\[
\xymatrix{
  & (G(f);G(a))\ar[dl]\ar[dr]^{G(f)} &\\
  G(a) & & G(b).
}
\]
Next, let $g:b\rightarrow c$ be another morphism composable with $f$. Then we have
\begin{align*}
\partial(G(g,f); \{0,1\}\otimes \{0,1\}; x) = (G(g,f);&  \{0,1\}\otimes \{0,1\}; \partial x)  + (G(f);  \{0,1\}; x) \\&+ (G(g\circ f); \{0,1\}; x)+ (G(g); \{0,1\}; G(f)(x)) \\ &+ G_{g,f}(x).
\end{align*}  
Notice that if $G(b)=0$, then $G(g,f)(x)=0$ and the last term
vanishes. (However, importantly, $(G(g\circ f); \{0,1\}; x)$ may be nonzero,  even though $g \circ f$ is a zero map.) Schematically we draw the following
picture:
\[
\xymatrix{
  & (G(g,f);G(a))\ar[dl]\ar[d]^{G_{g,f}}\ar[dr]^{G(f)} & \\
  (G(f);G(a)) & G(b) & (G(g);G(b)).
}
\]
  
Now let us look at the homotopy colimit of the homotopy-coherent diagram (\ref{squarecomplex3}). As a vector space, it is the direct sum of the following sixteen terms.

\begin{itemize}
\item The four chain complexes $C_{11}, C_{01}, C_{10}$, and $C_{00}$.
\item Terms corresponding to the eight morphisms in the diagram, to wit,
	\begin{align*}
	&(G(f_1); \{0,1\}; C_{11}) & 
	&(G(f_2); \{0,1\}; C_{01}) & 
	&(G(f_3); \{0,1\}; C_{11}) &
	&(G(f_4); \{0,1\}; C_{10}) \\
	&(G(f_5); \{0,1\}; C_{11}) &
	&(G(f_6); \{0,1\}; C_{01}) &
	&(G(f_7); \{0,1\}; C_{10}) &
	&(G(f_2,f_1); \{0,1\}; C_{11}).
	\end{align*}
\item Four terms corresponding to pairs of composable morphisms, to wit,
	\begin{align*}
	&(G(f_2; f_1); \{0,1\}\otimes \{0,1\}; C_{11}) &
	&(G(f_4; f_3); \{0,1\}\otimes \{0,1\}; C_{11}) \\
	&(G(f_6; f_1); \{0,1\}\otimes \{0,1\}; C_{11}) &
	&(G(f_7; f_3); \{0,1\}\otimes \{0,1\}; C_{11}).
	\end{align*}
\end{itemize}

From here on, no terms in $I_*^{\otimes n}$ will appear other than $\{0,1\}$ and $\{0,1\}\otimes \{0,1\}$, and which is appropriate will be clear from context. So, in the interests of readability we will omit these terms. The differentials on the homotopy colimit are as follows:
\begin{align*}
&\partial(G(f_2,f_1); x) = (G(f_2,f_1);\partial x) + (G(f_1);x) + (G(f_2 \circ f_1); x) + (G(f_2);G(f_1)x) + G_{f_2,f_1}x \\
&\partial(G(f_4,f_3);x) = (G(f_4,f_3);\partial x) + (G(f_3);x) + (G(f_2 \circ f_1);x) + (G(f_4); G(f_3)x) + G_{f_4,f_3}x \\
&\partial(G(f_6,f_1);) = (G(f_6,f_1); \partial x) + (G(f_5); x) + (G(f_1); x) + (G(f_6); G(f_1)x) \\
&\partial(G(f_7,f_3);x) = (G(f_7, f_3; \partial x) + (G(f_5);x)+ (G(f_3); x) + (G(f_6);G(f_3)x)\\ 
&\partial (G(f_2 \circ f_1); x) = (G(f_2 \circ f_1);\partial x) + x + G(f_2\circ f_1)x \\
&\partial (G(f_1); x) = (G(f_1); \partial x) + x + G(f_1)x \\
&\partial (G(f_2); x) = (G(f_2);\partial x) + x + G(f_2)x \\
&\partial (G(f_3); x) = (G(f_3);\partial x) + x + G(f_3)x \\
&\partial (G(f_4);x) = (G(f_4);\partial x) + x + G(f_4)x \\
&\partial (G(f_5);x) = (G(f_5);\partial x) + x \\
&\partial (G(f_6);x) = (G(f_6); \partial x) +x \\
&\partial (G(f_7;x) = (G(f_7); \partial x) + x 
\end{align*}  
  
We will break the total complex into seven subcomplexes, six of which are acyclic and one of which is the mapping cone. The six trivial subcomplexes are split off by edge reductions along blank arrows in the homotopy colimit expressed using the diagrammatic conventions explained above; the interested reader is encouraged to draw the diagram and follow along.

Our first subcomplex $C_1$ has a basis consisting of all elements 
\begin{align*}
&a^1_x = (G(f_6,f_1); x) \\ 
&b^1_x =(G(f_5); x) + G(f_1); x) + (G(f_6); G(f_1)x) + (G(f_6,f_1); \partial x)
\end{align*}
for $x$ an element of $C_{11}$. Then the differential $\partial$ on this complex is $\partial(a^1_x)=b^1_x$ and $\partial(b^1_x) = 0$.

Our second subcomplex $C_2$ has a basis consisting of all elements 
\begin{align*}
&a^2_x = (G(f_7, f_3); x) + (G(f_6,f_1);x) \\
&b^2_x = (G(f_1);x) + (G(f_3);x) + (G(f_6); G(f_1)x) \\ &\ \ \ \ + (G(f_7); G(f_3)x)+ (G(f_7, f_3); \partial x) + (G(f_6,f_1); \partial x)
\end{align*} 
for $x \in C_{11}$. As previously, $\partial(a^2_x)=b^2_x$ and $\partial(b^2_x) = 0$.

Our third subcomplex $C_3$ has basis consisting of all elements
\begin{align*}
&a^3_x = (G(f_2,f_1); x) + (G(f_7, f_3); x) + (G(f_6,f_1);x) \\
&b^3_x = (G(f_3);x) + (G(f_2 \circ f_1); x) + (G(f_6);G(f_1)x) \\ &\ \ \ \ + (G(f_2);G(f_1)x)+ (G(f_7);G(f_3)x)+ G_{f_2, f_1}x  \\ &\ \ \ \ + (G(f_2,f_1); \partial x)+ (G(f_7, f_3); \partial x) + (G(f_6,f_1); \partial x)
\end{align*}
for $x \in C_{11}$. As previously, $\partial(a^3_x)=b^3_x$ and $\partial(b^3_x) = 0$.

Our fourth subcomplex $C_4$ has a basis consisting of all elements 
\begin{align*}
&a^4_x = (G(f_2 \circ f_1); x) \\
&b^4_x = x + G(f_2 \circ f_1)x + (G(f_2 \circ f_1); \partial x)
\end{align*}
for $x \in C_{11}$. As previously, $\partial(a^4_x)=b^4_x$ and $\partial(b^4_x) = 0$.

Our fifth subcomplex $C_5$ has a basis consisting of all elements 
\begin{align*}
&a^5_x = (G(f_6); x) \\
&b^5_x = x + (G(f_6); \partial x)
\end{align*}
for $x \in C_{01}$. As previously, $\partial(a^5_x)=b^5_x$ and $\partial(b^5_x) = 0$.

Our sixth subcomplex $C_6$ has a basis consisting of all elements 
\begin{align*}
&a^6_x = (G(f_7); x) \\
&b^6_x = x + (G(f_7); \partial x)
\end{align*}
for $x \in C_{10}$. As previously, $\partial(a^6_x)=b^6_x$ and $\partial(b^6_x) = 0$.

Finally we come to the subcomplex remaining after splitting all six of these acyclic complexes. $C_7$ has a basis consisting of elements
\begin{align*}
&a^7_x = (G(f_4,f_3); x) + (G(f_2,f_1); x) + (G(f_7, f_3); x) + (G(f_6,f_1);x), \ \ x \in C_{11} \\
&b^7_x = (G(f_6);x) + (G(f_2); x), \ \ x \in C_{01} \\
&c^7_x = (G(f_4);x) + (G(f_7);x), \ \ x \in C_{10} \\
&e^7_x = x, \ \ x \in C_{00}.
\end{align*}
with differentials
\begin{align*}
&\partial a^7_x = a^7_{\partial x} + b^7_{G(f_1)x} + c^7_{G(f_3)x}+e^7_{(G_{f_2,f_1}+G_{f_4,f_3})x} \\
&\partial b^7_x = b^7_{\partial x} + e^7_{G(f_2)x} \\
&c^7_x = c^7_{\partial x} + d^7_{G(f_4)x}  \\
&e^7_x = e^7_{\partial x}
\end{align*}

This is the mapping cone on (\ref{square}). Using the previous diagrammatic conventions, it is the complex shown here.  
  \[
  \begin{tikzpicture}
    \node at (0,0) (top) {$(G(f_6,f_1);C_{11})+(G(f_7,f_3);C_{11})+(G(f_2,f_1);C_{11})+(G(f_4,f_3);C_{11})$};
    \node at (-5.5,-2) (centl) {$(G(f_6);C_{01})+(G(f_2);C_{01})$};
    \node at (5.5,-2) (centr) {$(G(f_4);C_{10})+(G(f_7);C_{10})$};
    \node at (0,-4) (bot) {$C_{00}$};
    \draw[->] (top) to node[above, sloped]{\lab{G(f_1)}} (centl);
    \draw[->] (top) to node[above, sloped]{\lab{G(f_3)}} (centr);
    \draw[->] (centl) to node[below, sloped]{\lab{G(f_2)}} (bot);
    \draw[->] (centr) to node[below,sloped]{\lab{G(f_4)}} (bot);
    \draw[->] (top) to node[right]{\lab{G_{f_2,f_1}+G_{f_4,f_3}}} (bot);
  \end{tikzpicture}
  \]
\end{example}


\bibliographystyle{hamsalpha}
\bibliography{heegaardfloer}
\end{document}
